\documentclass[12pt,oneside,openany,article]{memoir}
\usepackage{mempatch}
\nouppercaseheads 
\usepackage[amsthm,thmmarks,hyperref]{ntheorem}
\usepackage[noTeX]{mmap}
\usepackage{etex}

\usepackage[english]{babel}
\usepackage[leqno]{mathtools}

\usepackage{mathrsfs}
\usepackage{upgreek}
\usepackage[compress,square,comma,numbers]{natbib}
\usepackage{hypernat} 

\usepackage{sfmath}

\usepackage{microtype}

\usepackage{subfig}
\usepackage{wrapfig}
\usepackage{array}

\usepackage{graphicx,color,xcolor}
\pagecolor{white}

\usepackage{hyperxmp}
\usepackage{xr-hyper}
\usepackage{nameref}
\usepackage[pdftex,bookmarks,pdfnewwindow,plainpages=false,unicode]{hyperref}

\usepackage{bookmark}

\usepackage{enumitem}

\usepackage{amssymb} 

\hypersetup{
colorlinks=true,
linkcolor=black,
citecolor=black,
urlcolor=blue,
pdfauthor={Victor Ivrii},
pdftitle={2D-Magnetic Schroedinger Operator: Short Loops, Pointwise  Spectral Asymptotics and Asymptotics of Dirac Energy},
pdfsubject={Sharp Spectral Asymptotics},
pdfkeywords={Microlocal Analysis,  Sharp Spectral Asymptotics, Schroedinger Operator, Strong Magnetic Field, Loops},
bookmarksdepth={4}
}

\numberwithin{equation}{chapter}

\theoremstyle{plain}
\newtheorem{theorem}{Theorem}[section]

\newtheorem{proposition}[theorem]{Proposition}
\newtheorem{corollary}[theorem]{Corollary}

\theoremstyle{definition}

\newtheorem{definition}[theorem]{Definition}
\newtheorem{Problem}[theorem]{Problem}
\theoremstyle{remark}
\newtheorem{remark}[theorem]{Remark}

\newtheorem{problem}[theorem]{Problem}

\numberwithin{equation}{chapter}

\DeclareMathAlphabet{\mathpzc}{OT1}{pzc}{m}{it}

 \newcommand{\cI}{\mathcal{I}}
 \newcommand{\cJ}{\mathcal{J}}

 \newcommand{\cN}{\mathcal{N}}

 \newcommand{\sC}{\mathscr{C}}

 \newcommand{\sL}{\mathscr{L}}

\newcommand{\corr}{{\mathsf{corr}}}

\newcommand{\D}{{\mathsf{D}}}

\newcommand{\HS}{{\mathsf{HS}}}

\newcommand{\MW}{{\mathsf{MW}}}

\newcommand{\I}{{\mathsf{I}}}
\newcommand{\J}{{\mathsf{J}}}
\newcommand{\K}{{\mathsf{K}}}
\newcommand{\R}{{\mathsf{R}}}

\newcommand{\T}{{\mathsf{T}}}

\newcommand{\W}{{{\mathsf{W}}}}

\newcommand{\x}{{\mathsf{x}}}

\newcommand{\const}{{\mathsf{const}}}
\newcommand{\dist}{{{\mathsf{dist}}}}

\newcommand{\new}{{{\mathsf{new}}}}

\newcommand{\bC}{{\mathbb{C}}}

\newcommand{\bR}{{\mathbb{R}}}

\newcommand{\bZ}{{\mathbb{Z}}}

\newcommand{\fZ}{{\mathfrak{Z}}}
\newcommand{\fz}{{\mathfrak{z}}}

\def\1{\boldsymbol {|}}
\newcommand{\boldgamma}{{\boldsymbol{\gamma}}}

\newcommand{\Def}{\mathrel{\mathop:}=}

\newcommand{\Hess}{\operatorname{Hess}}

\newcommand{\osc}{\operatorname{osc}}

\renewcommand{\Re}{\operatorname{Re}}       

\newcommand{\sign}{\operatorname{sign}}

\newcommand{\supp}{\operatorname{supp}}

\newcommand{\Tr}{\operatorname{Tr}}

\newenvironment{claim}[1][{\textup{(\theequation)}}]{\refstepcounter{equation}\vglue10pt
\begin{trivlist}
\item[{\hskip\labelsep#1}]}{\vglue10pt\end{trivlist}}

\newenvironment{claim*}[1][{}]{\vglue10pt
\begin{trivlist}
\item[{\hskip\labelsep#1}]}{\vglue10pt\end{trivlist}}

\newcounter{note}

\DeclareTextCommand{\textinfty}{PU}{\9042\036}

\DeclareTextCommand{\textge}{PU}{\9042\145}
\DeclareTextCommand{\textle}{PU}{\9042\144}
\DeclareTextCommand{\texthat}{PD1}{\136}

\begin{document}
\title{$2D$- and -$3D$Magnetic Schr\"{o}dinger Operators: Short Loops, Pointwise  Spectral Asymptotics and Asymptotics of Dirac Energy}
\author{Victor Ivrii}

\maketitle
{\abstract%
We consider 2- and 3-dimensional Schr\"odinger or generalized Schr\"odinger-Pauli operators with the non-degenerating magnetic field in the open domain  under certain non-degeneracy assumptions we derive  pointwise spectral asymptotics.

We also consider asymptotics of some related expressions (see below). For all asymptotics loops rather than periodic trajectories play important role. 
\endabstract}

\setcounter{chapter}{-1}
\chapter{Introduction}
\label{sect-16-0}

In this paper we consider $2\D$- and $3\D$-magnetic Schr\"odinger operator (\ref{book_new-13-1-1}) satisfying assumptions (\ref{book_new-13-1-2})--(\ref{book_new-13-1-5}) and consider pointwise asymptotics of $e(x,x,0)$ and also asymptotics of expression (\ref{book_new-6-3-4}):
\begin{gather}
\I\Def\iint  e(x,y,0)e (y,x,0) \omega (x,y)\psi_2(x) \psi_1(y)\,dx\,dy
\label{16-0-1}
\shortintertext{as well as of}
\J\Def \iint e(x,x,0)e(y,y,0)\omega(x,y)\psi_2(x) \psi_1(y)\,dxdy
\label{16-0-2}
\end{gather}
with a function $\omega(x,y)$ satisfying assumption (\ref{book_new-6-3-5}):
\begin{claim}\label{16-0-3}
$\omega (x,y)\Def \Omega (x,y; x-y)$ where function $\Omega$ is smooth in
$B(0,1)\times B(0,1)\times B(\bR^2\setminus 0)$ and homogeneous of degree $-\kappa$ ($0<\kappa<2$) with respect to its third argument,
\end{claim}
and with smooth cut-off functions $\psi_1,\psi_2$.

These two expressions play a role in the applications to the multiparticle quantum theory of Part~VIII. Actually, instead of asymptotics of (\ref{16-0-2}) we consider related estimates of
\begin{multline}
\K\Def\\
\iint \bigl(e(x,x,\tau)-h^{-d}\cN_x (\tau)\bigr)
\bigl(e(y,y,\tau)-h^{-d}\cN_x  (\tau)\bigr)\omega(x,y)\,dxdy
\label{16-0-4}
\end{multline}
where $\cN_x(\tau)$ denotes some approximation to $e(x,x,\tau)$.

We assume that $V/F$ satisfies non-degeneracy conditions (\ref{book_new-13-3-45}) and (\ref{book_new-13-3-54}) i.e.
\begin{gather}
V\le -\epsilon _0\qquad \qquad \text{in\ \ } B(0,1),
\label{16-0-5}\\[2pt]
|F|\ge \epsilon \qquad \text{in\ \ } B(0,1),\label{16-0-6}\\[2pt]
|\nabla \frac{V}{F}|\asymp 1 \qquad \text{in\ \ } B(0,1);
\label{16-0-7}
\end{gather}
in asymptotics of expressions (\ref{16-0-1})--(\ref{16-0-2}) we will be able to drop (\ref{16-0-5}) and replace (\ref{16-0-7}) by
\begin{align}
 |\nabla \frac{V}{F}|\le \epsilon_0&\implies
|\det \Hess \frac{V}{F}|\ge \epsilon_0
\label{16-0-8}\\
\shortintertext{or}
 |\nabla \frac{V}{F}|\le \epsilon_0&\implies
\  \det \Hess \frac{V}{F}\ \ge \epsilon_0.
\tag*{$\textup{(\ref*{16-0-8})}^+$}\label{16-0-8-+}
\end{align}

As we consider Schr\"odinger-Pauli operator (\ref{book_new-13-5-3}) and $\mu h\gtrsim 1$ conditions (\ref{16-0-7})--\ref{16-0-8-+} will be modified to
\begin{align}
&|\frac{V}{F}+(2m+1-\varsigma) |+ |\nabla \frac{V}{F}|\asymp 1 \qquad \text{in\ \ } B(0,1) \ \forall m\in \bZ^+;
\tag*{$\textup{(\ref*{16-0-7})}'$}\label{16-0-7-'}\\
&|\frac{V}{F}+(2m+1-\varsigma) |+ |\nabla \frac{V}{F}|\le \epsilon_0\implies
|\det \Hess \frac{V}{F}|\ge \epsilon_0,
\tag*{$\textup{(\ref*{16-0-8})}'$}\label{16-0-8-'}\\
&|\frac{V}{F}+(2m+1-\varsigma) |+ |\nabla \frac{V}{F}|\le \epsilon_0\implies
\det \Hess \frac{V}{F}\ge \epsilon_0
\tag*{$\textup{(\ref*{16-0-8})}^{+\,\prime}$}\label{16-0-8-+'}
\end{align}
respectively.

Non-degeneracy assumptions (\ref{16-0-7})--\ref{16-0-8-+} eliminate periodicity at least as  $\mu h \ll 1$ but there are plenty of loops as on figures \ref{fig-16-1a} and \ref{fig-16-1b}
\begin{figure}[ht]
\centering
\subfloat[prolate cycloid]{
\includegraphics{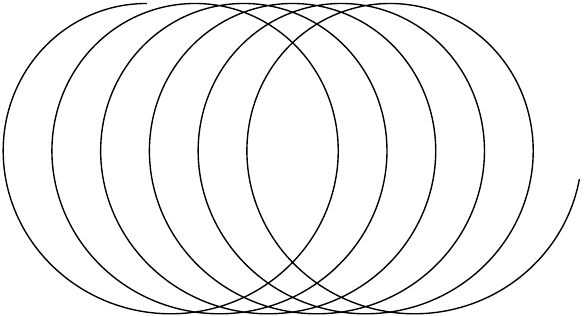}
\label{fig-16-1a}
}\quad
\subfloat[perturbed prolate cycloid]{
\includegraphics{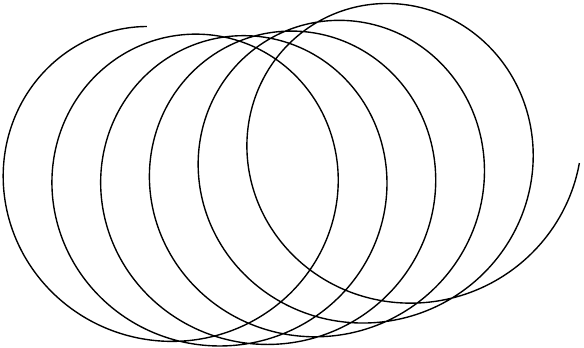}
\label{fig-16-1b}
}
\caption{\label{fig-16-1} Classical trajectories}
\end{figure}
and we know that loops play the same important role in the pointwise asymptotics (of $e(x,x,0)$) as periodic trajectories play in the local asymptotics (of $\int e(x,x,0)\psi(x)\,dx $). There also will be many paths from $x$ to $y$ and back and these paths play important role in the asymptotics of expression (\ref{16-0-1}).

Fortunately, even if there are plenty of looping points they are rather exceptional on the classical trajectory. Surely, one needs to remember that if we study $e(x,x,0)$ point $x$ is fixed but direction $\xi$ varies and thus trajectory\footnote{\label{foot-16-1} Even if one calls them trajectories they are only $x$-\emph{projections\/} of actual trajectories which live in $(x,\xi)$-space.} moves so it passes through $x$ and remains on the energy level $0$ (see figures \ref{fig-16-2a}--\ref{fig-16-2b} below). However for the geometrical simplicity we will sometimes talk about $x$ moving along trajectory. As we study expression (\ref{16-0-1}) $x,y$ both move.

In fact, situation is more complicated than this due to the uncertainty principle: as $\mu$ increases not  cyclotrons becomes smaller and tighter but the lines retain their width and become wider in comparison with the cyclotron, and symbolically we can consider figures \ref{fig-16-3a}--\ref{fig-16-3c} below.

We also generalize section~\ref{book_new-sect-6-3} and we consider expression (\ref{16-0-1}). It follows from section~\ref{book_new-sect-6-3} that \ $\I= \cI^\W +O(h^{1-d-\kappa})$ \ as $\mu=O(1)$ \ with $\cI^\W$ defined by the same formula but with $e(x,y,\tau)$ replaced by
\begin{equation}
e^\W_y (x,y,\tau) \Def (2\pi h)^{-d}\int_{g(y,\xi)\le V(y)+\tau} e^{ih^{-1}\langle x-y,\xi\rangle }\,d\xi.
\label{16-0-9}
\end{equation}
Then the standard rescaling technique implies the same asymptotics but with the remainder estimate $O(\mu h^{1-d-\kappa})$ provided $1\le \mu=o(h^{-1})$.

Let $d=2$. Then \emph{in the general case\/} it is the best remainder estimate possible while $O(\mu h^{-1})$ is the best possible remainder estimate for
\begin{equation}
 \int e(x,x,\tau)\psi (x)\,dx;
\label{16-0-10}
\end{equation}
recall  that  this estimate cannot be improved in the case of constant magnetic field $F$ and $g^{jk}=\const$, $V=\const$.

On the other hand, \emph{in the generic case\/} the remainder estimate for (\ref{16-0-10}) is $o(\mu^{-1}h^{-1})$ and the principal part is $\int e^\MW(x,\tau)\psi (x)\,dx$ if $F$ does not vanish and $\mu \le h^{-1}$.

Meanwhile for $d=3$ the remainder estimate for $\I'$ is
$O\bigl(h^{-2}+\mu h^{-1-\delta}\bigr)$ and  under very mild non-degeneracy assumptions it is $O\bigl(h^{-2}+\mu h^{-1}\bigr)$.

\medskip
Therefore our purpose is to get the sharper remainder estimate for (\ref{16-0-1}) under the same non-degeneracy assumptions. This is a very daunting task since for (\ref{16-0-10}) \emph{periodic trajectories\/} were the main source of trouble and they were broken in the generic case; for (\ref{16-0-1}) \emph{loops\/} are also the source of trouble, and in the generic case former periodic trajectories generate a lot of loops.

It is rather easy to derive asymptotics with the sharp remainder estimate $O(\mu^{-1}h^{-1-\kappa})$ as $d=2$ and $O(h^{-2})$ as $d=3$ but with the principal part given by very implicit Tauberian formula when $\I^\T$ is defined by (\ref{16-0-1}) but with $e(x,y,\tau)$ replaced by its Tauberian approximation
\begin{equation}
e^\T  (x,y,\tau)\Def h^{-1}\int _{-\infty}^\tau F_{t\to h^{-1}\tau} \bigl({\bar\chi}_T(t) u(x,y,t)\bigr)\,d\tau;
\label{16-0-11}
\end{equation}
such remainder estimate is a rather easy corollary of the results of Chapter~\ref{book_new-sect-13}. However deriving of an asymptotics with the sharp remainder estimate and rather explicit principal part is much more difficult.

\section*{Plan of the paper}

In sections~\ref{sect-16-1} and~\ref{sect-16-5} we consider respectively $2\D$- and $3\D$-pilot-models with the constant $g^{jk}$ and $F$ and linear $V$ and derive representations of solutions and pointwise asymptotics as magnetic field is weak enough. We also derive asymptotics of expression (\ref{16-0-10}) with $\psi(x)$ replaced by $\psi_\gamma(x)$ which is $\psi(x)$ scaled with a spatial scaling parameter $\gamma$; these asymptotics are instrumental in the asymptotics of (\ref{16-0-2}) and estimates of (\ref{16-0-4}).

In sections~\ref{sect-16-2} and~\ref{sect-16-6} we prove the same asymptotics respectively for the general $2\D$- and $3\D$-operators; as magnetic field is strong we use explicit expressions for the pilot-model operators as approximations.

In sections~\ref{sect-16-3} and~\ref{sect-16-7} we prove important preliminary results and derive sharp asymptotics for expression (\ref{16-0-1}) but with the Tauberian principal part.

In sections~\ref{sect-16-4} and~\ref{sect-16-8} we pass from the Tauberian approximations to Weyl, magnetic Weyl or pilot-model approximations.

Finally, in section~\ref{sect-16-9} we consider estimates of (\ref{16-0-4}) with different approximations, leaving to the reader very similar asymptotics of expression (\ref{16-0-2}).

The results of this paper are not always sharp or very explicit, but could be made either sharp or completely explicit.

\chapter{Pointwise asymptotics: $2\D$-pilot-model}
\label{sect-16-1}

\section{Pilot-model in $\bR^2$: propagator}
\label{sect-16-1-1}
Consider the \emph{pilot-model  operator\/}\index{operator!pilot-model}
\begin{equation}
A=\bar{A} \Def h^2 D_1^2 + (hD_2-\mu x_1)^2+2\alpha x_1.
\label{16-1-1}
\end{equation}

We are interested in the Schwartz kernel $U(x,y,t)$ of the propagator $e^{ih^{-1}tA}$. Making $h$-Fourier transform with respect to $x_2\mapsto \xi_2$ and rescaling $x_1\mapsto \mu x_1$, $t\mapsto \mu t$  we arrive to
\begin{equation}
U(x,y,t)= (2\pi h)^{-1}\mu \int \mathsf{u}(x_1,y_1;\eta,t)e^{ih^{-1}(x_2-y_2)\eta}\,d\eta
\label{16-1-2}
\end{equation}
with $\mathsf{u}(x_1,y_1; \eta ,t)$ the Schwartz kernel of
$e^{i\hbar ^{-1}t\mathbf{a}}$ with $1\D$-Harmonic oscillator
\begin{multline}
\mathbf{a}= \hbar^2 D_1^2 + (x_1-\eta)^2+2\alpha \mu^{-1} x_1=\\
\underbrace{\hbar^2 D_1^2 + (x_1-\eta +\alpha \mu^{-1})^2}_{\bar{\mathbf{a}}} +
\mu^{-1}\alpha\underbrace{(2 \eta -\alpha \mu^{-1})}_{\zeta (\eta)}.
\label{16-1-3}
\end{multline}
Recall that for the Harmonic oscillator $\mathbf{b}=D^2+x^2$ the Schwartz kernel of $e^{it\mathbf{b}}$~\footnote{\label{foot-16-2} This is a metaplectic operator, see \ref{book_new-sect-1-2-3} or H.~ter~Morsche and P.~J.~Oonincx~\cite{morsche:oonincx}.} is
\begin{equation}
(2\pi)^{-\frac{1}{2}} e^{\frac{i\pi}{4}\sigma(t)} |\sin (2t)|^{-\frac{1}{2}}
\exp\Bigl(-\frac{i}{2}\bigl( \cot(2t)(x^2+y^2)-2xy \csc (2t)\bigr)\Bigr)
\label{16-1-4}
\end{equation}
where $\sigma(t)=1,2,3,4$ at
$(0,\frac{\pi}{2}), (\frac{\pi}{2},\pi), (\pi,\frac{3\pi}{2}), (\frac{3\pi}{2},2\pi)$ respectively and it is $2\pi$-periodic and therefore for Harmonic oscillator $\mathbf{b}'=\hbar^2 D^2+x^2$ the Schwartz kernel of $e^{it\hbar^{-1}\mathbf{b}'}$ is
\begin{multline}
(2\pi\hbar)^{-\frac{1}{2}} e^{\frac{i\pi}{4}\sigma(t)}|\sin (2t)|^{-\frac{1}{2}} \times \\
\exp\Bigl(-\frac{i}{2}\hbar^{-1}
\bigl( \cot(2t)(x^2+y^2)-2xy\csc (2t)\bigr)\Bigr);
\label{16-1-5}
\end{multline}
one can prove it easily by rescaling $x \mapsto  x\hbar^{\frac{1}{2}}$.

For operator (\ref{16-1-3}) the Schwartz kernel of $e^{it\hbar^{-1}\mathbf{a}}$ is obtained by plugging $x \Def x_1-\eta+ \alpha \mu^{-1}$,
$y \Def y_1-\eta+ \alpha \mu^{-1}$ and multiplication by
$\exp\bigl(i\hbar^{-1}\mu^{-1}\alpha\zeta(\eta)\bigr)$. Therefore, multiplying (\ref{16-1-5}) by
\begin{equation}
(2\pi \hbar)^{-1}\mu^2\exp \Bigl(i\hbar^{-1}\bigl(\mu^{-1}t\alpha(2 \eta -\alpha \mu^{-1})+(x_2-y_2)\eta \bigr)\Bigr)
\label{16-1-6}
\end{equation}
and integrating by $\eta$ we arrive to\footnote{\label{foot-16-3}  After rescaling $x\mapsto \mu x$, $y\mapsto \mu y$, $t\mapsto \mu t$, $\mu\mapsto 1$, $h\mapsto \hbar=\mu h$; however we treat $U(x,y,t)$ and $e(x,y,\tau)$ as functions, not as densities with respect to $y$ and it leads to the factor $\mu^2$.}
\begin{equation}
U(x,y,t)=(2\pi\hbar)^{-\frac{3}{2}} \mu^2 \int e^{\frac{i\pi}{4}\sigma(t)}
|\sin (2t)|^{-\frac{1}{2}} e^{i\hbar^{-1}\bar{\varphi} (x,y,\eta,t)}\, d\eta
\label{16-1-7}
\end{equation}
with
\begin{multline}
\bar{\varphi}\Def
-\frac{1}{2}\cot(2t  )\bigl((x_1-\eta+\alpha \mu^{-1})^2 + (y_1-\eta+\alpha\mu^{-1})^2\bigr)+  \\[2pt]
\qquad\qquad\qquad \csc (2t)(x_1-\eta+\alpha \mu^{-1})(y_1-\eta+\alpha \mu^{-1}) +(x_2-y_2)\eta+\\[2pt]
t\mu^{-1}\alpha(  2\eta -\alpha \mu^{-1}).
\label{16-1-8}
\end{multline}
We can rewrite (\ref{16-1-7})--(\ref{16-1-8}) after integration by $\eta$ as
\begin{equation}
U(x,y,t)=i(4\pi\hbar)^{-1} \mu^2
\csc(t) \,e^{i\hbar^{-1}\bar{\phi}(x,y,t)}
\label{16-1-9}
\end{equation}
with
\begin{multline}
\bar{\phi}\Def
-\frac{1}{4}\cot(t) (x_1-y_1)^2\\
+\frac{1}{2} (x_1+y_1+2\alpha\mu^{-1})(x_2-y_2+2t\alpha\mu^{-1})
-\frac{1}{4}\cot(t)(x_2-y_2+2t\alpha\mu^{-1})^2 - t\alpha^2\mu^{-2};
\label{16-1-10}
\end{multline}
here the critical point with respect to $\eta$ is
\begin{equation}
\eta= \frac{1}{2}(x_1+y_1+2\alpha \mu^{-1}) -
\frac{1}{2}\cot(t)(x_2-y_2+2t\alpha \mu^{-1}).
\label{16-1-11}
\end{equation}
So, we have proven

\begin{proposition}\label{prop-16-1-1}
For the pilot-model operator \textup{(\ref{16-1-1})}~\footref{foot-16-3} in $\bR^2$  the Schwartz kernel $\bar{U}(x,y,t)$ of $e^{ih^{-1}t\bar{A}}$ is given by \textup{(\ref{16-1-9})}--\textup{(\ref{16-1-11})}  and  the Schwartz kernel $e(x,y,\tau)$ of the spectral projector is given by
\begin{multline}
\partial_\tau e(x,y,\tau)= (2\pi \hbar)^{-1} F_{t\to \hbar^{-1}\tau}U(x,y,.)=\\
(2\pi \hbar)^{-1} \int e^{-i\hbar^{-1}t\tau'}U(x,y,t)\,dt
\label{16-1-12}
\end{multline}
and
\begin{multline}
e(x,y,\tau)=
(2\pi \hbar)^{-1}\int^\tau\Bigl(\int e^{-i\hbar^{-1}t\tau'}U(x,y,t)\,dt\Bigr)\,d\tau=\\
(2\pi)^{-1} \int (-it)^{-1}e^{-i\hbar^{-1}t\tau}U(x,y,t)\,dt
\label{16-1-13}
\end{multline}
with the last integral taken in the sense of the essential value at $0$.
\end{proposition}

In the case of the strong magnetic field we will need an alternative representation.
Starting from formula for $1\D$-harmonic oscillator
\begin{equation}
\mathsf{e}(x_1,y_1,\tau) =
\sum _{m\in \bZ^+} \upsilon_m (x_1)\upsilon_m(y_1)
\uptheta \bigl(\tau - (2m+1)\bigr)
\label{16-1-14}
\end{equation}
with Hermite functions $\upsilon_m$ we after the same rescaling and transition to the pilot-model as before arrive to

\begin{proposition}\label{prop-16-1-2}
For the pilot-model operator \textup{(\ref{16-1-1})}~\footref{foot-16-3} in $\bR^2$ the  Schwartz kernel of the spectral projector is defined by
\begin{multline}
e(x,y,\tau)=\\
(2\pi)^{-1}\mu^2  \hbar^{-1}\sum_{m\in \bZ^+}\int
\upsilon_m \bigl(\eta+\hbar^{-\frac{1}{2}}(x_1-y_1)\bigr)
\upsilon_m \bigl(\eta-\hbar^{-\frac{1}{2}}(x_1-y_1)\bigr)\times\\[3pt]
\uptheta\Bigl( \tau - \alpha \mu^{-1}(x_1+y_1) -2\alpha \mu^{-1}\hbar^{\frac{1}{2}}\eta  -\alpha^2\mu^{-2}-(2m+1)\hbar\Bigr) e^{i\hbar^{-\frac{1}{2}} (x_2-y_2)\eta} \,d\eta
\label{16-1-15}
\end{multline}
and thus
\begin{multline}
\partial_\tau e(x,y,\tau)=\\
(4\pi\alpha)^{-1}\mu^3 \hbar^{-\frac{3}{2}}\sum_{m\in \bZ^+}\int
\upsilon_m \bigl(\eta+\hbar^{-\frac{1}{2}}(x_1-y_1)\bigr)
\upsilon_m \bigl(\eta-\hbar^{-\frac{1}{2}}(x_1-y_1)\bigr)\times\\[3pt]
e^{i\hbar^{-\frac{1}{2}} (x_2-y_2)\eta} \Bigr|_{\eta=\eta_m}
\label{16-1-16}
\end{multline}
with
\begin{equation}
\eta_m \Def (2\alpha )^{-1}\mu \hbar^{-\frac{1}{2}}
\bigl(\tau - \alpha \mu^{-1}(x_1+y_1)   -\alpha^2\mu^{-2}-(2m+1)\hbar\bigr).
\label{16-1-17}
\end{equation}
\end{proposition}

\section{Tauberian estimate}
\label{sect-16-1-2}
Consider now pilot-model (\ref{16-1-1}) in  $B(0,1)\subset X\subset \bR^2$. As dynamics (classical or microlocal) starts in $B(0,\epsilon)$ it is confined to $B(0,1)$ for $|t|\le T^*=\epsilon \mu$ and thus in this time interval we can use formulae of the previous subsection (modulo negligible term).

\subsection{Preparatory estimate}
\label{sect-16-1-2-1}

Let us rescale\footref{foot-16-3} and set $x=y=0$ in (\ref{16-1-9})--(\ref{16-1-11}):
\begin{gather}
U(0,0,t)\equiv i(4\pi\hbar)^{-1} \mu^2
\csc(t) \,e^{i\hbar^{-1}\bar{\phi}(t)}
\label{16-1-18}\\
\shortintertext{with}
\bar{\phi}(t)\Def
t^2\alpha^2\mu^{-2} \cot(t) + \alpha^2\mu^{-2}t;
\label{16-1-19}\\[2pt]
\eta=  -t\alpha \mu^{-1}  \cot(t)+\alpha \mu^{-1}.
\label{16-1-20}
\end{gather}
We are interested in
\begin{equation}
F_{t\to \hbar^{-1}\tau} \bar{\chi}_T(t) U(0,0,t)
\label{16-1-21}
\end{equation}
with $T=T^*$. Plugging (\ref{16-1-18}) we arrive to
\begin{multline}
F_{t\to \hbar^{-1}\tau} \bar{\chi}_T(t) U(0,0,t)\equiv \\
i(4\pi\hbar)^{-1} \mu^2 \int \bar{\chi}_T(t)
\csc(t) e^{i\hbar^{-1}(\bar{\phi}(t)-\tau)}\,dt =\\
(4\pi\hbar)^{-1} \mu^2 \int \bar{\chi}_T(t)
\csc(t)\exp\Bigl(i\hbar^{-1}
\bigl(-t^2 \alpha ^2 \mu^{-2}\cot (t ) + t\mu^{-2}\alpha^2 -t\tau\bigr) \Bigr)\,dt.
\label{16-1-22}
\end{multline}
Applying stationary phase to
\begin{gather}
\varphi(t)\Def -t^2 \mu^{-2}\alpha ^2 \cot (t ) + t\mu^{-2}\alpha^2 -t\tau
\label{16-1-23}
\\
\shortintertext{we get}
\mu^{-2} \alpha^2 (t^2 - t\sin (2t)) = (\tau-\mu^{-2}\alpha^2) \sin^2(t).
\label{16-1-24}
\end{gather}
\begin{remark}\label{rem-16-1-3}
Consider corresponding classical trajectory:
\begin{equation}
x_1=\rho \cos (2s)+\eta -\alpha \mu^{-1},\qquad
\xi_1= \rho \sin (2s)+\alpha \mu^{-1}s
\label{16-1-25}
\end{equation}
and there is a self-intersection iff
\begin{equation}
\left\{\begin{aligned}
&\cos (2s_1)=\cos (2s_2),\\
&-\rho \sin (2s_1)+2\alpha \mu^{-1}s_1=-\rho \sin (2s_2)+2\alpha \mu^{-1}s_2;
\end{aligned}\right.
\label{16-1-26}
\end{equation}
then $s_2+s_1=\pi k$ with $k\in \bZ$ (because we cannot fulfill the second equation as $s_1\ne s_2$ and $s_2-s_1=\pi k$); then $2s_1=\pi k-t$, $2s_2=\pi k+t$ and the second equation is $\rho \sin (\pi k-t)=  -\alpha \mu^{-1}t$ or, equivalently, $\rho = -\alpha \mu^{-1}t \csc (\pi k-t)$.

Further, we need to satisfy
$\rho \cos (\pi k-t)+\eta -\alpha \mu^{-1}=0$ (as $x_1=0$ is a level of intersection) and $(\eta -\alpha \mu^{-1})^2 + 2\alpha \mu^{-1}(\eta - \alpha\mu^{-1})+ \alpha^2\mu^{-2}=\tau$ (to be on the energy level $\tau$) and then $\rho^2 \cos^2 t -2\alpha \mu^{-1}\rho \cos(\pi k-t)= \tau-\alpha^2\mu^{-2}$.

Plugging $\rho = -\alpha \mu^{-1}t\csc  (\pi k-t)$ we conclude that this equation becomes (\ref{16-1-24}). We can find $\eta$ to satisfy condition
$\rho \cos (2s_1)+\eta-\alpha \mu^{-1}=0$ (as  $(0,.)$ is a point of intersection, or equivalently,
$\eta=-\alpha \mu^{-1}t \cot (t )+\alpha \mu^{-1}$.
\end{remark}

We need to justify the stationary phase method. Note first that the spacing between successive stationary points is $\sim \pi$. Therefore
\begin{equation}
t_k=-t_{-k},\quad t_k \sim \pi k, \quad
\sin (t_k)\sim \alpha \mu^{-1}\tau^{-\frac{1}{2}}\pi k
\label{16-1-27}
\end{equation}
and
\begin{claim}\label{16-1-28}
The number of stationary points is
$\sim 2\pi^{-1}|\alpha|^{-1}\mu \tau^{-\frac{1}{2}}$.
\end{claim}
Recall that
\begin{gather}
\varphi'(t)= \alpha^2\mu^{-2}\bigl(t^2 \csc^2(t) -2t\cot(t)+1\bigr)-\tau
\label{16-1-29}\\
\shortintertext{and}
\varphi''(t)= -2\alpha^2\mu^{-2}\csc(t)
\Bigl(t^2 \csc^2(t)\cos(t) -2t\csc(t)+\cos(t)\Bigr);
\label{16-1-30}
\end{gather}
therefore
\begin{claim}\label{16-1-31}
Let $k\ne 0$; then $\varphi''(t_k)\asymp -2\tau \cot(t_k)$ as
$|\cos (t_k)|\ge C_0\varepsilon$ and moreover
$\varphi''(t_k)\sim -2\tau \cot(t_k)$ as $|\cos (t_k)|\gg \varepsilon$.
\end{claim}

Therefore as $|\sin(t_k)|\ge \epsilon$, $|\cos(t_k)|\ge \epsilon$ stationary points are non-degenerate and we can use the stationary phase method.

In the \emph{near-pole zone\/}\index{zone!near-pole}
$\{|\sin (t)|\le \epsilon\}$ we need to remember about singularity as $\sin(t)=0$; so we need to introduce a scale $\ell\asymp |\sin(t)|$ and the stationary phase method is expected to work only as
$|\varphi ''|\ell^2 \ge C_0\hbar$ or equivalently
$|\sin (t_k)|\ge C_0 \hbar$ or, finally
\begin{gather}
|k| \ge \bar{k}\Def C_0 \varepsilon^{-1} \hbar
\label{16-1-32}\\
\shortintertext{with}
\varepsilon \Def \alpha \mu^{-1}(\tau-\alpha^2\mu^{-2})^{-\frac{1}{2}}.
\label{16-1-33}
\end{gather}
Summarizing, we arrive to

\begin{claim}\label{16-1-34}
Condition (\ref{16-1-32}) is fulfilled for all
$k\ne 0$ as $\varepsilon \ge C_0\hbar$; otherwise stationary phase method fails for $|k|\le \bar{k}$. In particular, $\bar{k}= C_0\mu^2h$ as $|\alpha|\asymp 1$, $\tau\asymp 1$ and condition (\ref{16-1-32}) is fulfilled for all
$k\ne 0$ as $\mu \le \epsilon_0h^{-\frac{1}{2}}$.
\end{claim}

Meanwhile, there is no singularity in  the \emph{near-equator zone\/}\index{zone!near-equator} $\{|\cos (t)|\le \epsilon\}$ but there is a degeneration as $\cos(t)=0$ and therefore we need to introduce a a scale
$\ell\asymp |\cos(t)|$ and the stationary phase method is expected to work only as $|\varphi ''|\ell^2 \ge C_0\hbar$ or equivalently
$|\cos (t)|\ge C_0\max(\varepsilon,\hbar^{\frac{1}{3}})$.

\begin{definition}\label{def-16-1-4}
(i)  $\{|\sin (t)|\le \max(\epsilon_0\varepsilon, C_0\hbar)\}$ is a \emph{near-pole singular zone\/}\index{zone!near-pole singular};

\medskip\noindent
(ii)  $\{|\cos (t)|\le C_0\max(\varepsilon,\hbar^{\frac{1}{3}})\}$  is a \emph{near-equator singular zone\/}\index{zone!near-equator singular};

\medskip\noindent
(iii) $\{|\sin (t)|\le \max(\epsilon_0\varepsilon, C_0\hbar),
|\cos (t)|\ge C_0\max(\varepsilon,\hbar^{\frac{1}{3}})\}$ is
a \emph{regular zone\/}\index{zone!regular}.
\end{definition}

We apply the stationary phase method in the regular zone and justify it and we estimate contributions of the singular zones.

Note that the simple integration by parts brings factors
$\hbar /(\varphi'\sin (t))$ and $\hbar \varphi'' /(\varphi')^2$ and contribution of zones where these expressions are  less than $h^\delta$ are negligible. However at this moment we want just to estimate.

\begin{remark}\label{rem-16-1-5}
As long as $\mu h\lesssim 1$ we will assume that  $\tau \asymp 1$; while we can always achieve $\alpha \asymp 1$ by rescaling $x\mapsto |\alpha| x$,
$\mu \mapsto \mu |\alpha|^{-1}$, $\hbar \mapsto \hbar$ we will need many intermediate results and it is simpler to deal with the general case $|\alpha|\le 1$ from the very beginning. Without any loss of the generality one can assume that
\begin{equation}
0<\alpha \le  1\qquad
(\implies \varepsilon \le \mu^{-1}).
\label{16-1-35}
\end{equation}
\end{remark}

First, consider zone
\begin{equation}
\bigl\{t:\ |\sin t|\le (1-\epsilon_0) \varepsilon|t|,\ |t|\ge \epsilon_0\bigr\}.
\label{16-1-36}
\end{equation}

Then $|\varphi'|\asymp \varepsilon^2 t^2 |\sin (t)|^{-2}$ and
$|\varphi ''|\lesssim \varepsilon^2  t^2|\sin (t)|^{-3}$ and both factors
$|\hbar /\varphi'\sin (t)|$ and $|\hbar \varphi'' /(\varphi')^2|$ do not exceed
$\hbar |\sin (t)|/\varepsilon^2 t^2\le \hbar  / \varepsilon |t|$ and
are less than $1$ as $\hbar |\sin (t)|\le  \varepsilon^2 t^2$. This always holds in zone (\ref{16-1-36}) if $\hbar\le \varepsilon$.

Then multiple integration by parts brings factor
$C(\hbar |\sin (t)|/ \varepsilon^2 t^2)^l$ with an arbitrarily large exponent $l$ and contribution of $k$-th ``tick''\footnote{\label{foot-16-4} I.e. interval $\{t:\,|t-t_k|\le \pi -\epsilon_0\}$.} to (\ref{16-1-22}) does not exceed
\begin{equation*}
C\mu^2\hbar^{-1}(\hbar |\sin (t)|/  \varepsilon^2 t^2)^l |\sin (t)|^{-1}.
\end{equation*}
Therefore summation with respect to $t$ along  zone (\ref{16-1-36}) intersected with $k$-th tick returns $C \mu h^{-1} (\hbar /\varepsilon |k|)^l$ as
$|k|\ge \bar{k}$ with $\bar{k}$ defined by (\ref{16-1-32});  then summation by $k$  results in $C \mu h^{-1} (\hbar/\varepsilon)^l$ as $\hbar\le \varepsilon$ ($1\le |k|$) and $C\mu h^{-1} \bar{k}\asymp  C\mu ^2\varepsilon^{-1}$  as $\hbar\ge \varepsilon$ ($\bar{k} \le |k|$).

In the latter case as $1\le |k|\le \bar{k}$ we  integrate by parts only as
$|\sin (t)|\le \hbar^{-1} \varepsilon^2 t^2$ and contribution of $k$-th tick to (\ref{16-1-22}) does not exceed
$C\mu  h ^{-1} \log  (\bar{k} /|k|)$ where logarithmic factor is an integral of $|\csc(t)| $ along $k$-th tick intersected with
$\{t:\, \hbar^{-1}  \varepsilon^2 t^2 \le |\sin (t)|\le \varepsilon |t|\}$. Then summation with respect to $1\le |k|\le \bar{k}$ returns $O(\mu h^{-1}\bar{k})$ again.

Therefore we arrive to

\begin{claim}\label{16-1-37}
Contribution of zone (\ref{16-1-36})
to expression (\ref{16-1-22}) does not exceed
$C \mu h^{-1} (\hbar/\varepsilon )^l$ as $\hbar \le \varepsilon$  and
$C\mu^2\varepsilon^{-1}$ as $\hbar \ge \varepsilon$.
\end{claim}

In the same time in the zone
\begin{equation}
\bigl\{t:\ |\sin (t)|\ge (1+\epsilon_0)\varepsilon|t|,\ |t|\ge \epsilon_0\bigr\}
\label{16-1-38}
\end{equation}
$|\varphi'|\asymp 1$ and $|\varphi ''|\lesssim |\sin (t)|^{-1}$ and both factors do not exceed $\hbar /|\sin (t)|$. Again summation with respect to $t$ along $k$-th tick intersected with zone (\ref{16-1-38}) returns
$C \mu h^{-1} (\hbar/\varepsilon |k|)^l$ as  $|k|\ge \bar{k}$ and summation by $k$ results in $C \mu h^{-1} (\hbar/\varepsilon)^l$ as $\hbar\le\varepsilon$ ($1\le |k|$) and $\mu^2\varepsilon^{-1}$ as  $\hbar\ge\varepsilon$ ($\bar{k}\le |k|$).

In the latter case as $|k|\le \bar{k}$ we  do not integrate by parts and contribution of $k$-th tick to (\ref{16-1-22}) does not exceed
$C\mu  h ^{-1} $; summation with respect to $1\le |k|\le \bar{k}$ returns $O(\mu^2\varepsilon^{-1})$. Therefore

\begin{claim}\label{16-1-39}
Contribution of zone (\ref{16-1-38}) to expression (\ref{16-1-22}) does not exceed $C \mu h^{-1} (\hbar/\varepsilon)^l$ as  $\hbar\le\varepsilon$  and $C\mu^2\varepsilon^{-1}$ as  $\hbar\ge\varepsilon$.
\end{claim}

Note that as  $\hbar\le \epsilon \varepsilon$ these  considered zones (\ref{16-1-36}) and (\ref{16-1-38}) cover completely the singular near-pole zone. On the other hand,  as  $\hbar\ge \epsilon \varepsilon$ we did not integrate by parts in the second zone as $|k|\le \bar{k}$ and therefore it can be extended to $\{t:\,|\sin (t)|\ge \epsilon_0 \varepsilon |t|\}$ and we arrive to

\begin{claim}\label{16-1-40}
Contribution of singular near-pole zone  to expression (\ref{16-1-22}) does not exceed $C \mu h^{-1} (\hbar/\varepsilon)^l$ as  $\hbar\le\varepsilon$ and $C\mu^2\varepsilon^{-1}$ as  $\hbar\ge\varepsilon$.
\end{claim}

Claims (\ref{16-1-37}), (\ref{16-1-39}), (\ref{16-1-39})  imply that as
$|\cos (t_k)|\ge \epsilon$ we can apply the stationary phase method.

Then contribution of $k$-th stationary point to (\ref{16-1-22}) does not exceed
$C\mu  h^{-1} (\hbar/\varepsilon |k|)^{\frac{1}{2}}$ and summation with respect to $k:\,\max (1,\bar{k})\le |k|\le C_0\varepsilon^{-1}$ returns
$C\mu  h^{-1} (\hbar /\varepsilon |k|)^{\frac{1}{2}}\times |k|$ calculated as
$|k|= \varepsilon^{-1}$ i.e. $C\mu^{\frac{3}{2}} \varepsilon^{-1}h^{-\frac{1}{2}} $ which is larger than what we got in (\ref{16-1-37})--(\ref{16-1-40}) thus resulting in

\begin{claim}\label{16-1-41}
Contribution of  zone $\{t:\, |\cos (t)|\ge \epsilon_0, \ |t|\ge \epsilon_0\}$ to expression (\ref{16-1-22}) does not exceed
$C\mu^{\frac{3}{2}}\varepsilon^{-1}h^{-\frac{1}{2}}$ as $ 1\le \mu \le h^{-1}$.
\end{claim}

Consider now the part of the near-equator zone
$\{|\cos (t)|\le \epsilon_0, \ |t|\ge \epsilon_0\}$ not covered by (\ref{16-1-37}), (\ref{16-1-39}); there
$(1-\epsilon) t^* \le |t| \le (1+\epsilon)t^*$ where
\begin{claim}\label{16-1-42}
$t^*= (k^*+\frac{1}{2})\pi$ is the closest to $\varepsilon^{-1}$  number of the form $ (k+\frac{1}{2})\pi$.
\end{claim}
Plugging $t=t^*-s$ with $|s|\le \epsilon t^*$, $|\sin (s)|\le \epsilon$  we find that
\begin{equation}
\tau^{-1} \varphi'(t)=-\varepsilon^2 s(2t^*-s) + \sin^2(s) +O(\varepsilon^2) + O(\varepsilon |\sin(s)|).
\label{16-1-43}
\end{equation}
Consider first zone
\begin{equation*}
\bigl\{s:\, C_0\le |s|\le \epsilon \varepsilon^{-1},\
|\sin (s)|^2\le 2(1-\epsilon_1)\varepsilon |s|\bigr\}
\end{equation*}
where $|\varphi'|\gtrsim \varepsilon |s|$. Integrating by parts we  acquire factor $(\hbar/\ell |\varphi'|) \asymp (\hbar/ (\varepsilon|s|)^{\frac{3}{2}})$ because $\ell =(\varepsilon |s|)^{\frac{1}{2}}$ is the scale; thus we integrate multiple times by parts as
$|s|\gtrsim \varepsilon^{-1} \hbar^{2/3}$; so contribution of $k$-th tick to (\ref{16-1-22}) does not exceed
\begin{equation*}
C\mu \varepsilon^{\frac{1}{2}} h^{-1}
(\hbar^{2/3}/ \varepsilon |k-k^*|)^{\frac{3}{2}l} |k-k^*|^{\frac{1}{2}}.
\end{equation*}
Then, as $\varepsilon  \ge \hbar^{\frac{2}{3}}$ summation with respect to
$k:|k-k^*|\ge C_0$ returns
$C\mu \varepsilon^{\frac{1}{2}} h^{-1} (\hbar^{\frac{2}{3}}/\varepsilon )^l$. On the other hand, as
$\varepsilon  \le \hbar^{\frac{2}{3}}$    summation with respect to
$k:|k-k^*|\ge \varepsilon^{-1} \hbar^{\frac{2}{3}} $ returns
$C\mu \varepsilon^{\frac{1}{2}} h^{-1}  \times
\hbar^{\frac{2}{3}} \varepsilon^{-1} =
C\mu^{\frac{5}{3}}\varepsilon^{-\frac{1}{2}} h^{-\frac{1}{3}}$.

Further, the contribution of the zone
\begin{equation*}
\bigl\{s:\, C_0\le |s|\le \epsilon \varepsilon^{-1},\
|\sin (s)|^2\ge 2(1+\epsilon_1)\varepsilon |s|\bigr\}
\end{equation*}
is estimated in the same way.
Finally, as $s<0$ both terms in the right-hand expression of (\ref{16-1-43}) have the same sign and we can extend zones in question to cover
$\{s:\, \sin^2 (s)\asymp \varepsilon |s|\}$ as well and we arrive to

\begin{claim}\label{16-1-44}
Contribution to (\ref{16-1-22}) of the zone
$\{s:\, |\sin (s)|\le \epsilon,\, |s|\le \epsilon\varepsilon^{-1}\}$ without subzone
\begin{equation}
\bigl\{s:\, s\ge C_0\max  (1, \varepsilon^{-1} \hbar^{\frac{2}{3}}),\
(2-\epsilon)\varepsilon s\le  \sin^2 (s)\le (2+\epsilon)\varepsilon s\bigr\}
\label{16-1-45}
\end{equation}
and the near-equator singular zone
\begin{equation}
\bigl\{s:\, |s|\le C_0\max  (1, \varepsilon^{-1} \hbar^{\frac{2}{3}}),\
|\sin  (s)|\le C_0\max  (\varepsilon,  \hbar^{\frac{1}{3}}) \bigr\}
\label{16-1-46}
\end{equation}
does not exceed
$C\mu \varepsilon^{\frac{1}{2}} h^{-1} (\hbar^{\frac{2}{3}}/\varepsilon)^l$ as
$\varepsilon \ge \hbar^{\frac{2}{3}}$ and
$C\mu^{\frac{5}{3}}\varepsilon^{-\frac{1}{2}} h^{-\frac{1}{3}}$ as
$\varepsilon \le \hbar^{\frac{2}{3}}$.
\end{claim}

Obviously,
\begin{claim}\label{16-1-47}
Contribution  to expression (\ref{16-1-22})  of zone (\ref{16-1-46}) does not exceed
$C\mu h^{-1} \times \max  (1, \varepsilon^{-1} \hbar^{\frac{2}{3}}) \times
\max (\varepsilon,  \hbar^{\frac{1}{3}})= C\mu h^{-1}\times
\max(\varepsilon, \varepsilon^{-1}\hbar ,\hbar^{\frac{1}{3}})$.
\end{claim}

Finally, in zone (\ref{16-1-45}) we can use the stationary phase method. Recall that $\varphi'' \asymp \sin (s)\asymp (\varepsilon|k-k^*|)^{\frac{1}{2}}$;
therefore the contribution of $k$-th tick does not exceed
$C\mu h^{-1} (\hbar /|\sin(s)|)^{\frac{1}{2}} \asymp
C\mu h^{-1} (\hbar  / (\varepsilon|k-k^*|)^{\frac{1}{2}})^{\frac{1}{2}} $
and summation with respect to $k:\ |k-k^*|\le \varepsilon^{-1}$ returns
$C\mu^{\frac{3}{2}}\varepsilon^{-1}h^{-\frac{1}{2}}$.

Combining all these estimates we arrive to  estimate (\ref{16-1-49}) below; adding
\begin{equation}
|F_{t\to \hbar^{-1}\tau} \bigl(\bar{\chi}_{\bar{T}}(t) U(0,0,t)\bigr)|\le
C\mu h^{-1}
\label{16-1-48}
\end{equation}
which is the standard semiclassical estimate $O(\hbar^{-1})$ with an extra factor $\mu^2$ due to rescaling we arrive to estimate (\ref{16-1-50}):

\begin{proposition}\label{prop-16-1-6}
For the pilot-model operator \textup{(\ref{16-1-1})} with  $\tau\asymp 1$, $\varepsilon\le \mu ^{-1}$, $\mu \le h^{-1}$ in $\bR^2$
\begin{equation}
|F_{t\to \hbar^{-1}\tau}
\bigl(\bigl(\bar{\chi}_T(t) -\bar{\chi}_{\bar{T}}\bigr)U(0,0,t)\bigr)|\le C\mu^{\frac{3}{2}} \varepsilon^{-1} h^{-\frac{1}{2}}
\label{16-1-49}
\end{equation}
as $\bar{T}= \epsilon_0 \le T \le \infty$ and
\begin{equation}
|F_{t\to \hbar^{-1}\tau} \bigl(\bar{\chi}_T(t) U(0,0,t)\bigr)|\le
C\mu h^{-1}+ C\mu^{\frac{3}{2}} \varepsilon^{-1} h^{-\frac{1}{2}}.
\label{16-1-50}
\end{equation}
\end{proposition}

\subsection{Tauberian estimate}
\label{sect-16-1-2-2}

\begin{remark}\label{rem-16-1-7}
(i) Recall that for the pilot-model operator  (\ref{16-1-1}) in domain $X$, $B(0,\ell)\subset X\subset \bR^2$ with $\ell\ge C_0\mu^{-1}$, fundamental solution $U$ coincides modulo negligible with one in $\bR^2$ as
$T\le T^*=\epsilon_0\mu \varepsilon^{-1} \ell $.

\medskip\noindent
(ii) We have $T^*=\epsilon_0\mu \varepsilon^{-1} \ell$ rather than $T^*=\epsilon_0\varepsilon^{-1} \ell$ as in Chapter~\ref{book_new-sect-13} only because we rescaled $t\mapsto\mu t$; without this rescaling we would have
$F_{t\to h^{-1}\tau}$ rather than  $F_{t\to \hbar^{-1}\tau}$ and the right-hand expressions (\ref{16-1-49}), (\ref{16-1-50}) should be multiplied by $\mu^{-1}$.
\end{remark}

Due to above we arrive to the statement (i) of proposition~\ref{prop-16-1-8} below; then the standard Tauberian arguments yield that the Tauberian error does not exceed the right-hand expression of (\ref{16-1-50}) divided by $T^*$:

\begin{proposition}\label{prop-16-1-8}
For a self-adjoint operator in $\sL^2(X)$ in the domain $X$, $B(0,\ell)\subset X\subset \bR^2$, $\ell\ge C_0\mu^{-1}$, coinciding in $B(0,\ell)$ with the pilot-model operator \textup{(\ref{16-1-1})} with  $\tau\asymp 1$, $\varepsilon\le\mu^{-1}$,
$\mu \le h^{-1}$,

\medskip\noindent
(i) Estimates  \textup{(\ref{16-1-49})}, \textup{(\ref{16-1-50})} hold as
$T\le T^*=\epsilon_0\mu\varepsilon^{-1}\ell$;

\medskip\noindent
(ii) Tauberian estimate
\begin{multline}
|e(0,0,\tau)-e^\T (0,0,\tau)|\le
C \bigl(\mu h^{-1}+ \mu^{\frac{3}{2}} \varepsilon^{-1} h^{-\frac{1}{2}}\bigr) \mu ^{-1}\varepsilon \ell^{-1}=\\
C\bigl(\varepsilon h^{-1}+ \mu^{\frac{1}{2}} h^{-\frac{1}{2}}\bigr)\ell^{-1}
\label{16-1-51}
\end{multline}
holds with the Tauberian expression
\begin{equation}
e^\T(x,y,\tau) \Def \hbar^{-1}\int_{-\infty}^\tau
\Bigl(F_{t\to \hbar^{-1}\tau} \bar{\chi}_T(t) U(x,y,t)\Bigr)\,d\tau
\label{16-1-52}
\end{equation}
with any $T\in (C_0\varepsilon^{-1}, \epsilon_0\mu\varepsilon^{-1} \ell)$.
\end{proposition}

\subsection{Micro-averaging}
\label{sect-16-1-2-3}

Note that the second terms in the right-hand expressions of (\ref{16-1-50}), (\ref{16-1-51}) are due to the loops and according to Chapter~\ref{book_new-sect-13} these expressions would not appear if instead of $\Gamma_x e(.,.,\tau)$ we looked at $\Gamma (e(.,.,\tau)\psi)$ with the uniformly smooth function $\psi$. Now we want to consider $\Gamma \bigl(e(.,.,\tau)\psi_\gamma\bigr)$ where $\psi_\gamma$ is $\gamma$-admissible function and $\gamma\le \ell $.

So, let us consider $\Gamma \bigl(e(.,.,\tau)\psi_\gamma\bigr)$. First of all note that all the above estimates will acquire factor $\gamma^2$ due to the integration, but there is more: integrating by parts with respect to $x_1$ and remembering that due to rescaling $x\mapsto \mu x$ we also rescale $\gamma\mapsto \mu \gamma$ we conclude that the contribution of $k$-th tick acquires factor
$(\hbar /\varepsilon |k|\mu \gamma)^l=(h / \varepsilon \gamma|k|)^l$ as
\begin{equation}
|k|\ge \hat{k}(\gamma)\Def h/\varepsilon \gamma.
\label{16-1-53}
\end{equation}

This factor ``forces'' $|k|=\max(\hat{k}(\gamma),1)$ in the estimate of the Fourier transform
\begin{equation}
F_{t\to \hbar{-1} } \Gamma (U\psi_\gamma)
\label{16-1-54}
\end{equation}
and there are four cases:
\begin{enumerate}[leftmargin=*,label=(\alph*)]

\item $\gamma \ge h/\varepsilon$; then $\hat{k}(\gamma)=1$ and expression (\ref{16-1-54}) does not exceed
$C\mu h^{-1}\min\bigl((\hbar/\varepsilon)^{\frac{1}{2}},1\bigr)\times
(h/\varepsilon\gamma)^l\gamma^2$;

\item $ h/\varepsilon \ge \gamma \ge \mu^{-1}$. Then
$1\le \hat{k}(\gamma)\le \bar{k}= \hbar/\varepsilon $ and  (\ref{16-1-54}) does not exceed $C\mu h^{-1}\hat{k}(\gamma)\gamma^2 = C\mu \varepsilon^{-1}\gamma $;

\item $\min(h/\varepsilon ,\mu^{-1}) \ge \gamma\ge h$. Then
$\bar{k}\le \hat{k}(\gamma)\le \varepsilon^{-1}$ and  (\ref{16-1-54}) does not exceed
$C\mu h^{-1} (\hbar/\varepsilon)^{\frac{1}{2}} |\hat{k}(\gamma)|^{\frac{1}{2}}\gamma^2=
C\mu^\frac{3}{2}\varepsilon^{-1} \gamma^{\frac{3}{2}}$;

\item $\gamma \le h$. Then $\hat{k}(\gamma)\ge \varepsilon^{-1}$ and micro-averaging brings no improvement and estimates (\ref{16-1-55})--(\ref{16-1-58}) below are due to estimates (\ref{16-1-49})--(\ref{16-1-51}) without micro-averaging.
\end{enumerate}

Therefore we arrive to

\begin{proposition}\label{prop-16-1-9}
For the pilot-model operator \textup{(\ref{16-1-1})} with  $\tau\asymp 1$, $\varepsilon\le \mu^{-1}$, $\mu \le h^{-1}$ in $\bR^2$ estimates
\begin{gather}
\gamma^{-2} |F_{t\to \hbar^{-1}\tau}
\bigl(\bigl(\bar{\chi}_T(t) -\bar{\chi}_{\bar{T}}\bigr)\Gamma (U\psi_\gamma)\bigr)|\le C\mu \varepsilon^{-1} R^\T(\gamma)
\label{16-1-55}\\
\shortintertext{and}
\gamma^{-2}
|F_{t\to \hbar^{-1}\tau} \bigl(\bar{\chi}_T(t) \Gamma (U\psi_\gamma)\bigr)|\le
C\mu h^{-1}+ C\mu \varepsilon^{-1} R^\T(\gamma)
\label{16-1-56}
\end{gather}
hold as  $\bar{T}= \epsilon_0 \le T \le \infty$  where
\begin{multline}
R^\T(\gamma)= \\
\left\{\begin{aligned}
&\mu^{\frac{1}{2}} h^{-\frac{1}{2}}\qquad
&&\text{as\ \ } \gamma\le h,\\
&\mu^ {\frac{1}{2}}\gamma^{-\frac{1}{2}}\qquad
&&\text{as\ \ } \min( h/\varepsilon ,\mu^{-1}) \ge \gamma\ge h ,\\
&\gamma^{-1}\qquad
&&\text{as\ \ }  h\varepsilon ^{-1}\ge \gamma \ge \mu^{-1},\\
&\varepsilon h^{-1}\min\bigl((\hbar/\varepsilon)^{\frac{1}{2}},1\bigr)\times
(h/\varepsilon\gamma)^l\quad
&&\text{as\ \ } \gamma\ge  h\varepsilon ^{-1}.
\end{aligned}\right.
\label{16-1-57}
\end{multline}
\end{proposition}

Then the standard Tauberian arguments imply

\begin{corollary}\label{cor-16-1-10}
In frames of proposition~\ref{prop-16-1-8}
\begin{gather}
\gamma^{-2}
|\Gamma \Bigl(\bigl(e(.,.,\tau)-e^\T (.,.,\tau)\bigr)\psi_\gamma\Bigr) |\le
C \bigl(\varepsilon h^{-1}+R^\T(\gamma)\bigr)\ell^{-1}.
\label{16-1-58}
\end{gather}
\end{corollary}

\section{Calculations}
\label{sect-16-1-3}

\subsection{Correction term}
\label{sect-16-1-3-1}

We want however more explicit expression for (\ref{16-1-52}) which we rewrite as
\begin{multline}
\hbar^{-1}\int_{-\infty}^\tau
\Bigl(F_{t\to \hbar^{-1}\tau} \bar{\chi}_{\bar{T}}(t) U(x,x,t)\Bigr)\,d\tau+\\
\int_{-\infty}^\tau
\Bigl(F_{t\to \hbar^{-1}\tau} (-it)^{-1}
\bigl(\bar{\chi}_{T}(t) -\bar{\chi}_{\bar{T}}(t)\bigr) U(x,x,t)\Bigr)\,d\tau
\label{16-1-59}
\end{multline}
and as $\mu \le h^{\delta-1}$   the first term modulo negligible is a Weyl expression $h^{-2}\cN^\W_x \Def (4\pi)^{-1}h^{-2} (\tau-V(x))$
(provided $V(x)\le \tau -\epsilon$) and the second expression is a \emph{correction term\/}\index{correction term} $h^{-2}\cN_{x, \corr}$ and our purpose is either to estimate or calculate it.

Let us apply the stationary phase method whenever in makes sense; otherwise we just temporarily skip the corresponding zone. Then the main part of the answer is
\begin{multline}
\sum_{k\in \fZ} \frac{1}{4\pi\sqrt\pi}\mu ^2  \hbar^{-\frac{1}{2}}
(t_k\sin (t_k))^{-1} |\varphi'' (t_k)|^{-\frac{1}{2}}\times \\
\exp\Bigl( \frac{i\pi}{4}\sign \varphi'' (t_k) +   i\hbar^{-1}\varphi(t_k)\Bigr)
\label{16-1-60}
\end{multline}
with
\begin{gather}
\varphi'' (t_k) \sim   2\tau  \cot(t_k)\qquad \text{as\ \ }
|\cos (t_k)|\ge C \varepsilon ^{\frac{1}{2}}
\label{16-1-61}\\
\shortintertext{and}
\fZ\Def \bigl\{k\ne 0:\ {\color{gray} |\sin (t_k)|\ge \hbar,} \
|\cos (t_k)|\ge
C\max \bigl(\varepsilon ^{\frac{1}{2}}, \hbar^{\frac{1}{3}}\bigr)\bigr\}
\label{16-1-62}
\end{gather}
where dimmed restriction is just temporary.

\subsection{Estimating correction term}
\label{sect-16-1-3-2}

First, let us estimate  the correction term. As $|\sin (t_k)|\ge \hbar $,
$|\cos (t_k)|\ge \epsilon$ we have $t_k\asymp  k$,
$\sin (t_k)\asymp \varepsilon k$,  $\varphi'' \asymp \varepsilon^{-1} k^{-1}$ and the  contribution of $k$-th tick does not exceed
$C\mu^{\frac{3}{2}}\varepsilon^{-\frac{1}{2}} h^{-\frac{1}{2}}|k|^{-\frac{3}{2}}$.

In comparison with the previous subsection we acquired factor $|k|^{-1}$ due to the factor $(-it)^{-1}$ in (\ref{16-1-59}) and it is a game-changer: summation with respect to $k$ returns $C\mu^{\frac{3}{2}}\varepsilon^{-\frac{1}{2}} h^{-\frac{1}{2}}|k|^{-\frac{1}{2}}$ with the \emph{smallest\/} possible $k$: as
$\hbar\le\varepsilon $ we sum from $|k|=1$ and get  $C\mu^{\frac{3}{2}}\varepsilon^{-\frac{1}{2}} h^{-\frac{1}{2}}$; as
$\hbar\ge \varepsilon $ we sum from $|k|=\bar{k}\Def \hbar \varepsilon^{-1}$ and get $C\mu h^{-1}$ which is the trivial estimate.

Further, contribution of $k$-th tick to the error is
$C\mu h^{-1}\times  (\hbar /\varepsilon |k|)^{r+\frac{1}{2}}|k|^{-1}$ if we use $r$-term stationary phase approximation.  Then summation with respect to $k:|k|\ge 1$ returns $ C\mu h^{-1}  (\hbar /\varepsilon )^{r+\frac{1}{2}}$ as
$\hbar \le \varepsilon$.

On the other hand, as $\hbar \ge \varepsilon$ summation with respect to $k:|k|\ge\bar{k}$ returns $C\mu h^{-1}$ and thus this approximation does not make any sense. Therefore

\begin{claim}\label{16-1-63}
We assume until the end of this subsection that
$\hbar\le\varepsilon$.
\end{claim}

This assumption takes care of many problems: it eliminates the near-pole singular zone and away from the stationary points we can many times integrate by parts; one can see easily that contribution of zones (\ref{16-1-37}) and (\ref{16-1-39}) does not exceed $C \mu h^{-1} (\hbar/\varepsilon)^l$ with arbitrarily large $l$.

Therefore we need an additional analysis only in the  the near-equator zone
$\{t:\ |\cos (t)|\le \epsilon,\ |t-t^*|\le \epsilon \varepsilon^{-1} \}$. It is easy to estimate its contribution to the correction term: in comparison with the previous subsection we gain factor $\varepsilon$ (from factor $(-it)^{-1}$) and we get $C\mu^2\varepsilon^{-\frac{1}{2}} h^{-\frac{1}{2}} \times \varepsilon =
C\mu^2\varepsilon^{\frac{1}{2}} h^{-\frac{1}{2}}$ which is less than the contribution of the near-polar zone.  Thus  we arrive to

\begin{proposition}\label{prop-16-1-11}
For a pilot-model operator \textup{(\ref{16-1-1})} with $\tau\asymp 1$, $\varepsilon\le \mu^{-1}$, $\hbar\le\varepsilon $

\medskip\noindent
(i) Correction term $h^{-2}\cN_{x,\corr}$ does not exceed
$C\mu^{\frac{3}{2}}\varepsilon^{-\frac{1}{2}} h^{-\frac{1}{2}}$ and therefore in frames of proposition~\ref{prop-16-1-8}
\begin{equation}
\R^\W_x\Def |e^\T(x,x,0) - h^{-2}\cN_x^\W|\le
C\mu^{-1}h^{-1}+ C\mu^{\frac{3}{2}}\varepsilon^{-\frac{1}{2}} h^{-\frac{1}{2}}.
\label{16-1-64}
\end{equation}
(ii) In particular,   $\R^\W_x\le C\mu^{-1}h^{-1}$ as
$ \varepsilon\ge  \mu^5 h $.
\end{proposition}

\subsection{Calculating correction term}
\label{sect-16-1-3-3}

Situation in the near-equator zone becomes more delicate if we want to estimate an error in $r$-term stationary phase approximation. Then the contribution of $k$-th tick does not exceed
\begin{equation*}
C \mu \varepsilon h^{-1}(\hbar/|\cos (t_k)|^3)^{r+\frac{1}{2}} \asymp
C\mu \varepsilon h^{-1} \bigl(\hbar/ (\varepsilon |k-k^*|)^{\frac{3}{2}} \bigr)^{r+\frac{1}{2}}
\end{equation*}
which as $r\ge 1$ sums to $C\mu \varepsilon h^{-1}
\bigl(\hbar/ (\varepsilon |k-k^*|)^{\frac{3}{2}} \bigr)^{r+\frac{1}{2}}|k-k^*|$
calculated at the least possible value of $|k-k^*|$; so we get
$C\mu \varepsilon h^{-1} \bigl(\hbar/\varepsilon^{\frac{3}{2}} \bigr)^{r+\frac{1}{2}}$
as we sum for $|k-k^*|\ge 1$, i.e. as $\varepsilon \ge \hbar^{\frac{2}{3}}$.

On the other hand, we get $C \mu \varepsilon h^{-1} \times \hbar^{\frac{2}{3}}\varepsilon^{-1} =C \mu^{\frac{5}{3}}h^{-\frac{1}{3}}$
as we sum for  $|k-k^*|\ge \hbar^{\frac{2}{3}}\varepsilon^{-1}$, i.e.
as $\varepsilon \le \hbar^{\frac{2}{3}}$. This takes care of zone (\ref{16-1-45}); analysis of zones covered by (\ref{16-1-44}) is easy and its contribution to the error is lesser.

Finally, the contribution to the error of zone (\ref{16-1-46}) does not exceed
$C \mu\varepsilon h^{-1}\max(\varepsilon , \varepsilon^{-1}\hbar , \hbar^{\frac{1}{3}})$ where in comparison with the analysis of the previous subsection  we acquired factor $\varepsilon$.

Therefore,  contribution of the near-equator zone to the error does not exceed
\begin{equation}
C\mu^{-1}h^{-1} + \left\{\begin{aligned}
&C\mu \varepsilon h^{-1} \bigl(\hbar/\varepsilon^{\frac{3}{2}} \bigr)^{r+\frac{1}{2}}+ C\mu^{\frac{4}{3}} \varepsilon h^{-\frac{2}{3}}\qquad
&&\text{as\ \ } \varepsilon\ge \hbar^{\frac{2}{3}},\\
&C\mu^{\frac{5}{3}}h^{-\frac{1}{3}}
&&\text{as\ \ } \varepsilon\le \hbar^{\frac{2}{3}}
\end{aligned}\right.
\label{16-1-65}
\end{equation}
and we arrive to

\begin{proposition}\label{prop-16-1-12}
For a pilot-model operator \textup{(\ref{16-1-1})} with  $\tau\asymp 1$, $\varepsilon\le \mu^{-1}$, $\hbar\le \varepsilon$ correction term $h^{-2}\cN_{x,\corr}$ is delivered with $r$-term stationary phase approximation $h^{-2}\cN_{x,\corr(r)}$ with an error not exceeding
\begin{multline}
C\mu h^{-1} \bigl(\hbar/\varepsilon \bigr)^{r+\frac{1}{2}} +
C\mu \varepsilon h^{-1} \bigl(\hbar/\varepsilon^{\frac{3}{2}} \bigr)^{r+\frac{1}{2}}+ C\mu^{\frac{4}{3}} \varepsilon h^{-\frac{2}{3}} + C\mu^{-1}h^{-1}
\\
\text{as\ \ } \varepsilon\ge \hbar^{\frac{2}{3}}
\label{16-1-66}
\end{multline}
and
\begin{equation}
C\mu h^{-1} \bigl(\hbar/\varepsilon\bigr)^{r+\frac{1}{2}} +
\mu^{\frac{5}{3}}h^{-\frac{1}{3}}+ C\mu^{-1}h^{-1}
\qquad\text{as\ \ } \varepsilon\le \hbar^{\frac{2}{3}}.
\label{16-1-67}
\end{equation}
\end{proposition}

\subsection{Micro-averaging}
\label{sect-16-1-3-4}

In the current framework (including assumption $\hbar \le \varepsilon$) micro-averaging does not improve estimate of the correction term or the error as
$|\cos(t)|\ge \epsilon_0$ in it unless $\gamma\ge h \varepsilon^{-1} $ (in comparison with the remainder estimate which was improved as  $\gamma\ge h$). Then, as $\gamma \ge h \varepsilon^{-1}$  in estimates (\ref{16-1-66}), (\ref{16-1-67}) the first term   gains factor $(h /\varepsilon\gamma)^l$, the last term is preserved and all other terms gain factor $(  h/ \gamma)^l\le \varepsilon^l$ which effectively kills them and we arrive to estimates (\ref{16-1-68}), (\ref{16-1-71}) below.

As $ h\le \gamma \le h \varepsilon ^{-1}$ the first and the last terms do  not improve but all further terms gain factor $( h/ \gamma)^l$ and we arrive to estimates (\ref{16-1-69}) and (\ref{16-1-70}) below. Finally, for
$\gamma\le \mu h$ micro-averaging brings no improvement.

\begin{proposition}\label{prop-16-1-13}
(i) For a pilot-model operator \textup{(\ref{16-1-1})} with  $\tau\asymp 1$,
$\varepsilon\le\mu^{-1}$, $\varepsilon\ge \hbar$,
\begin{multline}
\R^{\W}_{x(r),\gamma}\Def \gamma^{-2}
|\int \Bigl(e^\T (x,x,0)  -
h^{-2} \bigl(\cN_x^\W +\cN_{x,\corr(r)}\bigr)\Bigr)\psi_\gamma\,dx |\le \\
C\mu h^{-1} \bigl(\hbar/\varepsilon \bigr)^{r+\frac{1}{2}}
\bigl( h/\varepsilon  \gamma)^l + C\mu^{-1}h^{-1}
\qquad\text{as\ \ }  \gamma\ge h/\varepsilon,
\label{16-1-68}
\end{multline}
\begin{multline}
\R^{\W}_{x(r),\gamma}\le
C\mu h^{-1} \bigl(\hbar/\varepsilon)^{r+\frac{1}{2}}+
C\mu \varepsilon\Bigr( h^{-1}\bigl(\hbar/\varepsilon^{\frac{3}{2}}\bigr)^{r+\frac{1}{2}} +
\mu ^{\frac{1}{3}} h^{-\frac{2}{3}}\Bigr)\bigl( h/\gamma\bigr)^l
\\
+C\mu^{-1}h^{-1}\qquad \text{as\ \ }  \varepsilon\ge \hbar^{\frac{2}{3}},\
h \le  \gamma \le h/\varepsilon
\label{16-1-69}
\end{multline}
and
\begin{multline}
\R^{\W}_{x(r),\gamma}\le
C\mu h^{-1} \bigl(\hbar/\varepsilon)^{r+\frac{1}{2}}+ C\mu^{\frac{5}{3}}h^{-\frac{1}{3}}\bigl( h/\gamma \bigr)^l +C\mu^{-1}h^{-1}\\
\text{as\ \ } \varepsilon\le \hbar^{\frac{2}{3}},\  h\le \gamma\le  h/\varepsilon,
\label{16-1-70}
\end{multline}
(ii) For a pilot-model operator \textup{(\ref{16-1-1})} with  $\tau\asymp 1$,
$\varepsilon\le\mu^{-1}$, $\varepsilon\le \hbar$
\begin{multline}
\R^{\W}_{x,\gamma}\Def \gamma^{-2}
|\int \Bigl(e^\T (x,x,0)  - h^{-2} \cN_x^\W \Bigr)\psi_\gamma \,dx |\le \\
\mu h^{-1}  \bigl( h/\varepsilon  \gamma)^l + C\mu^{-1}h^{-1}
\qquad\text{as\ \ }  \gamma\ge h/\varepsilon.
\label{16-1-71}
\end{multline}
\end{proposition}

\section{Strong and superstrong magnetic field}
\label{sect-16-1-4}

\subsection{Tauberian estimate}
\label{sect-16-1-4-1}

Consider strong  $\mu h\asymp 1$ and superstrong $\mu h \gtrsim 1$ magnetic field for a Schr\"odinger-Pauli pilot-model operator
\begin{equation}
A\Def h^2 D_1^2 + (hD_2-\mu x_1)^2+2\alpha x_1-\fz \mu h
\label{16-1-72}
\end{equation}
which alternatively means that we assume that $|\tau - \fz\mu h|\le C$ in the framework of the standard pilot-model (\ref{16-1-1}). Formulae (\ref{16-1-15})--(\ref{16-1-16}) imply that

\begin{proposition}\label{prop-16-1-14}
For a pilot-model operator \textup{(\ref{16-1-1})} with $\mu h\gtrsim 1$ in $\bR^2$ as ${|x|\le C_0}$, $|y|\le C_0$

\medskip\noindent
(i) $e(x,y,\tau)\equiv 0 \mod O(\mu^{-\infty})$  as \
$\tau \le \mu h -\epsilon_0$ (the lower spectral gap);

\medskip\noindent
(ii) $e(x,y,\tau)\equiv e(x,y,\tau') $ and
$\partial_\tau e(x,y,\tau)\equiv 0 \mod O(\mu^{-\infty})$ as\\[2pt]
\hphantom{\qquad\qquad } $(2n-1)+\epsilon_0\le \tau'< \tau \le (2n+1)\mu h -\epsilon_0$ (other spectral gaps);

\medskip\noindent
(iii) $e(x,x,\tau)- e(x,x,\tau') \asymp \mu h^{-1}$ as
$\tau ' <(2n-1)\mu h - \epsilon_0 $ and \newline
$\tau > (2n-1)\mu h + \epsilon_0 $;

\medskip\noindent
(iv) As $\varepsilon>0$ the following estimates hold
\begin{gather}
|\partial_\tau e(x,y,\tau)|\le
C\mu ^{\frac{1}{2}} \varepsilon^{-1}h^{-\frac{3}{2}}
\label{16-1-73}\\
\shortintertext{and}
|F_{t\to \hbar^{-1}\tau} U(x,y,t) |\le
C\mu ^{\frac{3}{2}}\varepsilon^{-1}h^{-\frac{1}{2}}
\label{16-1-74}
\end{gather}
and these estimate are sharp as $x=y$ and $\tau$ is close to Landau levels $(2n+1)\mu h$.
\end{proposition}

Further, as before $T^*= \epsilon_0\mu\varepsilon^{-1} \ell$. Then we immediately arrive to

\begin{corollary}\label{cor-16-1-15}
For a self-adjoint operator in domain $X$, $B(0,\ell)\subset X\subset \bR^2$, $\ell\ge C_0\mu^{-\frac{1}{2}}h^{\frac{1}{2}}$, coinciding in $B(0,\ell)$ with the pilot-model operator \textup{(\ref{16-1-72})} with
$\tau\asymp 1$, $ \varepsilon <0$,  $\mu \ge h^{-1}$

\medskip\noindent
(i) Statements (i)-(iv) of proposition \ref{prop-16-1-14} remain true;

\medskip\noindent
(ii) Formula \textup{(\ref{16-1-15})}  holds modulo $O(\mu^{\frac{1}{2}}h^{-\frac{1}{2}}\ell^{-1})$;

\medskip\noindent
(iii) Moreover, formula \textup{(\ref{16-1-15})} holds modulo $O(\mu^{-\infty})$ \\[2pt] \hphantom{\qquad\qquad } as
$(2n-1)+\epsilon_0\le \tau \le (2n+1)\mu h -\epsilon_0$;
\end{corollary}

\subsection{Micro-averaging}
\label{sect-16-1-4-2}

Let us consider micro-averaging. First let us estimate Fourier transform where as before we rescaled $t\mapsto \mu t$; formulae (\ref{16-1-15})--(\ref{16-1-16}) imply immediately

\begin{proposition}\label{prop-16-1-16}
For a pilot-model operator \textup{(\ref{16-1-72})} in $\bR^2$  with
$\mu h\ge 1$ and $\fz=(2n+1)$
\begin{gather}
\gamma^{-2}| F_{t\to h^{-1}\tau} \bar{\chi}_T(t)\Gamma (U\psi_\gamma )|\le
C\mu^2 + C \mu \varepsilon^{-1} R^\T(\gamma)
\label{16-1-75}\\
\shortintertext{with}
R^\T (\gamma)= \left\{\begin{aligned}
&\mu^{\frac{1}{2}}  h^{-\frac{1}{2}}\qquad
&&\text{as\ \ } \gamma\le \mu^{-\frac{1}{2}} h^{\frac{1}{2}},\\
& \gamma^{-1}
\qquad
&&\text{as\ \ } \gamma\ge \mu^{-\frac{1}{2}} h^{\frac{1}{2}}.
\end{aligned}\right.
\label{16-1-76}
\end{gather}
\end{proposition}

\begin{corollary}\label{cor-16-1-17}
In frames of corollary~\ref{cor-16-1-15}
\begin{equation}
\gamma^{-2} |\Gamma \bigl(e(.,.,\tau)-e^\T (.,.,\tau)\bigr)\psi _\gamma |\le
C \bigl(\mu\varepsilon +R^\T (\gamma)\bigr)\ell^{-1}.
\label{16-1-77}
\end{equation}
\end{corollary}

\section{Weyl and magnetic Weyl approximation}
\label{sect-16-1-5}

We can try to use a host of the different approximations but restrict ourselves now to Weyl approximation, which makes sense only as $\mu h\le 1$. Recall that
\begin{gather}
|e^\T (x,x,\tau)-h^{-2}\cN^\W_x(\tau)|\le C R^\W,
\label{16-1-78}\\[3pt]
\gamma^{-2}
|\int \bigl(e^\T (x,x,\tau)-h^{-2}\cN^\W_x(\tau)\bigr)\psi_\gamma \,dx |\le
C R^\W(\gamma)\label{16-1-79}
\end{gather}
with
\begin{align}
R^\W &=
\left\{\begin{aligned}
&\mu ^{\frac{3}{2}}\varepsilon^{-\frac{1}{2}} h^{-\frac{1}{2}}
\qquad&&\text{as\ \ } \varepsilon\ge \hbar ,\\
&\mu h^{-1}\qquad&&\text{as\ \ }  \varepsilon\le \hbar,
\end{aligned}\right.
\label{16-1-80}\\[2pt]%
R^\W(\gamma)&=\left\{\begin{aligned}
&R^\W  \qquad&&\text{as\ \ } \gamma \le h/\varepsilon,\\
&R^\W  (  h/\varepsilon \gamma)^l \qquad&&\text{as\ \ } \gamma \ge h/\varepsilon.
\end{aligned}\right.
\label{16-1-81}
\end{align}
On the other hand, let us apply instead the magnetic Weyl approximation. Obviously, without micro-averaging $\cN^\MW_x(\tau)$ cannot produce uniform with respect to $\tau$ remainder estimate better than $\mu h^{-1}$ due to its own jumps. Meanwhile with an averaging (i.e. a micro-averaging with $\gamma=1$) $\cN^\MW_x(\tau)$ provides a better remainder estimate $O(\mu ^{-1}h^{-1})$ (see Chapter~\ref{book_new-sect-13}; for a pilot-model we need only assume that $\varepsilon \ge \mu^{-l}$). So, as $\mu  h\le 1$ the magnetic Weyl approximation may provide a better remainder estimate as $\gamma$ is not too small enough but for small $\gamma$ Weyl approximation is better. Let us investigate this.

\begin{proposition}\label{prop-16-1-18}
(i) For a pilot-model operator \textup{(\ref{16-1-1})} as $\mu h\lesssim 1$
\begin{multline}
\R^\MW (\gamma) \Def \gamma^{-2} | \int \bigl(e^\T (x,x,0)  - h^{-2}\int \cN^\MW \bigr)\psi_\gamma\,dx|\le \\
C \mu^{-1}h^{-1} +C R^\MW(\gamma)
\label{16-1-82}
\end{multline}
with
\begin{equation}
R^\MW(\gamma)= \left\{\begin{aligned}
&\mu^{-1} h^{-1}\gamma^{-2} (h/\varepsilon \gamma)^l\qquad
&& \text{as\ \ }\gamma \ge \max(\mu^{-1}, h/\varepsilon),\\[2pt]
&\mu^{-1} h^{-1}\gamma^{-2}\qquad
&& \text{as\ \ }\mu^{-1}\le  \gamma \le h/\varepsilon,\\[2pt]
&\mu h^{-1}( h/\varepsilon \gamma)^l\qquad && \text{as\ \ } h/\varepsilon  \le  \gamma \le \mu ^{-1},\\[2pt]
&\mu h^{-1} \qquad && \text{as\ \ }\gamma\le \min(\mu^{-1},h/\varepsilon);
\end{aligned}\right.
\label{16-1-83}
\end{equation}

\medskip\noindent
(ii) For a pilot-model operator \textup{(\ref{16-1-72})} as $\mu h\gtrsim 1$ and $\fz=2n+1$ estimate
\begin{equation}
\R^\MW (\gamma) \le C +C R^\MW(\gamma)
\label{16-1-84}
\end{equation}
holds with
\begin{equation}
R^\MW(\gamma)= \left\{\begin{aligned}
&\gamma^{-2}\qquad
&& \text{as\ \ } \gamma \ge \mu^{-\frac{1}{2}}h^{\frac{1}{2}},\\
&\mu   h^{-1} \qquad
&& \text{as\ \ }\gamma\le \mu^{-\frac{1}{2}}h^{\frac{1}{2}}.
\end{aligned}\right.
\label{16-1-85}
\end{equation}
\end{proposition}

\begin{proof}
We need to investigate only case $\mu ^{-1}\le \gamma \le h\varepsilon^{-1}$ as in all other cases we can easily pass from results for $\R^\W (\gamma)$.

Note that $h^{-2}\cN^\MW_y(\tau)$ delivers $e(y,y,\tau)$ for a pilot-model operator $\bar{A}\Def\bar{A}_y=\bar{A}_y(x,D_x)$ defined by (\ref{16-1-1}) or (\ref{16-1-72}) with $\alpha x_1$ frozen at the point $y$ (with respect to which we integrate later); so we consider our pilot-model operator as a perturbation of this one.

Let us plug into Tauberian formula the Schwartz kernel of
\begin{equation}
e^{i\hbar^{-1}tA}=e^{i\hbar^{-1}t\bar{A}}+
i\hbar^{-1} \int _0^t e^{i\hbar^{-1}t'\bar{A}}\bigl(A-\bar{A}\bigr)e^{i\hbar^{-1}(t-t')A}\,dt'
\label{16-1-86}
\end{equation}
or its iteration for $\pm t>0$
\begin{multline}
e^{i\hbar^{-1}tA}=e^{i\hbar^{-1}t\bar{A}_x}+
\sum_{1\le j\le m-1}\times \\
(i\hbar^{-1})^j
\smashoperator[lr]{\int \limits_ {\substack{  \{ \pm t_1>0, \pm t_2>0, \ldots ,\pm t_j >0,\\
t_1+t_2+\ldots +t_j=t\} }  } } e^{i\hbar^{-1}t_1\bar{A}}\bigl(A-\bar{A}\bigr)e^{i\hbar^{-1}t_2\bar{A}}
\bigl(A-\bar{A}\bigr)  \cdots \bigl(A-\bar{A}\bigr)e^{i\hbar^{-1}t_j\bar{A}}
\,dt_1dt_2\cdots dt_j +\\
(i\hbar^{-1})^m
\smashoperator[lr]{\int \limits_ {\substack{  \{ \pm t_1>0, \pm t_2>0, \ldots ,\pm t_m >0,\\
t_1+t_2+\ldots +t_j=t\} }  } } e^{i\hbar^{-1}t_1\bar{A}}\bigl(A-\bar{A}\bigr)e^{i\hbar^{-1}t_2\bar{A}}
\bigl(A-\bar{A}\bigr)  \cdots \bigl(A-\bar{A}\bigr)e^{i\hbar^{-1}t_m A}
\,dt_1dt_2\cdots dt_m.
\tag*{$\textup{(\ref*{16-1-86})}_m$}\label{16-1-86-m}
\end{multline}

So let us plug \ref{16-1-86-m} into our expression with $\chi_T(t)$ instead of $\bar{\chi}_T(t)$.  Then our standard methods imply that $j$-th term
($j=0,\ldots, m$) of what we got does not exceed
\begin{equation}
C  \mu h^{-1}  \bigl( T \|A-\bar{A}_x\| \hbar^{-1} \bigr)^j \times
\bigl(1+ \frac{T}{\bar{T_2}}\bigr)^{-l} \bigl(1+ \frac{T}{\bar{T}_1}\bigr)^{-l}
\label{16-1-87}
\end{equation}
with $\bar{T}_1 =C_0\mu \varepsilon^{-1} \bar{\gamma} $ where
$\bar{\gamma} \Def \max(\mu^{-1},\mu^{-\frac{1}{2}} h^{\frac{1}{2}})$   and
$\bar{T}_2=  h /\varepsilon  \gamma $; we leave easy details to the reader.

Then  we can replace the norm by the ``effective norm''  which is a norm restricted to $\mu\bar{\gamma}$-vicinity of $x$ (in rescaled coordinates) and equal to $C\alpha \bar{\gamma}$.

Note that $\bar{T}_2\le \bar{T}_1$ as $\gamma \ge \bar{\gamma}$; then summation with respect to $T$ returns
\begin{equation*}
C  \mu h^{-1}  \bigl( \bar{T}_2 \|A-\bar{A}_x\| \hbar^{-1} \bigr)^j \asymp
C  \mu h^{-1}  \bigl( \bar{\gamma}/\gamma \bigr)^j
\end{equation*}
which provides required estimate (\ref{16-1-83}) or (\ref{16-1-84}) for all terms with $j\ge 2$.

Finally note that while $\bar{A}_y(x-y,D_x)$ is even with respect to ${(x-y , D_x)}$ (we can always achieve it by the gauge transformation), perturbation ${(A-\bar{A}_y)(x-y,D_x)}$ is odd and therefore only terms with even $j$ survive as we plug $x=y$. This takes care of term with $j=1$.
\end{proof}

\begin{remark}\label{rem-16-1-19}
(i) For $\mu h\le 1$ magnetic Weyl formula provides better approximation as
$\mu ^{-1} \le \gamma \le   h/\varepsilon $; otherwise Weyl formula provides either better or equally bad ($\mu h^{-1}$) approximation.

\medskip\noindent
(ii) For $\mu h\ge 1$ magnetic Weyl formula provides  approximation rather than just main term estimate as $\gamma \gg \mu^{-\frac{1}{2}}h^{\frac{1}{2}}$.
\end{remark}

\section{Geometric interpretation}
\label{sect-16-1-6}

\begin{remark}\label{rem-16-1-20}
(i) As we mentioned $t_k$ are length of the classical loops at point $0$ (or $x$ if we consider $e(x,x,\tau)$; then $t_k=t_k(x)$). While picture of the fixed trajectory and a point $x$ on it (Figure~\ref{fig-16-2a}) is more geometrically appealing, the correct picture is of the fixed point and different trajectories passing through it  (Figure~\ref{fig-16-2b}); due to assumption $\alpha\ne 0$ ``radii'' of these trajectories differ by $O(\mu^{-1})$.
\begin{figure}[h]
\subfloat[Intuitive]{
\includegraphics{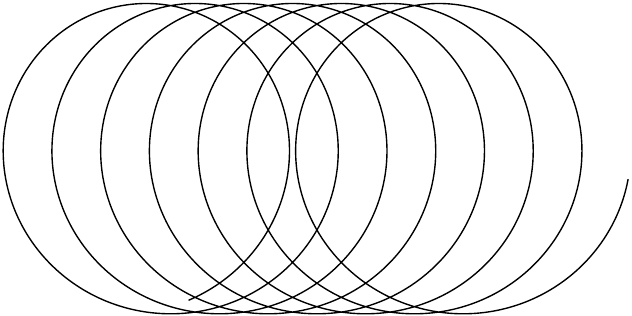}
\label{fig-16-2a}}
\qquad
\subfloat[Correct picture]{
\includegraphics{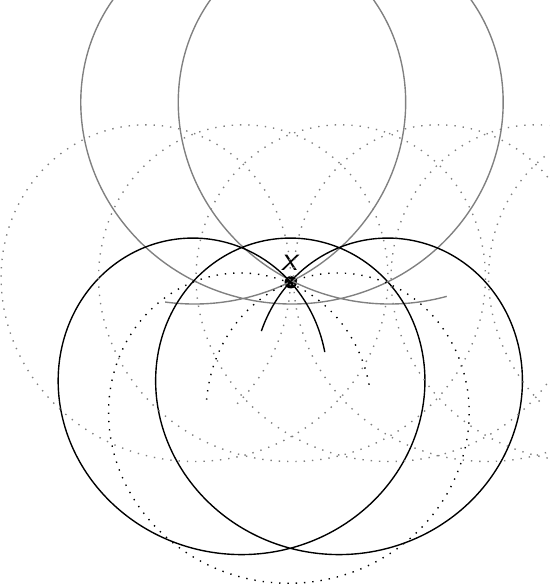}
\label{fig-16-2b}}
\caption{\label{fig-16-2} On the left more geometrically appealing picture, when $x$ moves along given trajectory; on the right the correct one when we consider different trajectories passing through a fixed point.}
\end{figure}

Therefore, when we talk here about \emph{loop\/} we mean a trajectory looping at the fixed point $x$ rather the self-intersections of the fixed trajectory.

Similarly, when we talk about \emph{pole\/} or \emph{equator\/} we mean not $x$ near a pole or an equator of the fixed trajectory but rather trajectory such that $x$ is near its pole or equator.

So, all further remarks should be interpreted correctly.

\medskip\noindent
(ii) How to interpret our results from heuristic uncertainty principle? Obviously these results mean that as $\varepsilon \gg \hbar$ then the majority of trajectories on Figure~\ref{fig-16-2b} do not return to the original point $x=0$ after $k=1$ tick\footnote{\label{foot-16-5} In the sense that the difference between two points is observable.}  while otherwise this happens only after $k$ ticks with $k\varepsilon \gg \hbar$. As the spatial shift is $\varepsilon k$ (after rescaling) this means that the thickness of the trajectory is
$\asymp \hbar$ after rescaling (and therefore $\asymp h$ before rescaling).

It also means that if $|\sin (\theta)|\asymp 1$ where $\theta$ is the polar angle (between direction of the trajectory at $0$ and the drift direction $(0,1)$) then we can take an interval in $\theta$ of the magnitude $1$, so $\theta$ is a dual variable to $x_2$. Conversely, with the exception of the interval of the length $\hbar/\varepsilon |k|$ around $\theta_k$ (where $\theta_k$ corresponds to the classical trajectory returning to $0$ after $k$ ticks) trajectories do not return at $0$; we can rewrite this condition as
\begin{equation}
|\sin(\theta-\theta_k) |\gg \hbar/(|k|\varepsilon).
\label{16-1-88}
\end{equation}
This characteristic length matches to the stationary phase estimates we derived rigorously.

\begin{figure}[h!]
\centering
\subfloat[]{
\includegraphics{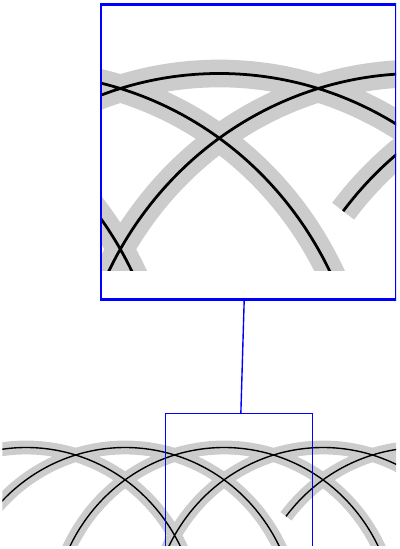}
\label{fig-16-3a}
}\
\subfloat[]{
\includegraphics{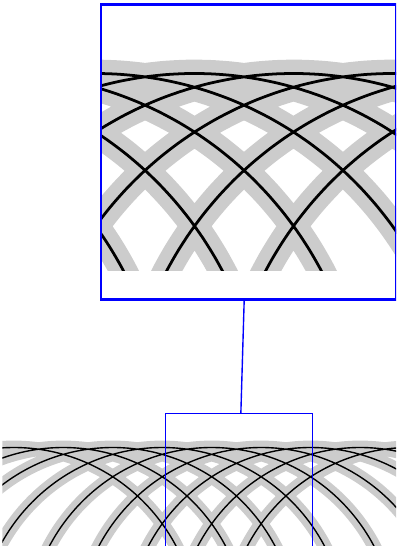}
\label{fig-16-3b}
}
\
\subfloat[]{
\includegraphics{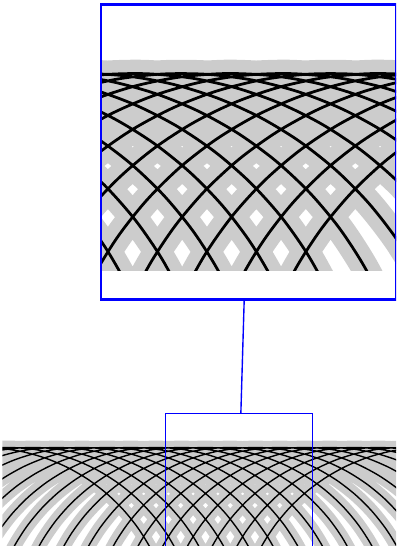}
\label{fig-16-3c}
}
\caption{\label{fig-16-3} Looping near pole from microlocal point of view}
\end{figure}

\noindent
(iii) Near equator situation changes drastically: spacing between self-in\-ter\-sections is $\asymp \varepsilon |\cos (\theta)|^{-1}$ as
$|\cos (\theta)|\ge \epsilon_0\varepsilon ^{\frac{1}{2}}$ and there are  exceptional $0,1,2$ self-in\-ter\-sections with
$|\cos (\theta)|\le \epsilon_0\varepsilon ^{\frac{1}{2}}$. Uncertainty principle may prevent us to know how many of them are no matter how ``not large'' $\mu$ is.
\begin{figure}[h!]
\centering
\subfloat[]{
\includegraphics{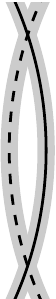}
\label{fig-16-4a}
}\qquad
\subfloat[]{
\includegraphics{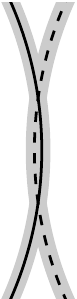}
\label{fig-16-4b}
}
\qquad
\subfloat[]{
\includegraphics{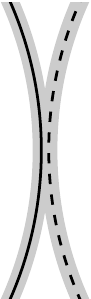}
\label{fig-16-4c}
}
\qquad
\subfloat[]{
\includegraphics{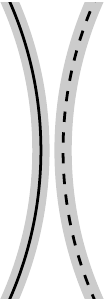}
\label{fig-16-4d}
}
\caption{\label{fig-16-4} Self-intersections near equator. The first winding shown by a solid line, the closest to it near equator by a dashed one.}
\end{figure}

Consequent self-intersections are distinguishable as
$|\cos (\theta)|\ge \hbar ^{\frac{1}{3}}$ according to our calculations.
\end{remark}

\begin{Problem}\label{Problem-16-1-21}
We know that estimates (\ref{16-1-73}), (\ref{16-1-74}) are sharp as
$\mu h \gtrsim 1$. Are these estimates sharp as $\mu h \ll 1$?
\end{Problem}

\chapter{Pointwise asymptotics: general $2\D$-operators}
\label{sect-16-2}

Let us recreate the approximation similar to the one of the previous section for general operators assuming that $\hbar \le \epsilon $ i.e.
\begin{equation}
\mu \le \epsilon  h^{-1}.
\label{16-2-1}
\end{equation}

We are interested in $|t|\le C_0\mu^{-1}\varepsilon^{-1}$  for the original (not rescaled) operator as for
$C_0\mu^{-1}\varepsilon^{-1}\le |t|\le \mu^{-1} T^*\Def \epsilon_0 \varepsilon^{-1} \ell$ dynamics leaves $B(0,c_0\mu^{-1})$ but remains in $B(0,\ell)$. Here

\begin{claim}\label{16-2-2}
$\varepsilon = \alpha^{-1}\mu^{-1}$ as we assume that
$|\nabla V/F|\asymp \alpha$ in $B(0,\ell)$ with
\begin{equation}
C_1 \mu^{-1} \le \alpha \le 1\qquad (\implies C_0\mu^{-2}\le
\varepsilon\le \mu^{-1})
\label{16-2-3}
\end{equation}
and  $\ell \ge c_0\mu^{-1}$ with sufficiently large constants $C_1,c_0$.
\end{claim}

\section{Classical dynamics}
\label{sect-16-2-1}

First we consider classical dynamics starting from point $\mathsf{x}=0$. So far we use not-rescaled $x,t$. Let us freeze $g^{jk}$ and $F$ there and replace $V_j$ and $V$ by their linear germs:
\begin{multline}
\bar{g}^{jk}=g^{jk}(0),\quad
\bar{V}(x)=V(0)+ \langle \nabla V(0),x\rangle,\\
\bar{V}_j(x)= V_j(0)+ \langle \nabla V_j(0),x\rangle \quad(\implies \bar{F}=F(0)),
\label{16-2-4}
\end{multline}
denote corresponding Hamiltonian by $\bar{a}(x,\xi)$ and all the object associated with it will have bar. We consider original our dynamic system as a perturbation.

Without any loss of the generality we can assume that
\begin{gather}
{g}^{jk}=\omega^2\updelta_{jk} \qquad (\implies a(x,\xi)=
\omega^2 (p_1^2+p_2^2)+V ), \label{16-2-5}\\[2pt]
\omega(0) =F(0)=1, \quad (\nabla\omega)(0)=0, \label{16-2-6}\\[2pt]
V(0)=0,\quad (\nabla V/F)(0)=(-\alpha,0)
\label{16-2-7}\\
\intertext{as we can achieve it by an appropriate change of variables (see footnote~\ref{foot-16-9}) and}
\quad V_1(0)=0, \quad V_2=0
\label{16-2-8}
\end{gather}
as we can achieve it by the gauge transformation.

\begin{proposition}\label{prop-16-2-1}
Under conditions \textup{(\ref{16-2-1})}--\textup{(\ref{16-2-8})} as
$|t|\le c_0\mu^{-1}\varepsilon^{-1}$

\medskip\noindent
(i) For drift flows $z_t=\Phi_t(0)$, $\bar{z}_t=\bar{\Phi}_t(0)$
\begin{equation}
z_t  =\bar{z}_t + O\bigl(\mu^{-1}\varepsilon  t^2\bigr);
\label{16-2-9}
\end{equation}
(ii) If $F=1$ then for Hamiltonian flows $(x_t,\xi_t)=\Psi_t(0,\eta)$, $(\bar{x}_t,\bar{\xi}_t)=\bar{\Psi}_t(0,\eta)$ with magnetic parameters $\mu$ and $\bar{\mu}=\mu(1+O(\mu^{-2}))$ defined by \textup{(\ref{16-2-18})} on the energy level $\tau \le c$
\begin{equation}
(x_t,\mu^{-1}\xi_t) = (\bar{x}_t, \mu^{-1}\bar{\xi}_t)
+O\bigl(\mu^{-2}  t\bigr).
\label{16-2-10}
\end{equation}
\end{proposition}

\begin{remark}\label{rem-16-2-2}
Assumption (\ref{16-2-1}) means that in $B(0,\mu^{-1})$ the constant part of
$\nabla V$ dominates over its variable part and it implies that as
\begin{gather}
\mu^{-1}\varepsilon  t^2\le c_0\mu^{-2}|t|\le \kappa \varepsilon |t|
\qquad\text{as\ \ } |t|\le c_0\mu^{-1}\varepsilon^{-1}
\label{16-2-11}\\
\shortintertext{with}
\kappa = c_0 \mu^{-2}\varepsilon^{-1}\le C_0^{-1}c_0
\label{16-2-12}
\end{gather}
and thus the perturbation of the drift is respectively small.  In particular drift line deviates from the straight one by no more than $c_0 \kappa$.

Also \emph{equator\/} is defined as a point where $x_t$ intersect its its first winding the last time deviates from $(\mu^{-1},0)$ by no more than
$c_0 \mu^{-1}\kappa$.
\end{remark}

\begin{proof}[Proof of proposition~\ref{prop-16-2-1}]
(i) Proof of assertion (i) is trivial and left to the reader. Just recall that the drift flow is defined by (\ref{book_new-13-6-5}) with an extra factor $\mu^{-1}$ in the right-hand expression and in our assumptions
$\bar{\Phi}_t: (x_1,x_2)\mapsto (x_1+\varepsilon t, x_2)$ after rescaling.

\medskip\noindent
(ii) Recall that $p_j=\xi_j-\mu V_j(x)$, $\{p_j,x_k\}=\updelta_{jk}$ and in virtue of assumption $F=1$
\begin{equation}
\{p_1,p_2\}=\mu \omega^{-2}.
\label{16-2-13}
\end{equation}
According to subsection~\ref{book_new-sect-13-2-1}  one can correct
\begin{gather}
y_1\Def x_1- \{p_1,p_2\}^{-1}p_2= x_1-\mu ^{-1}\omega^2 p_2,\label{16-2-14}\\
y_2\Def x_2+ \{p_1,p_2\}^{-1}p_1=x_2+\mu^{-1}\omega^2 p_1\notag
\end{gather}
modulo $O(\mu^{-2})$ so that corrected expressions satisfy drift equation modulo $O(\mu^{-2})$. More precisely, in the current setup
\begin{equation*}
\{a, y_1\}= \mu^{-1}\omega^2 \Bigl( \{p_2,\omega^2\} (p_1^2-p_2^2)
-2  \{p_1 ,\omega^2\} p_1p_2  + \{p_2,V\}\Bigr)
\end{equation*}
and for $y_1'=y_1+\mu^{-2}\beta_1 p_1p_2+\mu^{-2}\beta_2 (p_1^2-p_2^2)$ with
\begin{equation*}
\beta_1 =-\frac{1}{2} \omega^2\{p_2,\omega^2\}, \qquad
\beta_2 = \frac{1}{2} \omega^2 \{p_1,\omega^2\}
\end{equation*}
we have
\begin{equation*}
\{a, y'_1\}\equiv - \mu^{-1}\omega^2 \{p_2,V\}
\end{equation*}
modulo terms which are $O(\mu^{-2})$ and also homogeneous polynomials of degrees $1$ or $3$ with respect to $(p_1,p_2)$ and thus these terms could be corrected by adding to $y'_1$ terms which are $O(\mu^{-3})$ and also homogeneous polynomials of degrees $2$ or $4$; then we arrive to
\begin{equation*}
\{a, y''_k\} \equiv (-1)^k  \mu^{-1} \omega^2\{p_{3-k},V\} \mod O(\mu^{-3})\qquad
k=1,2
\end{equation*}
($k=2$ is considered in the same way as $k=1$). Note that the right-hand expressions are calculated at point $x=x_t$ rather than $y''=y''_t$; so we rewrite them as
\begin{multline*}
\{a, y''_k\}\equiv\\
(-1)^k \mu^{-1} \bigl(\omega^2 \{p_{3-k},V\}\bigr)(y'')
+ \sum_{j=1,2} \mu^{-2} (-1)^{j+k-1}
(\partial_{3-j}\omega^2 \partial_{3-k} V)(x)  p_j  \\
\mod O( \mu^{-3})\qquad k=1,2
\end{multline*}
and we can correct $y''_k$  by adding to $y''_k$ terms which are $O(\mu^{-3})$ and also homogeneous polynomials of degrees $2$ with respect to $(p_1,p_2)$.

Therefore
\begin{equation}
y'''_t\equiv z_t \mod O(\mu^{-3}t)
\label{16-2-15}
\end{equation}
where $z_t=\Phi_t(0)$.

Meanwhile (\ref{16-2-13}) implies
\begin{multline}
\frac{d\ }{dt}p_j =2(-1)^j\mu p_{3-j} - (\omega^2)_{x_j}(p_1^2+p_2^2)- V_{x_j}=\\
2 (-1)^j\mu  p_{3-j} - 2\omega_{x_j}\omega^{-1}W+ W_{x_j} ,\qquad j=1,2
\label{16-2-16}
\end{multline}
where we replaced $p_1^2+p_2^2$ by $\omega^{-2}W$, $W=\tau-V$ (no error as $\tau$ is an energy level). Let us plug
$x_1(t)\equiv z_1(t)+\mu^{-1}\omega^2(z_t) p_2$,
$x_2(t)\equiv  z_2(t)-\mu^{-1}\omega^2 (z_t) p_1$ modulo $O(\mu^{-2})$ in the right-hand expressions; we get
\begin{equation}
\frac{d\ }{dt}
\begin{pmatrix}p_1\\p_2\end{pmatrix}
= J\begin{pmatrix}p_1\\p_2\end{pmatrix} +K
\label{16-2-17}
\end{equation}
with
\begin{gather*}
J\Def 2\mu \begin{pmatrix}
-\mu^{-2}\omega^2(\omega_{x_1}\omega^{-1}W+ W_{x_1})_{x_2} &
-\bigl(1 -\mu^{-2}\omega^2(\omega_{x_1}\omega^{-1}W+ W_{x_1})_{x_1}\bigr)\\[3pt]
\bigl(1-\mu^{-2}\omega^2 (\omega_{x_2}\omega^{-1}W+ W_{x_2})_{x_2}\bigr) &
\mu^{-2}\omega^2(\omega_{x_2}\omega^{-1}W+ W_{x_2})_{x_1}
\end{pmatrix},\\[2pt]
\shortintertext{and}
K\Def\begin{pmatrix}
- 2(\omega_{x_1}\omega^{-1}W+ W_{x_1})\\[3pt]
-2\omega_{x_2}\omega^{-1}W+ W_{x_2}
\end{pmatrix}.
\end{gather*}
and here we can calculate elements of $J$ at any  point of $B(0,\mu^{-1})$ we choose (and we choose $0$) while elements of $K$ are calculated at $z_t$.
Then $J$ \,becomes a constant coefficient matrix and since $\nabla\omega(0)=0$, $\omega(0)=1$
\begin{equation*}
J=2\mu \begin{pmatrix}
-\mu^{-2} (\omega_{x_1x_2} W+ W_{x_1x_2}) &
-\bigl(1 -\mu^{-2}(\omega_{x_1x_1}W+ W_{x_1x_1})\bigr)\\[3pt]
\bigl(1-\mu^{-2}(\omega_{x_2x_2}W+ W_{x_2x_2})\bigr) &
\mu^{-2}(\omega_{x_2x_1}W+ W_{x_2x_1})
\end{pmatrix}
\end{equation*}
and its eigenvalues are $\pm 2\bar{\mu} i$ with
\begin{multline}
\bar{\mu}=\\
\mu \Bigl(\bigl(1 -\mu^{-2}(\omega_{x_1x_1}W+ W_{x_1x_1})\bigr) \bigl(1-\mu^{-2}(\omega_{x_2x_2}W+ W_{x_2x_2}) - \mu^{-4} (\omega_{x_1x_2} W+ W_{x_1x_2})^2\Bigr)^{\frac{1}{2}}\\
= \mu + \mu^{-1}(\Delta (\omega W))+O(\mu^{-3}).
\label{16-2-18}
\end{multline}
Then $J= Q ^{-1}\bar{J}Q$ with
$\bar{J}= \bar{\mu}\begin{pmatrix} 0 & -1\\ 1&0 \end{pmatrix}$ and
$Q= I+O(\mu^{-2})$.

Due to (\ref{16-2-17})
\begin{multline}
p(t)\equiv  e^{tJ}p(0) + \int_0^t e^{(t-t')J}K(z_{t'})\,dt'=\\
e^{tJ}p(0) + \int_0^t e^{(t-t')J}J^{-1} \frac{d\ }{dt'}K(z_{t'})\, dt'-
J^{-1} K(z_t)+ e^{t J}J^{-1} K(z_0).
\label{16-2-19}
\end{multline}
The second term in the right-hand expression is $O(\mu^{-1}\varepsilon t)$ as $J^{-1}=O(\mu^{-1})$ and $\frac{d\ }{dt'}K(z_{t'})=O(\varepsilon)$. Note that with $O(\mu^{-2}t)$ error one can replace in the two last terms $J$ by $\bar{J}$ calculated for a pilot-model. Finally, with an error $O(\mu^{-1}\varepsilon t)$ one can replace $K(z_t)$ by $K(z_0)$; however $K(z_0)$ and $\bar{K}(z_0)$ coincide as  $\nabla \omega(0)=0$.

Replacing $J$ by $\bar{J}$ we reduced evolution of $p$ to those of the pilot-model albeit with $\bar{\mu}$ instead of $\mu$. Therefore
$\mu^{-1} p_t\equiv \mu^{-1} \bar{p}_t \mod O( \mu^{-2}t)$. Replacing $\mu$ by $\bar{\mu}$ does not affect drift (in frames of the indicated precision).

But then $x_t$ could be found from $z_t$ and $p_t$ and their drift also is described by a pilot-model in frames of the same error and then it is true for $\mu^{-1}\xi_t$ as well.
\end{proof}

Consider now the general case, i.e. $F$ different from $1$ and variable. Then differential equations describing $(x(t),\xi(t))$ for $a(x,\xi)$ coincide with equations for $F^{-1}a(x,\xi)$ with $F\mapsto 1$ and
$\tau-V\mapsto F^{-1}(\tau-V)$ but with the ``time'' $\theta$ satisfying
\begin{multline*}
\frac{d\theta}{dt}= F(x_t)\equiv
F(z_t) +  (\omega^2 F^{-1})(z_t) \bigl(F_{x_1}(z_t) p_2 - F_{x_2}(z_t)p_1\bigr) \mod O(\mu^{-2})
\end{multline*}
where $z_t=\Phi_t(0)$.
We can correct $\theta$ by $O(\mu^{-1})$ eliminating linear terms in the right hand expression and therefore
\begin{multline}
\theta \equiv \int F(z_t)\,dt\equiv
F(0)t + \frac{1}{2} \bigl(\frac{d\ }{dt}F(z_t)\bigr)(0) t^2 \\
\mod O(\mu^{-1}+\mu^{-2} |t|)
\label{16-2-20}
\end{multline}
with the second term in the right-hand expression
$O(\varepsilon \mu^{-1}t^2)$.

From now on we use rescaling $x\mapsto \mu x$, $t\mapsto \bar{\mu} t$, $\theta \mapsto \mu\theta$. Our analysis implies immediately

\begin{proposition}\label{prop-16-2-3}
Let  conditions \textup{(\ref{16-2-1})}--\textup{(\ref{16-2-8})} be fulfilled\footnote{\label{foot-16-6} Before rescaling.}. Then (after rescaling $x\to \bar{\mu}x$, $t\mapsto \bar{\mu}t$)   as $|t|\le c_0\varepsilon^{-1}$
\begin{equation}
|D^\beta \Psi_{t}|\le C_\beta\qquad \forall \beta.
\label{16-2-21}
\end{equation}
\end{proposition}

\section{Semiclassical approximation to $U(x,y,t)$}
\label{sect-16-2-2}

\begin{proposition}\label{prop-16-2-4}
Let  conditions \textup{(\ref{16-2-1})}--\textup{(\ref{16-2-8})} be fulfilled\footref{foot-16-6}. Then

\medskip\noindent
(i) Uniformly with respect to $|t|\le c\varepsilon^{-1}$ \ $e^{i\hbar^{-1}tA}$ is an $\hbar$-Fourier integral operator corresponding to Hamiltonian flow $\Psi_t$;

\medskip\noindent
(ii) As $|\sin (2\theta)|\ge \epsilon$ with $\theta$ defined by \textup{(\ref{16-2-20})}
\begin{equation}
U(x,y,t)\equiv (4\pi\hbar)^{-1} i
(\sin(\theta))^{-1}  e^{i\hbar^{-1}\phi (x,y,t)} \sum_m b_m (x,y,t) \hbar^m
\label{16-2-22}
\end{equation}
with $\phi$ defined by \textup{(\ref{16-2-25})}--\textup{(\ref{16-2-29})} below and satisfying (with all derivatives)
\begin{gather}
\phi = \bar{\phi} (\theta) +O(\mu^{-2}\varepsilon^{-1} ),\label{16-2-23}\\
b_m =\updelta_{0m}+O(\mu^{-2}\varepsilon^{-1} )\label{16-2-24}
\end{gather}
with $\bar{\phi}$ defined by \textup{(\ref{16-1-10})}.
\end{proposition}

\begin{proof}
Both assertions of proposition are standard as $|t|\le T=\const$ but we need to extend them for larger $t$. Further, $\phi$ is an \emph{action\/}\index{action}\footnote{\label{foot-16-7} Sign ``$-$'' is due to considering of propagator $e^{i\hbar^{-1}tA}$ rather than  $e^{-i\hbar^{-1}tA}$.}
\begin{equation}
\phi = -\int_0^t L\bigl(x(t'),\dot{x}(t')\bigr)\,dt'
\label{16-2-25}
\end{equation}
with the Lagrangian
\begin{multline}
L(x,\dot{x})= \bigl(\sum_{k} \dot{x}_k \xi_k - a(x,\xi)\bigr)|_{\xi_k=\textup{(\ref{16-2-29})}}=\\
\frac{1}{4}\sum _{j,k} g_{jk}(x) \dot{x}_j\dot{x}_k +
\frac{1}{2} \sum_j \dot{x}_jV_j(x)-V(x).
\label{16-2-26}
\end{multline}
which is the Legendre transformation of the Hamiltonian
\begin{gather}
a(x,\xi)=
\sum_{j,k} g^{jk}(x) \bigl(\xi_j-V_j(x)\bigr) \bigl(\xi_k-V_k(x)\bigr) +V(x)
\label{16-2-27}\\
\dot{x}_j=2\sum_k g^{jk}(x) \bigl(\xi_k-V_k(x)\bigr),\label{16-2-28}\\
\xi_k = \frac{1}{2}\sum \sum_j g_{jk} \dot{x}_j + V_k(x),\label{16-2-29}
\end{gather}
$\dot{f}\Def {df}/{dt'}$ and we are talking about trajectories from $y$ as $t'=0$ to $x$ as $t'=t$.

To extend assertions (i), (ii) to $t$: $|t|\le T=c\varepsilon^{-1}$ note that
\begin{equation}
e^{\pi t \hbar^{-1}A}= - e^{i\mu^{-1}\hbar^{-1}B}
\label{16-2-30}
\end{equation}
with $\hbar$-pseudo-differential operator $B$; $B=0$ if $A$ is a pilot-model operator with $\alpha=0$. Therefore as usual
\begin{equation}
e^{\pi t \hbar^{-1}A}=  e^{ik\hbar^{-1}A}e^{i(t- \pi k)\hbar^{-1}A}=
(-1)^k e^{ik\mu^{-1}\hbar^{-1}B}e^{i(t- \pi k)\hbar^{-1}A}
\label{16-2-31}
\end{equation}
with $k=\lfloor t/\pi\rfloor$ and the right-hand expression is $\hbar$-Fourier integral operator as  $|k|\le c\varepsilon^{-1}$. Assertion (i) is proven.

\medskip
Note that the corresponding canonical manifold is well-projected to $(x,y)$ space ($\bR^4$) as $|\sin (2\theta)|\ge \epsilon$ and compare it with the canonical manifold for the pilot-model operator; this proves (ii) as either $F=1$ or $|t|\le \epsilon_1\varepsilon$ with sufficiently small $\epsilon_1=\epsilon_1(\epsilon)$.

But then one can represent $e^{i\hbar^{-1}tA}$ as
$e^{i\hbar^{-1}t_1 A}e^{i\hbar^{-1}t_2 A}\cdots e^{i\hbar^{-1}t_nA}$ with
$n\le c_1$ and $|t_j|\le \epsilon_1\mu$, $|\sin (2t_j)|\ge \frac{1}{2}\epsilon$ which implies (ii) in the general case.
\end{proof}

So far we exclude both vicinities of $\sin(\theta)=0$ which matches $x=y$ (poles) and $\cos(\theta)=0$ which matches $x$ and $y$ being antipodal points of the trajectory (equator) but we need to approach both of them. Actually exclusion of the latter was no more than a precaution but poles require a modification:

\begin{proposition}\label{prop-16-2-5}
Let  conditions \textup{(\ref{16-2-1})}--\textup{(\ref{16-2-8})} be fulfilled\footref{foot-16-6}. Then

\medskip\noindent
(i) Decomposition \textup{(\ref{16-2-22})} remains valid as
${|\cos (\theta)|\le \epsilon}$;

\medskip\noindent
(ii) Decomposition \textup{(\ref{16-2-22})} remains valid as
\begin{equation}
C\max\bigl(\hbar, \mu^{-1}\varepsilon |t|\bigr) \le |\sin (\theta)|\le \epsilon
\label{16-2-32}
\end{equation}
albeit with an error not exceeding
\begin{gather}
C\hbar^{-1}|\sin (\theta)|^{-1} (\hbar/|\sin (\theta)|)^l\label{16-2-33}\\
\intertext{and with $b_m$ such that}
|D^\beta (\phi-\bar{\phi})|\le
C_\beta \varepsilon |\sin (\theta)|^{-|\beta|}
\qquad\forall \beta,\label{16-2-34}\\[2pt]
|D^\beta (b_m-\updelta_{m0})  |\le
C_{m\beta} \varepsilon |\sin (\theta)|^{-m-|\beta|}
\qquad \forall \beta,m. \label{16-2-35}
\end{gather}
\end{proposition}

\begin{proof}
As
\begin{equation*}
U(x,y,t)= \mu^{-2} \int U(x,z,t')U(z,y,-t'')\,dz
\end{equation*}
for $t=t'-t''$ (where factor $\mu^{-d}$ is due to rescaling $U$ as a function) we need to consider this oscillatory integral. If we consider oscillatory integral with the propagator for the pilot-model operator we note that the standard stationary phase method applies with an effective semiclassical parameter $\hbar$ as $(\cot(t')-\cot(t''))=-\sin(t)\sec(t')\sec(t'')$ disjoint from $0$ i.e. also as $|\cos (t)|\le \epsilon$.

On the other hand, an effective semiclassical parameter is $\hbar/|\sin (t)|$ near poles.

In virtue of proposition~\ref{prop-16-2-4} both these claims remain true for a general operator as well albeit with $t$ replaced by $\theta$.
\end{proof}

\begin{remark}\label{rem-16-26}
So far we  need only $|\nabla V/F| \lesssim \mu\varepsilon$ rather than $|\nabla V/F| \asymp \mu\varepsilon$~\footref{foot-16-6}.
\end{remark}

\section{Semiclassical approximation to $e(x,x,\tau)$}
\label{sect-16-2-3}

Therefore in zone $\{|\sin (\theta)|\ge C\hbar\}$ all arguments of the pilot-model theory work (with obvious modifications) under non-degeneracy assumptions
\begin{gather}
V-\tau \le -\epsilon_0,\label{16-2-36}\\
|\nabla (\tau-V)F^{-1}|\asymp \varepsilon \label{16-2-37}
\end{gather}
we conclude that
\begin{claim}\label{16-2-38}
Contribution of zone $\{|\sin (\theta)|\ge C\hbar \}$ to
$F_{t\to h^{-1}\tau} (1-\bar{\chi}_1 (t))\Gamma_x U$ does not exceed
$C\mu h^{-1}+C\mu^{\frac{3}{2}}\varepsilon^{-1} h^{-\frac{1}{2}}$.
\end{claim}

Consider now zone $\{|\sin (\theta) |\le \epsilon|\}$. Both operators $A$ and $B$ (defined by (\ref{16-2-30}) are $\xi$-microhyperbolic and unless factors they are coming with (namely $(\theta -\pi k)$ (with the closest $\pi k$) and $\varepsilon t$ respectively are of the same magnitude we can use $\xi$-microhyperbolicity to prove

\begin{claim}\label{16-2-39}
Contribution of zone
$\{|t|\asymp T\ge \varepsilon^{-1} \max(\mu \hbar, C |\sin (\theta)|)\}$  to $F_{t\to h^{-1}\tau} \Gamma_x U$ does not exceed
\begin{equation*}
C\mu^2 \hbar^{-2} \times T\times \varepsilon T \times
( \hbar/\varepsilon T)^{l+1}=
C\mu^2 h ^{-1}T (\hbar/\varepsilon T)^l
\end{equation*}
(with arbitrarily large $l$) while its contribution to expression \textup{(\ref{16-1-59})}  does not exceed the same expression albeit without factor $T$ i.e.
\begin{equation*}
C\mu ^2 h ^{-1} (\hbar/\varepsilon T)^l,
\end{equation*}
\vglue-10pt
\end{claim}
where factor $\mu^2$ is due to rescaling, $T\times \varepsilon T$ is the  measure of the zone $\{\theta: |\theta|\asymp T,\ |\sin (\theta)|\le \varepsilon T\}$ and other factors are standard; recall that $B$ comes with the factor $\varepsilon$ and therefore effective semiclassical parameter is $\hbar/\varepsilon$. We leave easy details to the reader.

After summation with respect to $T$ we conclude  that
\begin{claim}\label{16-2-40}
As $\varepsilon\ge C \hbar $ contributions of zone
$\{|t|\ge 1,\ |\sin (\theta)|\le c\hbar\}$  to both
$F_{t\to h^{-1}\tau} \Gamma_x U$ and  expression \textup{(\ref{16-1-59})}
do not exceed  $C \mu h^{-1}(\hbar/\varepsilon)^l$
\end{claim}
and
\begin{claim}\label{16-2-41}
As $ \varepsilon \le C_0\hbar \le 1$ contribution of zone
$\{|t|\ge C_0\hbar/\varepsilon ,\ |\sin (\theta)|\le c\hbar\}$  to
$F_{t\to h^{-1}\tau} \Gamma_x U$ does not exceed $C\mu^2\varepsilon^{-1}$
while  its contribution to expression \textup{(\ref{16-1-59})} does not exceed  $C \mu h^{-1}$.
\end{claim}

So, as $\varepsilon \ge C\hbar $ we are done but as
$ \varepsilon \le C\hbar $ we need to consider zone
$\{1\le |t|\le c\hbar/\varepsilon ,\ |\sin (\theta)|\le c\hbar\}$ and the same arguments imply

\begin{claim}\label{16-2-42}
As $ \varepsilon \le C_0\hbar \le 1$ contribution of zone
$\{1\le |t|\le c\hbar/\varepsilon ,\ |\sin (\theta)|\le c\hbar\}$  to
$F_{t\to h^{-1}\tau} \Gamma_x U$ does not exceed
\begin{equation*}
C\mu^2 \hbar^{-2} \times \varepsilon^{-1} \hbar  \times \hbar = C\mu^2\varepsilon^{-1}
\end{equation*}
where factor $\mu^2$ is due to rescaling, $\varepsilon^{-1}\hbar\times \hbar$ is the measure of the zone $\{\theta: |\theta|\le \varepsilon^{-1} \hbar ,\ |\sin (\theta)|\le \hbar\}$  and  factor $\hbar^{-2}$ is the standard one.
\end{claim}

Combining with the results for zone $\{ |\sin (\theta)|\ge c\hbar\}$ we arrive to

\begin{proposition}\label{prop-16-2-7}
For magnetic Schr\"odinger operator in domain $X$,
$B(0,\ell)\subset X \subset \bR^2$  under standard smoothness assumptions and non-degeneracy assumptions \textup{(\ref{16-0-6})} and \textup{(\ref{16-2-36})}--\textup{(\ref{16-2-37})}, \textup{(\ref{16-2-3})} after rescaling $x\mapsto \mu x$, $t\mapsto \mu t$ (as $\mu \le h^{-1}$) estimates \textup{(\ref{16-1-49})}--\textup{(\ref{16-1-50})} of proposition~\ref{prop-16-1-6} and estimates \textup{(\ref{16-1-55})}--\textup{(\ref{16-1-57})} of proposition~\ref{prop-16-1-9} hold as $T\le c\varepsilon^{-1}$.

Further, estimates \textup{(\ref{16-1-78})}--\textup{(\ref{16-1-81})} hold as well.
\end{proposition}

Now we apply the standard Tauberian arguments. We are interested mainly in the case of $\alpha \asymp 1$; then automatically $\ell\asymp 1$ and $T^*\asymp \mu^2$ inside of domain. However in more general case $\ell \asymp \alpha\asymp \mu \varepsilon$ automatically and $T^*\asymp \ell/\varepsilon \mu \mu^2\asymp \mu^2$ again and we arrive to

\begin{proposition}\label{prop-16-2-8}
For magnetic Schr\"odinger operator in domain $X$,
$B(0,\ell)\subset X \subset \bR^2$  under standard smoothness assumptions and non-degeneracy assumptions \textup{(\ref{16-0-6})} and \textup{(\ref{16-2-36})}--\textup{(\ref{16-2-37})}, \textup{(\ref{16-2-3})}
Tauberian estimates \textup{(\ref{16-1-51})}  of proposition~\ref{prop-16-1-8} and estimates \textup{(\ref{16-1-57})}--\textup{(\ref{16-1-58})}  of corollary~\ref{cor-16-1-10} hold with the Tauberian expression \textup{(\ref{16-1-52})} and $\ell\asymp \alpha$.
\end{proposition}

\section{Stationary phase calculations}
\label{sect-16-2-4}

Let apply the stationary phase method to the Tauberian expression. To do so we need to remember that assumptions $|t|\ll \varepsilon^{-1}$, $\varphi'=0$ yield
$\varphi'' \asymp |\sin(\theta)|^{-1}$ and all other estimates hold; therefore the stationary phase construction is available there. Also this construction works as long as $|t|\le c\varepsilon^{-1}$ and $|\varphi'' |\ge \epsilon$ i.e. $|t-t^*|\ge \epsilon \varepsilon^{-1}$ where $T^*=\varepsilon^{-1}$.

Moreover,   $|\varphi''|\ge \epsilon_0$ as
$|t-t^*|\le \epsilon \varepsilon^{-1}$ and $\varphi'=0$ and therefore all our estimates work here, in the near equator zone as well where

\begin{definition}\label{def-16-2-9}
Equator just moves to the point where $|\varphi''|$ is minimal for $\varphi'=0$.
Let us in the near equator zone redefine $\theta$ in such way that equator is $\cos(\theta)=0$.
\end{definition}

Let us introduce

\begin{definition}\label{def-16-2-10}
Consider $r$-term representation \textup{(\ref{16-2-22})} (i.e. with summation over $m<r$) and plug it into Tauberian expression, excepting $|t|\le \epsilon_0$ and calculate it by the stationary phase method with $r$ terms again. Let us call the result \emph{$r$-term semiclassical approximation\/}\index{$r$-term semiclassical approximation} and denote it by $h^{-2}\cN_{x,\corr(r)}$.
\end{definition}

Then its main term is delivered by modified (\ref{16-1-60})--(\ref{16-1-62})
\begin{multline}
\sum_{k\in \fZ} \frac{1}{4\pi\sqrt\pi}\mu   \hbar^{-\frac{1}{2}}
(t_k\sin (\theta_k))^{-1} |\varphi'' (t_k)|^{-\frac{1}{2}}b_0(t_k)\times \\
\exp\Bigl( \frac{i\pi}{4}\sign \varphi'' (t_k) +   i\hbar^{-1}\varphi(t_k)\Bigr)
\label{16-2-43}
\end{multline}
with
\begin{gather}
\varphi'' (t_k) \sim   2\tau  \cot(\theta_k)\qquad \text{as\ \ }
|\cos (\theta_k)|\ge C \varepsilon^{\frac{1}{2}}
\label{16-2-44}\\
\shortintertext{and}
\fZ\Def \bigl\{k\ne 0:\
|\cos (\theta_k)|\ge C\max \bigl(\varepsilon^{\frac{1}{2}}, \hbar^{\frac{1}{3}}\bigr)\bigr\}
\label{16-2-45}
\end{gather}
where $t_k$ is a time of $k$-th return to $x$ (along $k$-th loop), $\theta_k=\theta(t_k)$ and $\varphi (t_k)$ is the corresponding action.

Then repeating arguments of the previous section we arrive to

\begin{proposition}\label{prop-16-2-11}
For magnetic Schr\"odinger operator in domain $X$,
$B(0,1)\subset X \subset \bR^2$  under standard smoothness assumptions and non-degeneracy assumptions \textup{(\ref{16-0-6})} and \textup{(\ref{16-2-36})}--\textup{(\ref{16-2-37})}, \textup{(\ref{16-2-3})}

\medskip\noindent
(i) As $\varepsilon \ge \hbar$ estimate \textup{(\ref{16-1-64})}   of proposition~\ref{prop-16-1-11} holds;

\medskip\noindent
(ii) Estimates \textup{(\ref{16-1-66})}--\textup{(\ref{16-1-67})} of proposition~\ref{prop-16-1-12}, estimates \textup{(\ref{16-1-68})}--\textup{(\ref{16-1-71})} of proposition~\ref{prop-16-1-13}  and
estimates \textup{(\ref{16-1-82})}--\textup{(\ref{16-1-85})} of proposition~\ref{prop-16-1-18} hold.
\end{proposition}

\section{Approximation by a pilot-model operator}
\label{sect-16-2-5}
\subsection{Weak magnetic field case}
\label{sect-16-2-5-1}
We cannot do better in this framework; however using the pilot-model operator as an approximation we could improve these results. Without any loss of the generality we can assume that (\ref{16-2-5})--(\ref{16-2-7}) are fulfilled.

Assume temporarily that $F=1$. Then one can prove easily that
\begin{equation}
t_k=\theta (t_k)\equiv \bar{t}_k,\qquad \phi(t_k)\equiv \bar{\phi}(\bar{t}_k)
\mod \mu^{-2}|k|
\label{16-2-46}
\end{equation}
where bar denotes objects related to the pilot-model operator.
Then $\hbar^{-1}\bigl(\phi(t_k)- \bar{\phi}(\bar{t}_k)\bigr)=O(\mu^{-2}\hbar^{-1}|k|)$ and it is less than $1$ as
\begin{equation}
|k|\le \tilde{k}\Def \mu^3 h.
\label{16-2-47}
\end{equation}
Therefore it  makes sense to apply this approach only  as $\tilde{k}\ge 1$ i.e.
\begin{equation}
\mu\ge h^{-\frac{1}{3}}.
\label{16-2-48}
\end{equation}
Under this assumption let us compare $r$-th terms for our operator and the pilot model as $|k|\le \tilde{k}$.

One can see easily that  their difference does not exceed
\begin{equation}
C\sum_k \mu h^{-1}  (\hbar /\varepsilon |k|)^{r+\frac{1}{2}}|k|^{-1}
\times \mu^{-3}h^{-1}|k|
\label{16-2-49}
\end{equation}
which as $r\ge 1$ does not exceed
$C\mu^{-2} h^{-2}  (\hbar /\varepsilon |k|)^{r+\frac{1}{2}}|k|$ calculated for the smallest possible $|k|$ which is $1$ provided $\hbar\le \varepsilon$. Therefore we get $C\mu^{-2} h^{-2}  (\hbar /\varepsilon )^{r+\frac{1}{2}}$. On the other hand in (\ref{16-2-49}) one should replace the last factor by $1$ as $|k|\ge \tilde{k}$.

Therefore, if we start from $l$-term approximation and apply estimate (\ref{16-1-66}), (\ref{16-1-67}) as $\varepsilon\ge \hbar^{\frac{2}{3}}$,
$ \hbar\le \varepsilon\le \hbar^{\frac{2}{3}}$ respectively and pass from $l$-term to $r$-term approximation we get estimate
\begin{multline}
\R^{\W\prime\prime}_{x(r)\gamma}\Def \\
|e^\T (x,x,\tau) -h^{-2}\cN_{x,\corr(r)}-
\bar{e}^\T (x,x,\tau) +h^{-2}\bar{\cN}_{x,\corr(r)}|\le \\[3pt]
C\mu^{-1}h^{-1}+
C\mu^{-2} h^{-2}  (\hbar /\varepsilon )^{r+\frac{1}{2}} +
C\mu h^{-1} (\hbar /\varepsilon )^{l+\frac{1}{2}}+\\[3pt]
\left\{\begin{aligned}
&C\mu \varepsilon h^{-1} \bigl(\hbar/\varepsilon^{\frac{3}{2}} \bigr)^{r+\frac{1}{2}}+ C\mu^{\frac{4}{3}} \varepsilon h^{-\frac{2}{3}} &&
\text{as\ \ } \varepsilon\ge \hbar^{\frac{2}{3}}\\[2pt]
&C\mu^{\frac{5}{3}}h^{-\frac{1}{3}}
\qquad&&\text{as\ \ } \hbar\le \varepsilon\le \hbar^{\frac{2}{3}}
\end{aligned}\right.
\label{16-2-50}
\end{multline}
where terms $-h^{-2}\cN_x^\W$ and $h^{-2}\bar{\cN}_x^\W$ obviously cancel one another.

\subsection{Successive approximations}
\label{sect-16-2-5-2}

Now we need to get rid off the third term in the right-hand expression of (\ref{16-2-50}) without condition $\varepsilon\gg \hbar$ which allows to eliminate it. To do this we need  more sophisticated arguments. Namely, let us apply the successive approximations.

Assuming that the original operator is perturbed by $O(\mu^{-2})$ we arrive to the heuristic conclusion that
\begin{claim}\label{16-2-51}
Contribution of $|t|\asymp T$ to $R^{\W\prime\prime}$ does not exceed the previous estimate of the contribution to $\R^{\W\prime}$ multiplied by $\mu^{-2}T\hbar^{-1}$;
\end{claim}
this statement needs justification. We can achieve a better estimate by taking more terms in the successive approximation but it leads to a rather overcomplicated formula; so we take just one term.

Then, in frames of (\ref{16-2-51}) we have three cases:
\begin{enumerate}[label=(\alph*), leftmargin=*]
\item $\varepsilon \ge \hbar$; then all what successive approximation does is to eliminate the third term in the right-hand expression of (\ref{16-2-50}).

\item $\mu^{-3}h^{-1}\le \varepsilon \le \hbar$; then $1\ge \bar{k}\le \varepsilon^{-1}$ and we need to adjust our arguments treating separately $|k|\le \bar{k}$ and $|k|\ge \bar{k}$ and recalculate the second term in the right-hand expression of (\ref{16-2-50}); however equator zone is not affected.

\item $\varepsilon \le \min(\hbar, \mu^{-3}h^{-1})$. In this case in addition we need to recalculate the last term in the right-hand expression of (\ref{16-2-50}).
\end{enumerate}

Let us reconsider the second term in the right-hand expression of (\ref{16-2-50}). First, expression (\ref{16-2-49}) with summation over $k: |k|\ge \bar{k}$ returns $C\mu^{-2}h^{-2}(\hbar/\varepsilon|k|)^{r+\frac{1}{2}}|k|$ calculated as $k= \bar{k}$; so we get $C\mu^{-2}h^{-2}\bar{k}=C\mu^{-1}h^{-1}\varepsilon^{-1}$.

Further, as $|k|\le \bar{k}$ we need to replace
$(\hbar/\varepsilon |k|)^{r+\frac{1}{2}}$ by $1$; so we sum
$C\mu h^{-1}|k|^{-1}\times \mu^{-3}h^{-1}|k|$ which returns $C\mu^{-2}h^{-2}\bar{k}$ again. Therefore\footnote{\label{foot-16-8} Pending (\ref{16-2-51}).} we conclude that

\begin{claim}\label{16-2-52}
As $\hbar\ge \varepsilon$, $r\ge 1$ one should replace the second term in (\ref{16-2-50}) by $C\mu^{-1}h^{-1}\varepsilon^{-1}$.
\end{claim}

Furthermore, as  $\varepsilon \le \min(\hbar, \mu^{-3}h^{-1})$ one should multiply the last term in (\ref{16-2-50}) i.e. $C\mu^{\frac{5}{3}}h^{-\frac{1}{3}}$ by $\mu^{-3}\varepsilon^{-1}h^{-1}$ resulting in
$C\mu^{-\frac{4}{3}}\varepsilon^{-1}h^{-\frac{4}{3}}$. Therefore\footref{foot-16-8} we conclude that

\begin{claim}\label{16-2-53}
As $\varepsilon \le \min(\hbar, \mu^{-3}h^{-1})$, $r\ge 1$ one should replace the second term in (\ref{16-2-50}) by $C\mu^{-\frac{1}{3}}\varepsilon^{-1}h^{-\frac{4}{3}}$.
\end{claim}

Therefore we (heuristically, the proof follows) arrive to

\begin{proposition}\label{prop-16-2-12}
Consider two magnetic Schr\"odinger operators $A$ and $\bar{A}$ satisfying the standard smoothness assumptions, non-degeneracy assumptions \textup{(\ref{16-0-6})} and \textup{(\ref{16-2-36})}--\textup{(\ref{16-2-37})}, \textup{(\ref{16-2-3})} and such that
\begin{equation}
g^{jk}=\bar{g}^{jk}+O(\mu^{-2}),\qquad V=\bar{V}+O(\mu^{-2}),\qquad V_j=\bar{V}_j+O(\mu^{-3}).
\label{16-2-54}
\end{equation}
Then for $r\ge 1$
\begin{multline}
\R^{\W\prime\prime}_{x(r)} \le
C\mu^{-1}h^{-1}+
C\left\{\begin{aligned}
&\mu^{-2} h^{-2}  (\hbar /\varepsilon )^{r+\frac{1}{2}} \qquad&&\text{as\ \ } \varepsilon\ge \hbar\\
&C\mu^{-1}h^{-1}\varepsilon^{-1}\qquad&&\text{as\ \ } \varepsilon\le \hbar
\end{aligned}\right.\quad +\\[3pt]
\left\{\begin{aligned}
&C\mu \varepsilon h^{-1} \bigl(\hbar/\varepsilon^{\frac{3}{2}} \bigr)^{r+\frac{1}{2}}+ C\mu^{\frac{4}{3}} \varepsilon h^{-\frac{2}{3}} &&
\text{as\ \ } \varepsilon\ge \hbar^{\frac{2}{3}}\\[2pt]
&C\mu^{\frac{5}{3}}h^{-\frac{1}{3}}
\qquad&&\text{as\ \ } \min(\hbar,\mu^{-3}h^{-1})\le \varepsilon\le \hbar^{\frac{2}{3}}\\
&C\mu^{-\frac{4}{3}}\varepsilon^{-1}h^{-\frac{4}{3}}
\qquad&&\text{as\ \ } \varepsilon\le  \min(\hbar,\mu^{-3}h^{-1})
\end{aligned}\right.
\label{16-2-55}
\end{multline}
and
\begin{multline}
\R^{\W\prime\prime}_{x(r)\gamma}\Def \\
\gamma^{-2}
|\int \Bigl(e^\T (x,x,\tau) -h^{-2}\cN_{x,\corr(r)}-
\bar{e}^\T (x,x,\tau) +h^{-2}\bar{\cN}_{x,\corr(r)}\Bigr)\psi_\gamma\,dx |\le \\[3pt]
C\mu^{-1}h^{-1}+
C\left\{\begin{aligned}
&\mu^{-2} h^{-2}  (\hbar /\varepsilon )^{r+\frac{1}{2}} (\hbar/\varepsilon\gamma)^l \qquad&&\text{as\ \ } \gamma \varepsilon\ge \hbar\\
&\mu^{-2} h^{-2}  (\hbar /\varepsilon )^{r+\frac{1}{2}} \qquad&&\text{as\ \ } \varepsilon\ge \hbar\\
&\mu^{-1}h^{-1}\varepsilon^{-1}\qquad&&\text{as\ \ } \varepsilon\le \hbar
\end{aligned}\right.\quad +\\[3pt]
C\min \bigl( 1,(h/\gamma)^l\bigr)\times \left\{\begin{aligned}
&\mu \varepsilon h^{-1}
\bigl(\hbar/\varepsilon^{\frac{3}{2}} \bigr)^{r+\frac{1}{2}} + C\mu^{\frac{4}{3}} \varepsilon h^{-\frac{2}{3}} &&
\text{as\ \ } \varepsilon\ge \hbar^{\frac{2}{3}}\\[2pt]
&\mu^{\frac{5}{3}}h^{-\frac{1}{3}}
\qquad&&\text{as\ \ } \min(\hbar,\mu^{-3}h^{-1})\le \varepsilon\le \hbar^{\frac{2}{3}}\\
&\mu^{-\frac{1}{3}}\varepsilon^{-1}h^{-\frac{4}{3}}
\qquad&&\text{as\ \ } \varepsilon\le  \min(\hbar,\mu^{-3}h^{-1})
\end{aligned}\right.
\label{16-2-56}
\end{multline}
while for $r=0$
\begin{multline}
\R^{\W\prime\prime}_{x(0)}\Def |e^\T (x,x,\tau) - \bar{e}^\T (x,x,\tau)|\le C\mu^{-1}h^{-1}+ \\
C\left\{\begin{aligned}
& h^{-1}\varepsilon^{-\frac{1}{2}}\qquad
&&\text{as\ \ }\mu^3h\le \varepsilon^{-1},\\
&\mu ^{-\frac{3}{2}}h^{-\frac{3}{2}}\varepsilon^{-1}
&&\text{as\ \ }\mu^3h\ge \varepsilon^{-1}
\end{aligned}\right.
\label{16-2-57}
\end{multline}
and
\begin{multline}
\R^{\W\prime\prime}_{x(0)\gamma}\Def
\gamma^{-2} |\int \Bigl(e^\T (x,x,\tau) - \bar{e}^\T (x,x,\tau)\Bigr)\psi_\gamma\,dx|\le C\mu^{-1}h^{-1}+ \\
C\left\{\begin{aligned}
&\mu ^{-\frac{3}{2}} h^{-1} \varepsilon^{-1}\gamma^{-\frac{1}{2}} \qquad
&&\text{as\ \ }\mu^{-3}\varepsilon^{-1}\le \gamma \le \min(h\varepsilon^{-1},\mu^{-1}),\\
&\mu ^{-2} h^{-1} \varepsilon^{-1}\gamma^{-1}
\qquad
&&\text{as\ \ }\mu^{-1}\le \gamma \le h\varepsilon^{-1},\\
&\mu h^{-1}\min
\bigl(1,(\mu h/\varepsilon)^{\frac{1}{2}}\bigr)(h/\varepsilon\gamma)^l
\qquad
&&\text{as\ \ }\gamma \ge h\varepsilon^{-1}.
\end{aligned}\right.
\label{16-2-58}
\end{multline}
\end{proposition}

\subsection{Justification}
\label{sect-16-2-5-3}

To justify (\ref{16-2-51}) and thus to prove proposition~\ref{prop-16-2-12} we need

\begin{proposition}\label{prop-16-2-13}
Consider two magnetic Schr\"odinger operators $A$ and $\bar{A}$ satisfying \textup{(\ref{16-2-54})}. Then as $|t|\le \epsilon \mu^3 h$ and
$|\sin (t)|\ge \max\bigl(\epsilon \varepsilon|t|, C\hbar\bigr)$
\begin{multline}
U(x,y,t) -\bar{U}(x,y,t)\equiv \\
\shoveright{(4\pi\hbar)^{-1} i
(\sin(\bar{\theta}(t)))^{-1}  e^{i\hbar^{-1}\bar{\phi} (x,y,t)} \sum_m b'_m (x,y,t) \hbar^m}\\
\mod
O\Bigl(\hbar^{-2}\mu^{-3}|t| |\sin(\theta(t))|^{-1}\times (\hbar/|\sin(t)|)^l\Bigr)
\label{16-2-59}
\end{multline}
with
\begin{equation}
|D^\beta b'_m|\le C_{m\beta }\mu^{-3}|t|\hbar^{-1} |\sin(\theta(t))|^{-m-|\beta|}.
\label{16-2-60}
\end{equation}
\end{proposition}

\begin{proof}
Obviously in our assumptions
\begin{multline*}
U(x,y,t) \equiv
(4\pi\hbar)^{-1} i
(\sin(\bar{\theta}(t)))^{-1}e^{i\hbar^{-1}\bar{\phi}(x,y,t)} \sum_m b''_m(x,y,t) \hbar^m  \\
\mod O\Bigl(\hbar^{-1} |\sin(\bar{\theta}(t))|^{-1}\times (\hbar/|\sin(\bar{\theta}(t))|)^l\Bigr)
\end{multline*}
with $b''_m$ satisfying
\begin{equation*}
|D^\alpha b''_m|\le C_{m\alpha } |\sin(\bar{\theta}(t))|^{-m-|\alpha|}
\end{equation*}
and plugging it into the right-hand expression of
\begin{equation}
e^{i\hbar^{-1}tA}- e^{i\hbar^{-1}t\bar{A}}=
i\hbar^{-1} \int _0^t e^{i\hbar^{-1}t'A} (A-\bar{A}) e^{i\hbar^{-1}(t-t')\bar{A}}\,dt'
\label{16-2-61}
\end{equation}
we arrive to (\ref{16-2-59}). Here to cover zones where either
$|\sin (t')|\le \epsilon$ or $|\sin (t-t')|\le \epsilon$ we just pass to the standard representation through non-singular oscillatory integrals (with one or two extra variable) and then apply the stationary phase method. We leave easy but tedious details to the reader.
\end{proof}

\begin{corollary}\label{cor-16-2-14}
In frames of proposition \ref{prop-16-2-13} equality \textup{(\ref{16-2-59})} holds with $U(x,y,t)$ and $\bar{U}(x,y,t)$ replaced respectively by
\begin{gather}
(4\pi\hbar)^{-1} i
(\sin(\theta(t)))^{-1}e^{i\hbar^{-1}\phi(x,y,t)} \sum_m b_m(x,y,t) \hbar^m
\label{16-2-62}\\
\shortintertext{and}
(4\pi\hbar)^{-1} i
(\sin(\bar{\theta}(t)))^{-1}e^{i\hbar^{-1}\bar{\phi}(x,y,t)}
\sum_m \bar{b}_m(x,y,t) \hbar^m.
\label{16-2-63}
\end{gather}
\end{corollary}

\subsection{Reduction to the pilot-model operator}
\label{sect-16-2-5-4}

From now we are interested \emph{only\/} in the estimates without micro-averaging.

The problem however is that $\bar{A}$ at this moment is a more general operator than the pilot-model: it contains linear with respect to $x$ terms in the original (before rescaling) $g^{jk}$ and $V$, and also it contains quadratic terms in
$V_2=x_1 + \frac{1}{2}\beta_{11}x_1^2 +\beta_{12}x_1x_2 +\frac{1}{2}\beta_{22}x_2^2 +\ldots$ all of which generate $O(\mu^{-1})$ perturbation.

However  one select a conformal coordinate system in such way that
$\nabla \omega(0)=0$~\footnote{\label{foot-16-9} Really, changing conformal system to another conformal system is done with $y_1=f$,
$\nabla y_2= (-f_{x_2},f_{x_1})$ which can be satisfied iff $\Delta f=0$; then $\omega$ is replaced by $\omega |\nabla f|$ and if
$f=a x_1+  bx_1x_2+\frac{1}{2}c(x_1^2-x_2^2)$ then
$|\nabla f |= a+ bx_2+ cx_1+O(|x|^2)$ and selecting $a,b,c$ we can make $\omega(0)=1$, $\nabla\omega(0)=0$.}. Then if we know that $F=F(0)+O(|x|^2)$ we conclude that $\beta_{11}=\beta_{12}=0$ and by a gauge transformation we can make $\beta_{22}=0$.

Further, we can assume that $\tau=0$ and then make $V(0)=-1$ by division $V(0)\mapsto -1$,
$h\mapsto h (-V(0))^{-\frac{1}{2}}$,
$\mu\mapsto \mu (-V(0))^{-\frac{1}{2}}$ before rescaling. Rescaling $\mu$ (and changing coordinate orientation if needed we can make $F(0)=1$.

Furthermore, we can assume that $F=1$. Really, as we can consider energy level $\tau=0$ we starting from equation $-\hbar D_tu = Au$, go to
\begin{equation}
F^{-1}Au=-\hbar F^{-1}D_tu=-hD_t + \hbar F^{-1}(F-1)D_tu
\label{16-2-64}
\end{equation}
and on an interval $|\tau|\le \hbar h^{-\delta}T^{-1}$ (which is completely microlocally admissible as $|t|\asymp T$ the last term is
$O(\mu^{-1} \hbar h^{-\delta}T^{-1})$ and multiplying by $T\hbar^{-1}$ we get
$O(\mu^{-1}h^{-\delta})$ which allows us to use the successive approximation method.

We leave to the reader to use our standard methods and solve a rather easy
\begin{problem}\label{problem-16-2-15}
Prove that effectively in estimates one can make $\delta=0$.
\end{problem}

Note that  $\mu^{-1}\le \mu^{-3}h^{-1}$ as $\mu \le h^{-\frac{1}{2}}$ and therefore this new error ($F$ is replaced by $F(0)$) does not exceed the right-hand expressions of (\ref{16-2-55}) and (\ref{16-2-57}) as $r\ge 1$ and $r=0$ respectively.

Meanwhile our new error does not exceed $C\mu h^{-1}\times \mu^{-1}=C h^{-1}$ anyway and as $\mu\ge h^{-\frac{1}{2}}$ ($\implies \varepsilon \le \hbar$) it is less than $\mu^{\frac{5}{3}}h^{-\frac{1}{3}}$ leave alone $\varepsilon^{-\frac{1}{2}}h^{-1}$ which are parts of right-hand expressions of (\ref{16-2-55}) and (\ref{16-2-57}) respectively. Thus we arrive to

\begin{claim}\label{16-2-65}
Estimates (\ref{16-2-55}) and (\ref{16-2-57}) hold with $\bar{A}$ which is our original pilot-model operator.
\end{claim}

\section{Superstrong magnetic field}
\label{sect-16-2-6}

Finally, let $\mu \ge h^{-1}$. Then as only matter is what is inside of the circle of the radius $C_0\mu^{-\frac{1}{2}}h^{\frac{1}{2}}$ (we can get rid off $h^{-\delta}$ factor in estimates) we get in successive approximations
$C\mu h^{-1} \times \mu^{-1}h  \times \mu^{-\frac{1}{2}}h^{\frac{1}{2}}\varepsilon^{-1} \times h^{-1}$ where $\mu^{-1}h$ is the size of the perturbation, $\mu^{-\frac{1}{2}}h^{\frac{1}{2}}\varepsilon^{-1} $ is $T$ (without rescaling); so we get $C\mu^{-\frac{1}{2}}h^{-\frac{1}{2}}\varepsilon^{-1}$.

We leave to the reader a rather easy
\begin{problem}\label{problem-16-2-16}
Prove that effectively in estimates one can make $\delta=0$ as well in $\mu^{-\frac{1}{2}+\delta}h^{\frac{1}{2}}$ radius and $\mu^{-\frac{1}{2}+\delta}h^{\frac{1}{2}}\varepsilon^{-1}$ time.
\end{problem}

So we arrive to
\begin{proposition}\label{prop-16-2-17}
Under assumptions \textup{(\ref{16-2-36})}, \textup{(\ref{16-2-37})}, $F=1$  and \textup{(\ref{16-2-54})} estimate
\begin{equation}
|e^\T(0,0,\tau)-\bar{e}^\T (0,0,\tau)|\le C\mu^{-\frac{1}{2}}h^{-\frac{1}{2}} \varepsilon^{-1}
\qquad \text{as\ \ }  \mu \ge h^{-1}.
\label{16-2-66}
\end{equation}
\end{proposition}

We leave to the reader yet another easy

\begin{problem}\label{problem-16-2-18}
Write correction to $\tau$ and may be $\alpha$ as $F\ne \const$ and  $g^{jk}\ne \const$.
\end{problem}

\section{Main theorem}
\label{sect-16-2-7}

Now we can write our main theorem of these two sections:

\begin{theorem}\label{thm-16-2-19}
For magnetic Schr\"odinger operator in domain $X\subset \bR^2$ such that $B(0,1)\subset X$ and satisfying conditions \textup{(\ref{16-0-5})}--\textup{(\ref{16-0-7})} the following estimates hold for $x\in B(0,\frac{1}{2})$

\medskip\noindent
(i) As $1\le \mu\le h^{-\frac{1}{2}}$
\begin{equation}
|e(x,x,0)- h^{-2}\cN^\W_x (0)|\le C\mu^{-1}h^{-1}+C\mu^{\frac{1}{2}}h^{-\frac{1}{2}}+ C\mu^2 h^{-\frac{1}{2}}
\label{16-2-67}
\end{equation}
and
\begin{multline}
\R^\W_{x(r)}\Def |e(x,x,0)-
h^{-2}\bigl(\cN^\W_x (0)+\cN^\W_{x,\corr(r)} (0)\bigr)|\le\\[3pt]
C\mu^{-1}h^{-1}+ C\mu^{\frac{1}{2}}h^{-\frac{1}{2}} +
C\mu h^{-1} \bigl(\mu^2 h)^{r+\frac{1}{2}}+ \\[3pt]
C\left\{\begin{aligned}
&\Bigr( h^{-1}\bigl(h \mu^{\frac{5}{2}}\bigr)^{r+\frac{1}{2}} +
\mu ^{\frac{1}{3}} h^{-\frac{2}{3}}\Bigr)\qquad&&\text{as \ \ } \mu\le h^{-\frac{2}{5}},\\
&\mu^{\frac{5}{3}}h^{-\frac{1}{3}}\qquad&&\text{as \ \ } \mu\ge h^{-\frac{2}{5}}.
\end{aligned}\right.
\label{16-2-68}
\end{multline}
(ii) As $h^{-\frac{1}{2}} \le \mu\le h^{-\frac{1}{2}}$
\begin{multline}
\R^{\W\prime\prime}_{x(r)} \Def \\
| \Bigl(e  (x,x,\tau) -h^{-2}\cN_{x,\corr(r)}-
\bar{e}_x  (x,x,\tau) +h^{-2}\bar{\cN}_{x,\corr(r)}\Bigr)|\le\\
C\mu^{\frac{1}{2}}h^{-\frac{1}{2}}+
C\left\{\begin{aligned}
&\mu^{-2} h^{-2}  (\mu^2h )^{r+\frac{1}{2}} \qquad&&\text{as\ \ } \mu \le h^{-\frac{1}{2}}\\
&h^{-1} \qquad&&\text{as\ \ }  \mu \le h^{-\frac{1}{2}},
\end{aligned}\right.\quad +\ \\
C\left\{\begin{aligned}
& h^{-1} \bigl(\mu ^{\frac{5}{2}}h \bigr)^{r+\frac{1}{2}}+ \mu^{\frac{1}{3}}  h^{-\frac{2}{3}} \qquad&&
\text{as\ \ } \mu\le h^{-\frac{2}{5}}\\
&\mu^{\frac{5}{3}}h^{-\frac{1}{3}} \qquad&&\text{as\ \ }  h^{-\frac{2}{5}}\le \mu\le h^{-\frac{1}{2}},\\
&\mu^{-\frac{1}{3}}h^{-\frac{4}{3}}
\qquad&&\text{as\ \ } \mu \ge h^{-\frac{1}{2}}
\end{aligned}\right.
\label{16-2-69}
\end{multline}
while for $r=0$
\begin{multline}
\R^{\W\prime\prime}_{x(0)}\Def |e  (x,x,\tau) - \bar{e}_x  (x,x,\tau)|\le C\mu^{\frac{1}{2}}h^{-\frac{1}{2}}+ \\
C\left\{\begin{aligned}
& h^{-1}\mu^{\frac{1}{2}}\qquad
&&\text{as\ \ }\mu\le h^{-\frac{1}{2}},\\
&\mu ^{-\frac{1}{2}}h^{-\frac{3}{2}}
&&\text{as\ \ }\text{as\ \ }\mu\le h^{-\frac{1}{2}};
\end{aligned}\right.
\label{16-2-70}
\end{multline}
where here and in (iii) $\bar{e}_x$ is constructed for a pilot-model in $x$;

\medskip\noindent
(iii) For magnetic Schr\"odinger-Pauli operator with $\mu \ge h^{-1}$
\begin{equation}
|e (x,x,\tau)-\bar{e}_x  (x,x,\tau)|\le C\mu^{\frac{1}{2}}h^{-\frac{1}{2}}.
\label{16-2-71}
\end{equation}
\end{theorem}

\section{Problems}
\label{sect-16-2-8}

Finally, let us formulate a series of the problems with probably simple parts (i) and really difficult parts (ii):

\begin{Problem}\label{problem-16-2-20}
As $\mu \le h^{-1}$ drop condition (\ref{16-0-5}) in the pointwise asymptotics:

\medskip\noindent
(i) Use the simple rescaling technique
$\mu \mapsto \mu \rho{\frac{1}{2}}\mu$,
$h\mapsto h\rho ^{-\frac{3}{2}}$ with the scaling function
$\rho (x) =  \max(\epsilon_0|V(x)|,\mu h, \mu^{-1})$ reducing operator to the similar one either satisfying condition (\ref{16-0-5}) \underline{or} with
$\mu h\asymp 1$ \underline{or} with $\mu=1$ and considered in subsection~\ref{book_new-sect-5-2-1}; in two latter cases condition (\ref{16-0-5})  is not needed;

\medskip\noindent
(ii) Using propagation, improve what follows from the simple rescaling technique.
\end{Problem}

\begin{Problem}\label{problem-16-2-21}
Derive  asymptotics  without condition (\ref{16-0-6}) replaced by the non-degeneracy assumption (\ref{book_new-14-1-1}) i.e. $|F|+|\nabla F|\asymp 1$:

\medskip\noindent
(i) Use the simple rescaling technique
$\mu \mapsto \mu \rho^{2}$, $h\mapsto h\rho ^{-1}$ with the scaling function
$\rho (x) = \max(\epsilon_0 |F|, \mu^{-\frac{1}{2}})$;

\medskip\noindent
(ii) Improve what follows from the simple rescaling technique.
\end{Problem}

\begin{Problem}\label{problem-16-2-22}
Derive asymptotics  under non-degeneracy assumption (\ref{16-0-8}) rather than (\ref{16-0-7}). The results of this section provide them at point $x$ where
$|\nabla V/F|\gg \bar{\gamma}\Def \max(\mu^{-1},\mu^{-\frac{1}{2}}h^{\frac{1}{2}})$.

\medskip\noindent
(i) Can one improve those results using that under condition \ref{16-0-8-+} and  $\textup{(\ref{16-0-8})}^-$ (when (\ref{16-0-8}) holds but \ref{16-0-8-+} does not) the drift dynamic is elliptic elliptic and hyperbolic respectively?

\medskip\noindent
(ii) Can one improve a trivial estimate $O(\mu h^{-1})$ at point $x$ where $|\nabla V/F|\lesssim \bar{\gamma}$? Will be results affected by the difference between cases of \ref{16-0-8-+} and  $\textup{(\ref{16-0-8})}^-$?
\end{Problem}

\chapter{Dirac energy: $2\D$-estimates}
\label{sect-16-3}

In this and the next sections we consider asymptotics of expression $\I$ defined by (\ref{16-0-2}).

\section{Tauberian formula}
\label{sect-16-3-1}

Let us consider first contribution of zone $\{|x-y|\ge C\gamma\}$.

\begin{proposition}\label{prop-16-3-1} Under conditions \textup{(\ref{16-0-5})}--\textup{(\ref{16-0-7})} the contribution of zone $\{|x-y|\ge C\gamma\}$, to the remainder is $O(\mu^{-1}h^{-1}\gamma ^{-\kappa})$ while the main part is given by the Tauberian approximation $\I^\T$, i.e. by the same expression \textup{(\ref{16-0-1})} with $e(x,y,0)$ replaced by its standard implicit Tauberian approximation with $T\asymp \epsilon \mu $
\textup{(\ref{16-0-11})}.
\end{proposition}

\begin{proof}
Recall that $e^{ih^{-1}tA}$ is the propagator of $A$ and $U(x,y,t)$ is its Schwartz' kernel.

Consider expression (\ref{16-0-1}) with $\omega (x,y)$ replaced by $\omega_\gamma(x,y)$ which is a cut-off of $\omega(x,y)$ in the zone $\{|x-y|\asymp\gamma\}$ and with the original functions $\psi_1,\psi_2$ replaced by $1$.
Let us replace \emph{one\/} copy of $e(x,y,\tau)$ by $e(x,y, \tau,\tau')= \bigl(e(x,y,\tau)-e(x,y,\tau')\bigr)$ with $\tau' \le \tau$ and the second copy by $e(x,y,\tau'')$ and denote the resulting expression by
$\I_\gamma (\tau,\tau',\tau'')$.

Now let us use decomposition
\begin{equation}
\omega_\gamma (x,y)=\gamma^{-d-\kappa}\int \psi_{1,\gamma}(x,z)\psi_{2,\gamma}(y,z)\,dz
\label{16-3-1}
\end{equation}
with some $\psi_1,\psi_2\in \sC^\infty_0$; as before subscript $\gamma$ means rescaling.

Then $\I_\gamma (\tau,\tau',\tau'')$ does not exceed
\begin{equation}
\sum_j C\gamma^{-\kappa} \|\varphi_j E(\tau,\tau')\varphi _j\|_1
\label{16-3-2}
\end{equation}
where $E(\tau,\tau')= E(\tau)-E(\tau')$, $\varphi _j$ are real-valued $\gamma$-admissible functions supported in $C_0\gamma$-vicinities of $z_j$ and balls $B(z_j,2C_0\gamma)$ cover domain $X$ with the multiplicity not exceeding $C_0$. Here we used that $\|E(\tau'')\|=1$. Since $E(\tau,\tau')$ is a positive operator and $\varphi_j=\varphi_j^*$,  one can replace the trace norm  by the trace itself and get
\begin{equation}
\sum_j C\gamma^{-\kappa} \Tr \varphi_j E(\tau,\tau' )\varphi _j =
C\gamma^{-\kappa} \Tr E(\tau,\tau') \bar{\psi}
\label{16-3-3}
\end{equation}
with $\bar{\psi}=\sum_j\varphi_j^2$.

Further, we know from the standard theory of Chapter~\ref{book_new-sect-13} that under conditions (\ref{16-0-5})--(\ref{16-0-7})
\begin{multline}
\| E(\tau, \tau' ) \bar{\psi} \|_1\le Ch^{-2}\bigl(|\tau -\tau'|+ CT^{-1}h\bigr) \\
\forall \tau ,\tau' \in [-\epsilon,\epsilon], \ T=\epsilon \mu
\label{16-3-4}
\end{multline}
and therefore
\begin{equation}
|  \I_\gamma (\tau,\tau',\tau'') | \le C\gamma^{-\kappa} h^{-2}\bigl(|\tau-\tau'|+ CT^{-1}h\bigr)
\label{16-3-5}
\end{equation}
in the same framework and therefore due to the standard Tauberian arguments we conclude that the contribution of zone $\{|x-y|\asymp\gamma\}$ to the Tauberian remainder estimate does not exceed $C\mu^{-1}h^{-1}\gamma^{-\kappa}$ which implies the statement immediately.
\end{proof}

However we need to consider also zone $\{|x-y|\le C\gamma \}$, complementary to one above. Assume that
\begin{equation}
\Omega _\kappa (z)= \sum_j D_{z_j} \Omega_{\kappa -1,j} + \Omega_{\kappa -1,0} \label{16-3-6}
\end{equation}
with the first subscript at $\Omega$ showing the degree of the singularity. Then
\begin{equation}
\omega _\kappa (x,x-y) \psi _\gamma(x-y) =
\sum_j D_{x_j} \bigl(\omega_{\kappa -1,j} \psi _\gamma \bigr)+
\omega_\kappa \psi'_\gamma+\omega_{\kappa-1}\psi'' _\gamma
\label{16-3-7}
\end{equation}
where $\psi_\gamma =\psi ( (x-y)\gamma^{-1})$ with $\psi$ supported in $B(0,1)$ and equal 1 in $B(0, {\frac 1 2})$ while $\psi'_\gamma$ is defined similarly with $\psi'$ supported in $B(0,1)\setminus B(0, {\frac 1 2})$ and the last term gains $1$ in the regularity.

After integration by parts expression $\I_{\kappa,\gamma}$, defined by (\ref{16-0-10}) with $\Omega$ replaced by $\Omega \psi_\gamma$, becomes
\begin{equation}
-h^{-1} \sum_j\iint \omega_{\kappa-1,j} (x,y) (hD_{x_j})
\bigl( e(x,y,\tau) \cdot e(y,x,\tau)\bigr)\, dxdy
\label{16-3-8}
\end{equation}
plus two other terms: the term defined by (\ref{16-0-1}) with the kernel
$\Omega '_{\kappa,j}$ of the same singularity $\kappa$, albeit without factor $h^{-1}$ and supported in the zone $\{|x-y|\ge \frac{1}{2}\gamma\}$, and the term defined by (\ref{16-0-1}) with the kernel $\Omega '_{\kappa-1,j}$, also without factor $h^{-1}$ and of singularity $(\kappa-1)$.

The former term could be considered as before yielding to the same remainder estimate $O(\mu^{-1}h^{1-d}\gamma^{-\kappa})$. To the latter term we can apply the same trick again and again raising power (and these terms are treated in the same manner (but simpler) as we deal below with (\ref{16-3-8}).

So, one needs to consider (\ref{16-3-8}) and thus, denoting the second copy of $e(y,x,\tau)$ by $f(y,x,\tau)$ and without using that they are equal
\begin{align}
&(hD_{x_j}) \bigl( e(x,y,\tau) \cdot f(y,x,\tau)\bigr) =\label{16-3-9}\\
&\bigl( hD_{x_j}e(x,y,\tau)\bigr) f(y,x,\tau)\ -\
e(x,y,\tau) \bigr(f(y,x,\tau)\,^t\!(hD_{x_j})\bigr) =\notag\\
&\bigl( P_{j,x}e(x,y,\tau)\bigr) f(y,x,\tau)\ \ -\
e(x,y,\tau)\bigr(f(y,x,\tau)\,^t\!P_{j,x}\bigr).\notag
\end{align}
Recall that $P_j=hD_j-\mu V_j(x)$ and $\,^t\!P_j=-hD_j-\mu V_j(x)$ is the dual operator. Also recall that if $e(x,y,\tau)$ and $f(y,x,\tau)$ are Schwartz kernels of $E(\tau)$ and $F(\tau)$, then $P_{j,x}e(x,y,\tau)$ and $f(y,x,\tau)\,^t\!P_{j,x}$ are those of $P_jE(\tau)$ and $F(\tau)P_j$.

Therefore we are  interested in the expressions of the type
\begin{equation}
h^{-1}\iint \omega_{\kappa-1} (x,y) e(x,y,\tau) f(x,y,\tau)
\psi_\gamma \,dxdy.
\label{16-3-10}
\end{equation}

If $\kappa \le 1$ then replacing $e(x,y,\tau)$ and $f(y,x,\tau)$  by their standard Tauberian expressions one gets an error not exceeding
$Ch^{-1} \times \mu^{-1}h^{-2}\gamma^{1-\kappa}$ because
$\|P_jE(\tau)\|\le C_0$, $\|P_jF(\tau)\|\le C_0$ where $F(\tau)$ is an operator with the Schwartz kernel $f(x,y,\tau)$ and also because

\begin{multline}
\sum_j\| \varphi_j   P_jE(\tau,\tau')\varphi_j \|_1\le
\sum_j\| \varphi_j   P_jE(\tau,\tau') \|_2 \cdot \|   E(\tau,\tau')\varphi_j  \|_2\le\\
\sum _j \| \varphi_j   P_jE(\tau,\tau')\|_2^2 +
\sum_j \|   E(\tau,\tau') \varphi_j\|_2^2=\\
\sum_j \Tr \varphi_j   P_jE(\tau,\tau')  P_j ^* \varphi_j +
\sum_j \Tr \varphi_j   E(\tau,\tau') \varphi_j  \\
\le Ch^{-2}\bigl(|\tau -\tau'|+ C\mu^{-1}h\bigr)\qquad \forall \tau,\tau' \in [-\epsilon,\epsilon]\label{16-3-11}
\end{multline}
which  also easily follows from Chapter~\ref{book_new-sect-13}.

So, in this case one gets remainder estimate
$O\bigl(\mu^{-1}h^{-1}\gamma^{-\kappa}+\mu^{-1}h^{-2}\gamma^{1-\kappa}\bigr)$ which is optimized to $O(\mu^{-1}h^{-1-\kappa})$ as $\gamma\asymp h$.

On the other hand, as $1<\kappa <2$ one can apply the same trick again since we did not use the fact that $e(.,.,.)$ and $f(.,.,.)$ coincide; then we arrive to the same estimates with $P_j$ replaced by $P_jP_k$ or even by
$P^J\Def P_{j_1}P_{j_2}\cdots P_{j_l}$:
\begin{equation}
\Tr P^J E(\tau,\tau')(P^J)^* \le Ch^{-2}(|\tau -\tau'|+ \mu^{-1}h)\qquad \forall
\tau,\tau'\in [-\epsilon,\epsilon].
\label{16-3-12}
\end{equation}

Finally,

\begin{remark}\label{rem-16-3-2}
note that as $\kappa \ne 1$, decomposition (\ref{16-3-6}) is always possible. Further, as $\kappa =1$ this decomposition  is possible as well provided one adds term $\varkappa (x) |x-y|^{-1}$ with an appropriate coefficient. On the other hand, if $\kappa=1$ and $\omega(x,y)=\varkappa (x) |x-y|^{-1}$ then this decomposition  is also possible but with $\omega_{0,j}(x,y)=\varkappa (x) (x_j-y_j)|x-y|^{-1}\log |x-y|$.
\end{remark}

So we arrive to

\begin{proposition}\label{prop-16-3-3}
Let conditions \textup{(\ref{16-0-5})}--\textup{(\ref{16-0-7})} be fulfilled. Then

\medskip\noindent
(i) As $0<\kappa < 2$ and \underline{either} $\kappa\ne 1$ \underline{or} $\kappa=1$ and $\omega(x,y)$ is replaced by
$\omega(x,y)-\varkappa (\frac{1}{2}(x+y)) |x-y|^{-1}$ with an appropriate smooth coefficient $\varkappa(x)$, with the error $O(\mu^{-1}h^{-1-\kappa})$ one can replace $e(x,y,\tau)$ by its standard Tauberian expression \textup{(\ref{16-0-11})} in the formula \textup{(\ref{16-0-11})} for $\I$.

\medskip\noindent
(ii) As $\kappa= 1$ and $\omega=\varkappa (\frac{1}{2}(x+y)) |x-y|^{-1}$, with the error $O(\mu^{-1}h^{-1-\kappa}|\log h|)$ one can replace $e(x,y,\tau)$ by its standard Tauberian expression \textup{(\ref{16-0-11})} in the formula \textup{(\ref{16-0-11})} for $\I$.
\end{proposition}

\begin{remark}\label{rem-16-3-4}
(i) The arguments above show that in an appropriate sense one can consider arbitrary $\kappa \in \bR$ and even in $\bC$.

\medskip\noindent
(ii) One needs only  (\ref{16-3-12}) rather than (\ref{16-0-5})--(\ref{16-0-7}), and (\ref{16-3-12}) holds as (\ref{16-0-7}) is replaced by a weaker non-degeneracy condition \ref{16-0-8-+};

\medskip\noindent
(iii) Furthermore (\ref{16-3-12}) with an extra factor $(1+\mu h |\log h|)$ in the right-hand expression holds under condition \textup{(\ref{16-0-8})}.
\end{remark}

\begin{Problem}\label{problem-16-3-5}
Can one prove the similar result for $\I_m$ defined by (1.0.6) V.~Ivrii \cite{ivrii:DE1} with $m\ge 3$?
\end{Problem}

Thus we  arrive to

\begin{proposition}\label{prop-16-3-6}
Let conditions \textup{(\ref{16-0-5})} and  \textup{(\ref{16-0-6})} be fulfilled.

\medskip\noindent
(i) Further, let \underline{either} condition  \ref{16-0-8-+}  be fulfilled \underline{or} condition \textup{(\ref{16-0-8})} be fulfilled
and $\mu \le h^{-1}|\log h|^{-1}$.  Then \textup{(\ref{16-0-11})} and statements (i), (ii) of proposition \ref{prop-16-3-3} hold.

\medskip\noindent
(ii) On the other hand,  let condition \textup{(\ref{16-0-8})} be fulfilled
and $h^{-1}|\log h|^{-1}\le \mu \le  h^{-1}$.  Then \textup{(\ref{16-0-11})} and statements (i), (ii) of proposition \ref{prop-16-3-3} hold  with an extra factor $(1+\mu h |\log h|)$ in the right-hand expressions.
\end{proposition}

\begin{remark}\label{rem-16-3-7}
Under certain assumptions (see \cite{ivrii:IRO4}) this result could be generalized for $d\ge 4$. However in calculations we use that $d=2$.\end{remark}

\section{Superstrong magnetic field}
\label{sect-16-3-2}

Consider Schr\"odinger-Pauli operator as $\mu h\ge \epsilon_0$. Then we arrive immediately

\begin{proposition}\label{prop-16-3-8}
Let $\mu h\ge \epsilon_0$. Then

\medskip\noindent
(i) Contribution of $\{|x-y|\ge \gamma\}$ to $\I$ does not exceed
$C\mu h^{-1}\gamma^{-\kappa}$;

\medskip\noindent
(ii) Further,
\begin{equation}
|\I|\le C\mu^{1+\frac{1}{2}\kappa} h^{-1-\frac{1}{2}\kappa}.
\label{16-3-13}
\end{equation}

\medskip\noindent
(iii)  Furthermore,  $\I=O(\mu^{-\infty})$ provided
\begin{equation}
V+\mu h F \ge \epsilon\qquad \forall x\in B(0,1);
\label{16-3-14}
\end{equation}
in particular it is the case as $\fz<1$ and $\mu h\ge C_0$.
\end{proposition}

\begin{proof}
Statement (i) trivially follows from the fact that $\sL^2$ norm of $e(.,,,\tau)$ does not exceed $C\mu^{\frac{1}{2}} h^{-\frac{1}{2}}$.

Meanwhile contribution of
$\{(x,y): x\in B(0,1) , y\in B(0,1), |x-y|\le \gamma\}$ does not exceed
$C\mu^2 h^{-2} \gamma^{2-\kappa}$. Adding to (i) we conclude that the sum reaches its  minimal $C\mu^{1+\frac{1}{2}\kappa} h^{-1-\frac{1}{2}\kappa}$ value as $\gamma=\mu^{-\frac{1}{2}}h^{\frac{1}{2}}$.

Finally, statement (iii) is due to the fact that   $e(x,y,0)=O(\mu^{-\infty})$ as condition (\ref{16-3-14}) is fulfilled in one of points $x,y$; see subsection~\ref{book_new-sect-13-5-1}.
\end{proof}

\begin{proposition}\label{prop-16-3-9}
Let conditions \textup{(\ref{16-0-1})}--\textup{(\ref{16-0-5})} be fulfilled. Let $\mu h\ge \epsilon$ and one of the nondegeneracy conditions
\begin{gather}
|\frac{V}{F}+ (2m+1)\mu h| + |\nabla \frac{V}{F}|\ge \epsilon
\label{16-3-15}\\
\shortintertext{or}
|\frac{V}{F}+ (2m+1)\mu h| + |\nabla \frac{V}{F}|\le \epsilon\implies
\det \Hess \frac{V}{F}\ge \epsilon
\tag*{$\textup{(\ref*{16-3-17})}^+$}\label{16-3-17-+}
\end{gather}
be fulfilled.

\medskip\noindent
(i) As  $\gamma \ge C_0\mu ^{-\frac{1}{2}}h^{-\frac{1}{2}}$ contribution of $\{|x-y|\ge \gamma\}$ to $\I$ does not exceed $C\gamma^{-2-\kappa}$;

\medskip\noindent
(ii) Estimate
\begin{equation}
|\I-\I^\T|\le C\mu^{\frac{1}{2}\kappa}h^{-\frac{1}{2}\kappa}
\label{16-3-16}
\end{equation}
holds as $\kappa\ne1$; as $\kappa=1$ this estimate holds with an extra factor $|\log \mu|$ in its right-hand expression;

\medskip\noindent
(iii) If instead of \textup{(\ref{16-3-15})} or  \ref{16-3-17-+} we assume that
\begin{equation}
|\frac{V}{F}+ (2m+1)\mu h| + |\nabla \frac{V}{F}|\le \epsilon\implies
|\det \Hess \frac{V}{F}\ge \epsilon|
\label{16-3-17}
\end{equation}
then (i), (ii) hold with an extra factor $|\log \mu|$ in the right-hand expressions;

\medskip\noindent
(iv) $\I-\I^\T=O(\mu^{-\infty})$ under ellipticity condition
$|V+(2m+1)\mu hF|\ge \epsilon$. In particular, it happens as
$1<\fz\notin 2\bZ+1$  and $\mu h\ge C_0$.
\end{proposition}

\begin{proof}
To prove statement (i) recall that the drift speed does not exceed $C\mu^{-1}$ and therefore Hilbert-Schmidt norm of $\psi E(\tau) \psi'$ does not exceed $C$ as $\psi,\psi'$ are $\sC^\infty_0$-functions with
$\dist (\supp \psi,\supp \psi')\asymp  1$.

Really, it is true for a Hilbert-Schmidt norm of
$\bigl(E(\tau )-E(\tau')\bigr)\psi$ with
$|\tau-\tau'|\le \mu^{-1}h$ and then by Tauberian theorem it is true for
a Hilbert-Schmidt norm of $\bigl(E(\tau )-E^\T(\tau)\bigr)\psi'$ with $E^\T$ operator with the Schwartz kernel $e^\T$ with time $T\asymp \mu$. However
$\psi E^\T \psi'$ is negligible as $T\le \epsilon \mu$ due to propagation results.

Then rescaling $x\mapsto x/\gamma$,
$y\mapsto y/\gamma$, $h\mapsto h/\gamma$, $\mu \mapsto \mu\gamma$ we conclude that for $\gamma$-admissible $\psi,\psi'$ with
$\dist (\supp \psi,\supp \psi')\asymp  1$  Hilbert-Schmidt norm of
$\psi E(\tau) \psi'$ does not exceed $C$ and therefore the contribution of
\begin{equation*}
K(z)=\{(x,y): |x-z|\le 2\gamma, |y-z|\le 2\gamma, |x-y|\ge \gamma\}
\end{equation*}
does not exceed $\gamma^{-\kappa}$; therefore taking partition we conclude that  the contribution of zone
$\{(x,y)\in B(0,1) \times B(0,1), |x-y|\ge \gamma\}$ does not exceed
$C \gamma^{-2-\kappa}$.

\medskip\noindent
Statement (iv) follows from the fact $e-e^\T=O(\mu^{-\infty})$ under ellipticity condition.

\medskip\noindent
To prove statement (ii) we apply the same arguments as in proposition~\ref{prop-16-3-3}. While (\ref{16-3-1})--(\ref{16-3-3}), (\ref{16-3-6})--(\ref{16-3-10}) remain true, (\ref{16-3-4}), (\ref{16-3-5})  should be modified: factor $h^{-2}$ must be replaced by $\mu h^{-1}$.

We claim that
\begin{claim}\label{16-3-18}
As $0<\kappa<1$ (\ref{16-3-11}) holds with factor $h^{-2}$ replaced by
$\mu ^{\frac{3}{2}}h^{-\frac{1}{2}}$.
\end{claim}
Really, the same arguments are applied as before albeit now $P_jE(\tau)$ is bounded by $C_0(\mu h)^{\frac{1}{2}}$ rather than by $C_0$.

Then we arrive to the remainder estimate
$O\bigl( \gamma^{-\kappa} + \mu^{\frac{1}{2}}h^{-\frac{3}{2}}\gamma^{1-\kappa}\bigr)$ and optimizing by $\gamma= \bar{\gamma}\Def C\mu ^{-\frac{1}{2}}h^{\frac{1}{2}}$ we arrive to the remainder estimate $O(\bar{\gamma}^{-\kappa})$.

As $\kappa=1$ we get the same estimate albeit with a logarithmic factor.

To tackle the case $1<\kappa <2$ one should replace (\ref{16-3-12}) by
\begin{multline}
\Tr P^J E(\tau,\tau')(P^J)^* \le
C\mu h^{-1}(|\tau -\tau'|+ \mu^{-1}h) (\mu h) ^{|J|} \\
\forall \tau,\tau'\in [-\epsilon,\epsilon]
\tag*{$\textup{(\ref*{16-3-12})}'$}\label{16-3-12-'}
\end{multline}
as $\mu h\ge 1$; one can prove it easily as
$\|P^J E(\tau)\|\le C (1+\mu h) ^{\frac{1}{2}|J|}$.

Then as $1<\kappa<2$ we arrive to the remainder estimate
$O(\gamma^{-\kappa}+ \mu  h^{-1}\gamma^{2-\kappa})$ and optimizing by
$\gamma= \bar{\gamma}$ we again arrive to the remainder estimate $O(\bar{\gamma}^{-\kappa})$.

Finally, the above arguments imply (iii).
\end{proof}

\chapter{Dirac energy: $2\D$-calculations}
\label{sect-16-4}

\section{Pilot Model}
\label{sect-16-4-1}

\subsection{Transformations}
\label{sect-16-4-1-1}

Consider pilot-model operator (\ref{16-1-1}). Let us rescale as before\footnote{\label{foot-16-10} We  need to add factor
$\mu^{-4+\kappa}=\mu^{-2}\times \mu^{-2}\times \mu^{-\kappa}$ where two factors $\mu^{-2}$ are coming from $dx$ and $dy$ and  $\mu ^\kappa$ from $\omega$.}. Then as  $U(x,y,t)$ is defined by (\ref{16-1-9})--(\ref{16-1-10}), and
$e (x,y,\tau)$ is given by (\ref{16-1-13}), and $^\dag$ means a complex conjugation we get
\begin{multline}
\I  \Def  (2\pi)^{-2} \mu^{-4+\kappa}\iint\iint \omega(\mu^{-1}x,\mu^{-1}y) (t't'')^{-1}
\times \\[4pt]
\shoveright{e^{-i\hbar^{-1}(t'-t'')\tau}
U(x,y,t')U^\dag (x,y,t'')\,dt'dt''\,dxdy=}\\[4pt]
\frac{1}{4}(2\pi)^{-4}\hbar^{-2}\mu^\kappa\iint\iint \omega(x,y) \csc(t')\csc(t'') (t't'')^{-1}  \times\\[4pt]
\exp
\Bigl(i\hbar^{-1}\bigl(\bar{\phi}(x,y,t')-\bar{\phi}(x,y,t'')-(t'-t'')\tau \bigr)\Bigr)
\,dt'dt''\, dxdy
\label{16-4-1}
\end{multline}
with
\begin{multline}
\bar{\phi}(x,y,t')- \bar{\phi}(x,y,t'') -(t'-t'')\tau =\\
-\frac{1}{4}\bigl(\cot(t') -\cot(t'')\bigr)(x_1-y_1)^2\\
-\frac{1}{4}\cot(t')(x_2-y_2+2t'\varepsilon )^2  +\frac{1}{4}\cot(t'')(x_2-y_2+2t''\varepsilon)^2\\ + (x_1+y_1+2\varepsilon)(t'-t'')\varepsilon -(t'-t'')(\tau+\varepsilon^2).
\label{16-4-2}
\end{multline}
Recall that in the original coordinates
$\omega(x,y)= \Omega \bigl(\frac{1}{2}(x+y);x-y\bigr)$ where $\Omega (.,.)$ is uniformly smooth with respect to the first variable and positively homogeneous of degree $-\kappa$ with respect to the second one as $|x-y|\le 1$. Therefore without any loss of the generality one can assume that
\begin{equation}
\omega(x,y)= \Omega \bigl(\frac{1}{2}(x+y);x-y\bigr)\psi  (x-y)
\label{16-4-3}
\end{equation}
where now $\Omega$ is positively homogeneous of degree $-\kappa$ and
$\psi\in \sC^\infty_0 \bigl(B(0,1)\bigr)$ equal $1$ in $B(0,\frac{1}{2})$; the difference would be of the same nature but with $\kappa$ replaced by $0$.

We can replace variables $x,y$ with new variables $\x \Def \frac{1}{2}(x+y)$ and $z\Def (x-y)$ and rescale. Note that (\ref{16-4-2}) depends on $\x$ in a very specific way: it does not depend on $\x_2$ at all and  it is linear with respect to $\x_1$; thus (\ref{16-4-1}) becomes\footnote{\label{foot-16-11}
In these calculations we skip $\x_2$ as an argument of $\Omega$ and integration by $d\mu^{-1}\x_2$}
\begin{multline*}
\frac{1}{4}(2\pi)^{-3}\hbar^{-2}\mu^{\kappa+2} \iint\iint
\hat{\Omega} \bigl(2h^{-1}\varepsilon (t'-t''), z\bigr) \csc(t')\csc(t'')(t't'')^{-1} \times \\[4pt]
\exp \Bigl(i\hbar^{-1}\Bigl(-\frac{1}{4}\bigl(\cot(t') -\cot(t'')\bigr)z_1 ^2
-\frac{1}{4}\cot(t')(z_2+2t'\alpha\mu^{-1})^2  +\\
\frac{1}{4}\cot(t'')(z_2+2t''\varepsilon )^2  -(t'-t'')(\tau-\varepsilon^2)\Bigr)\Bigr)
\,dt'dt''\,\psi(\mu^{-1}z)\,dz_1dz_2
\end{multline*}
where $\hat{\Omega}(\lambda;z)=F_{\mathsf{x}_1\to \lambda}\Omega$ is a partial Fourier  transform. Rewriting the third line in (\ref{16-4-2}) as
\begin{multline*}
-\frac{1}{4}\bigl(\cot(t) -\cot(t')\bigr)(x_2-y_2 +(t+t')\varepsilon)^2\\
-\frac{1}{4}\bigl(\cot(t) -\cot(t')\bigr)(t-t')^2\varepsilon^2-\\
\frac{1}{2} \bigl(\cot(t) + \cot(t')\bigr)
\bigl(x_2-y_2 +(t+t')\varepsilon\bigr)(t-t') \varepsilon
\end{multline*}
and replacing $t'=t+s$, $t''=t-s$ we arrive to
\begin{multline}
\I  =\\
\frac{1}{2}(2\pi)^{-3}\hbar^{-2}\mu^{\kappa+2} \iint\iint
\hat{\Omega} (4h^{-1}\varepsilon s, z)\csc(t+s)\csc(t-s)(t+s)^{-1}(t-s)^{-1} \times
\\[3pt]
\exp \Bigl(i\hbar^{-1}\Bigl[-\frac{1}{4}\bigl(\cot(t+s) -\cot(t-s)\bigr) \bigl(z_1 ^2 + (z_2+2t\varepsilon)^2  + 4s^2\varepsilon^2\bigr)\\
- \bigl(\cot(t+s) +\cot(t-s)\bigr) s\varepsilon -2s(\tau-\varepsilon^2)\Bigr]\Bigr)
\,dt ds\,\psi(\mu^{-1}z) dz_1dz_2.
\label{16-4-4}
\end{multline}

\begin{remark}\label{rem-16-4-1}
(i) For a sake of simplicity we assume (without any loss of the generalization) that $\alpha>0$.

\medskip\noindent
(ii) If $\alpha  \ll 1$ in the above analysis applied to the general operators we need to take $\alpha$-admissible with respect to $\mathsf{x}$ function $\Omega_\alpha$ instead of $\Omega$; thus up to the shift we replace $\Omega(\mathsf{x},\cdot)$ by
$\Omega(\mathsf{x}\alpha^{-1},\cdot)$ and thus
$\hat{\Omega}(2h^{-1}\varepsilon s,\cdot)$ by
$\alpha^2\hat{\Omega}(2h^{-1}\varepsilon \alpha s,\cdot)$. In this case we will use notation $\I_\alpha$ instead of $\I$.

In the same time we need to consider separately $|z|\le \mu \alpha$ and
$|z|\ge \mu \alpha$. For a while we will not assume that $\varepsilon \asymp \mu^{-1}\alpha$.
\end{remark}

\subsection{Case $ \alpha^2 \gg \mu h$}
\label{sect-16-4-1-2}

Assume first that
\begin{equation}
\alpha^2 \ge \mu h^{1-\delta}
\label{16-4-5}
\end{equation}
where for a sake of simplicity we assume that $\alpha >0$ (and therefore $\varepsilon>0$).

\begin{remark}\label{rem-16-4-2}
(i) In the virtue of the factor $\hat{\Omega}(2\hbar^{-1} \alpha ^2 s,\cdot)$ under assumption (\ref{16-4-5})  we need to consider only
$|s|\le h^\delta$  and therefore we can consider separately
$|t'|\le \epsilon_0,|t''|\le \epsilon_0$ and
$|t'|\ge \epsilon_0,|t''|\ge \epsilon_0$;

\medskip\noindent
(ii) Note that due to section~\ref{book_new-sect-6-3} contribution of zone
$\{|t'|\le \epsilon_0,|t''|\le \epsilon_0\}$ defined by integral expressions (\ref{16-4-1}) or  (\ref{16-4-4}) with an extra factor  $\bar{\chi}_{\epsilon_0}(t')$ or
$\bar{\chi}_{\epsilon_0}(t'')$ or $\bar{\chi}_{\epsilon_0}(t)$ differs from the same expression for non-magnetic Schr\"odinger operator  by
$O(\mu h^{-1-\kappa} \times \hbar^\kappa)=O(\mu^{\kappa+1}h^{-1})$ as
$\kappa \ne 1$ and by $O(\mu^2 h^{-1}|\log \mu|)$ as $\kappa = 1$.

\medskip\noindent
(iii) Furthermore, if we remove from this expression for a non-magnetic Schr\"odinger operator cut-off $\{|t'|\ge \epsilon_0\}$ then the error would not exceed the same expression as well.
\end{remark}

Let us consider contribution of zone
$\{|t'|\ge \epsilon_0,|t''|\ge \epsilon_0\}$ defined by an integral expressions (\ref{16-4-1}) or (\ref{16-4-4}) with an extra factor  $\bigl(1-\bar{\chi}_{\epsilon_0}(t')\bigr)$. Due to remark~\ref{rem-16-4-2}(i) we need to consider only $t',t''$ belonging to the same tick.

Let us consider first zone
\begin{equation}
\{|s|\ge \hbar/\alpha^2, |\sin (t)|\ge C|s|\}.
\label{16-4-6}
\end{equation}
Then $\cot(t')-\cot(t'')\asymp -\sin^{-2} (t)s$ and integration by parts with respect to $z$ delivers one of the factors
\begin{gather}
|s|^{-1}\hbar
\bigl(|z_1|^2 +|z_2|\cdot|z_2+\varepsilon t| \bigr)^{-1}|\sin(t)|^2, \label{16-4-7}\\[3pt]
|s|^{-1}\hbar
\bigl(|z_1|^2 +|z_2+\varepsilon t|^2\bigr)^{-1} |\sin(t)|^2.\label{16-4-8}
\end{gather}
Thus integrating by parts many times in the zone where both of these factors are less than $1$ we acquire factors
\begin{gather}
\quad\;\Bigl(1+|s|\hbar^{-1}
\bigl(|z_1|^2 +|z_2|\cdot |z_2+\varepsilon t|)|\sin(t)|^{-2}\Bigr)^{-l}, \label{16-4-9}\\[3pt]
\Bigl(1+|s|\hbar^{-1}
\bigl(|z_1|^2 +|z_2+\varepsilon t|^2\bigr)|\sin(t)|^{-2}\Bigr)  ^{-l}
\label{16-4-10}
\end{gather}
respectively. Multiplying by $|z|^{-\kappa}$ and integrating we get after multiplication by $|\sin (t)|^{-2}$
\begin{equation}
\asymp |\sin (t)|^{-\kappa} |s|^{-1+\frac{1}{2}\kappa}\hbar^{1-\frac{1}{2}\kappa}
\label{16-4-11}
\end{equation}
and integrating by $t$ over one tick intersected with $\{t:\ |\sin(t)|\ge |s|\}$ we get
\begin{equation}
C\hbar^{1-\frac{1}{2}\kappa}\left\{\begin{aligned}
&|s|^{-1+\frac{1}{2}\kappa}\qquad &&\kappa<1,\\
&|s|^{-\frac{1}{2}\kappa}\qquad &&\kappa>1,\\
&|s|^{-\frac{1}{2}}(1+|\log |s||)\qquad &&\kappa=1.
\end{aligned}\right.
\label{16-4-12}
\end{equation}
This expression (\ref{16-4-12}) must be \underline{either} integrated by with respect to $s$: $|s|\le \hbar/\alpha^2$ \underline{or} multiplied by
$(\hbar/\alpha^2)^l |s|^{-l}$ due to factor $\hat{\Omega}$ and integrated over $|s|\ge \hbar/\alpha^2$, resulting in both cases in the same answer
which is the value of $\textup{(\ref{16-4-12})}\times |s|$ calculated as
$s= \hbar/\alpha^2$ i.e.
\begin{equation}
C\hbar ^{1-\frac{1}{2}\kappa}\left\{\begin{aligned}
&(\hbar/\alpha^2)^{\frac{1}{2}\kappa}\qquad &&\kappa<1,\\
&(\hbar/\alpha^2)^{1-\frac{1}{2}\kappa}\qquad &&\kappa>1,\\
&(\hbar/\alpha^2)^{\frac{1}{2}}(1+|\log (\hbar/\alpha^2)|)\qquad &&\kappa=1.
\end{aligned}\right.
\label{16-4-13}
\end{equation}

In addition to zone (\ref{16-4-6}) we  need to consider  zone
\begin{equation}
\{|\sin(t')|\asymp |s|,\ |\sin(t'')|\le |s| \};
\label{16-4-14}
\end{equation}
its tween  $\{|\sin(t')|\asymp |s|,\ |\sin(t'')|\le |s| \}$ is considered in the same way.

In zone (\ref{16-4-14})  $|\cot (t')-\cot(t'')|\asymp |\sin (t'')|^{-1}$ and in this case (\ref{16-4-7}), (\ref{16-4-8}) are replaced by
\begin{gather}
\hbar \bigl(|z_1|^2 +|z_2|\cdot |z_2+\varepsilon t|\bigr)^{-1}|\sin(t'')|,
\tag*{$\textup{(\ref*{16-4-7})}'$}\label{16-4-7-'}\\[3pt]
\hbar \bigl(|z_1|^2 +|z_2+\varepsilon t|^2\bigr)^{-1}|\sin(t'')|
\tag*{$\textup{(\ref*{16-4-8})}'$}\label{16-4-8-'}
\end{gather}
and (\ref{16-4-9}), (\ref{16-4-10}) by
\begin{gather}
\quad\;\Bigl(1+\hbar^{-1}
\bigl(|z_1|^2 +|z_2|\cdot |z_2+\varepsilon t|\bigr)|\sin(t'')|^{-1}\Bigr)^{-l}, \tag*{$\textup{(\ref*{16-4-9})}'$}\label{16-4-9-'}\\[3pt]
\Bigl(1+\hbar^{-1}
\bigl(|z_1|^2 +|z_2+\varepsilon t|^2)|\sin(t'')|^{-1}\Bigr)^{-l}.
\tag*{$\textup{(\ref*{16-4-10})}'$}\label{16-4-10-'}
\end{gather}
Then, multiplying by $|z|^{-\kappa}$ and integrating we get  after multiplication by $|\sin(t'')|^{-1}|\sin(t')|^{-1}$
\begin{equation}
\hbar ^{1-\frac{1}{2}\kappa}
|\sin (t'')|^{-\frac{1}{2}\kappa}|s|^{-1};
\tag*{$\textup{(\ref*{16-4-11})}'$}\label{16-4-11-'}
\end{equation}
then integrating by $|t''|$ over one tick but intersected with
$\{|\sin(t'')|\le |s|\}$ we get
\begin{equation}
\hbar ^{1-\frac{1}{2}\kappa} |s|^{-\frac{1}{2}\kappa}.
\tag*{$\textup{(\ref*{16-4-12})}'$}\label{16-4-12-'}
\end{equation}
Finally, \underline{either} integrating  over $|s|\le \hbar/\alpha^2$ \underline{or} multiplying by $|s|^{-l}(\hbar/\alpha^2)^l$ and integrating over $|s|\ge \hbar/\alpha^2$ we get in both cases the same answer
$\hbar ^{1-\frac{1}{2}\kappa}(\hbar\alpha^2)^{1-\frac{1}{2}\kappa}$ not exceeding (\ref{16-4-13}).

Therefore the total contribution of zones (\ref{16-4-6}) and (\ref{16-4-14}) is given by expression (\ref{16-4-13}). Then multiplying by
$|k|^{-2}\mu^\kappa h^{-2}\alpha^2$ we get after summation with respect to $k:|k|\ge 1$ the value as $k=1$ i.e.
$\textup{(\ref{16-4-13})}\times \mu^\kappa h^{-2}\alpha^2$.

Therefore we arrive to

\begin{proposition}\label{prop-16-4-3}
For the pilot-model operator with  $1\le \mu \le h^{-1}$ and
$\alpha\ge \epsilon_0\mu^{-1}$  under additional assumption \textup{(\ref{16-4-5})}
\begin{multline}
|\I_\alpha ^\T - \I_\alpha ^{\T\,\prime}|\le
CR^\W (\alpha)\Def \\
C \left\{\begin{aligned}
&\mu^{1+\kappa} h^{-1} \alpha^{2-\kappa}\quad &&0<\kappa<1,\\
&\mu^2  h^{-\kappa} \alpha^\kappa \quad &&1<\kappa<2,\\
& \mu^2 h^{-1} \alpha(1+|\log (\mu h/\alpha^2)|)\quad &&\kappa=1
\end{aligned}\right.
\label{16-4-15}
\end{multline}
where $\I_\alpha ^{\T\,\prime}$ is a Tauberian expression for $\I$ albeit with $T=\epsilon_0 \mu^{-1}$.
\end{proposition}

Now we need to explore the difference between $\I_\alpha ^{\T\,\prime}$ and
$\cI^\W_\alpha$ defined by (\ref{16-0-9}).

\begin{proposition}\label{prop-16-4-4}
For the pilot-model operator with  $1\le \mu \le h^{-1}$ and
$\alpha\ge \epsilon_0\mu^{- 1}$
\begin{equation}
|\I_\alpha ^{\T\,\prime} - \cI^\W_\alpha | \le C\mu^2h^{-\kappa}\alpha^2.
\label{16-4-16}
\end{equation}
\end{proposition}

\begin{proof}
Repeating arguments of subsection~\ref{book_new-sect-6-3-4} one can prove easily that
\begin{equation}
|\I_\alpha ^{\T\,\prime} -
\sum_{l:0\le l\le L-1,j=0,1} \cI^\W_{(l)j,\alpha} \mu^l|
\le C(\mu h)^L h^{-2-\kappa}\alpha^2 + C\mu^{\kappa+1}h^{-1}\alpha^2
\label{16-4-17}
\end{equation}
where $\cI^\W_{(l)j,\mu^{-1}}$ is defined by (\ref{16-0-9}) albeit with $\Omega(\frac{1}{2}(x+y),x-y)$ multiplied by a homogeneous polynomial of degree $(l+j)$ with respect to $(x-y)$ and for $l=j=0$ this polynomial is $1$ and with integrals  taken \underline{only} over zone $\{(x,): |x-y|\lesssim \mu^{-1}\}$.

Since without any loss of the generality we can assume that $\Omega$ is even with respect to the second argument,  all terms with odd $(l+j)$ vanish and since replacing $\mu$ by $-\mu$ we should arrive to the same result,   all terms with odd $l$ vanish and therefore
\begin{claim}\label{16-4-18}
In (\ref{16-4-17})  all terms vanish except those with $j=0$ and even $l$.
\end{claim}
Picking $L=2$ we arrive then to (\ref{16-4-16}).
\end{proof}

Combining propositions~\ref{prop-16-4-3} and~\ref{prop-16-4-4} and noting that the right-hand expression of (\ref{16-4-16}) does not exceed $CR^\W(\alpha)$ we arrive to

\begin{corollary}\label{cor-16-4-5}
For the pilot-model operator with  $1\le \mu \le h^{-1}$ and
$\alpha\ge \epsilon_0\mu^{-1}$  under additional assumption \textup{(\ref{16-4-5})}
\begin{equation}
|\I_\alpha ^{\T} - \cI^\W_\alpha | \le CR^\W(\alpha).
\label{16-4-19}
\end{equation}
\end{corollary}

\subsection{Improvement}
\label{sect-16-4-1-3}

In a certain case (under assumption (\ref{16-4-22}) below) we can improve the results of the previous subsubsection~\ref{sect-16-4-1-2}. To do this we note that (an easy proof is left to the reader)

\begin{remark}\label{rem-16-4-6}
Let $\zeta \le \varepsilon |t|$ with
\begin{equation}
\zeta \Def (\hbar/|s|)^{\frac{1}{2}}|\sin(t)|.
\label{16-4-20}
\end{equation}
Then the sum of (\ref{16-4-9}) and  (\ref{16-4-10}), multiplied by $|z|^{-\kappa}\,dz$ and integrated, does not exceed
\begin{equation}
C\left\{\begin{aligned}
&\zeta^2 (\varepsilon |t|)^{-\kappa}\qquad &&0<\kappa <1,\\
&\zeta ^{4-2\kappa} (\varepsilon |t|)^{\kappa-2}\qquad &&1<\kappa <2,\\
& \zeta^2  (\varepsilon |t|)^{-1}(1+|\log (\zeta/\varepsilon |t|))\qquad &&\kappa=1.
\end{aligned}\right.
\label{16-4-21}
\end{equation}
\end{remark}

We need to multiply by $|\sin(t)|^{-2}\,dt$ and integrate. Then as $0<\kappa<1$ main contribution came from $\sin(t)\asymp 1$ and improvement is possible iff then $\zeta \le \varepsilon $ (as $|k|=1$ provided the large part of contribution). But as $s \asymp \hbar/\alpha^2$, $\sin(t)\asymp 1$  (which was the main contributor in the previous subsubsection) $\zeta \asymp \alpha$ while $\varepsilon =\mu^{-1}\alpha$.

So, improvement is possible only for $1<\kappa<2$. In this case main contribution came from $\sin(t)\asymp s$ and improvement is possible if then
$\zeta \asymp (\hbar  |s|)^{\frac{1}{2}} \le \mu^{-1}\alpha$.
Setting $|s|=\hbar/\alpha^2$ we have $\zeta \asymp \hbar/\alpha$ and therefore  improvement is possible if
\begin{equation}
\alpha^2\ge \mu^2h.
\label{16-4-22}
\end{equation}

Consider zone (\ref{16-4-6}) first. Then multiplication by  $|\sin(t)|^{-2}\,dt$ and integration results instead of (\ref{16-4-12}) in
\begin{equation}
C\left\{\begin{aligned}
&\hbar ^{\frac{1}{2}}|s|^{-\frac{1}{2}} \varepsilon^{1-\kappa} \qquad
&&1<\kappa <\frac{3}{2},\\
&\hbar^{2-\kappa}|s|^{1-\kappa}\varepsilon^{\kappa-2}\qquad
&&\frac{3}{2}<\kappa <2,\\
& \hbar^{\frac{1}{2}}|s|^{-\frac{1}{2}}  \varepsilon^{-\frac{1}{2}}
(1+|\log (\hbar|s|/\varepsilon ^2))\qquad &&\kappa=\frac{3}{2}
\end{aligned}\right.
\label{16-4-23}
\end{equation}
where we set $|k|=1$.
Finally, \underline{either} integrating  over $|s|\le \hbar/\alpha^2$ \underline{or} multiplying by $|s|^{-l}(\hbar/\alpha^2)^l$ and integrating over $|s|\ge \hbar/\alpha^2$ we get instead of (\ref{16-4-13})
\begin{equation}
C\left\{\begin{aligned}
&\mu^\kappa h \alpha^{-\kappa} \qquad
&&1<\kappa <\frac{3}{2},\\
&\mu^{6-3\kappa}h ^{4-2\kappa} \alpha^{3\kappa-6} \qquad
&&\frac{3}{2}<\kappa <2,\\
& \mu^{\frac{3}{2}} h \alpha^{-\frac{3}{2}}
(1+|\log (\mu^2 h / \alpha^2)|)\qquad &&\kappa=\frac{3}{2}.
\end{aligned}\right.
\label{16-4-24}
\end{equation}

Consider zone (\ref{16-4-14}); we apply remark~\ref{rem-16-4-6} albeit with
\begin{equation}
\zeta \Def (\hbar |\sin(t'')|)^{\frac{1}{2}}.
\tag*{$\textup{(\ref*{16-4-20})}'$}\label{16-4-20-'}
\end{equation}
Again multiplying by $|\sin(t')|^{-1}|\sin(t'')|^{-1}$ and
integrating by $|t''|$ over one tick but intersected with
$\{|\sin(t'')|\le |s|\}$ and setting $|\sin(t')|=|s|$, and finally \underline{either} integrating  over $|s|\le \hbar/\alpha^2$ \underline{or} multiplying by $|s|^{-l}(\hbar/\alpha^2)^l$ and integrating over
$|s|\ge \hbar/\alpha^2$, we arrive to the same answer, not exceeding (\ref{16-4-24}).

Finally, multiplying by $|k|^{-2}\mu^\kappa h^{-2}\alpha^2$ we get after summation with respect to $k: |k|\ge 1$ the following improvement to proposition~\ref{prop-16-4-3}

\begin{proposition}\label{prop-16-4-7}
For the pilot-model operator with  $1\le \mu \le h^{-1}$ and
$\alpha\ge \max (\mu^{-1},\mu h^{\frac{1}{2}},
\mu ^{\frac{1}{2}}h^{\frac{1}{2}-\delta})$
\begin{multline}
|\I_\alpha ^\T - \I_\alpha ^{\T\,\prime} |\le CR_1^\W(\alpha)\Def \\
C\left\{\begin{aligned}
&\mu^{2\kappa} h^{-1} \alpha^{2-\kappa} \qquad
&&1<\kappa <\frac{3}{2},\\
&\mu^{6-2\kappa}h ^{2-2\kappa} \alpha^{3\kappa-4} \qquad
&&\frac{3}{2}<\kappa <2,\\
& \mu^{3} h^{-1} \alpha^{\frac{1}{2}}
(1+|\log (\mu^2 h / \alpha^2)|)\qquad &&\kappa=\frac{3}{2}.
\end{aligned}\right.
\label{16-4-25}
\end{multline}
\end{proposition}

\begin{remark}\label{rem-16-4-8}
Right-hand expression of (\ref{16-4-25}) in comparison with (\ref{16-4-15}) has factor
\begin{equation}
C\left\{\begin{aligned}
&(\mu^2h/\alpha^2)^{\kappa-1}\qquad
&&1<\kappa <\frac{3}{2},\\
&(\mu^2h/\alpha^2)^{2-\kappa}  \qquad
&&\frac{3}{2}<\kappa <2,\\
& (\mu^2h/\alpha^2)^{\frac{1}{2}}
(1+|\log (\mu^2 h )|)\qquad &&\kappa=\frac{3}{2}.
\end{aligned}\right.
\label{16-4-26}
\end{equation}
\end{remark}

While we cannot directly apply estimate (\ref{16-4-16}) now as $\mu^2h^{-\kappa}$ may be greater than $R^\W_1(\alpha)$,  we can apply (\ref{16-4-17})--(\ref{16-4-18}) with $L=4$ resulting in

\begin{proposition}\label{prop-16-4-9}
For the pilot-model operator with  $1\le \mu \le h^{-1}$ and
$\alpha\ge \max (\mu^{-1},\mu h^{\frac{1}{2}},
\mu ^{\frac{1}{2}}h^{\frac{1}{2}-\delta})$
\begin{equation}
|\I_\alpha ^\T - \cI_\alpha ^\W -\cI_{\corr}^\W |\le CR_1^\W(\alpha)
\label{16-4-27}
\end{equation}
with $\cI_\corr^\W \Def \cI_{(2)0,\alpha}^\W $ but taken \underline{only} over zone $\{|x-y|\le \mu^{-1}\}$.
\end{proposition}

\subsection{Case $\alpha^2 \not\gg \mu h$}
\label{sect-16-4-1-4}

Assume now that $\alpha \ge \mu^{-1}$ but  (\ref{16-4-5}) fails. Then in contrast to the previous we will need to compute contributions of pair of ticks $(k',k'')$ with $k'\ne k''$. Then if $t'$, $t''$ belong to $k'$-th and $k''$-th ticks respectively we denote $r=k'-k''$ and $s=t'-t''-2\pi r$.

\paragraph{Contribution of $k'=k''\ne 0$.}
\label{sect-16-4-1-4-1}

Consider case $k'-k''=r=0$ first. As $\alpha^2 \ge \hbar$  applying the same arguments as before   we get the same answer (\ref{16-4-13}) as before. Meanwhile, as  $\alpha^2 \le \hbar$ we need to integrate (\ref{16-4-12}) over $\{s:|s|\le 1\}$ and we arrive to
$\hbar ^{1-\frac{1}{2}\kappa}$ for all $\kappa:0<\kappa <2$.

Multiplication by $\mu^\kappa h^{-2}|k|^{-2}\alpha^2$ and summation with respect to $k$ results in the right hand expression of (\ref{16-4-15}) as
$\alpha^2 \ge \hbar$ and
\begin{equation}
C\mu^{\kappa} h^{-2}\hbar  ^{1-\frac{1}{2}\kappa}\alpha^2=
C\mu^{\frac{1}{2}\kappa +1} h^{-1-\frac{1}{2}\kappa}\alpha^2
\label{16-4-28}
\end{equation}
as $\alpha^2 \le \hbar$.

\paragraph{Contribution of $0\ne k'\ne k''\ne 0$.}
\label{sect-16-4-1-4-2}

Consider now  case $k'-k''=r\ne 0$. As $\alpha^2 \ge \hbar$ expression (\ref{16-4-13}) acquires factor $(|k'-k''|\alpha^2 /\hbar)^{-l}$ and multiplying by $\mu^\kappa h^{-2}\alpha^2 |k'|^{-1}|k''|^{-1}$ we get after summation with respect to $k'\ne 0$, $k''\ne 0$ the same right-hand expression of (\ref{16-4-15})\,\footnote{\label{foot-16-12} Actually, with an extra factor $(\hbar/\alpha^2)^l$ but we do not need it here.}.

Meanwhile, as  $\alpha^2 \le \hbar$ instead of (\ref{16-4-13}) we get
\begin{equation*}
C\hbar ^{1-\frac{1}{2}\kappa}\bigl(1+|k'-k''|\alpha^2/\hbar\bigr)^{-l}.
\end{equation*}
Multiplication by
$\mu^\kappa h^{-2} \alpha^2|k'|^{-1}|k''|^{-1}$ and summation with with respect to $k'$, $k''$ returns
\begin{equation}
C\mu^{1+\frac{1}{2}\kappa} h^{-1-\frac{1}{2}\kappa}\alpha^2
\bigl(1+(\log (\hbar/\alpha^2))_+\bigr)^2.
\label{16-4-29}
\end{equation}

\paragraph{Contribution of $0=k'\ne k''$.}
\label{sect-16-4-1-4-3}

Consider now  $k'= 0, k''\ne 0$; its tween case $k'\ne 0, k''= 0$ is addressed in the same way. To do this we need to modify our expression fo $\I$: namely, we do not divide by $t'$ and then take Fourier transform
$F_{t'\to \hbar^{-1}\tau}$; instead we take Fourier transform
$F_{t'\to \hbar^{-1}\tau'}$, then integrate by $\tau'$ to $\tau$ and divide by $\hbar$; modifying this way  (\ref{16-4-1})  we arrive instead of  (\ref{16-4-4})  to
\begin{multline}
\I_\alpha =\\
\frac{1}{2}(2\pi)^{-3}\hbar^{-3}\mu^{\kappa+2} \alpha^2\int\iint\iint
\hat{\Omega} (4\hbar^{-1} \alpha^2 s, z)\times\\
\csc(t+s)\csc(t-s)(t+s)^{-1}(t-s)^{-1} \times
\\[3pt]
\exp \Bigl(i\hbar^{-1}\Bigl[-\frac{1}{4}\bigl(\cot(t+s) -\cot(t-s)\bigr) \bigl(z_1 ^2 + (z_2+2t\varepsilon)^2  + 4s^2\varepsilon^2\bigr)\\
- \bigl(\cot(t+s) +\cot(t-s)\bigr) s\varepsilon -2s(\tau-\varepsilon^2) +(t+s)(\tau-\tau') \Bigr]\Bigr)
\,dt ds\,\psi(\mu^{-1}z) dz_1dz_2\,d\tau'.
\label{16-4-30}
\end{multline}
Then repeating our above arguments we conclude that
\begin{claim}\label{16-4-31}
As $\alpha^2\ge \hbar$ the resulting term (instead of (\ref{16-4-13})) does not exceed $\textup{(\ref{16-4-13})}\times(\hbar/\alpha^2|k''|)^l \times \hbar^{-1}$
and as $\alpha^2\le \hbar$ the resulting term does not exceed
$C\hbar^{1-\frac{1}{2}\kappa}\times \bigl( 1+ \alpha^2|k''|\hbar^{-1}\bigr)^{-l} \times \hbar^{-1}$
\end{claim}
and we can rewrite the result in both cases as
\begin{equation*}
C\hbar^{-\frac{1}{2}\kappa}  \bigl( 1+ \alpha^2|k''|\hbar^{-1}\bigr)^{-l}.
\end{equation*}

Multiplying by $\mu^\kappa h^{-2}|k''|^{-1}\alpha^2 $ we get after summation with respect to $k''\ne 0$
\begin{align}
& C\mu^{\frac{1}{2}\kappa} h^{-2-\frac{1}{2}\kappa} \alpha^2 (\hbar/\alpha^2)^l \qquad&&\text{as\ \ } \alpha^2\ge h,
\label{16-4-32}\\[2pt]
& C\mu^{\frac{1}{2}\kappa} h^{-2-\frac{1}{2}\kappa} \alpha^2 \bigl(1+(\log (\hbar/\alpha^2))_+\bigr) \qquad&&\text{as\ \ } \alpha^2\le \hbar.
\label{16-4-33}
\end{align}

So we have proven

\begin{proposition}\label{prop-16-4-10}
For the pilot-model operator with  $1\le \mu \le h^{-1}$, $\alpha\ge \mu^{-1}$
\begin{align}
&|\I_{(\alpha)}^\T - \cI_{(\alpha)}^\W |\le \textup{(\ref{16-4-15})} + \textup{(\ref{16-4-32})}
&&\text{as\ \ } \alpha^2\ge \hbar,\label{16-4-34}\\[2pt]
&|\I_{(\alpha)}^\T - \cI_{(\alpha)}^\W |\le \textup{(\ref{16-4-28})} + \textup{(\ref{16-4-29})} + \textup{(\ref{16-4-33})}
&&\text{as\ \ } \alpha^2\le \hbar.\label{16-4-35}
\end{align}
\end{proposition}

\section{General operators}
\label{sect-16-4-2}

Consider now general operators satisfying (\ref{16-0-5})--(\ref{16-0-6}) \underline{and} either (\ref{16-0-7}) or (\ref{16-0-8}) or \ref{16-0-8-+}, $\tau=0$.

First, using proposition~\ref{prop-16-3-1}, remark~\ref{rem-16-3-4}(ii),(ii) and propagation results we conclude that

\begin{proposition}\label{prop-16-4-11}
Let conditions \textup{(\ref{16-0-5})}--\textup{(\ref{16-0-6})} be fulfilled.
Then

\medskip\noindent
(i) Contribution of zone $\{|x-y|\ge C_0\mu^{-1}\}$ to $\I^\T$ does not exceed $C\mu^{\kappa+1} h^{-1}$ and

\medskip\noindent
(ii) Under condition \textup{(\ref{16-0-7})} contribution of time $\{|t|\ge C_0\}$ (before rescaling) to $\I^\T$ does not exceed $C\mu^{\kappa+1} h^{-1}$.
\end{proposition}

Note that $O(\mu^{\kappa+1}h^{-1})$ is smaller than the remainder estimate $R^\W(1)$ or $R^\W_1(1)$ and as long as we are interested in Weyl approximation $\cI^\W$ (even with correction term) we should be completely happy with it.

Under condition (\ref{16-0-7}) we are done, immediately arriving to statement (i) and related part of (iii) of theorem~\ref{thm-16-4-13} below.

Under condition (\ref{16-0-8}) let us introduce
\begin{equation}
\alpha (x)= \epsilon_0 |\nabla VF^{-1}|+ \frac{1}{2}\bar{\alpha},\qquad
\bar{\alpha}=C_0\max\bigl( \mu^{-1}, (\mu h)^{\frac{1}{2}-\delta}\bigr)
\label{16-4-36}
\end{equation}
and consider covering of $B(0,1)$ by $\alpha$-admissible elements. In virtue of proposition~\ref{prop-16-4-11}(i) we need to consider only pairs $(x,y)$ belonging to the same element. Due to rescaling, contribution of one such element to $|I^\T-\cI^\W|$ does not exceed $C\mu h^{-1-\kappa}\alpha^2$ and therefore

\begin{claim}\label{16-4-37}
Under condition (\ref{16-0-8}) contribution of zone
\begin{equation*}
\{(x,y): |x-y|\le C_0\mu^{-1},\  |\nabla VF^{-1}| \le \alpha \}
\end{equation*}
to  $|I^\T-\cI^\W|$ does not exceed $C\mu  h^{-1-\kappa}\alpha^2$.
\end{claim}

Therefore we need to consider only zone where $|\nabla VF^{-1}(x)| \ge \alpha $ and $|x-y|\le C_0\mu^{-1}$. We are going to prove that

\begin{proposition}\label{prop-16-4-12}
Let conditions \textup{(\ref{16-0-5})}--\textup{(\ref{16-0-6})} be fulfilled.
Further, let $\mu \le h^{-1}$. Then

\medskip\noindent
(i) For each $z$ with $\alpha (z) \ge \bar{\alpha}$ contribution of zone
\begin{equation*}
\{(x,y)\in B(z,\alpha (z)),\ |x-y|\le C_0\mu^{-1}|\}
\end{equation*}
to $|\I^\T -\cI^\W|$ does not exceed $R^\W(\alpha)$  which is defined as the right-hand expression of \textup{(\ref{16-4-15})} or \textup{(\ref{16-4-25})}.

\medskip\noindent
(ii) Further, as $\alpha^2\ge \mu ^2h$, $1<\kappa<2$ this contribution does not exceed $R_1^\W(\alpha)$  which is defined as the right-hand expression of \textup{(\ref{16-4-25})}.
\end{proposition}

Note that both $R^\W(\alpha)$ and $R^\W_1(\alpha)$ contain $\alpha$ in the positive powers and thus after integration with respect $\alpha^{-1}\,d\alpha$ we get their values as $\alpha=1$.

Therefore we arrive to statement (ii) and related part of (iii) of theorem~\ref{thm-16-4-13}:

\begin{theorem}\label{thm-16-4-13}
Let conditions \textup{(\ref{16-0-5})}--\textup{(\ref{16-0-6})} be fulfilled.
Then as $\mu \le h^{-1}$

\medskip\noindent
(i) Under condition \textup{(\ref{16-0-7})}
\begin{gather}
|\I -\cI^\W|\le C\mu^{-1}h^{-1-\kappa}+ R^\W
\label{16-4-38}\\
\shortintertext{with}
R ^\W \Def
C\left\{\begin{aligned}
& \mu^{\kappa+1} h^{-1}  \qquad && \text{as\ \ } \kappa <1,\\
&\mu^2 h^{-\kappa}   \qquad && \text{as\ \ }\kappa>1,\\
&\mu^2 h^{-1}  |\log (\mu h)|  \qquad && \text{as\ \ } \kappa=1.
\end{aligned}\right.
\label{16-4-39}
\end{gather}

\medskip\noindent
(ii) Under condition \textup{(\ref{16-0-8})}~\footnote{\label{foot-16-13} If \ref{16-0-8-+} fails  the first term should acquire factor $(1+\mu h|\log h|)$ but it is important only as $\mu$ is close to $h^{-1}$ but then the other terms dominate.}
\begin{equation}
|\I -\cI^\W|\le C\mu^{-1}h^{-1-\kappa}+ R^\W +
C\mu^2 h^{-\kappa}(\mu h)^{-\delta};
\label{16-4-40}
\end{equation}
(iii) As $\mu^2h\le 1$ and $1<\kappa<2$ in these estimates one can replace $\cI^\W$ by $\cI^\W+\cI^\W_\corr$ and $R^\W$ by
\begin{equation}
R^\W_1\Def \left\{\begin{aligned}
&\mu^{2\kappa} h^{-1} \qquad
&&1<\kappa <\frac{3}{2},\\
&\mu^{6-2\kappa}h ^{2-2\kappa}  \qquad
&&\frac{3}{2}<\kappa <2,\\
& \mu^{3} h^{-1}
(1+|\log (\mu^2 h )|)\qquad &&\kappa=\frac{3}{2}.
\end{aligned}\right.
\label{16-4-41}
\end{equation}
\end{theorem}

\begin{remark}\label{rem-16-4-14}
(i) Actually as $\mu^2 h \le 1$, $\kappa\ge \frac{3}{2}$ this correction term does not exceed $C\mu^{-1}h^{-1-\kappa}$.

\medskip\noindent
(ii) Consider other terms in the estimates. Then  (\ref{16-4-28}), (\ref{16-4-29}), (\ref{16-4-33})  contain $\alpha$ in the positive powers and (\ref{16-4-32}) contains $\alpha$ in the negative power and thus after integration with respect $\alpha^{-1}\,d\alpha$ we get their values as $\alpha=(\mu h)^{\frac{1}{2}}$ i.e. (\ref{16-4-28}), (\ref{16-4-29}) become $C\mu^{\frac{1}{2}\kappa +2} h^{-\frac{1}{2}\kappa}$ which is less than (\ref{16-4-39}) and (\ref{16-4-41}) while (\ref{16-4-32}), (\ref{16-4-33}) become $C\mu^{\frac{1}{2}\kappa +1} h^{-1-\frac{1}{2}\kappa} $ which is larger than (\ref{16-4-39}). Therefore selecting
$\alpha \ge (\mu h)^{\frac{1}{2}-\delta}$ saves us from all these problems for the price of $\delta\ne 0$.

Even as $h^{-1+\delta}\le \mu \le h^{-1}$ it works as
$C\mu^{\frac{1}{2}\kappa +1} h^{-1-\frac{1}{2}\kappa}$ acquires factor
$(\mu h)^l$ which makes it subordinate to (\ref{16-4-39}).
\end{remark}

\begin{Problem}\label{problem-16-4-15}
Prove above results with $\delta=0$.
\end{Problem}

\begin{proof}[Proof of proposition~\ref{prop-16-4-12}]

First of all applying proposition~\ref{prop-16-2-5} and the same arguments as in the analysis of section~\ref{sect-16-2} we arrive to

\begin{claim}\label{16-4-42}
Contribution of zone
$\{|\sin (\theta(t'))|\ge C_0\hbar , \ |\sin (\theta(t''))|\ge C_0\hbar \}$ to the error $|\I^\T-\cI^\W|$ does not exceed expression (\ref{16-4-15}) or (\ref{16-4-41}) (in its frames).
\end{claim}

Therefore we need to explore contributions of three remaining zones
\begin{gather}
\{|\sin (\theta(t'))|\le 2C_0\hbar, \ |\sin (\theta(t''))|\le 2C_0\hbar \}
\label{16-4-43}
\shortintertext{and}
\{|\sin (\theta(t'))|\ge 2C_0\hbar, \ |\sin (\theta(t''))|\le C_0\hbar \}
\label{16-4-44}
\end{gather}
(the same analysis would cover its tween  zone
$\{|\sin (\theta(t'))|\ge 2C_0\hbar,\ |\sin (\theta(t''))|\le C_0\hbar \}$).

Note first that the standard propagation arguments imply that we get a factor
$C(|z_1|+|z_2|)^{-l}(|\sin (\theta(t'))|+ \varepsilon |t'|)^l$ and the similar factor with $t'$ replaced by $t''$.

\medskip
\emph{Zone \textup{(\ref{16-4-43})} as $\varepsilon \ge C_1\hbar$.\/}  Then in the classical propagation shift $\varepsilon |k'|$ is always greater than $\hbar$.

Obviously contribution of  zone  (\ref{16-4-43}), intersected with $(k',k'')$ windings $\times \{(x,y):|x-y|\ge \rho\}$ to the Tauberian expression does not exceed
\begin{equation}
C\mu^2 h^2 |k'|^{-1}|k''|^{-1} \rho^{-\kappa}
\sup_t \iint |\tilde{U}(x,y,t)|^2 \, dxdy \le C\mu^2 h^{2-d} \rho^{-\kappa} \alpha^d
\label{16-4-45}
\end{equation}
where factor $\hbar^2 $ is due to integration with respect to $t',t''$ and $\tilde{U}$ means cut-off of $U$  to the energy level $|\tau|\le c$.

Plugging (as $d=2$)
\begin{equation}
\rho = \varepsilon \mu^{-1}\max(|k'|,|k''|),
\label{16-4-46}
\end{equation}
$\varepsilon=\mu^{-1}\alpha$ we get
\begin{equation}
C\mu^{2\kappa+2}  |k'|^{-1} |k''|^{-1} \alpha^{2 -\kappa}
\bigl( \max(|k'|,|k''|)\bigr)^{-\kappa}.
\label{16-4-47}
\end{equation}
Summation by $k',k''$ returns
$C\mu^{2\kappa +2}  \alpha ^{2-\kappa}$ which is  less than the right-hand expression of \textup{(\ref{16-4-15})} or (as $\alpha^2\ge \mu^2h$, $1<\kappa<2$)  the right-hand expression of \textup{(\ref{16-4-25})}.

Meanwhile, estimating $\tilde{U}(x,y,t)$ by $h^{-d}$ and integrating by
$|x-y|\le \rho $ and then integrating by $t',t''$ over
$|\theta(t')-\pi k'|\le \hbar$, $|\theta(t'')-\pi k''|\le \hbar$
we get
\begin{equation}
Ch^{2-2d}  \mu^2  \rho ^{d-\kappa} |k'|^{-1}|k''|^{-1} \bigl(h/ \rho\bigr)^l\alpha^d
\label{16-4-48}
\end{equation}
where the last factor is due to propagation results; plugging $d=2$, (\ref{16-4-46}) we get
\begin{equation}
C \mu^2 h ^{-\kappa} |k'|^{-1}|k''|^{-1}\alpha^2
\bigl(\hbar/ \varepsilon \max(|k'|,|k''|)\bigr)^l ;
\label{16-4-49}
\end{equation}
after summation we get
$C \mu^{2} h ^{-\kappa} \bigl(\hbar/ \varepsilon \bigr)^l\alpha^2$
which is again less than the right-hand expression of \textup{(\ref{16-4-15})} or (as $\alpha^2\ge \mu^2h$, $1<\kappa<2$)  the right-hand expression of \textup{(\ref{16-4-25})}.

\medskip
\emph{Zone \textup{(\ref{16-4-43})} as $\varepsilon \le C_1\hbar$.\/}
Taking sum of (\ref{16-4-47})  over
$k',k'': \max(|k'|,|k''|) \ge \hbar/\varepsilon$ we get
\begin{equation*}
C\mu^{2\kappa+2}\alpha^{2-\kappa} (1+ |\log \mu^2h/\alpha|) \times (\mu^2h/\alpha)^{-\kappa}=
C\mu^2h^{-\kappa}\alpha^2(1+ |\log \mu^2h/\alpha|)
\end{equation*}
while taking sum of (\ref{16-4-49}) we get the right-hand expression as well.
Obviously, it is less than the right-hand expression of \textup{(\ref{16-4-15})}.

Thus we need to consider only $k',k'': \max(|k'|,|k''|) \le \hbar/\varepsilon$.

Using the same arguments as before, we conclude that the contribution of  zone  (\ref{16-4-43}), intersected with $(k',k'')$ windings $\times \{(x,y):|x-y|\ge \hbar \}$ to the Tauberian expression does not exceed $C\mu^2h^{-\kappa}|k'|^{-1}|k''|^{-1} $; similarly contribution of  zone  (\ref{16-4-43}), intersected with $(k',k'')$ windings
$\times \{(x,y):|x-y|\le \hbar \}$ to the Tauberian expression does not exceed
$C \mu^2 h ^{-\kappa} |k'|^{-1}|k''|^{-1}\alpha^2$ as well.

After summation we get $C\mu^2 h^{-\kappa} (1+ |\log \mu^2h/\alpha|)^2$ which again is less than the right-hand expression of \textup{(\ref{16-4-15})}.

\medskip
\emph{Zone \textup{(\ref{16-4-44})} as $\varepsilon \ge C_1\hbar$.\/} Now shift with respect to $(x-y)$ is
$\mu^{-1} \max\bigl(\varepsilon |t|, |\varepsilon t'\mp \sin (t') |\bigr)$; then repeating  analysis in zone (\ref{16-4-43}) we arrive to (\ref{16-4-47}) for $1<\kappa<2$ and to (\ref{16-4-49}) for $0<\kappa<2$; as $0<\kappa\le 1$ one should replace  (\ref{16-4-47})  by
\begin{equation}
C\mu^{2\kappa+2} \hbar |k'|^{-1} |k''|^{-1}
\label{16-4-50}
\end{equation}
with a logarithmic factor as $\kappa=1$ and summation with respect to $k',k''$ results in $C\mu^{3\kappa+2}h(1+|\log h|)^2 $ which again is less than $R^\W(\alpha)$.

\medskip
\emph{Zone \textup{(\ref{16-4-44})} as $\varepsilon \le C_1\hbar$.\/} With the above adjustment arguments of zone \textup{(\ref{16-4-43})} work here as well.
\end{proof}

\section{Perturbations}
\label{sect-16-4-3}

\subsection{General scheme}
\label{sect-16-4-3-1}

\paragraph{Contribution of $k'=k''\ne 0$.}
\label{sect-16-4-3-1-1}
Consider first case $k'=k''=k$. Assume first that $\alpha^2 \ge \hbar$.

Recall that $k$-th tick' contribution to $\I^\T$ was not exceeding expression (\ref{16-4-13}) multiplied by $\mu^\kappa h^{-2}\alpha^2|k|^{-2}$ i.e.
\begin{equation}
C\mu^{\kappa+2} \alpha^2|k|^{-2}\hbar ^{-1-\frac{1}{2}\kappa}
\left\{\begin{aligned}
&(\hbar/\alpha^2)^{\frac{1}{2}\kappa}\qquad &&\kappa<1,\\
&(\hbar/\alpha^2)^{1-\frac{1}{2}\kappa}\qquad &&\kappa>1,\\
&(\hbar/\alpha^2)^{\frac{1}{2}}(1+|\log (\hbar/\alpha^2)|)\qquad &&\kappa=1
\end{aligned}\right.
\label{16-4-51}
\end{equation}
(as $|\sin(t')|\le \hbar$ or $|\sin(t'')|\le \hbar$ one should use proposition~\ref{prop-16-4-12}).
Then we can suspect that the successive approximations bringing factor
$\sigma T/\hbar$ as $U$ is replaced by $U^0$  simultaneously brings the same factor when we replace $U$ by $\bar{U}$ in $\I$; so contribution of $k$-th tick to the error is estimated by the previous expression multiplied by
$C\sigma |k| /\hbar$ i.e.
\begin{equation}
C\sigma \mu^{\kappa+2} \alpha^2|k|^{-1}
\hbar ^{-2-\frac{1}{2}\kappa}\left\{\begin{aligned}
&(\hbar/\alpha^2)^{\frac{1}{2}\kappa}\qquad &&\kappa<1,\\
&(\hbar/\alpha^2)^{1-\frac{1}{2}\kappa}\qquad &&\kappa>1,\\
&(\hbar/\alpha^2)^{\frac{1}{2}}(1+|\log (\hbar/\alpha^2)|)\qquad &&\kappa=1.
\end{aligned}\right.
\label{16-4-52}
\end{equation}
and summation over $k: |k|\le \bar{k}\Def \hbar/\sigma$ results in
\begin{multline}
C\sigma \mu^{\kappa+2} \alpha^2|\log \hbar/\sigma|\hbar ^{-2-\frac{1}{2}\kappa}
\left\{\begin{aligned}
&(\hbar/\alpha^2)^{\frac{1}{2}\kappa}\qquad &&\kappa<1,\\
&(\hbar/\alpha^2)^{1-\frac{1}{2}\kappa}\qquad &&\kappa>1,\\
&(\hbar/\alpha^2)^{\frac{1}{2}}(1+|\log (\hbar/\alpha^2)|)\quad &&\kappa=1.
\end{aligned}\right.
\label{16-4-53}
\end{multline}

\begin{remark}\label{rem-16-4-16}
(i) Here $\sigma$ is the magnitude of perturbation and as we approximate by pilot-model it is $O(\mu^{-2})$ but if we are trying to use the magnetic Weyl approximation it is $O(\varepsilon)$.

\medskip\noindent
(ii) Expression (\ref{16-4-53}) provides us with the answer as
$\alpha^2\ge \hbar$; as $\alpha^2\le \hbar$ we need to recall that if we are interested only in zone $\{|x-y|\le C_0\mu^{-1}\}$ before rescaling then $|z|\le C_0$ and we need therefore there reset $\hbar/\alpha^2$ to $1$ and also $|s|$ to $1$ resulting in
\begin{equation}
C\sigma \mu^{\kappa+2} \alpha^2|\log \hbar/\sigma|\hbar ^{-2-\frac{1}{2}\kappa}
\label{16-4-54}
\end{equation}
(iii) In these settings we need to assume that
$T\le C_0\varepsilon^{-1}$ i.e. $\bar{k}\le C_0\varepsilon^{-1}$ i.e.
\begin{equation}
\mu h\le \sigma /\varepsilon.
\label{16-4-55}
\end{equation}
This inequality always holds as $\sigma =\varepsilon$ but it may fail as
$\sigma =\mu^{-2}$; in this case we either reset $\bar{k}$ to $C\varepsilon^{-1}$ or reevaluate $\sigma$ and analyze separately in both cases (see in the next subsubsection).

\medskip\noindent
(iv) Note, as $\mu h\le \sigma /\varepsilon$ summation (\ref{16-4-51}) for $k:|k|\ge \bar{k}$ brings the same expression albeit without logarithmic factor; otherwise if we reset $\bar{k}= C\varepsilon^{-1}$ this sum returns
$C\mu^\kappa h^{-1}$.

\medskip\noindent
(iv) However for the pilot-model  approximation we reexamine zone
$\{|x-y|\ge C_0\mu^{-1}\}$ and improve the remainder estimate of (iv)
\end{remark}

\paragraph{Contribution of $0\ne k'\ne k''\ne 0$.}
\label{sect-16-4-3-1-2}

Next we need to consider  $0\ne k'\ne k''\ne 0$. Recall that as $\alpha^2\ge \hbar$ the contribution of such pair to $\I^\T$ does not exceed
\begin{multline}
C\mu^{\kappa+2} \alpha^2|k'|^{-1} |k''|^{-1} (\hbar/\alpha^2|k'-k''|)^l
\hbar ^{-1-\frac{1}{2}\kappa}\times\\[3pt]
\left\{\begin{aligned}
&(\hbar/\alpha^2)^{\frac{1}{2}\kappa}\qquad &&\kappa<1,\\
&(\hbar/\alpha^2)^{1-\frac{1}{2}\kappa}\qquad &&\kappa>1,\\
&(\hbar/\alpha^2)^{\frac{1}{2}}(1+|\log (\hbar/\alpha^2)|)\qquad &&\kappa=1;
\end{aligned}\right.
\label{16-4-56}
\end{multline}
multiplying by $C\sigma |k'|/\hbar$ (as perturbation goes to the first ``factor'') we get
\begin{multline}
C\sigma \mu^{\kappa+2} \alpha^2 |k''|^{-1} (\hbar/\alpha^2|k'-k''|)^l
\hbar ^{-2-\frac{1}{2}\kappa}\times\\[2pt]
\left\{\begin{aligned}
&(\hbar/\alpha^2)^{\frac{1}{2}\kappa}\qquad &&\kappa<1,\\
&(\hbar/\alpha^2)^{1-\frac{1}{2}\kappa}\qquad &&\kappa>1,\\
&(\hbar/\alpha^2)^{\frac{1}{2}}(1+|\log (\hbar/\alpha^2)|)\qquad &&\kappa=1.
\end{aligned}\right.
\label{16-4-57}
\end{multline}
Then summation with respect to $k'$ of $(\hbar/\alpha^2|k'-k''|)^l$ returns
$(\hbar/\alpha^2)^l$ as $\hbar\le \alpha^2$ and summation with respect to $k''$ then returns (\ref{16-4-53}) with an extra factor $(\hbar/\alpha^2)^l$ and therefore is less than (\ref{16-4-53}).

On the other hand, as $\hbar\ge \alpha^2$ summation with respect to $k'$ of $1$ returns $\hbar/\alpha^2$ and as we recall that big left brace expression should be replaced by $1$ we arrive after summation  with respect to $k''$ to
\begin{equation}
C\sigma  \mu^{\kappa+2}   \hbar ^{-1-\frac{1}{2}\kappa}(1+(\log \hbar/\sigma))
\label{16-4-58}
\end{equation}
which is (\ref{16-4-54})  with an extra factor $(\hbar/\alpha^2)$ and therefore is greater than (\ref{16-4-53}).

\paragraph{Contribution of $0=k''\ne k'$.}
\label{sect-16-4-3-1-3}

If $k'\ne 0$, $k''= 0$ we recall that $|k''|^{-1}$ should be replaced by $\hbar^{-1}$ in (\ref{16-4-56})--(\ref{16-4-57}) and perturbation brings factor $\sigma |k'|/\hbar$ so we get instead of (\ref{16-4-57})
\begin{multline*}
C\sigma \mu^{\kappa+2} \alpha^2  (\hbar/\alpha^2|k'|)^l
\hbar ^{-3-\frac{1}{2}\kappa}
\left\{\begin{aligned}
&(\hbar/\alpha^2)^{\frac{1}{2}\kappa}\qquad &&\kappa<1,\\
&(\hbar/\alpha^2)^{1-\frac{1}{2}\kappa}\qquad &&\kappa>1,\\
&(\hbar/\alpha^2)^{\frac{1}{2}}(1+|\log (\hbar/\alpha^2)|)\qquad &&\kappa=1.
\end{aligned}\right.
\end{multline*}
Then summation with respect to $k'$ returns
\begin{equation}
C\sigma \mu^{\kappa+2} \alpha^2  (\hbar/\alpha^2)^l
\hbar ^{-3-\frac{1}{2}\kappa}
\label{16-4-59}
\end{equation}
as $\hbar\le \alpha^2$ and
\begin{equation}
C\sigma  \mu^{\kappa+2}  \hbar ^{-2-\frac{1}{2}\kappa}
\label{16-4-60}
\end{equation}
as $\hbar\ge \alpha^2$; (\ref{16-4-60}) is even greater than (\ref{16-4-58}).

\paragraph{Contribution of $0=k'$.}
\label{sect-16-4-3-1-4}
With perturbation factors are no more of the same rights and we need to consider this case as well.

\medskip\noindent
(i) If $k'=0$, $k''\ne 0$ we recall that $|k'|^{-1}$ should be replaced by $\hbar^{-1}$ in (\ref{16-4-56})--(\ref{16-4-57}) but perturbation brings factor $\sigma$ so together we get factor $\sigma\hbar^{-1}$ which falls into into case $k'=1$, $k''\ne 0$.

\medskip\noindent
(ii) Similarly, if $k'=k''=0$ we fall into case $k'=1$, $k''= 0$.

In total we get (\ref{16-4-53})+(\ref{16-4-59}) as $\alpha^2 \ge \hbar$ and (\ref{16-4-60}) as $\alpha^2\le \hbar$. So, we arrive to

\begin{proposition}\label{prop-16-4-17}
Let conditions \textup{(\ref{16-0-5})}--\textup{(\ref{16-0-6})} be fulfilled in $B(0,1)$ and let  $\alpha \Def \epsilon_0 |\nabla V/F (\bar{x})|\ge C_0\mu^{-1}$
where $\bar{x}\in B(0,\frac{1}{2})$.

Let $\bar{e}_z(x,y,\tau)$ be a Schwartz kernel of the spectral projector of \underline{either}  the magnetic Weyl approximation (then
$\sigma= \alpha \mu^{-1}$) \underline{or} the pilot-model approximation (then $\sigma=  \mu^{-2}$) at point $z$. Let   $T\ge
\min \bigl(\epsilon _0 \hbar/\sigma,C_0 \mu \alpha^{-1}\bigr)\ge 1$.

Let $\psi,\psi_1\in \sC^\infty_0 (B(0,1))$.

\medskip\noindent
(i) Then as $\alpha^2\ge \hbar$, $\gamma=C_0\mu^{-1}$
\begin{multline}
\R_2= \R_2(\alpha,\gamma)\Def\\
|\int \Omega_{\alpha,\gamma} \bigl(\frac{1}{2}(x+y),x-y\bigr)
\Bigl( |e^\T (x,y,0)|^2- |\bar{e}^\T _{\frac{1}{2}(x+y)}(x,y,0)|^2\Bigr)\,dxdy| \le \\
\textup{(\ref{16-4-53})}+\textup{(\ref{16-4-59})}
\label{16-4-61}
\end{multline}
where
\begin{equation*}
\Omega_{\alpha,\gamma}(x,z) \Def
\psi \bigl(\alpha^{-1}(x-\bar{x})\bigr) \psi_1\bigl(\gamma^{-1}z\bigr);
\end{equation*}
(ii)  As $\alpha^2\le \hbar$, $\gamma=C_0\mu^{-1}$ we get
\begin{equation}
\R_2\le \textup{(\ref{16-4-54})}+\textup{(\ref{16-4-60})}.
\label{16-4-62}
\end{equation}
(ii) Further, in this inequality one can replace $\bar{e}^\T$ by $\bar{e}$.
\end{proposition}

\begin{proof}
The easy albeit tedious proof following our standard technique we leave to the reader.
\end{proof}

\subsection{Improvement}
\label{sect-16-4-3-2}
We want to apply arguments of the previous subsubsection~``\nameref{sect-16-4-1-3}'' to improve  results of the previous subsubsection. One can see easily that  improvement makes sense only if we can improve $k'=k''=1$ and thus as $\alpha^2 \ge \mu^2 h$. This excludes magnetic Weyl approximation.

In this case $R^\W_1(\alpha)$ defined by (\ref{16-4-25}) acquires factor $\mu^{-3}h^{-1}$:
\begin{equation}
C\left\{\begin{aligned}
&\mu^{2\kappa-3} h^{-2} \alpha^{2-\kappa} \qquad
&&1<\kappa <\frac{3}{2},\\
&\mu^{3-2\kappa}h ^{1-2\kappa} \alpha^{3\kappa-4} \qquad
&&\frac{3}{2}<\kappa <2,\\
&  h^{-2} \alpha^{\frac{1}{2}}
(1+|\log (\mu^2 h / \alpha^2)|)\qquad &&\kappa=\frac{3}{2}.
\end{aligned}\right.
\label{16-4-63}
\end{equation}
This replaces (\ref{16-4-53}) and (\ref{16-4-59}) becomes negligible.

So, we arrive to

\begin{proposition}\label{prop-16-4-18}
Let conditions of proposition~\ref{prop-16-4-17} be fulfilled, we consider the pilot-model approximation and let $\alpha^2\ge \mu^2h$. Then

\medskip\noindent
(i) $\R_2(\alpha,\gamma)$ does not exceed \textup{(\ref{16-4-63})}.

\medskip\noindent
(ii) Further, in this inequality one can replace $\bar{e}^\T$ by $\bar{e}$.
\end{proposition}

\begin{proof}
The easy albeit tedious proof following our standard technique we leave to the reader.
\end{proof}

\subsection{Results. I}
\label{sect-16-4-3-3}

Let condition (\ref{16-0-7}) be fulfilled.

\paragraph{Magnetic Weyl approximation.}
\label{sect-16-4-3-3-1}
If we are looking for the magnetic Weyl approximation, we should assume that $h^{-\frac{1}{2}}\le \mu \le h^{-1}$ and pick up  $\sigma=\varepsilon=\mu^{-1}$, $\alpha=1$ and look what happens with expressions (\ref{16-4-53}) and (\ref{16-4-59}). The former returns the last term in the estimate (\ref{16-4-64}) below and the latter returns  the lesser expression
$C \mu^{\kappa-1} h^{-2}  \hbar ^l$. So we arrive to estimate (\ref{16-4-64}) of theorem \ref{thm-16-4-19} below albeit with an extra factor
$\bigl(1+(\log \mu^2h)_+\bigr)$ in front of the big left brace.

To get rid off it let us consider two-term successive approximation. Then to estimate a remainder one needs to sum the same expressions as before but with  an extra factor which is either $\mu^{-2}h^{-1}|k'|$ or $\mu^{-2}h^{-1}|k''|$.

This really does not matter for $k'=k''\ne 0$ and we sum $\mu^{-2}h^{-1}$ instead of $|k|^{-1}$ from $|k|=1$ to $|k|=\mu^2h$ resulting in $1$ rather than logarithm. As $0\ne k'\ne k''\ne 0$ we get the same answer albeit with a factor $(\mu h)^l$ as $\alpha\asymp 1$. If $k'=0$ and/or $k''=0$ we replace $|k'|^{-1}$ and/or $|k''|^{-1}$ by $\hbar^{-1}$ but this is more than compensated by $(\mu h)^l$ but effectively $T'=\hbar$ and/or $T''=\hbar$ so in fact we reset these cases to $k'=1$ and/or $k''=1$.

On the other hand, the second term in the successive approximation refers to perturbation being strictly $\alpha (x_1-y_1)$ but this leads to $0$ as $\Omega$ is  even with respect to $(x-y)$. Thus we arrive to estimate (\ref{16-4-64}) of theorem \ref{thm-16-4-19} below.

\begin{theorem}\label{thm-16-4-19}
Let conditions \textup{(\ref{16-0-5})}--\textup{(\ref{16-0-7})} be fulfilled. Then as $h^{-\frac{1}{2}}\le \mu \le h^{-1}$
\begin{multline}
|\I-\cI^\MW| \le
C\mu^{-1}h^{-1-\kappa}+
C\left\{\begin{aligned}
& \mu^{\kappa-1} h^{-2} \; &&0<\kappa<1,\\
& h^{-1-\kappa} \; &&1<\kappa<2,\\
&h^{-2} (1+|\log (\mu h)|)\; &&\kappa=1.
\end{aligned}\right.
\label{16-4-64}
\end{multline}
\end{theorem}

\paragraph{Pilot-model approximation.}
\label{sect-16-4-3-3-2}
If we are looking for the pilot-model approximation, we should assume that $h^{-\frac{1}{3}}\le \mu \le h^{-1}$ and pick up  $\sigma=\mu^{-2}$, $\varepsilon=\mu^{-1}$, $\alpha=1$. We need to consider variants:

\medskip\noindent
(i)  $h^{-\frac{1}{3}}\le \mu \le h^{-\frac{1}{2}}$.  As $0<\kappa \le 1$ (\ref{16-4-53}) becomes $C\mu ^{\kappa-2} h^{-2}(1+|\log \mu^3 h|)$ (with another logarithmic factor as $\kappa=1$ and it is larger than
$C\mu^\kappa h^{-1}$ (which is the contribution of $\{|x-y|\ge C_0\mu^{-1}\}$).

As $1<\kappa < 2$  we can replace (\ref{16-4-53}) by (\ref{16-4-63}) which is then $O(\mu^{-1}h^{-1-\kappa})$ and $\mu^\kappa h^{-1}$ is also $O(\mu^{-1}h^{-1-\kappa})$.

Noting that (\ref{16-4-59}) is negligible we arrive to statement (i) of theorem \ref{thm-16-4-21} below.

\medskip\noindent
(ii)  $h^{-\frac{1}{2}}\le \mu \le h^{-1}$. In this case we note that (\ref{16-4-59}) is still less than (\ref{16-4-53}) but $\mu^\kappa h^{-1}$ may become leading term. To improve this term   let us consider contribution of zone
$|x-y|\asymp \gamma\ge C_0\mu^{-1}$ into an error: we need to take $\sigma=\gamma^2$ and $T\asymp \mu^2 \gamma$ so we get the estimate of perturbation
\begin{equation}
C\mu h^{-1}T^{-1}\gamma^{-\kappa} \times \sigma T/\hbar \asymp
Ch^{-2}  \gamma^{2-\kappa}
\label{16-4-65}
\end{equation}
recall that unperturbed term does not exceed
\begin{equation}
C\mu h^{-1}T^{-1}\gamma^{-\kappa} \asymp
C\mu^{-1} h^{-1} \gamma^{-\kappa-1}
\label{16-4-66}
\end{equation}
because the drift propagation does not expand $B(x,\rho)$ with
$\rho \ge C_0\mu^{-1}$ but merely shifts  it. We sum (\ref{16-4-65}) as
$\gamma^2 T\le \hbar$ i.e. $\gamma^3 \le \mu^{-1}h$ i.e.
$\gamma \le \bar{\gamma}=\mu^{-\frac{1}{3}}h^{\frac{1}{3}}=
\mu^{-1}(\mu ^2h)^{\frac{1}{3}}$.

Then we get
$C h^{-\frac{4}{3}-\frac{1}{3}\kappa }\mu^{\frac{1}{3}\kappa-\frac{2}{3}}$ and we arrive to statement (ii) of theorem \ref{thm-16-4-21} below.

\begin{remark}\label{rem-16-4-20}
Obviously arguments of the magnetic Weyl approximation allowing us to get rid off logarithmic factor still work albeit the second term does not necessarily vanish as the perturbation is quadratic. The logarithmic factor in estimates (\ref{16-4-67}) and (\ref{16-4-69}) does not seem be worth of trouble to write this second term.
\end{remark}

\begin{theorem}\label{thm-16-4-21}
Let conditions \textup{(\ref{16-0-5})}--\textup{(\ref{16-0-7})} be fulfilled. Then

\medskip\noindent
(i)  As $h^{-\frac{1}{3}}\le \mu \le h^{-\frac{1}{2}}$
\begin{multline}
|\I-\bar{\I}| \le C\mu^{-1}h^{-1-\kappa}+\\
C |\log \mu^3 h|
\left\{\begin{aligned}
&\mu^{\kappa-2} h^{-2}\qquad &&0<\kappa<1,\\
&0 \qquad &&1<\kappa<2,\\
&\mu^{-1} h^{-2}(1+|\log h|)\qquad &&\kappa=1
\end{aligned}\right.
\label{16-4-67}
\end{multline}
where
\begin{equation}
\bar{I}\Def\\
\int \Omega _{1,1} \bigl(\frac{1}{2}(x+y),x-y\bigr)
|\bar{e}  _{\frac{1}{2}(x+y)}(x,y,0)|^2 \,dxdy
\label{16-4-68}
\end{equation}
(ii)  As $h^{-\frac{1}{2}}\le \mu \le h^{-1}$
\begin{multline}
|\I-\bar{\I}| \le C\mu^{-1}h^{-1-\kappa}+  Ch^{-\frac{4}{3}-\frac{1}{3}\kappa}\mu^{\frac{1}{3}\kappa-\frac{2}{3}} +\\
C |\log h|
\left\{\begin{aligned}
&\mu^{\kappa-2} h^{-2}\qquad &&0<\kappa<1,\\
&\mu^{-1}  h ^{-1-\kappa}\qquad &&1<\kappa<2,\\
&\mu^{-1} h^{-2}(1+|\log (\mu h)|)\qquad &&\kappa=1.
\end{aligned}\right.
\label{16-4-69}
\end{multline}
\end{theorem}

\subsection{Results. II}
\label{sect-16-4-3-4}

Assume now that condition (\ref{16-0-8}) is fulfilled instead of (\ref{16-0-7}) and estimate an approximation error again.

\paragraph{Magnetic Weyl approximation}
\label{sect-16-4-3-4-1}

Let us consider the magnetic Weyl approximation first. Then we should assume that $h^{-\frac{1}{2}}\le \mu \le h^{-1}$ and pick up  $\sigma=\varepsilon=\alpha\mu^{-1}$. Recall that the contribution of
$\{(x,y): |x-y|\ge C_0\mu^{-1}\}$ does not exceed
$C\mu^\kappa h^{-1}$ and therefore we need to consider zone
\begin{equation}
\{(x,y): |x-y|\le C_0\mu^{-1}, \ \alpha(x)\ge C_1\mu^{-1}\}
\label{16-4-70}
\end{equation}
where also
$(1-\epsilon) \le \alpha (y)/\alpha(x)\le (1+\epsilon)$~\footnote{\label{foot-16-14} And add contribution of the zone $\{(x,y): \alpha (x)\le \bar{\alpha}\Def C_1 \mu^{-1}, \ \alpha (y)\le \bar{\alpha}\}$ which would be $C\mu h^{-1-\kappa}\times \bar{\alpha}^2= C\mu^{-1}h^{-1-\kappa}$ in virtue (\ref{16-0-8}).}.

Therefore we can use $\alpha$-admissible partition in zone (\ref{16-4-69}) and we need just look what happens with expressions (\ref{16-4-53}) and (\ref{16-4-59}) after integration with respect to $\alpha^{-1}d\alpha$ as $\alpha \ge (\mu h)^{\frac{1}{2}}$ and  (\ref{16-4-60})
$\mu^{-1}\le \alpha \le (\mu h)^{\frac{1}{2}}$.

Obviously integral of (\ref{16-4-53}) resets to its value as $\alpha=1$ as under assumption (\ref{16-0-7}); integrals of (\ref{16-4-59}) and (\ref{16-4-60}) reset to their values as $\alpha= (\mu h)^{\frac{1}{2}}$ which are equal and less than what we got already.

Therefore estimate (\ref{16-4-64}) still holds\footnote{\label{foot-16-15} Albeit the first term $C\mu^{-1}h^{-1-\kappa}$ acquires factor $|\log h|$ if assumption \ref{16-0-8-+} fails: see  Tauberian estimate.} and we arrive to statement (i) of theorem~\ref{thm-16-4-22} below after we apply the same arguments with two-term approximation as in the previous subsubsection to get rid off the logarithmic factor.

\paragraph{Pilot-model approximation}
\label{sect-16-4-3-4-2}
Let us consider the pilot-model approximation; then again $\sigma=\mu^{-2}$. Then we need to consider two cases.

\medskip\noindent
(a) $h^{-\frac{1}{3}}\le \mu\le h^{-\frac{1}{2}}$. In this case we are completely happy with the zone $\{(x,y): |x-y|\ge C_0\mu^{-1}\}$.

As $0<\kappa\le 1$ we need to integrate with respect to $\alpha^{-1}d\alpha$ (\ref{16-4-53}) and (\ref{16-4-59}) as $\alpha \ge (\mu h)^{\frac{1}{2}}$ and (\ref{16-4-60}) as $\alpha \le (\mu h)^{\frac{1}{2}}$. The first integral resets to (\ref{16-4-53}) as $\alpha=1$; integrals of (\ref{16-4-59}) and (\ref{16-4-60}) reset to their values as $\alpha= (\mu h)^{\frac{1}{2}}$ which are  equal (up to a logarithmic factor) and less than what we got already. Thus we arrive to statement (ii) of theorem~\ref{thm-16-4-22} below.

As $1<\kappa <2$ we need to replace integral of (\ref{16-4-53}) by its integral over zone $(\mu h)^{\frac{1}{2}}\le \alpha \le \mu h^{\frac{1}{2}}$ plus integral of (\ref{16-4-63}) over zone $\mu h^{\frac{1}{2}}\le \alpha \le 1$. These integrals are reset to their integrands values as
$\alpha=(\mu h)^{\frac{1}{2}}$ and $\alpha =1$ respectively; the latter is larger than the former albeit not necessarily larger than integral of (\ref{16-4-60}).

Thus we arrive to statement (iii) of theorem~\ref{thm-16-4-22} below.

\medskip\noindent
(b) $h^{-\frac{1}{2}}\le \mu\le h^{-1}$. In this case the contribution of $\{(x,y): |x-y| \le C_0\mu^{-1}\}$ does not exceed the right-hand expression of (\ref{16-4-69}) but we need to reexamine zone
$\{(x,y): |x-y| \ge C_0\mu^{-1}\}$ which we split into two:
$\{|x-y|\le \epsilon_0\alpha(x)\}$ and $\{|x-y|\ge \epsilon_0\alpha(x)\}$.

Repeating corresponding arguments of the analysis under assumption (\ref{16-0-7}) we conclude that the contribution of the former zone does not exceed what we got then. In the latter zone $T\asymp \mu^2$ and its contribution does not exceed $C\mu^{-1}h^{-1-\kappa}$.

Thus we arrive to statement (iv) of theorem~\ref{thm-16-4-22} below.

\begin{theorem}\label{thm-16-4-22}
Let conditions \textup{(\ref{16-0-5})}--\textup{(\ref{16-0-6})} and  \textup{(\ref{16-0-8})} be fulfilled. Then\footref{foot-16-15}

\medskip\noindent
(i) As $h^{-\frac{1}{2}}\le \mu \le h^{-1}$ estimate \textup{(\ref{16-4-64})} holds;

\medskip\noindent
(ii)  As $h^{-\frac{1}{3}}\le \mu \le h^{-\frac{1}{2}}$, $0<\kappa \le 1$ estimate \textup{(\ref{16-4-67})} holds;

\medskip\noindent
(iii)  As $h^{-\frac{1}{3}}\le \mu \le h^{-\frac{1}{2}}$, $1<\kappa <2$
\begin{equation}
|\I-\bar{I}| \le  C\mu ^{-1}h^{-1-\kappa}+ C\mu^{\frac{1}{2}\kappa-2}h^{-\frac{1}{2}\kappa-2}(1+|\log \mu^3h|);
\label{16-4-71}
\end{equation}

\medskip\noindent
(iv) As $h^{-\frac{1}{2}}\le \mu \le h^{-1}$ estimate \textup{(\ref{16-4-69})} holds.
\end{theorem}

\section{Superstrong magnetic field}
\label{sect-16-4-4}

Consider under assumption (\ref{16-0-8}) the Schr\"odinger-Pauli operator  with $\mu \ge h^{-1}$. Recall that according to propositions~\ref{prop-16-3-8}(ii), \ref{prop-16-3-9}(ii)  respectively
\begin{gather*}
|\I|\le C\mu^{1+\frac{1}{2}\kappa} h^{-1-\frac{1}{2}\kappa}\\
\shortintertext{and}
|\I-\I^\T|\le C\mu^{\frac{1}{2}\kappa}h^{-\frac{1}{2}\kappa}
\end{gather*}
(as $\kappa\ne1$; as $\kappa=1$ the latter estimate holds with an extra factor
$|\log \mu|$ in its right-hand expression).

Let us assume that $F=1$ to avoid some unpleasant correction terms as we reduce operator to one with $F=1$. Recall that now  $\hslash\Def \mu^{-\frac{1}{2}}h^{\frac{1}{2}}$ is the magnitude of the cyclotron radius.

\paragraph{Magnetic Weyl approximation.}
\label{sect-16-4-4-1}
Consider first the magnetic Weyl approximation. Then
$\sigma=\alpha\hslash = \alpha \mu^{-\frac{1}{2}}h^{\frac{1}{2}}$ and
$T=\sigma ^{-1}\hbar= \alpha ^{-1}\mu^{\frac{3}{2}}h^{\frac{1}{2}}$; then
the the shift (unrescaled) due to magnetic drift does not exceed
$C\alpha \mu^{-2}T= C\hslash$, which justifies the choice of $\sigma$.

Therefore as $\alpha \ge \bar{\alpha}=C_1 \mu^{-\frac{1}{2}}h^{\frac{1}{2}}$ the contribution of $\alpha$-ball intersected with
$\{(x,y):|x-y|\le \bar{\gamma}\Def C_0\mu^{-\frac{1}{2}}h^{\frac{1}{2}}\}$
to $|\I^\T- \cI^\MW|$ does not exceed
$C\mu h^{-1}\sigma \times \mu^{\frac{1}{2}\kappa}h^{-\frac{1}{2}\kappa} \asymp
C \alpha \mu^{\frac{1}{2}+\frac{1}{2}\kappa}h^{-\frac{1}{2}-\frac{1}{2}\kappa}$.
Then summation with respect to all such balls does not exceed
$C \mu^{\frac{1}{2}+\frac{1}{2}\kappa}h^{-\frac{1}{2}-\frac{1}{2}\kappa}$, as in  estimate (\ref{16-4-72}) below.

Meanwhile the contributions to both $\I$ and $\bar{\I}$ of zone $\{(x,y):|x-y|\ge \bar{\gamma}\}$ does not exceed
$C\bar{\gamma}^{-1-\kappa}$ which is also of the same magnitude.

Finally, the contributions to both $\I$ and $\bar{\I}$ of zone
$\{(x,y): \alpha (x) \le \bar{\alpha}, \ \alpha (y) \le \bar{\alpha}\}$ do not exceed
$C\mu^{1+\frac{1}{2}\kappa} h^{-1-\frac{1}{2}\kappa}\times \bar{\alpha}^2 \asymp
C\mu^{\frac{1}{2}\kappa} h^{-\frac{1}{2}\kappa}$ which is less  than the right-hand expressions in both (\ref{16-4-72}) and (\ref{16-4-73}).

Thus we arrive to estimate (\ref{16-4-72}) below.

\paragraph{Pilot-model approximation.}
\label{sect-16-4-4-2}
Meanwhile for the pilot-model approximation $\sigma = \mu^{-1}h$ as long as
$T\le T_*\Def \alpha ^{-1}\mu^{\frac{3}{2}}h^{\frac{1}{2}}$; so we conclude that as $\alpha\ge \bar{\alpha}$, the contribution to $(\I-\bar{\I})$ of $\alpha$-ball intersected with $\{(x,y): |x-y|\le \bar{\alpha}\}$  does not exceed $C\alpha\mu^{\frac{1}{2}\kappa}h^{-\frac{1}{2}\kappa}$. After summation with respect to $\alpha^{-1}d\alpha$ we get $C\mu^{\frac{1}{2}\kappa}h^{-\frac{1}{2}\kappa}$ which is less than the right-hand expression in (\ref{16-4-73}).

Now we need to consider larger $T$ and $\gamma$; then $\sigma =\gamma^2$ and
$T\asymp \alpha^{-1} \mu ^2\gamma$. Then we should define $\bar{\gamma}$ so that
$\sigma T \le  \hbar$ iff $\gamma \le \bar{\gamma}$. Therefore we pick up
$\bar{\gamma}= \mu^{-\frac{1}{3}}h^{\frac{1}{3}}\alpha^{\frac{1}{3}}$ and matching $\bar{T}= \mu^{\frac{4}{3}}h^{\frac{1}{3}}\alpha^{-\frac{2}{3}}$
and $\bar{\gamma}= \mu^{-\frac{2}{3}}h^{\frac{2}{3}}\alpha^{\frac{2}{3}}$.

Using the same arguments as above we conclude that the contribution to $(\I-\bar{\I})$ of the zone $\{(x,y): \bar{\alpha}\le |x-y|\le \bar{\gamma}\}$ and contributions to both $\I$ and $\bar{I}$  of the zone
$\{(x,y): \bar{\alpha}\ge |x-y|\le \bar{\gamma}\}$ also do not exceed the right-hand expression  in (\ref{16-4-73}).

Thus we arrive  to  estimate (\ref{16-4-73}) below.

\begin{theorem}\label{thm-16-4-23}
For Schr\"odinger-Pauli operator  under conditions \textup{(\ref{16-0-5})}--\textup{(\ref{16-0-6})}, \textup{(\ref{16-0-8})}, $\mu \ge h^{-1}$ and $F=1$
\begin{gather}
|\I - \cI^\MW|\le C \mu^{\frac{1}{2}+\frac{1}{2}\kappa}h^{-\frac{1}{2}-\frac{1}{2}\kappa}
\label{16-4-72}\\
\shortintertext{and}
|\I-\bar{\I}|\le C  \mu^{\frac{1}{3}+\frac{1}{3}\kappa}h^{-\frac{1}{3}-\frac{1}{3}\kappa}.
\label{16-4-73}
\end{gather}
\end{theorem}

\section{Problems and remarks}
\label{sect-16-4-5}

\begin{remark}\label{rem-16-4-24}
The main difference between cases $1<\kappa<2$ and $0<\kappa<1$ as $\mu h \ll 1$ is that in the former case the main contribution to the remainder is delivered by $(x,y)$ close to one another ($|x-y|\ll \mu^{-1}$) while in the latter case by $(x,y)$ with $|x-y|\asymp \mu^{-1}$.
\end{remark}

We can calculate $\cI^\MW$ and $\bar{\I}$ plugging corresponding expressions $e^\MW_z(x,y,0)$ and $\bar{e}_z(x,y,0)$ with $z=\frac{1}{2}(x+y)$ into $\I$.

\begin{Problem}\label{problem-16-4-25}
Find nice expressions for $\cI^\MW$ and $\bar{\I}$.
\end{Problem}

\begin{Problem}\label{problem-16-4-26}
As $\mu h\le 1$ get rid off condition (\ref{16-0-5}); according to Chapter~\ref{book_new-sect-13} we do not need it for estimate $|\I-\I^\T|$ but we want to get rid of it in estimates for  $|\I-\cI^\W|$, $|\I-\cI^\MW|$ and $|\I-\bar{\I}|$. To do this

\medskip\noindent
(i) Assume first that condition
\begin{equation}
|V|+|\nabla V|\asymp 1
\label{16-4-74}
\end{equation}
is fulfilled and consider the scaling function $\alpha(x)= \max\bigl(\epsilon |V(x)|,\mu h, C\mu^{-1}\bigr)$.

\medskip\noindent
(ii) Assume then that condition
\begin{equation}
|V|+|\nabla V|+|\det \Hess V| \asymp 1
\label{16-4-75}
\end{equation}
is fulfilled and consider the scaling function
\begin{equation*}
\alpha(x)= \max\bigl(\epsilon (|V(x)|+|\nabla V(x)|^2)^{\frac{1}{2}},(\mu h)^{\frac{1}{2}}, C\mu^{-1}\bigr).
\end{equation*}
\end{Problem}

\begin{Problem}\label{problem-16-4-27}
Get rid off condition (\ref{16-0-6}) assuming instead that
$|F|+|\nabla F|\ge \epsilon$.

\medskip\noindent
(i) Repeating arguments of section~\ref{sect-16-3} and using results of Chapter~\ref{book_new-sect-14} one can prove easily that
$|\I-\I^\T|\le C\mu^{-1/2}h^{-1-\kappa}$ as $\mu \le h^{-1}$.
Derive similar estimates for Schr\"odinger and Schr\"odinger-Pauli  operators as $h^{-1}\le \mu\le h^{-2}$ and for Schr\"odinger-Pauli  operator as
$\mu\ge h^{-2}$;

\medskip\noindent
(ii) Introducing scaling function $\ell(x)=\epsilon |F(x)|$ in zone
$\{x:\ \ell(x)\ge \mu^{-\frac{1}{2}}\}$ one can rescale
$\mu \mapsto \mu \ell^2$, $h\mapsto h\ell^{-1}$,
$\mu h^{-1-\kappa}\mapsto \mu^{-1}h^{-1-\kappa} \ell^{-1+\kappa}$, after integration over $\ell^{-2}d\ell$ we get $C\mu^{-1}h^{-1-\kappa} \ell^{-2+\kappa}$ calculated as $\ell=\mu^{-\frac{1}{2}}$ i.e. $\mu^{-\frac{1}{2}\kappa}h^{-1-\kappa}$. This is an estimate for the contribution of the regular zone.

Prove the same estimate for the contribution of the singular zone as
$\mu \le h^{-1}$. Derive similar estimates  as $h^{-1}\le \mu\le h^{-2}$ and  as  $ \mu\ge h^{-2}$;

\medskip\noindent
(iii) Is it possible to upgrade the above estimate  to $C\mu^{-\frac{1}{2}}h^{-1-\kappa}$ as $\kappa < 1$? Explore also cases  $h^{-1}\le \mu\le h^{-2}$ and  $ \mu\ge h^{-2}$.
\end{Problem}

\chapter{Pointwise asymptotics: $3\D$-pilot-model}
\label{sect-16-5}

\section{Pilot-model in $\bR^3$: propagator}
\label{-16-5-1}
Consider the \emph{pilot-model  operator\/}\index{operator!pilot-model}
\begin{equation}
A=\bar{A} \Def h^2 D_1^2 + (hD_2-\mu x_1)^2+h^2 D_3^2+
2\alpha x_1+2\beta x_3
\label{16-5-1}
\end{equation}
which is the sum of $2\D$-pilot-model (\ref{16-1-1}) $\bar{A}_{(2)}$ and $1\D$ Schr\"odinger operator
\begin{equation}
B\Def h^2 D_3^2  +\beta x_3
\label{16-5-2}
\end{equation}
and therefore propagator of $\bar{A}$ is
\begin{equation}
U(x,y,t)=U_{(2)}(x',y',t) U_{(1)} (x_3,y_3,t)
\label{16-5-3}
\end{equation}
where $U_{(2)}(x',y',t)$ is the Schwartz kernel of the propagator $e^{ih^{-1}\bar{A}_{(2)}}$ and therefore \emph{after rescaling\/} it is defined by (\ref{16-1-1}) while $U_{(1)}(x_3,y_3,t)$ is the Schwartz kernel of the propagator $e^{ih^{-1}tB}$ and then one can prove easily that \emph{before rescaling}
\begin{multline}
U_{(1)}(x_3,y_3,t)=\\[2pt]
(2\pi h)^{-1}\int \exp \Bigl(ih^{-1}\bigl(
(\zeta +\beta t)x_3 - (\zeta -\beta t)y_3 -2\zeta^2 t -\frac{2}{3}\beta^2t^3
\bigr)\Bigr)\,d\zeta=\\[3pt]
\frac{1}{2} (2\pi ht)^{-\frac{1}{2}}
\exp \Bigl(ih^{-1}\bigl(\beta t (x_3+y_3)+\frac{1}{8}t^{-1}(x_3-y_3)^2 -\frac{2}{3}\beta^2t^3\bigr)\Bigr);
\label{16-5-4}
\end{multline}
in particular
\begin{equation}
U_{(1)}(x_3,x_3,t)=
\frac{1}{2} (2\pi ht)^{-\frac{1}{2}}
\exp \Bigl(ih^{-1}\bigl(2\beta t x_3    -\frac{2}{3}\beta^2t^3\bigr)\Bigr).
\label{16-5-5}
\end{equation}
Therefore
\begin{claim}\label{16-5-6}
After standard rescaling $x\mapsto \mu x$, $t\mapsto \mu t$ $U(x,x,t)$ in comparison with $U_{(2)}(x',x',t)$ acquires factor
\begin{equation}
\frac{1}{2} \mu (2\pi \hbar t)^{-\frac{1}{2}}
\exp\Bigl( i\hbar^{-1}
\bigl(2\beta tx_3-\frac{2}{3}\mu^{-2}\beta^2t^3\bigr)\Bigr)
\label{16-5-7}
\end{equation}
\end{claim}
and therefore (\ref{16-1-23}) and (\ref{16-1-24}) are replaced respectively by
\begin{gather}
\varphi (t)\Def -t^2 \mu^{-2}\alpha ^2 \cot (t ) + t\mu^{-2}\alpha^2 +\frac{2}{3}\mu^{-2}\beta^2t^3-t\tau
\label{16-5-8}
\\
\shortintertext{and}
\mu^{-2} \alpha^2 (t^2 - t\sin (2t)) + 2\mu^{-2}\beta^2t^2\sin^2(t) = (\tau-\mu^{-2}\alpha^2) \sin^2(t);
\label{16-5-9}
\end{gather}
the latter equation could be rewritten as (\ref{16-1-24}) with $\tau$ replaced by $\tau'$
\begin{gather}
\mu^{-2} \alpha^2 (t^2 - t\sin (2t))  = (\tau'-\mu^{-2}\alpha^2) \sin^2(t)
\label{16-5-10}\\
\shortintertext{and}
2\mu^{-2}\beta^2t^2= (\tau-\tau');\label{16-5-11}
\end{gather}
the former equation shows the return times of $2\D$-movement on the energy level $\tau'$ and the former the looping time of $1\D$-movement associated with $B$ (after rescaling) on the energy level $(\tau-\tau')$.

\begin{remark}\label{rem-16-5-1}
We should set $T=\epsilon_0\mu $ (in contrast to $2\D$-case). Really, unless we want to make assumptions outside of the ball $B(0,\ell)$ with a small constant $\ell$ we need to take $T\le \epsilon_0 $ (before rescaling) and then  $T\le \epsilon_0 \mu$ (after rescaling). This absolves us from the analysis of the equatorial zone.
\end{remark}

Let us compare solutions $t^*_k$ of (\ref{16-5-9}) and solutions
$t_k$ to (\ref{16-1-24}). One can see easily that

\begin{claim}\label{16-5-12}
As $|k|\le \epsilon |\beta|^{-1}\mu$
\begin{gather}
t^*_k = t_k\bigl(1+O(\mu^{-2}\beta^2 k^2)\bigr)
\label{16-5-13}\\
\shortintertext{and}
\varphi''(t^*_k)=
\varphi_{(2)}''(t_k)\bigl(1+O(\mu^{-2}\beta^2 k^2)\bigr)
\label{16-5-14}
\end{gather}
with $\varphi_{(2)}$ defined by (\ref{16-1-23}). Furthermore other properties of $\varphi_{(2)}$ are fulfilled for $\varphi$ as well.
\end{claim}

\section{Tauberian estimates}
\label{-16-5-2}

Then as $|k|\le \epsilon \mu^{-1}$ in virtue of (\ref{16-5-6}) and calculations of section~\ref{sect-16-1} contribution of $k$-th tick to
$F_{t\to \hbar^{-1}\tau}\Gamma_x U $ does not exceed
\begin{equation}
C\left\{ \begin{aligned}
&\mu h^{-1} \times \bigl(\mu /h |k|\bigr)^{\frac{1}{2}}
\qquad&&\text{as\ \ } 1\le |k|\le \bar{k},\\
&\mu h^{-1} (\mu^2 h/\alpha |k|)^{\frac{1}{2}}
\times \bigl(\mu /h |k|\bigr)^{\frac{1}{2}}
\qquad&&\text{as\ \ } \bar{k}\le |k|\le \epsilon \mu^{-1},
\end{aligned}\right.
\label{16-5-15}
\end{equation}
where as before $\bar{k}= \epsilon \mu^2h \alpha^{-1}$. Recall that
without an extra factor $\bigl(\mu /h |k|\bigr)^{\frac{1}{2}}$ we would have estimate for $F_{t\to \hbar^{-1}\tau}\Gamma_{x'} U_{(2)}. $

However, this extra factor is a game changer. As before there are three cases:
\begin{enumerate}[leftmargin=*, label=(\alph*)]
\item $\alpha \ge \mu^2h$; then we set $\bar{k}=1$ and the summation with respect to $k:1\le |k|\le \epsilon \mu \ell$ returns
\begin{equation*}
C \mu h^{-1} (\mu^2 h/\alpha)^{\frac{1}{2}}\cdot (\mu /h )^{\frac{1}{2}}
\log (\mu\ell )\asymp
C \mu^{\frac{5}{2}} h^{-1} \alpha^{-\frac{1}{2}}\log (\mu\ell);
\end{equation*}

\item $\mu h \ell^{-1} \le \alpha \le \mu^2h$; then
$1\le \bar{k}\le \epsilon \mu \ell $ and summation with respect to
$k: \bar{k}\le |k|\le \epsilon \mu\ell $ returns
\begin{equation*}
C \mu^{\frac{5}{2}} h^{-1} \alpha^{-\frac{1}{2}}\log (\mu\ell)/\bar{k}\asymp
C \mu^{\frac{5}{2}} h^{-1}
\alpha^{-\frac{1}{2}}\bigl(1+ (\log (\alpha\ell/\mu h))_+\bigr)
\end{equation*}
while summation with respect to $k: 1\le k|\le \bar{k}$ returns the same expression albeit without logarithmic factor;

\item $\alpha \le \mu h \ell^{-1}$; then $\bar{k}\ge \epsilon \mu \ell$ and we reset set $\bar{k}=\epsilon \mu \ell$ and summation with respect to $k$ returns
\begin{equation*}
C \mu h^{-1} \times \bigl(\mu /h \bigr)^{\frac{1}{2}}\bar{k}^{\frac{1}{2}}\asymp
C \mu^2 h^{-\frac{3}{2}}  \ell  ^{\frac{1}{2}}.
\end{equation*}
\end{enumerate}

So we have proved

\begin{proposition}\label{prop-16-5-2} (cf. proposition~\ref{prop-16-1-6}).
(i) After rescaling  as $T=\epsilon \mu \ell$, $\ell\ge C_0\mu^{-1}$
\begin{multline}
|F_{t\to \hbar^{-1}\tau}\bar{\chi}_T(t)\Gamma_x \mathsf{U}|\le
C\mu h^{-2} + \\[2pt]
C\left\{\begin{aligned}
& \mu^{\frac{5}{2}} h^{-1}
\alpha^{-\frac{1}{2}}\bigl(1+ (\log (\alpha\ell/\mu h))_+\bigr)
\qquad && \text{as\ \ } \alpha \ell \ge \mu h ,\\
&\mu^2 h^{-\frac{3}{2}}  \ell  ^{\frac{1}{2}}
\qquad && \text{as\ \ } \alpha \ell \le \mu h;
\end{aligned}\right.
\label{16-5-16}
\end{multline}
(ii) Therefore as operator $A$ coincides with the pilot-model operator \textup{(\ref{16-5-1})} in $B(0,\ell)$
\begin{multline}
|e(0,0,\tau)-e^\T(0,0,\tau)|\le
C h^{-2}\ell^{-1} + \\[2pt]
C\left\{\begin{aligned}
& \mu^{\frac{3}{2}} h^{-1}
\alpha^{-\frac{1}{2}}\bigl(1+ (\log (\alpha\ell/\mu h))_+\bigr)\ell^{-1}
\qquad && \text{as\ \ } \alpha \ell \ge \mu h ,\\
&\mu  h^{-\frac{3}{2}}  \ell  ^{-\frac{1}{2}}
\qquad && \text{as\ \ } \alpha \ell \le \mu h;
\end{aligned}\right.
\label{16-5-17}
\end{multline}
(iii) In particular, as $\alpha\asymp \ell\asymp 1$ the right-hand expression in \textup{(\ref{16-5-17})} does not exceed
$C h^{-2}+ C\mu^{\frac{3}{2}} h^{-1}\bigl(1+ |\log \mu h| \bigr)$
which is $O(h^{-2})$ as $\mu \le (h|\log h|)^{-\frac{2}{3}}$.
\end{proposition}

\section{Weyl estimates}
\label{-16-5-3}

Even simpler are Weyl estimates: in comparison with $2\D$-case they acquire factor $\mu^{\frac{1}{2}}h^{-\frac{1}{2}}$ and (\ref{16-1-64}) becomes
\begin{equation}
\R^\W_x\Def |e^\T(x,x,0) - h^{-3}\cN_x^\W|\le
C\mu^{-\frac{1}{2}}h^{-\frac{3}{2}}+ C\mu^{\frac{5}{2}}\alpha^{-\frac{1}{2}} h^{-1}
\label{16-5-18}
\end{equation}
and more generally we arrive to statement (i) below

\begin{proposition}\label{prop-16-5-3} (cf. proposition~\ref{prop-16-1-12}).
For a pilot-model operator \textup{(\ref{16-5-1})} with  $\tau\asymp 1$,

\medskip\noindent
(i) As $\mu^2h\le \alpha\le 1$ estimate \textup{(\ref{16-5-18})} holds; moreover if we  replace $\cN_x^\W$ by $\cN_x^\W +\cN_{x,\corr(r)}$ with the correction term $h^{-3}\cN_{x,\corr(r)}$  delivered by $r$-term stationary phase approximation then an error does not exceed
\begin{equation}
C\mu^{-\frac{1}{2}}h^{-\frac{3}{2}} + C\mu ^{\frac{3}{2}}h^{-\frac{3}{2}} \bigl(\mu^2h/\alpha \bigr)^{r+\frac{1}{2}};
\label{16-5-19}
\end{equation}
(ii) As $ \alpha\le \mu^2h$ estimate holds:
\begin{equation}
\R^\W_x\Def |e^\T(x,x,0) - h^{-3}\cN_x^\W|\le
C\mu^{\frac{3}{2}} h^{-\frac{3}{2}};
\label{16-5-20}
\end{equation}
(iii) In particular, as $\mu \le h^{-\frac{1}{3}}$ we have $\R^\W =O(h^{-2})$ without any assumptions to $\alpha,\beta$: $|\alpha|\le 1$, $|\beta|\le 1$.
\end{proposition}

Recall that here we do not need to analyze the equatorial zone and therefore the difference between cases $\alpha \ge \mu \hbar^{\frac{2}{3}}$ and $\alpha \le \mu \hbar^{\frac{2}{3}}$ disappears together with corresponding terms in (\ref{16-1-66}) or (\ref{16-1-67}).

\section{Micro-averaging}
\label{-16-5-4}

Micro-averaging adds yet another layer of complexity. Let us consider an \emph{isotropic micro-averaging}\index{micro-averaging!isotropic}
with function $\psi_\gamma(x)=\psi (x/\gamma)$ and an
\emph{anisotropic micro-averaging}\index{micro-averaging!anisotropic}
with function $\psi_\boldgamma(x)=\psi (x'/\gamma,x_3/\gamma_3)$
where $\boldgamma=(\gamma,\gamma_3)$ is a scale with respect to $(x',x_3)$.

Then one needs to integrate by parts with respect to $x_3$ taking in account factor $\exp \bigl(2ih^{-1}\beta t x_3 \bigr)$ in (\ref{16-5-5}) and rescaling with respect to $t$ which brings factor $(\mu h/|\beta|\gamma_3 |k|)^l$ as
$(\mu h/|\beta|\gamma_3 |k|)\le 1$ in addition to factor
$(\mu h/|\alpha|\gamma  |k|)^l$ as  $(\mu h/|\alpha|\gamma |k|)\le 1$. Thus

\begin{claim}\label{16-5-21}
Micro-averaging brings factor
$(\mu h/\nu (\boldgamma) |k|)^l$ as $(\mu h/\nu (\boldgamma) |k|)\le 1$ with
$\nu (\boldgamma)\Def |\alpha|\gamma +|\beta|\gamma_3 $.
\end{claim}
Then we need to modify (\ref{16-1-53}) replacing $|\alpha|\gamma$ by $\nu(\boldgamma)$:
\begin{equation}
|k|\ge \hat{k}(\boldgamma)\Def \mu h/\nu (\boldgamma).
\label{16-5-22}
\end{equation}

To understand how it affects a Tauberian estimate we need to consider the following cases:

\begin{enumerate}[label=(\alph*),leftmargin=*]
\item \label{case-a} Micro-averaging has no effect as
$\hat{k}(\boldgamma)\ge \epsilon \mu \ell $ i.e. as
$ \nu (\boldgamma)  \le h \ell^{-1} $;

\item  \label{case-b} Micro-averaging has a minimal effect (only decreasing the logarithmic factor) as $\epsilon \mu \ell \ge \hat{k}(\boldgamma)  \ge \bar{k}$. Obviously  in this case $\bar{k}=\epsilon \max\bigl(\mu^2h/\alpha,1 \bigr)$.

So this case holds iff
$ h \ell^{-1} \le \nu(\boldgamma) \le \min(\mu h , \alpha \mu ^{-1})$. Then the logarithmic factor becomes $\log \bigl(\hat{k}(\boldgamma)/\bar{k}\bigr)$.

\item \label{case-c} Micro-averaging kills logarithmic factor and further leads to $\bar{k}^{\frac{1}{2}}$ replaced by $\hat{k}(\boldgamma)^{\frac{1}{2}}$ as
$1\le \hat{k}(\boldgamma)\le \bar{k}$. Obviously  in this case
$\bar{k}=\epsilon \min\bigl(\mu^2h/\alpha,\mu \ell  \bigr)$.

So this case holds iff $\max(h \ell^{-1},\alpha \mu^{-1}) \le \nu(\boldgamma) \le \mu h  $. Then the right-hand expression in (\ref{16-5-24}) becomes
$C\mu^2 h^{-1}\nu^{-\frac{1}{2}}\gamma^{-\frac{1}{2}}$;

\item  \label{case-d} Micro-averaging resets $\bar{k}$ to $1$ and brings factor
$(\mu h/\nu(\boldgamma))^l$ as $\hat{k}(\boldgamma)\le 1$ i.e. as
$\nu(\boldgamma)\ge \mu h$. Then the right-hand expression in (\ref{16-5-24}) becomes $C\mu^{\frac{3}{2}}h^{-\frac{3}{2}}(\mu h/\nu(\boldgamma))^l$.
\end{enumerate}

Thus we arrive to

\begin{proposition}\label{prop-16-5-4}
For the pilot-model operator \textup{(\ref{16-5-1})} with  $\tau\asymp 1$,
$0\le \alpha\le 1$, $|\beta|\le 1$, $\mu \le h^{-1}$ in $\bR^3$ estimates
\begin{gather}
\gamma^{-2}\gamma_3^{-1} |F_{t\to \hbar^{-1}\tau}
\bigl(\bigl(\bar{\chi}_T(t) -\bar{\chi}_{\bar{T}}\bigr)
\Gamma (U\psi_\boldgamma)\bigr)|\le C\mu  R^\T(\boldgamma)
\label{16-5-23}\\
\shortintertext{and}
\gamma^{-2}\gamma_3^{-1}
|F_{t\to \hbar^{-1}\tau} \bigl(\bar{\chi}_T(t)
\Gamma (U\psi_\boldgamma)\bigr)|\le
C\mu h^{-2}+ C\mu  R^\T(\boldgamma)
\label{16-5-24}
\end{gather}
hold as  $T =\epsilon \mu \ell$  where
\begin{multline}
R^\T(\boldgamma)= \\[2pt]
\left\{\begin{aligned}
& \mu^{\frac{3}{2}} h^{-1}
\alpha^{-\frac{1}{2}}\bigl(1+ (\log (\alpha\ell/\mu h))_+\bigr)
\qquad && \text{as\ \ }
\alpha \ell \ge \mu h , \ \nu(\boldgamma) \le h\ell^{-1},\\[2pt]
&\mu  h^{-\frac{3}{2}}  \ell  ^{\frac{1}{2}}
\qquad && \text{as\ \ }
\alpha \ell \le \mu h, \ \nu(\boldgamma) \le h \ell^{-1},\\[2pt]
& \mu^{\frac{3}{2}} h^{-1}
\alpha^{-\frac{1}{2}}\bigl(1+ (\log (\alpha/\mu \nu(\boldgamma)))_+\bigr)
\qquad && \text{as\ \ }
h\ell^{-1}\le \nu(\boldgamma) \le \min(\mu h , \alpha \mu^{-1}),\\[2pt]
&\mu  h^{-1}\nu(\boldgamma)^{-\frac{1}{2}}
\qquad && \text{as\ \ }
\max(h \ell^{-1},\alpha \mu^{-1}) \le \nu(\boldgamma)\le \mu h,\\[2pt]
&\mu^ {\frac{1}{2}}h^{-\frac{3}{2}}(\mu h/\nu (\boldgamma))^l\qquad
&&\text{as\ \ } \nu(\boldgamma)\ge \mu h.
\end{aligned}\right.
\label{16-5-25}
\end{multline}
\end{proposition}

In the right-hand expression of (\ref{16-5-25}) two first lines correspond to the case \ref{case-a}.

\begin{corollary}\label{cor-16-5-5}
In frames of proposition~\ref{prop-16-5-4} the following estimate holds
\begin{equation}
\gamma^{-2}\gamma_3^{-1}
| \Gamma \bigl(e(.,.,\tau) -e^\T(.,.,\tau)\bigr)\psi_\boldgamma|\le
C \bigl(h^{-2}+   R^\T(\boldgamma)\bigr)\ell^{-1}
\label{16-5-26}
\end{equation}
\end{corollary}

Consider now how micro-averaging affects Weyl estimate (with the right-hand expression (\ref{16-5-19})). There is no effect as $\nu(\boldgamma) \le \mu h$ and factor $(\mu h/\nu (\boldgamma))^l$ comes out as $\nu(\boldgamma) \ge \mu h$. Thus we arrive to

\begin{proposition}\label{prop-16-5-6}
For the pilot-model operator \textup{(\ref{16-5-1})} with  $\tau\asymp 1$,
$0\le \alpha\le 1$, $\mu \le h^{-1}$ in $\bR^3$

\medskip\noindent
(i) As $\mu ^2h\le \alpha\le 1$ estimate holds:
\begin{multline}
\R^{\W}_{x(r)}(\boldgamma)\Def\\
\gamma^{-2}\gamma_3^{-1}|\int \Bigl(e^\T (x,x,0) -
h^{-3} \bigl(\cN_x^\W +\cN_{x,\corr(r)}\bigr)\Bigr)\psi_\boldgamma\,dx |\le \\
Ch^{-2}+ C\mu ^{\frac{3}{2}}h^{-\frac{3}{2}}
\bigl(\mu^2h/\alpha \bigr)^{r+\frac{1}{2}}
\left\{\begin{aligned}
&\bigl( \mu h/\nu  (\boldgamma))^l
&&\qquad\text{as\ \ }  \nu(\boldgamma)\ge \mu h ,\\
&1 &&\qquad\text{as\ \ }  \nu(\boldgamma)\le \mu h
\end{aligned}\right.
\label{16-5-27}
\end{multline}

\medskip\noindent
(ii) As $\alpha\le \mu ^2h$ estimate holds:
\begin{multline}
\R^{\W}_{x}(\boldgamma)\Def \gamma^{-2}\gamma_3^{-1}
|\int \Bigl(e^\T (x,x,0)  -
h^{-3} \cN_x^\W \Bigr)\psi_\boldgamma\,dx |\le \\
Ch^{-2}+ C\mu ^{\frac{3}{2}}h^{-\frac{3}{2}}
\left\{\begin{aligned}
&\bigl( \mu h/\nu  (\boldgamma))^l
&&\qquad\text{as\ \ }  \nu(\boldgamma)\ge \mu h,\\
&1 &&\qquad\text{as\ \ }  \nu(\boldgamma)\le \mu h.
\end{aligned}\right.
\label{16-5-28}
\end{multline}
\end{proposition}

\section{Strong and superstrong magnetic field}
\label{-16-5-5}

\subsection{Tauberian estimate}
\label{-16-5-5-1}

Consider strong  $\mu h\asymp 1$ and superstrong $\mu h \gtrsim 1$ magnetic field for the Schr\"odinger-Pauli pilot-model operator
\begin{equation}
A=\bar{A} \Def h^2 D_1^2 + (hD_2-\mu x_1)^2+h^2 D_3^2+
2\alpha x_1+2\beta x_3-\mu h
\label{16-5-29}
\end{equation}
which alternatively means that we assume that $|\tau - \fz\mu h|\le C$ in the framework of the standard pilot-model (\ref{16-5-1}), $\fz=1$.

\begin{proposition}\label{prop-16-5-7}
For a pilot-model operator \textup{(\ref{16-5-1})} with $\mu h\gtrsim 1$ in $\bR^3$ as ${|x|\le C_0}$, $|y|\le C_0$, $\tau \le c\mu h$

\medskip\noindent
(i) $e(x,y,\tau)\equiv 0 \mod O(\mu h^{\infty})$  as \
$\tau \le \mu h -\epsilon_0$ (lower spectral gap);

\medskip\noindent
(ii) As $\alpha>0$ the following estimates hold
\begin{multline}
| e(x,x,\tau+h)- e(x,y,\tau)|\le\\
C\mu  h^{-1}\max_{j\in \bZ^+}
\bigl(\max(|V(x)+2j\mu h-\tau |,h)\bigr)^{-\frac{1}{2}}
\label{16-5-30}
\end{multline}
and
\begin{equation}
|F_{t\to \hbar^{-1}\tau} U(x,y,t) |\le
C\mu^2 h^{-1}
\max_{j\in \bZ^+}\bigl(\max(|V(x)+2j\mu h-\tau |,h)\bigr)^{-\frac{1}{2}}
\label{16-5-31}
\end{equation}
and these estimate are sharp as $x=y$ and $\tau$ is close to Landau level
$\mu h$. Recall that in \textup{(\ref{16-5-31})} $t$ is rescaled.
\end{proposition}

\begin{proof}
Statement (i) is obvious. To prove (ii) note that
\begin{multline}
e(x,y,\tau)= \int e_B(x_3,y_3,\tau-\tau')\, d_{\tau'} e_{(2)}(x',y',\tau')=\\
\int e_{(2)}(x',y',\tau-\tau')\, d_{\tau'} e_B(x_3,y_3,\tau')
\label{16-5-32}
\end{multline}
and therefore
\begin{multline}
e(x,y,\tau+h)- e(x,y,\tau)=\\ \int
\bigl(e_B(x_3,y_3,\tau-\tau'+h)- e_B(x_3,y_3,\tau-\tau')\bigr)\, \partial_{\tau'} e_{(2)}(x',y',\tau')\,d\tau'.
\label{16-5-33}
\end{multline}
Recall that due to our analysis in section~\ref{sect-16-1} $\partial_{\tau'}e_{(2)}(x',y',\tau')$ is essentially supported in $C\mu^{-\frac{1}{2}}h^{\frac{1}{2}}\alpha $-vicinity of Landau level and fast decays out of it.

Finally recall that due to the analysis of subsubsections~\ref{book_new-sect-5-2-1-3}--\ref{book_new-sect-5-2-1-4}
\begin{equation}
|e_B(x_3,y_3,\tau+h)-e_B(x_3,y_3,\tau)|\le C(\tau')^{-\frac{1}{2}}\qquad
\text{as\ \ } \tau'\ge C_0 h.
\label{16-5-34}
\end{equation}

Combining it with estimate (\ref{16-1-73}) for
$|\partial_{\tau'} e_{(2)}(x',y',\tau')|$ we conclude that the left-hand expression of (\ref{16-5-30}) does not exceed
\begin{equation*}
C\mu ^{\frac{3}{2}} \alpha^{-1}h^{-\frac{3}{2}})\times \alpha \mu^{-\frac{1}{2}}h^{\frac{1}{2}} \times
\max_{j\in \bZ^+} \bigl(\max(|V(x)+2j\mu h-\tau|,h)\bigr)^{-\frac{1}{2}}.
\end{equation*}
Estimate (\ref{16-5-31}) is proven; estimate (\ref{16-5-31}) follows from it.
\end{proof}

Further, as before $T^*= \epsilon_0\mu \ell$. Then we immediately arrive to

\begin{corollary}\label{cor-16-5-8}
For a self-adjoint operator in domain $X$, $B(0,\ell)\subset X\subset \bR^2$, $\ell\ge C_0h$, coinciding in $B(0,\ell)$ with the pilot-model operator \textup{(\ref{16-5-29})} with
$\tau\asymp 1$, $0< \alpha \le 1$, $|\beta|\le 1$, $\mu \ge h^{-1}$

\medskip\noindent
(i) Statements (i)-(ii) of proposition \ref{prop-16-5-7} remain true;

\medskip\noindent
(ii) Formula \textup{(\ref{16-5-32})}  holds modulo
$O(\mu h^{-\frac{3}{2}}\ell^{-1})$.
\end{corollary}

\subsection{Micro-averaging}
\label{-16-5-5-2}

Let us consider micro-averaging. First let us estimate Fourier transform where as before we rescaled $t\mapsto \mu t$; formulae (\ref{16-5-16})--(\ref{16-5-17}) imply immediately

\begin{proposition}\label{prop-16-5-9}
For a pilot-model operator \textup{(\ref{16-5-29})} in $\bR^3$  with
$\mu h\ge 1$, $\gamma_3 \ge h $
\begin{gather}
\gamma^{-2}\gamma_3^{-1}| F_{t\to h^{-1}\tau} \bar{\chi}_T(t)
\Gamma (U\psi_\boldgamma )|\le
C\mu^2 h^{-1} \gamma_3^{-\frac{1}{2}}
\label{16-5-35}
\end{gather}
\end{proposition}

\begin{corollary}\label{cor-16-5-10}
In frames of corollary~\ref{cor-16-5-8} as $\ell\ge \gamma_3\ge h$
\begin{equation}
\gamma^{-2}\gamma_3^{-1} |\Gamma
\bigl(e(.,.,\tau)-e^\T (.,.,\tau)\bigr)\psi _\boldgamma |\le
C \mu h^{-1}\gamma_3^{-\frac{1}{2}} \ell^{-1}.
\label{16-5-36}
\end{equation}
\end{corollary}

\section{Magnetic Weyl approximation}
\label{-16-5-6}

We can try to use a host of the different approximations but restrict ourselves now to the magnetic Weyl approximation.  Recall (\ref{16-5-32}) and note that
\begin{multline}
e^\MW(x,y,\tau)= \int e^\W_B(x_3,y_3,\tau-\tau')\,
d_{\tau'} e_{(2)}^\MW (x',y',\tau')=\\
\int e_{(2)}^\MW (x',y',\tau-\tau')\,
d_{\tau'} e^\W_B(x_3,y_3,\tau').
\label{16-5-37}
\end{multline}

\subsection{Pointwise asymptotics}
\label{-16-5-6-1}

\begin{proposition}\label{prop-16-5-11}
For a pilot-model operator \textup{(\ref{16-5-1})}  with
$\tau\le c$, $\mu h\lesssim 1$ and for a pilot-model operator \textup{(\ref{16-5-29})}  with  $\tau\le c$, $\mu h\gtrsim 1$
\begin{equation}
|e^\MW(x,y,\tau)-e^\MW(x,y,\tau')|\le C
\left\{\begin{aligned}
&\mu ^{\frac{3}{2}}h^{-1}\qquad&&\text{as\ \ } \mu h\le 1,\\
&\mu h^{-\frac{3}{2}}\qquad&&\text{as\ \ } \mu h\ge 1
\end{aligned}\right.
\label{16-5-38}
\end{equation}
as $|\tau-\tau'|\le h$.
\end{proposition}
\begin{proof}
Obviously the left-hand expression does not exceed
\begin{equation*}
C\mu h^{-\frac{3}{2}}\sum_{j\le J} j^{-\frac{1}{2}} \asymp
C\mu h^{-\frac{3}{2}}J^{\frac{1}{2}}
\end{equation*}
with $J=C_0\max ( (\mu h)^{-1},1)$ which implies (\ref{16-5-38}).
\end{proof}
Therefore one can hardly expect that the magnetic Weyl approximation provides a better error than the right hand expression of (\ref{16-5-38}).

\begin{proposition}\label{prop-16-5-12}
Let $|\tau|\le \epsilon$. Then for a pilot-model operator \textup{(\ref{16-5-1})} as $\mu h\lesssim 1$ and for a pilot-model operator \textup{(\ref{16-5-1})} as $\mu h\gtrsim 1$
\begin{equation}
\R^\MW \Def \ |  e^\T (x,x,0)  - h^{-3}  \cN^\MW |\le \\
C h^{-2} +C\mu h^{-1}+ C \mu h^{-\frac{5}{3}}|\beta|^{\frac{1}{3}}.
\label{16-5-39}
\end{equation}
\end{proposition}

\begin{proof}  Without any loss of the generality we can assume that $x=0$, $\beta>0$. As contribution of $k$-th tick to the Tauberian expression is $O(\mu^{\frac{3}{2}}h^{-\frac{3}{2}}|k|^{-\frac{3}{2}})$, its contribution to the error when we replace $\beta$ by $0$ does not exceed $C\mu^{\frac{3}{2}}h^{-\frac{3}{2}}|k|^{-\frac{3}{2}}\times
\mu^{-3}h^{-1}\beta^2 |k|^3$ and summation with respect to
\begin{gather}
k:|k|\le \tilde{k}_1\Def \mu \beta^{-\frac{2}{3}}h^{\frac{1}{3}}
\label{16-5-40}\\
\shortintertext{returns}
C\mu^{-\frac{3}{2}}h^{-\frac{5}{2}}\beta^2 |k|^{\frac{5}{2}}=
C\mu h^{-\frac{5}{3}} \beta^{\frac{1}{3}}
\label{16-5-41}
\end{gather}
On the other hand, summation of $C\mu^{\frac{3}{2}}h^{-\frac{3}{2}}|k|^{-\frac{3}{2}}$ as
$|k|\ge \tilde{k}_1$ returns $C\mu^{\frac{3}{2}}h^{-\frac{3}{2}}\tilde{k}_1^{-\frac{1}{2}}$ which is the same expression because $\tilde{k}_1$ was defined from $\mu^{-3}h^{-1}\beta^2k^3=1$.

Meanwhile, contribution of $k$-th tick to the error when we replace $\alpha$ by $0$ does not exceed $C\mu^{\frac{3}{2}}h^{-\frac{3}{2}}|k|^{-\frac{3}{2}}\times
\mu^{-2}h^{-1}\alpha  |k|$  and summation with respect to
\begin{gather}
k:|k|\le \tilde{k}_2\Def \mu ^2 h\alpha^{-1}
\label{16-5-42}\\
\shortintertext{returns}
C\mu^{-\frac{1}{2}}h^{-\frac{5}{2}} \alpha  \tilde{k}_2^{\frac{1}{2}} =
C\mu^{-\frac{1}{2}}h^{-2} \alpha^{\frac{1}{2}}
\label{16-5-43}
\end{gather}
which is less than $Ch^{-2}$. On the other hand, summation of $C\mu^{\frac{3}{2}}h^{-\frac{3}{2}}|k|^{-\frac{3}{2}}$ as
$|k|\ge \tilde{k}_2$ returns $C\mu^{\frac{3}{2}}h^{-\frac{3}{2}}\tilde{k}_2^{-\frac{1}{2}}$  which is the same expression because $\tilde{k}_2$ was defined from equation $\mu^{-2}h^{-1}\alpha \tilde{k}_2=1$.
\end{proof}

\begin{remark}\label{rem-16-5-13}
(i)  We summed with respect to all $k$ while it would be enough only with respect to $k:|k|\le \epsilon \mu$. However it provides us by no improvement.

\medskip\noindent
(ii) So far we have not used factor  $C(\mu^2 h/\alpha |k|)^{\frac{1}{2}}$; using it we acquire in (\ref{16-5-41}) factor
$(\mu^2 h/\alpha \tilde{k}_1)^{\frac{1}{2}}=
\mu^{\frac{1}{2}} h^{\frac{1}{3}}\beta^{\frac{1}{3}}\alpha^{-\frac{1}{2}}$ resulting in
\begin{gather}
\R^\MW \le Ch^{-2}+ C\mu h^{-1}+
C\mu ^{\frac{3}{2}} h^{-\frac{4}{3}} \beta^{\frac{2}{3}}\alpha^{-\frac{1}{2}}
\label{16-5-44}\\
\shortintertext{provided}
\alpha \ge \mu h^{\frac{2}{3}}\beta^{\frac{2}{3}}.
\label{16-5-45}
\end{gather}
(iii) In contrast to $2\D$ the magnetic Weyl approximation now is  better than Weyl approximation as $\mu \ge h^{-\frac{1}{3}}$ (when it matters) but  Weyl approximation with the correction terms may be better still.
\end{remark}

\subsection{Micro-averaging}
\label{-16-5-6-2}

Consider now micro-averaging. We need to redo only the first step of the proof of proposition~\ref{prop-16-5-12} and only in the case
$\mu \ge h^{-\frac{1}{3}}\beta^{-2}$ and it is useful only if
$\tilde{k}_1 \ge \hat{k}(\boldgamma)= \mu h/\nu(\boldgamma)$ or equivalently
\begin{equation}
\nu(\boldgamma)=\alpha \gamma + \beta \gamma_3\ge \beta^{\frac{2}{3}}h^{\frac{2}{3}}.
\label{16-5-46}
\end{equation}
In this case in the left-hand expression of (\ref{16-5-41}) should be reset to $\tilde{k}_1$ replaced by $\hat{k}(\boldgamma)$ as $\hat{k}(\boldgamma) \ge 1$ or by $1$ with an extra factor $(\mu h/\nu(\boldgamma))^l$ as
$\hat{k}(\boldgamma) \le 1$ resulting in $CR_1^\MW(\boldgamma)$ with
\begin{multline}
R_1^\MW (\boldgamma)\Def\\[2pt]
C\mu  \left\{\begin{aligned}
& \beta^2 \nu (\boldgamma)^{-\frac{5}{2}}
\bigl(\mu h/\nu(\boldgamma)\bigr)^l\qquad
&&\text{as\ \ } \nu(\boldgamma)\ge \mu h,\ \alpha \le \mu^2h,\\[2pt]
& \beta^2 \nu (\boldgamma)^{-\frac{5}{2}}
\qquad&&\text{as\ \ } \nu(\boldgamma)\le \mu h,\
\mu \nu(\boldgamma)\ge \alpha,\\[2pt]
&\beta^2 \nu (\boldgamma)^{-\frac{5}{2}}(\mu^2h/\alpha)^{\frac{1}{2}}
\bigl(\mu h/\nu(\boldgamma)\bigr)^l\qquad
&&\text{as\ \ } \nu(\boldgamma)\ge \mu h,\ \alpha \ge \mu^2h,\\[2pt]
& \beta^2 \nu (\boldgamma)^{-\frac{5}{2}}
(\mu \nu(\boldgamma)/\alpha)^{\frac{1}{2}}
\qquad&&\text{as\ \ } \nu(\boldgamma)\le \mu h,\ \mu \nu(\boldgamma)\le \alpha
\end{aligned}\right.
\label{16-5-47}
\end{multline}
and therefore we arrive to

\begin{proposition}\label{prop-16-5-14}
Let $|\tau|\le \epsilon$ and \textup{(\ref{16-5-46})} be fulfilled. Then for a pilot-model operator \textup{(\ref{16-5-1})} as $\mu h\lesssim 1$ and for a pilot-model operator \textup{(\ref{16-5-1})} as $\mu h\gtrsim 1$ a
\begin{multline}
\R^\MW  (\boldgamma)  \Def \gamma^{-2}\gamma_3^{-1} | \int \bigl(e^\T (x,x,0)  - h^{-3}\int \cN^\MW \bigr)\psi_\boldgamma\,dx|\le \\[2pt]
C h^{-2} +C \mu h^{-1} +C R_1^\MW(\boldgamma )
\label{16-5-48}
\end{multline}
with $R_1^\MW(\boldgamma)$ defined by \textup{(\ref{16-5-47})}.
\end{proposition}

\begin{remark}\label{rem-16-5-15}
If we replace only $e_{(2)}(.,.,.)$ by $e^\MW (.,.,.)$ we can skip the last term.
\end{remark}

\section{Geometric interpretation}
\label{-16-5-7}

Note that there are two classical dynamics: in $x'$ and in $x_1$. The former has return times $t'_k\Def t'_k(\tau')$ where $\tau'$ is a corresponding part of energy and $t'_k$ are defined in section~\ref{sect-16-1} as $t_k(\tau)$ while the latter is
\begin{equation}
x_3(t)=x_3(0)+2\xi_3(0)t-\beta t^2,\qquad \xi_3(t)=\xi_3(0)-\beta t.
\label{16-5-49}
\end{equation}
and  has return times (to $0$)
$t''_k(\tau-\tau')= 2\beta^{-1}(\tau -\tau')$ and therefore the total system has return times\footnote{\label{foot-16-16} To point $0$.} $t_k$ and return energy partitions $(\tau'_k,\tau-\tau'_k)$  defined from the pair of equations
\begin{equation}
t_k =t'_k(\tau'_k)= t''_k(\tau-\tau')
\label{16-5-50}
\end{equation}
which is equivalent to (\ref{16-5-9})--(\ref{16-5-10}).

Therefore relatively sets of return directions $\xi_k$ is thinner in $3\D$ than in $2\D$ which explains errors acquiring lesser factor (in comparison with the principal part) in $3\D$ than in $2\D$.

In particular, as $\beta=0$ $1\D$-dynamics does not return at all unless it stays at $0$, therefore in this case (\ref{16-5-49}) mean exactly that $\tau'_k=\tau$ and $t'_k(\tau)=t_k(\tau)$.

\chapter{Pointwise asymptotics: general $3\D$-operators}
\label{sect-16-6}

\section{Set-up}
\label{sect-16-6-1}

We assume that
\begin{equation}
F\Def \bigl(\sum _{jk} g_{jk}F^jF^k\bigr)^{\frac{1}{2}}
\label{16-6-1}
\end{equation}
where $\mathbf{F}=(F^1,F^2,F^3)$ and $F$ are vector and scalar intensities of the magnetic field respectively. Temporarily we assume that $F=1$ (as we can reduce the general case to this one by dividing operator by $F$ and using our standard arguments).

Further, without any loss of the generality we will assume locally that
\begin{equation}
F^1=F^2=0.
\label{16-6-2}
\end{equation}

\begin{remark}\label{rem-16-6-1}
(i) Due to Frobenius theorem we can make locally $g^{3j}=0$ for $j=1,2$ and $F^1=F^2=0$ simultaneously if and only if $f_*\wedge df_*=0$ where
$f_*= \sum_{j,k} F^jg_{jk}dx^k$ is $1$-form.

\medskip\noindent
(ii) As
\begin{gather}
F=\bigl(\frac{1}{2}\sum_{i,j,k,l} g^{ij}g^{kl}F_{ik}F_{jl}\bigr)^{\frac{1}{2}},
\qquad F_{jk}=\partial_k A_j-\partial_j A_k,\label{16-6-3}\\
\intertext{under assumption (\ref{16-6-2})}
F= \bigl(g^{11}g^{22}-g^{12}g^{21}\bigr)^{\frac{1}{2}}|F_{12}|.
\label{16-6-4}
\end{gather}
\end{remark}

However
\begin{claim}\label{16-6-5}
We can assume  that $F^1=F^2=0$ locally and simultaneously that  $g^{jk}=\updelta_{jk}$ along a single magnetic line $\{x'=0\}$.
\end{claim}

Really, we can satisfy condition $f_*\wedge df_*=0$ changing $g^{jk}$ but preserving them at the chosen magnetic line. Note that $\mathbf{F}$ then acquires some scalar factor which does not affect this condition. So, we can assume $g^{3j}=0$ for $j=1,2$ and $F^1=F^2=0$ at the chosen magnetic line which (after shift) becomes then $\{x'=0\}$. Changing $x_3\mapsto \phi(x_3)$ makes $g^{33}=1$ along this magnetic line. Then $F^3 =\pm F=1$ (for an appropriate orientation). Changing $x'= B(x_3)x'$ with an appropriate matrix $B(x')$ makes $g^{jk}=\updelta_{jk}$ along this line.

Further
\begin{claim}\label{16-6-6}
In frames of assumption (\ref{16-6-5}) we can assume that $g^{31}=0$.
\end{claim}

Really, one can achieve it by $x'\mapsto x'$, $x_3\mapsto \phi(x)$ with
$\phi (x)=x_3+O(|x'|^2)$.

Further, one can make
\begin{equation}
g^{jk}=\updelta_{jk}+O(|x'|^2)\qquad \text{as\ \ } x_3=0 \quad j,k=1,2
\label{16-6-7}
\end{equation}
by $x'\mapsto x'+Q(x')$ with $Q$ quadratic  respect to $x'$.

By a gauge transform one can make  $A_3=0$. Then  $F^1=F^2=0$ imply that $A_1,A_2$ do not depend on $x_3$ and therefore we can assume that $A_1=0$ as well:
\begin{equation}
A_1=A_3=0,\qquad A_2=A_2(x')= x_2 + O(|x'|^3)
\label{16-6-8}
\end{equation}
where the last equality is due to $F=1$ and  assumption (\ref{16-6-5}).

\section{Classical dynamics}
\label{sect-16-6-2}

Let us consider a classical dynamics. We know that for time $|t|\le T=\epsilon_0$ Hamiltonian trajectory starting from
$\{|x|\le c_0\mu^{-1}\}$ remains confined in $C_0\mu^{-1}$-tube
$\{|x'|\le C_0\mu^{-1}\}$. Then in virtue of our assumptions
\begin{gather}
\frac{dx_3}{dt}=2\xi_3 +O(\mu^{-1})
\qquad
\frac{d\xi _3}{dt} = -V_{x_3}(0,0,x_3) +O(\mu^{-1})
\label{16-6-9}\\
\intertext{and therefore}
(x_3,\xi_3)(t)=(x^0_3,\xi^0_3)(t)+O(\mu^{-1}|t|)
\label{16-6-10}
\end{gather}
where here $(x^0_3,\xi^0_3)(t)$ denotes dynamics for a $1\D$-Schr\"odinger $\xi_3^2 +V(0,0,x_3)$ with the same initial data (as $t=0$).

Then one can prove easily that
\begin{multline}
|\xi_3(t)|\le C\bigl(|\beta| \cdot |t|+|\xi_3(0)|\bigr),\qquad
|x_3|\le C\bigl(|\beta| t^2+|\xi_3(0)|\cdot |t|\bigr)\\[2pt]
\text{if\ \ } x_3(0)=0 \text{\ \ and \ \ } |t|\le \epsilon
\label{16-6-11}
\end{multline}
with
\begin{equation}
\beta\Def -V_{x_3}(0).
\label{16-6-12}
\end{equation}
Then
\begin{align}
\xi_3 &= \xi_3(0) \bigl(1+ O(t^2)\bigr) +2\beta t \bigl(1+O(t^2)\bigr),
\label{16-6-13}\\[2pt]
x_3 &= \xi_3(0)t \bigl(1+ O(t^2)\bigr) +\beta t^2 \bigl(1+O(t^2)\bigr).
\label{16-6-14}
\end{align}

Let us analyze evolution of $(x',\xi')(t)$ but we start from $(p_1,p_2)(t)$. Recall that
\begin{equation}
\mu^{-1}\{p_1,p_2\}\equiv 1,\qquad g^{11}g^{22}-g^{12}g^{21}\equiv 1
\mod O(\mu^{-2})
\label{16-6-15}
\end{equation}
in the tube in question.

Let $p'_1\Def p_1+\gamma p_3$; then
$ p'_2\Def (g^{12}p_1 + g^{22}p_2+g^{23}p_3)\equiv -\frac{1}{2}\mu^{-1}\{a,p'_1\}\mod O(\mu^{-1})$.
Then
\begin{multline*}
\frac{1}{2}\mu^{-1}\{a,p'_2\}\equiv (g^{11}p_1+g^{12}p_2+g^{13}p_3) g^{22}-
(g^{12}p_1+g^{22}p_2+g^{23}p_3)g^{12}\equiv\\
p_1 + (g^{13}g^{22}-g^{23}g^{12})p_3.
\end{multline*}
So we pick up $\gamma = g^{13}g^{22}-g^{23}g^{12}$ and then
\begin{gather}
p'_1\Def (g^{22})^{-\frac{1}{2}} (p_1+\gamma p_3),\qquad \gamma=g^{13}g^{22}-g^{23}g^{12},
\label{16-6-16}\\[2pt]
p'_2\Def (g^{22})^{-\frac{1}{2}} (g^{12}p_1 + g^{22}p_2+g^{23}p_3),
\label{16-6-17}\\
\shortintertext{satisfy}
\frac{1}{2}\mu^{-1}\{a,p'_1\}\equiv -p'_2,\qquad \frac{1}{2}\mu^{-1}\{a,p'_2\}\equiv p'_1,\qquad
\mu^{-1}\{p'_1,p'_2\}\equiv 1
\label{16-6-18}
\end{gather}
modulo $O(\mu^{-1})$. We set $p'_3=p_3$.

We do not assume that $g^{13}=0$ identically; instead its choice will be different and it would be only $O(\mu^{-1})$ in the tube in question. Note that \begin{gather}
\gamma \equiv g^{13}g^{22}\mod O(\mu^{-2})\equiv g^{13}\mod O(\mu^{-2}+\mu^{-1}|x_3|),\label{16-6-19}\\[2pt]
\gamma_{x_j} \equiv  g^{13}\mod O(\mu^{-1}).\label{16-6-20}
\end{gather}
Really, $g^{jk}=\updelta_{jk}+O(|x'|^2+|x'|\cdot |x_3|)$ ($j=1,2$), $g^{3j}=O(|x'|)$.

Let us consider again $\{a,p'_1\}$, this time more precisely. Note that
\begin{equation*}
\{a,p'_1\}= -2\mu p'_2 + 2\sum_j \gamma_{x_j}p_jp_3 -V_{x_1}+ O(\mu^{-1})
\qquad\text{as\ \ } x_3=0
\end{equation*}
as $|x'|=O(\mu^{-1})$ due to our assumptions.

Therefore we arrive to the first equation below and the second is proven in the same way
\begin{align}
&\{a,p'_1\}\equiv  -2\mu p'_2+ 2 \sum_j g^{13}_{x_j}p_jp_3 -V_{x_1},
\label{16-6-21}\\
&\{a,p'_2\}\equiv \hphantom{-}2\mu p'_1+ 2\sum_j g^{23}_{,x_j} p_j p_3-V_{x_2}\mod O(\mu^{-1}).
\label{16-6-22}
\end{align}
We want
\begin{equation}
g^{23}_{x_2}=g^{13}_{x_1}=0, \quad
-g^{23}_{x_1}=g^{13}_{x_2}=\kappa (x_3)\qquad \text{as\ \ } x'=0
\label{16-6-23}
\end{equation}
which is possible to arrange as the only value at $x'=0$ which is fixed is \begin{equation}
\kappa(x_3)\Def \frac{1}{2}\bigl(g^{13}_{x_2}-g^{23}_{x_1}\bigr)\bigr|_{x'=0}=
-\frac{1}{2}\bigl(g_{13,x_2}-g_{23,x_1}\bigr)\bigr|_{x'=0};
\label{16-6-24}
\end{equation}
one can prove easily this by transformation $x_3\mapsto x_3+Q(x_3;x')$ with $Q(x_3;x')$ quadratic with respect to $x'$.  Note that
\begin{equation}
f_*\wedge df_*=
2\kappa(x_3) dx_1\wedge dx_2\wedge dx_3\qquad \text{as\ \ } x'=0.
\label{16-6-25}
\end{equation}
Then (\ref{16-6-21})--(\ref{16-6-22}) become
\begin{align}
&\{a,p'_1\}\equiv  -2(\mu + \kappa p_3) p'_2 -V_{x_1},\label{16-6-26}\\
&\{a,p'_2\}\equiv \hphantom{-}2(\mu+\kappa p_3) p'_1-V_{x_2}\qquad
\mod O(\mu^{-1}).\label{16-6-27}
\end{align}
Obviously
\begin{claim}\label{16-6-28}
We can replace in equations (\ref{16-6-26})--(\ref{16-6-27}) $V_{x_j}$ by $V_{x_j}(x_3)\Def V_{x_j}|_{x'=0}$ and $p_3$ by $p^0_3(t)$.
\end{claim}

If $\kappa=0$ and $V_{x_j}$ did not depend on $x_3$ we would get the pilot-model operator. Unfortunately, it is not the case and (\ref{16-6-26}), (\ref{16-6-27}) coincide with those for the pilot-model only  modulo $O(\mu^{-1}+|p_3|+|x_3|)$. However note that
$\kappa (x_3)p_3= \frac{d\ }{dt}K(x_3)$ where
\begin{equation}
K =\frac{1}{2}\int \kappa\,dx_3.
\label{16-6-29}
\end{equation}
Therefore
\begin{multline}
p'(t)\equiv
\Omega(t)\Omega^{-1}(0) p'(0) + \int^t_0 \Omega(t)\Omega^{-1} (t')V' \,dt \\[2pt]
\mod O(\mu^{-1}|t|)
\label{16-6-30}
\end{multline}
with
\begin{equation}
\Omega(t)=\begin{pmatrix} \cos(\mu \theta) & -\sin (\mu \theta)\\
\sin (\mu \theta) &\cos(\mu \theta) \end{pmatrix},\quad
\theta (t)= t+ \mu ^{-1}K(x_3(t))
\label{16-6-31}
\end{equation}
$p=(p_1,p_2)^t$, $V'=(\nabla 'V)^t$, $\Omega(t)$ fundamental matrix of
system  (\ref{16-6-26})--(\ref{16-6-27}). Expressing $\Omega^{-1}(t')$ via its derivative and integrating by parts we conclude that
\begin{equation}
p'(t)\equiv\Omega(t)p'(0) + \mu^{-1} \Omega(t) V'(x_3(0))- \mu^{-1}  V'(x_3(t)).
\label{16-6-32}
\end{equation}

Meanwhile, as $j=1,2$
\begin{equation}
\{a, (-1)^j \mu x_{3-j}+ p_j \} \equiv - \sum_{k,l} g^{kl}_{x_j} p_kp_l-V_{x_j}
\mod O(\mu^{-1})
\label{16-6-33}
\end{equation}
and we can replace in the right-hand expression $x'$ by $0$, $p_k$ by $p'_k$ and as $k=1,2$ we can replace $p'_k$ by the corresponding component of
$\Omega (t)p'(0)$. Then integrating by parts we arrive to
\begin{multline}
\bigr((-1)^j \mu x_{3-j}+ p_j \bigl)|^{t=t}_{t=0}\equiv \\
-\int_0^t
\Bigl( \frac{1}{2}\sum_{k,l}  \bigl(g^{11}+g^{22}\bigr)_{x_j} (p_1^2+p_2^2)+ V_{x_j} (x_3(t))\Bigr)\,dt \equiv  -\int_0^t  V_{x_j}(x_3(t)) \,dt
\label{16-6-34}
\end{multline}
modulo $O(\mu^{-1})$ as $\bigl(g^{11}+g^{22}\bigr)_{x_j}=O(\mu^{-1})$ in virtue of our assumptions including (\ref{16-6-15}). Further, considering out-of integral terms we conclude that

\begin{claim}\label{16-6-35}
Equality (\ref{16-6-34}) holds modulo
$O\bigl(\mu^{-1}(|t|+|\sin (2\mu \theta)+|p_3|\bigr)$ (where
$p_3(t)\equiv p_3(0)\mod |t|$).
\end{claim}

\section{Semiclassical approximation}
\label{sect-16-6-3}

So, after rescaling, the classical dynamics of the general operator is close to one of the \emph{generalized pilot-model}\index{pilot-model!generalized}
\begin{multline}
A^0  \Def h^2 D_1^2 + (hD_2-\mu x_1)^2+h^2 D_3^2+
V^0(x),\\
V^0(x)\Def V(x_3)+\sum_{j=1,2}\alpha_j (x_3) x_j
\label{16-6-36}
\end{multline}
with
\begin{equation}
V(x_3)=V(0,0,x_3),\qquad \alpha_j(x_3)=V_{x_j}(0,0,x_3), \quad j=1,2
\label{16-6-37}
\end{equation}
and the latter one is close albeit with a larger error to  the dynamics of the pilot-model (\ref{16-5-1}) with $V(x_3)$, $\alpha_j(x_3)$  replaced by $V(0)+\beta x_3$ ($\beta=V_{x_3}(0)$), $\alpha_j=\alpha_j(0)$ respectively.

\begin{proposition}\label{prop-16-6-2}(cf.~proposition~\ref{prop-16-2-4}).
Let  conditions \textup{(\ref{16-2-1})}--\textup{(\ref{16-2-3})}, \textup{(\ref{16-6-7})}--\textup{(\ref{16-6-8})} be fulfilled\footref{foot-16-6}.  Then after rescaling

\medskip\noindent
(i) Uniformly with respect to $|t|\le \epsilon\mu$ the propagator $e^{i\hbar^{-1}tA}$ is an $\hbar$-Fourier integral operator corresponding to the Hamiltonian flow $\Psi_t$;

\medskip\noindent
(ii) As $|\sin (2\theta_1)|\ge \epsilon$, $|t|\ge \epsilon_0$
\begin{multline}
U(x,y,t)\equiv \\
\mu (2\pi \hbar )^{-\frac{3}{2}}  t^{-\frac{1}{2}} i
(\sin(\theta_1))^{-1}   e^{i\hbar^{-1}\phi (x,y,t)}
\Bigl(\sum_m b_m (x,y,t)\Bigr)  \hbar^m
\label{16-6-38}
\end{multline}
with $\theta_1$ defined by
\begin{equation}
\theta_1(t) = \theta (t)+ K\bigl(x_3(t)\bigr)-K\bigl(x_3(0)\bigr)
\label{16-6-39}
\end{equation}
with $\theta(t)$, $K(x_3)$ defined by \textup{(\ref{16-2-20})}, \textup{(\ref{16-6-29})} respectively and with $\phi$ defined by \textup{(\ref{16-2-25})}--\textup{(\ref{16-2-29})}  and satisfying (with all derivatives)
\begin{gather}
\phi = \bar{\phi} (\theta) + O(\mu^{-1}t ),\label{16-6-40}\\[2pt]
b_m =\updelta_{0m}+O(\mu^{-1}t )\label{16-6-41}\\
\shortintertext{with}
\bar{\phi}=\beta \mu^{-1} t  (x_3 +y_3) + \frac{1}{8} (x_3-y_3)^2t^{-1} -\frac{2}{3}\beta^2\mu^{-2} t^3 + \bar{\phi}_2
\label{16-6-42}
\end{gather}
with $\bar{\phi}_2$ defined by \textup{(\ref{16-1-10})}.
\end{proposition}

\begin{proof}
Easy but tedious details we leave to the reader: comparing with the pilot-model or generalized pilot model we see that the structure of the canonical graph is close to one of the pilot-model.
\end{proof}

\begin{proposition}\label{prop-16-6-3}(cf.~proposition~\ref{prop-16-2-5}).
Let  conditions \textup{(\ref{16-2-1})}--\textup{(\ref{16-2-3})}, \textup{(\ref{16-6-7})}--\textup{(\ref{16-6-8})} be fulfilled\footref{foot-16-6}. Then after rescaling

\medskip\noindent
(i) Decomposition \textup{(\ref{16-6-39})} remains valid as $|t|\ge \epsilon_0$ and
${|\cos (\theta_1)|\le \epsilon}$;

\medskip\noindent
(ii) Further, after rescaling decomposition \textup{(\ref{16-6-39})} remains valid as $|t|\ge \epsilon_0$ and
\begin{equation}
C\max\bigl(\hbar, \mu^{-2} |t|\bigr) \le |\sin (\theta_1)|\le \epsilon
\label{16-6-43}
\end{equation}
albeit with an error not exceeding
\begin{gather}
C\hbar^{-\frac{3}{2}}|t|^{-\frac{1}{2}} |\sin (\theta_1)|^{-1}
\bigl(\hbar/|\sin (\theta_1)|\bigr)^l\label{16-6-44}\\
\intertext{and with $\phi$, $b_m$ such that}
|D^\beta (\phi-\bar{\phi})|\le
C_\beta \mu^{-1} |t| \alpha^{-1}|\sin (\theta_1)|^{-|\beta|}
\qquad\forall \beta,\label{16-6-45}\\[2pt]
|D^\beta (b_m-\updelta_{m0})  |\le
C_{m\beta} \mu^{-1}|t|  |\sin (\theta_1)|^{-m-|\beta|}
\qquad \forall \beta,m. \label{16-6-46}
\end{gather}
\end{proposition}

\begin{proof}
Again a tedious but an easy proof we leave to the reader.
\end{proof}

\section{Tauberian estimates}
\label{sect-16-6-4}

\begin{remark}\label{rem-16-6-4}
(i) Let us set $x=y$. Then we get expression (\ref{16-6-38}) for $U(y,y,t)$ albeit now due to definition $\theta_1=\theta$.

\medskip\noindent
(ii) Now we can continue in the same manner as before but to get a better estimate one needs to observe that the shift in $x'$ (after rescaling) is now $\mu^{-1}\int \nabla_{x'}V \cdot F^{-1}\,dt$ \ and its difference with $\mu^{-1}\alpha t$ does not exceed  ${C\mu^{-1}|t|\cdot |\osc (\nabla'V)}|$ where $\osc (W)$ means an oscillation of $W$ along trajectory (until time $t$).

However this trajectory must return to $x_3= y_3$ and therefore one can see easily that along it $p_3=O(\beta \mu^{-1} t)$ and
$x_3-y_3=O(\beta \mu^{-2} t^2)$ and we need to assume that this is less than $\epsilon\alpha$ i.e.
\begin{equation}
|t|\le T_* =\epsilon \mu \min(1,\alpha^{\frac{1}{2}}\beta^{-\frac{1}{2}}),\quad
\alpha \Def |\nabla'V(y)|, \ \beta \Def |\partial_{y_3} V(y)|.
\label{16-6-47}
\end{equation}
\end{remark}

Therefore only as $|t|\le T_*$ one can use nondegeneracy condition but it makes sense to consider $t:\ T_*\le |t|\le T^*=\epsilon\mu$. Then we have the following counterpart of proposition~\ref{prop-16-5-2} where we need to sum (\ref{16-5-15}) to $k= T_*$ and compare
$\bar{k}=\mu^2 h\alpha^{-1}$ with $T$: namely, $\bar{k}\le T_*$ iff
\begin{equation}
\alpha \ge
\max \bigl(\mu h, \beta^{\frac{1}{3}}(\mu h)^{\frac{2}{3}},\mu^{-1}\bigr).
\label{16-6-48}
\end{equation}
Setting here $\beta=1$ (as without spatial averaging or micro-averaging $\partial_{\parallel \mathbf{F}}V/F$ is our foe and we need only estimate it from above) and thus $T_*=\epsilon \mu \alpha^{\frac{1}{2}}$ we arrive to estimate (\ref{16-6-50}) below.

On the other hand, if we do not use non-degeneracy condition we need to sum
\begin{equation}
C\mu^{\frac{3}{2}}h^{-\frac{3}{2}} |k|^{-\frac{1}{2}}
\label{16-6-49}
\end{equation}
which results in $C\mu^{\frac{3}{2}}h^{-\frac{3}{2}} T^{\frac{1}{2}}$ and we arrive to estimate (\ref{16-6-51}) below.

\begin{proposition}\label{prop-16-6-5}(cf.~proposition~\ref{prop-16-5-2}).
Let  conditions \textup{(\ref{16-2-1})}--\textup{(\ref{16-2-3})}, \textup{(\ref{16-6-7})}--\textup{(\ref{16-6-8})},
\textup{(\ref{16-0-5})}--\textup{(\ref{16-0-6})} be fulfilled\footref{foot-16-6}. Then

\medskip\noindent
(i) Assume that conditions \textup{(\ref{16-6-47})}--\textup{(\ref{16-6-48})} be fulfilled as well. Then after rescaling  as $T$ is given by \textup{(\ref{16-6-47})} with $\beta=1$
\begin{multline}
|F_{t\to \hbar^{-1}\tau}\bar{\chi}_T(t)\Gamma_x \mathsf{U}|\le
C\mu h^{-2} + \\[2pt]
C\left\{\begin{aligned}
& \mu^{\frac{5}{2}} h^{-1}
\alpha^{-\frac{1}{2}}\bigl(1+ (\log (\alpha /(\mu h)^{\frac{2}{3}}))_+\bigr)
\qquad && \text{as\ \ }
\alpha \ge (\mu h)^{\frac{2}{3}} ,\\
&\mu^{\frac{3}{2}} h^{-\frac{3}{2}}  T  ^{\frac{1}{2}}
\qquad && \text{as\ \ }
\alpha \le (\mu h)^{\frac{2}{3}};
\end{aligned}\right.
\label{16-6-50}
\end{multline}
(ii) In the general case as $|t|\le T^*$
\begin{equation}
|F_{t\to \hbar^{-1}\tau}\bar{\chi}_T(t)\Gamma_x \mathsf{U}|\le
C\mu h^{-2} + C\mu^{\frac{3}{2}} h^{-\frac{3}{2}}  T  ^{\frac{1}{2}}.
\label{16-6-51}
\end{equation}
\end{proposition}

Then picking-up $T=T_*$ and dividing the left-hand expression of (\ref{16-6-50}) by $T_*$ \underline{or}  picking-up $T=T^*$ and dividing  the left-hand expression of (\ref{16-6-51}) by $T^*$ we arrive to estimates
(\ref{16-6-52}), (\ref{16-6-53}) below.

\begin{corollary}\label{cor-16-6-6}
(i) In frames of proposition~\ref{prop-16-6-5}(i) with $\alpha \ge (\mu h)^{\frac{3}{2}}$ (which implies $\mu\le h^{-1}$)
\begin{multline}
\R^\T\Def |e(0,0,\tau)-e^\T(0,0,\tau)|\le\\[2pt]
C  h^{-2}\alpha^{-\frac{1}{2}} +
C\mu^{\frac{3}{2}}h^{-1}\alpha^{-1}\bigl(1+ (\log (\alpha /(\mu h)^{\frac{2}{3}} ))_+\bigr);
\label{16-6-52}
\end{multline}
(ii) In the general case
\begin{equation}
\R^\T \le C  h^{-2} +  C\mu h^{-\frac{3}{2}};
\label{16-6-53}
\end{equation}
(iii) In particular, according to \textup{(\ref{16-6-52})}, as $\mu \le h^{-1}$, $\alpha\asymp 1$
\begin{equation}
\R^\T \le C h^{-2}+ C\mu^{\frac{3}{2}} h^{-1}\bigl(1+ |\log \mu h| \bigr)
\label{16-6-54}
\end{equation}
which is $O(h^{-2})$ as $\mu \le (h|\log h|)^{-\frac{2}{3}}$;

\medskip\noindent
(iv) In particular, according to \textup{(\ref{16-6-53})}, $\R^\T=O(h^{-2})$ as $\mu \le h^{-\frac{1}{2}}$;

\medskip\noindent
(v) Estimate \textup{(\ref{16-6-52})} is better iff
$h^{-\frac{1}{2}}\le \mu \le h^{-1}$ and
\begin{equation}
\alpha \ge \mu^{-2}h^{-1}+ (\mu h)^{\frac{1}{2}}|\log (\mu h)|.
\label{16-6-55}
\end{equation}
\end{corollary}

\section{Weyl estimates}
\label{sect-16-6-5}

The proof of the following statement is rather obvious:

\begin{proposition}\label{prop-16-6-7}(cf.~proposition~\ref{prop-16-5-3}).
Let  conditions \textup{(\ref{16-2-1})}--\textup{(\ref{16-2-3})}, \textup{(\ref{16-6-7})}--\textup{(\ref{16-6-8})},
\textup{(\ref{16-6-47})}--\textup{(\ref{16-6-48})},
\textup{(\ref{16-0-5})}--\textup{(\ref{16-0-6})} be fulfilled\footref{foot-16-6}. Then

\medskip\noindent
(i) As $\mu^2h\le \alpha\le 1$ estimate \textup{(\ref{16-5-18})} holds and correction term $h^{-3}\cN_{x,\corr}$ is delivered by $r$-term stationary phase approximation $h^{-3}\cN_{x,\corr(r)}$ with an error not exceeding
\textup{(\ref{16-5-19})}.

\medskip\noindent
(ii) As $ \alpha\le \mu^2h$ estimate \textup{(\ref{16-5-19})} holds;

\medskip\noindent
(iii) In particular without any nondegeneracy condition
$\R_x^\W= O(\mu^{\frac{3}{2}}h^{-\frac{3}{2}})$  and it is $O(h^{-2})$ as $\mu\le h^{-\frac{1}{3}}$;

\medskip\noindent
(iv) On the other hand,
$\R_x^\W= O(\mu^{-\frac{1}{2}}h^{-\frac{3}{2}}+ \mu^{\frac{5}{2}}h^{-1})$ as $\alpha \asymp 1$ and $\mu \le h^{-\frac{1}{2}}$   and therefore $\R_x^\W=O(h^{-2})$ as $\alpha\asymp 1$ and $\mu\le h^{-\frac{2}{5}}$.
\end{proposition}

Combining with corollary~\ref{cor-16-6-6} we conclude that

\begin{theorem}\label{thm-16-6-8}
Let  conditions \textup{(\ref{16-2-1})}--\textup{(\ref{16-2-3})}, \textup{(\ref{16-6-7})}--\textup{(\ref{16-6-8})},
\textup{(\ref{16-6-48})}, \textup{(\ref{16-0-5})}--\textup{(\ref{16-0-6})} be fulfilled\footref{foot-16-6}. Then

\medskip\noindent
(i) In the general case
\begin{equation}
|e (x,x,0)- h^{-3}\cN_x^\W (x,x,0)|\le Ch^{-2}+  C\mu^{\frac{3}{2}}h^{-\frac{3}{2}};
\label{16-6-56}
\end{equation}
(ii) Under non-degeneracy condition
\begin{equation}
|\nabla_{\perp \mathbf{F}} V/F|\asymp 1
\label{16-6-57}
\end{equation}
as $\mu \le h^{-\frac{1}{2}}$ estimates
\begin{equation}
|e (x,x,0)- h^{-3}\cN_x^\W (x,x,0)|\le Ch^{-2}+ C\mu^{\frac{5}{2}}h^{-1}
\label{16-6-58}
\end{equation}
and
\begin{multline}
|e (x,x,0)- h^{-3}\cN_x^\W (x,x,0)-h^{-3}\cN _{x,\corr(r)}|\le\\
Ch^{-2}+ C\mu^{\frac{3}{2}}h^{-\frac{3}{2}}(\mu^2h)^{r+\frac{1}{2}}
\label{16-6-59}
\end{multline}
hold.
\end{theorem}

\section{Successive approximations}
\label{sect-16-6-6}

Let us try successive approximation method. As an approximation we successively try the generalized pilot-model operator $A^0$, the pilot-model operator $\bar{A}$ and the magnetic Weyl approximation. As for $\mu\le h^{-\frac{1}{3}}$ Weyl approximation delivers $O(h^{-2})$ error we assume that
\begin{equation}
\mu \ge h^{-\frac{1}{3}}.
\label{16-6-60}
\end{equation}

\subsection{Generalized pilot-model approximation, $\mu h\le 1$}
\label{sect-16-6-6-1}

We claim that

\begin{claim}\label{16-6-61}
Effectively if $x=y(=0)$  we can estimate $\|A-A^0\|$ by
\begin{equation}
\zeta(k)\Def
C\bigl(\mu^{-2}|k|+(\mu h)^{\frac{1}{2}}\mu^{-1}|k|^{-\frac{1}{2}}\bigr).
\label{16-6-62}
\end{equation}
\end{claim}
Really, due to our assumptions
$g^{jl}=\updelta_{jl} +O(|x'|^2)+O(|x'|\cdot|x_3|)$ with $|x'|=O(\mu^{-1})$, $x_3=O(\mu^{-1}|k|)$ ($j,l=1,2$), $g^{33}=1+O(\mu^{-2})$ (before rescaling) the only term which does not allow such estimate is $g^{j3}D_jD_3$ with $j=1,2$.

However for a classical trajectory to return back to time $t\asymp k$ we need to have $p_3=O(\mu^{-1}k)$; in our analysis  we will need also to satisfy uncertainty principle: $p_3^2 |k|\ge \mu h$; so we estimate $p_3$ by
$C\mu^{-1}|k| + C (\mu h)^{\frac{1}{2}}|k|^{-\frac{1}{2}}$ and as
$g^{j3}=O(|x'|)$ this leads to (\ref{16-6-62}).

Then  the error in $k$-th winding does not exceed
\begin{multline}
C\mu^{\frac{3}{2}} h^{-\frac{3}{2}} |k|^{-\frac{3}{2}} \times \bigl(\mu^{-2}|k|+(\mu h)^{\frac{1}{2}}\mu^{-1}|k|^{-\frac{1}{2}}\bigr) \times \mu^{-1}h^{-1}|k|\times\\[3pt]
\left\{\begin{aligned}
&\bigl(\frac{\mu^2h}{\alpha |k|}\bigr)^{r+\frac{1}{2}}\qquad &&\text{as\ \ }
\alpha^{-1}\mu^2h\le |k| \le (\mu^3h)^{\frac{1}{2}} ,\\
&1 \qquad &&\text{as\ \ }
|k|\le \min\bigl( \alpha^{-1}\mu^2h, (\mu^3h)^{\frac{1}{2}}\bigr).
\end{aligned}
\right.
\label{16-6-63}
\end{multline}
Really, the successive approximation makes sense only as $\zeta(k)|k|\le \mu h$ which in view of (\ref{16-6-60}) is equivalent to
$|k|\le (\mu^3h)^{\frac{1}{2}}$.

Here as $|k|\ge \alpha^{-1}\mu^2 h$ we apply $r$-term stationary phase approximation as well. Also recall that this interval originally was $\alpha^{-1}\mu ^2h \le |k|\le \epsilon \mu \alpha^{\frac{1}{2}}$~\footnote{\label{foot-16-17} Here the upper limit meant that $|x_3|\le \alpha$ along trajectory before $k$-th winging.  Since we look at the trajectory returning at $k$-th winging,
$|x_3|\le (\mu^{-1}|k|)^2$ before it and we need to satisfy
$(\mu^{-1}|k|)^2\le \alpha$.} but it is non-empty iff
$\alpha \ge (\mu h)^{\frac{2}{3}}$.

Then
$(\mu^3h)^{\frac{1}{2}}\le \mu \alpha^{\frac{1}{2}}$ and further
$(\mu^3h)^{\frac{1}{2}}\ge \alpha^{-1}\mu ^2h$ iff
$\alpha \ge (\mu h)^{\frac{1}{2}}$. Only under this assumption we need to consider the first case in (\ref{16-6-63}) but it is the most important case as it happens for $\alpha\asymp 1$.

Let us break (\ref{16-6-63})  into two expressions: one with the second factor $\mu^{-2}|k|$ and another with the second factor
$(\mu h)^{\frac{1}{2}}\mu^{-1}|k|^{-\frac{1}{2}}$.

In the general case (when we do not know if $\alpha \ge (\mu h)^{\frac{1}{2}}$) summation of the first expression returns $Ch^{-2}(\mu^3 h)^{\frac{1}{4}}$ and of the second expression returns a lesser $Ch^{-2}(1+\log (\mu^3h))$. Contribution of $k:|k|\ge (\mu^3h)^{\frac{1}{2}}$ to the Tauberian expression is a sum of $C\mu^{\frac{3}{2}}h^{-\frac{3}{2}}|k|^{-\frac{3}{2}}$, which gives us
$C\mu^{\frac{3}{2}}h^{-\frac{3}{2}}(\mu^3h)^{-\frac{1}{4}}$ i.e. exactly the same answer. Therefore we arrive to

\begin{claim}\label{16-6-64}
The Tauberian expressions for the original operator and its generalized pilot-model approximation differ by no more than
$Ch^{-2}(\mu^3 h)^{\frac{1}{4}}$.
\end{claim}

Assume now that $\alpha \ge (\mu h)^{\frac{1}{2}}$. Again break (\ref{16-6-63}) into two expressions the same way as before. Obviously summation with respect to $k$ of the second expression returns
$Ch^{-2} \bigl(1+(\log (\mu^2h/\alpha))_+\bigr)$ no matter what $r$ is.

As $r=0$ summation of the first expression with respect to $k$ returns
$C\mu h^{-\frac{3}{2}}\alpha^{-\frac{1}{2}}$.
One can see easily that the contribution of $k:|k|\ge (\mu^3h)^{\frac{1}{2}}$ to the Tauberian expression does not exceed
$C\mu h^{-\frac{3}{2}}\alpha^{-\frac{1}{2}}$.

Therefore we arrive to

\begin{claim}\label{16-6-65}
As $\alpha\ge (\mu h)^{\frac{1}{2}}$ the Tauberian expressions for the original operator and its generalized pilot-model approximation differ by no more than
$Ch^{-2}+ C\mu h^{-\frac{3}{2}}\alpha^{-\frac{1}{2}}$.
\end{claim}

As $r\ge 2$ summation of the first expression with respect to $k$ returns its value ($\times |k|$) as  $k=\max(1, \mu^2 h/\alpha)$:
\begin{equation}
C\mu^{\frac{3}{2}} h^{-1}\alpha^{-\frac{3}{2}} \times
\left\{\begin{aligned}
& \bigl(\frac{\mu^2h}{\alpha }\bigr)^{r-1}\quad &&\text{as\ \ }
1\ge \mu^2h/\alpha,\\
&1  \quad &&\text{as\ \ }  1\le \mu^2h/\alpha
\end{aligned}
\right.
\label{16-6-66}
\end{equation}
and as $r=1$ we get
$C\mu^{\frac{3}{2}} h^{-1}\alpha^{-\frac{3}{2}} (1+ \log (\alpha(\mu h)^{-\frac{1}{2}})$.
One can see easily that the contribution of $k:|k|\ge (\mu^3h)^{\frac{1}{2}}$ to the Tauberian expression does not exceed these expressions as $r\ge 1$.

Therefore we arrive to

\begin{claim}\label{16-6-67}
As $\alpha\ge (\mu h)^{\frac{1}{2}}$ the Tauberian expressions with the subtracted $h^{-3}\cN_{x,\corr(r)}$
i.e.
\begin{equation}
e^\T (x,x, \tau) -h^{-3} \cN_{x,\corr(r)}
\label{16-6-68}
\end{equation}
for the original operator and its generalized pilot-model approximation differ by no more than
\begin{multline}
Ch^{-2}+ Ch^{-2}(\log (\mu^2h/\alpha))_+  +\\[3pt]
C\mu^{\frac{3}{2}} h ^{-1}\alpha^{-\frac{3}{2}}
\left\{\begin{aligned}
& \bigl(1+ \log (\alpha(\mu h)^{-\frac{1}{2}})\bigr) \qquad &&\text{as\ \ } r=1,  \alpha \ge \mu^2h,\\
& \bigl(1+ \log (\mu^3h)\bigr) &&\text{as\ \ }  r=1, \alpha \le \mu^2h,\\
& (\mu^2h/\alpha)^{r-1} &&\text{as\ \ } r\ge 2, \alpha \ge \mu^2h,\\
& 1  &&\text{as\ \ } r\ge 2, \alpha \le \mu^2h.
\end{aligned}\right.
\label{16-6-69}
\end{multline}
\end{claim}

Combining with corollary~\ref{cor-16-6-6} we estimate expressions (\ref{16-6-70})--(\ref{16-6-72}) below for different $\alpha$; in particular we arrive to

\begin{theorem}\label{thm-16-6-9}
Let  conditions \textup{(\ref{16-2-1})}--\textup{(\ref{16-2-3})}, \textup{(\ref{16-6-7})}--\textup{(\ref{16-6-8})}, \textup{(\ref{16-6-48})},
\textup{(\ref{16-0-5})}--\textup{(\ref{16-0-6})} be fulfilled\footref{foot-16-6}. Let $e^0(.,.,.)$ be a Schwartz kernel of the spectral projector for operator $A^0$ defined by \textup{(\ref{16-6-36})}--\textup{(\ref{16-6-37})}. Then as $h^{-\frac{1}{3}}\le \mu\le h^{-1}$

\medskip\noindent
(i) In the general case
\begin{equation}
|e (x,x,0)- e^0(x,x,0)|\le  C\mu^{\frac{3}{4}}h^{-\frac{7}{4}};
\label{16-6-70}
\end{equation}
(ii) Under non-degeneracy condition \textup{(\ref{16-6-57})} estimates
\begin{equation}
|e (x,x,0)- e^0(x,x,0)|\le Ch^{-2}+ C\mu h^{-\frac{3}{2}}
\label{16-6-71}
\end{equation}
and
\begin{multline}
|e (x,x,0)- e^0(x,x,0)-h^{-3}\bigl(\cN _{x,\corr(r)}-\cN^0_{x,\corr(r)}\bigr)|
\le \\[2pt]
Ch^{-2}\bigl(1+(\log \mu^2h)_+\bigr)+
C\mu^{\frac{3}{2}}h^{-1}\bigl(1+|\log \mu h|\bigr)+\\[2pt]
C\mu^{\frac{3}{2}} h ^{-1}
\left\{\begin{aligned}
& \bigl(1+ \log (\mu^3h)\bigr) &&\text{as\ \ }  r=1,\\
& (\mu^2h)^{r-1} &&\text{as\ \ } r\ge 2, \mu\le h^{\frac{1}{2}},\\
& 1  &&\text{as\ \ } r\ge 2, \mu \ge h^{-\frac{1}{2}}
\end{aligned}\right.
\label{16-6-72}
\end{multline}
hold.
\end{theorem}

\begin{remark}\label{rem-16-6-10}
Replacing $\alpha(x_3)x_j$ by $\alpha(0)x_j$ we estimate the norm of the perturbation  by $\zeta_1(k)= \mu^{-1}\cdot (\mu^{-2}|k|^2+ \mu^{-1})$ which is less than expression (\ref{16-6-62}) and thus leads to the lesser error. On the other hand, this new approximation is already a direct sum $A=A_{(2)}+B$ with the pilot model operator $A_{(2)}$ and
\begin{equation}
B= h^2 D_3^2 + V^0(x_3).
\label{16-6-73}
\end{equation}

Therefore all above estimate remain valid with $A^0$ defined by (\ref{16-6-36})--(\ref{16-6-37}) albeit with $\alpha_j=V_{x_j}(0,0,0)$.
\end{remark}

\subsection{Pilot-model approximation, $\mu h\le 1$}
\label{sect-16-6-6-2}

Consider now the pilot-model approximation $\bar{A}$ which differ from the approximation $A^0$ of remark~\ref{rem-16-6-10} by $O(x_3^2)$. Then

\begin{claim}\label{16-6-74}
Effectively if $x=y$ we can estimate $\|A^0-\bar{A}\|$ by
\begin{equation}
\zeta(k)\Def C\bigl(\mu^{-4}k^4+\mu^{-1}h|k| \bigr).
\label{16-6-75}
\end{equation}
\end{claim}

Really, we need to estimate $O(x_3^2)$ and according to above arguments we estimate $|x_3|$ by
$C\bigl(\mu^{-2}k^2 + (\mu h)^{\frac{1}{2}}|k|^{-\frac{1}{2}}\times \mu^{-1}|k|\bigr)$ which leads to (\ref{16-6-74}).

Then we have estimate of the error in $k$-th winding not exceeding
\begin{multline}
C\mu^{\frac{3}{2}} h^{-\frac{3}{2}} |k|^{-\frac{3}{2}} \times \bigl(\mu^{-4}k^4+\mu^{-1}h|k| \bigr)\times \mu^{-1}h^{-1}|k|\times\\[3pt]
\left\{\begin{aligned}
&\bigl(\frac{\mu^2h}{\alpha |k|}\bigr)^{r+\frac{1}{2}}\qquad &&\text{as\ \ }
|k|\ge \alpha^{-1}\mu^2h,\\
&1 \qquad &&\text{as\ \ }  |k|\le  \alpha^{-1}\mu^2h.
\end{aligned}
\right.
\label{16-6-76}
\end{multline}

The perturbation factor is $O(1)$ iff $|k|\lesssim \mu h^{\frac{1}{5}}$ and we need to compare it with $ \alpha^{-1}\mu^2h $: we have cases
$\alpha \ge \mu h^{\frac{4}{5}}$ when $ \alpha^{-1} \mu^2h$ is less and
$\alpha \le \mu h^{\frac{4}{5}}$ when $\mu h^{\frac{4}{5}}$ is less.

\paragraph{Case $\alpha \ge \mu h^{\frac{4}{5}}$.}
\label{sect-16-6-6-2-1}
Then we need to sum
\begin{enumerate}[leftmargin=*,label=(\alph*)]
\item Expression (\ref{16-6-76}) with the last factor $1$ from $|k|=1$ to $|k|= \alpha^{-1} \mu^2h$ which returns
$C\mu^{\frac{1}{2}} h^{-\frac{5}{2}} |k|^{\frac{1}{2}} \times \bigl(\mu^{-4}k^4+\mu^{-1}h|k| \bigr)$ calculated as $|k|= \alpha^{-1}\mu^2 h$:
\begin{equation*}
C\mu^{\frac{11}{2}} h^{2} \alpha^{-\frac{9}{2}} +
C\mu^{\frac{5}{2}} \alpha^{-\frac{3}{2}};
\end{equation*}
as $\alpha\asymp 1$ the first term does not exceed $C\mu^{\frac{3}{2}}h^{-1}$ iff $\mu \le h^{-\frac{3}{4}}$ while for the second term it is always true;

\item  Expression (\ref{16-6-76}) with the last factor
$\bigl(\mu^2h/\alpha |k| \bigr)^{r+\frac{1}{2}}$ from  $|k|= \alpha^{-1}\mu^2h $ to $k=\mu h^{\frac{1}{5}}$ which returns what we got earlier plus value of
$C\mu^{\frac{1}{2}} h^{-\frac{5}{2}} |k|^{\frac{1}{2}} \times \bigl(\mu^{-4}k^4+\mu^{-1}h|k| \bigr)\times (\mu^2h/\alpha |k|)^{r+\frac{1}{2}}$ calculated at $k= \mu h^{\frac{1}{5}}$ (may be with the logarithmic term for exceptional values of $r=4,1$) i.e.
\begin{equation*}
C\mu  h^{-\frac{8}{5}} (\mu h^{\frac{4}{5}}/\alpha )^{r+\frac{1}{2}}+
C\mu h^{-\frac{6}{5}}(\mu h^{\frac{4}{5}}/\alpha )^{r+\frac{1}{2}};
\end{equation*}
as $r=1$ we get $C\mu^{\frac{5}{2}}h^{-\frac{2}{5}}\alpha^{-\frac{3}{2}}$;
as $\alpha\asymp 1$ this is less than $Ch^{-2}$ as $\mu\le h^{-\frac{16}{25}}$;
note that the second term always is $O(\mu^{\frac{3}{2}}h^{-1}+h^{-2})$;

\item $C\mu^{\frac{3}{2}}h^{\frac{3}{2}}
|k|^{-\frac{3}{2}}(\mu^2h/\alpha)^{r+\frac{1}{2}}$ as $|k|\ge C\mu h^{\frac{4}{5}}$ which returns the  expression above.
\end{enumerate}

In total, we get modulo $O\bigl(h^{-2}+\mu ^{\frac{3}{2}}h^{-1}\bigr)$
\begin{equation}
C\mu^{\frac{11}{2}} h^{2} \alpha^{-\frac{9}{2}} +
C\mu^{\frac{5}{2}} \alpha^{-\frac{3}{2}}+
C\mu  h^{-\frac{8}{5}} (\mu h^{\frac{4}{5}}/\alpha )^{r+\frac{1}{2}}.
\label{16-6-77}
\end{equation}

\paragraph{Case $\alpha \le \mu h^{\frac{4}{5}}$.}
\label{sect-16-6-6-2-2}

Then we need to sum
\begin{enumerate}[leftmargin=*,label=(\alph*)]
\item Expression (\ref{16-6-76}) with the last factor $1$ from $|k|=1$ to
$|k|=\mu h^{\frac{1}{5}}$ which returns
$C\mu^{\frac{1}{2}} h^{-\frac{5}{2}} |k|^{\frac{1}{2}} \times \bigl(\mu^{-4}k^4+\mu^{-1}h|k| \bigr)$ calculated at $k= \mu h^{\frac{1}{5}}$ i.e.
\begin{equation*}
C\mu  h^{-\frac{8}{5}} + C\mu h^{-\frac{6}{5}};
\end{equation*}
here the first term is larger than $C\mu^{\frac{3}{2}}h^{\frac{3}{2}}+Ch^{-2}$  as $\mu\ge h^{-\frac{2}{5}}$ but the second term is always smaller than this;

\item $C\mu^{\frac{3}{2}}h^{\frac{3}{2}}|k|^{-\frac{3}{2}}$ as
$|k|\ge C\mu h^{\frac{4}{5}}$ which returns an above expression.
\end{enumerate}

So, we get in this case
\begin{equation}
C\mu h^{-\frac{8}{5}}.
\label{16-6-78}
\end{equation}

Thus we arrive to

\begin{theorem}\label{thm-16-6-11}
Let  conditions \textup{(\ref{16-2-1})}--\textup{(\ref{16-2-3})}, \textup{(\ref{16-6-7})}--\textup{(\ref{16-6-8})},
\textup{(\ref{16-6-48})}, \textup{(\ref{16-0-5})}--\textup{(\ref{16-0-6})} be fulfilled\footref{foot-16-6}. Let $\bar{e}(.,.,.)$ be a Schwartz kernel of the spectral projector for the pilot-model operator $\bar{A}$. Then as
$h^{-\frac{1}{3}}\le \mu\le h^{-1}$

\medskip\noindent
(i) In the general case
\begin{equation}
|e (x,x,0)- \bar{e}(x,x,0)|\le  C\mu^{\frac{3}{4}}h^{-\frac{7}{4}} +
C\mu h^{-\frac{8}{5}};
\label{16-6-79}
\end{equation}
(ii) Under non-degeneracy condition \textup{(\ref{16-6-57})} as
$\mu \le h^{-\frac{4}{5}}$ estimates
\begin{equation}
|e (x,x,0)- \bar{e}(x,x,0)|\le Ch^{-2}+
C\mu h^{-\frac{3}{2}} +
C\mu^{\frac{11}{2}} h^{2}  +
C\mu^{\frac{5}{2}} +
C\mu ^{\frac{3}{2}} h^{-\frac{6}{5}}
\label{16-6-80}
\end{equation}
and
\begin{multline}
|e (x,x,0)- \bar{e}(x,x,0)-
h^{-3}\bigl(\cN _{x,\corr(r)}-\cN^0_{x,\corr(r)}\bigr)|
\le \\[2pt]
Ch^{-2}\bigl(1+(\log \mu^2h)_+\bigr)+
C\mu^{\frac{3}{2}}h^{-1}\bigl(1+|\log \mu h|\bigr)+\\[2pt]
C\mu^{\frac{3}{2}} h ^{-1}
\left\{\begin{aligned}
& \bigl(1+ \log (\mu^3h)\bigr) &&\text{as\ \ }  r=1,\\
& (\mu^2h)^{r-1} &&\text{as\ \ } r\ge 2, \mu \le h^{\frac{1}{2}},\\
& 1  &&\text{as\ \ } r\ge 2, \mu \ge h^{-\frac{1}{2}}
\end{aligned}\right. \quad +\\[2pt]
C\mu^{\frac{11}{2}} h^{2}  + C\mu^{\frac{5}{2}}
C\mu  h^{-\frac{8}{5}} (\mu h^{\frac{4}{5}} )^{r+\frac{1}{2}}.
\label{16-6-81}
\end{multline}
hold.
\end{theorem}

Let us improve the above results. Note that we need only consider components of the estimate which are (in unrescaled $x_3,p_3$) due to $p_3= O(\mu^{-1}k)$, $x_3=O(\mu^{-2}k^2)$ rather than those which are due to the uncertainty principle
$p_3= O((\mu h)^{\frac{1}{2}}|k|^{-\frac{1}{2}})$,
$x_3=O(\mu^{-1}(\mu h)^{\frac{1}{2}}|k|^{\frac{1}{2}})$ as the latter brought a proper estimate in the above analysis.

\medskip\noindent
Consider only case $\alpha \asymp 1$, $\beta\lesssim 1$. Then we need to consider only $\mu \ge h^{-\frac{16}{25}}$. Note that actually instead of $p_3=O(t)$, $x_3=O(t^2)$ we can use estimates $p_3=O(\beta t)$,
$x_3=O(\beta t^2)$ and then our estimates acquire factor $\beta^2$ (recall that we are not discussing terms which are due to the uncertainty principle).

\paragraph{Case $h^{-16/25}\le \mu \le h^{-4/5}$.}
\label{sect-16-6-6-2-3}

Consider (\ref{16-6-77}) with $r=1$; then the last term is the largest and the estimate with above improvement is
$O(\beta^2 \mu^{5/2}h^{-2/5})$ and it is $O(\mu ^{3/2}h^{-1})$ as
$\beta\le (\mu h^{\frac{3}{5}})^{-\frac{1}{2}}$. So we need to consider case
\begin{equation}
\beta \ge (\mu h^{\frac{3}{5}})^{-\frac{1}{2}}\qquad (\ge h^{1/10}).
\label{16-6-82}
\end{equation}
As $r=0$ we get $\mu ^{\frac{3}{2}}h^{-\frac{6}{5}}$ which is less than $\mu h^{-\frac{3}{2}}$ in (\ref{16-6-71});

\paragraph{Case $h^{-4/5}\le \mu \le h^{-1}$.}

Then estimate is (modulo $O(\mu ^{-3/2}h^{-1})$)
$O(\beta^2 \mu h^{-8/5})$ and it is $O(\mu ^{3/2}h^{-1})$ as
$\beta\le (\mu h^{\frac{6}{5}})^{\frac{1}{4}}$. So we need to consider case
\begin{equation}
\beta \ge (\mu h^{\frac{6}{5}})^{\frac{1}{4}}\qquad (\ge h^{1/10}).
\label{16-6-83}
\end{equation}

Therefore we conclude that

\begin{theorem}\label{thm-16-6-12}
As $\beta \Def |\nabla_{\parallel \mathbf{F}} V/F|\le h^{-1/10}$ estimates \textup{(\ref{16-6-70})}--\textup{(\ref{16-6-72})} of theorem~\ref{thm-16-6-9} remain true for $A^0$ replaced by $\bar{A}$.
\end{theorem}

So we need to consider the case when  $1\D$-Schr\"odinger operator is almost non-degenerate and the corresponding dynamics is confined to  zone
$x_3\le \beta$; rescaling
$x_3\mapsto x_{3,\new}=x_3/\beta$ we arrive to operator
\begin{equation}
\beta^2 \Bigl(h_3^2 D_3^2 + V_3( x_3)\Bigr),\qquad h_3=h/\beta^2\ll 1, \quad
V_3(x_3)=\beta^{-2}V(x_3)
\label{16-6-84}
\end{equation}
satisfies $\partial_3 V_3(x_3)\asymp 1$. Then we can construct its solution by WKB method. Note that before any rescaling
\begin{multline}
U_{(1)}(x_3,y_3,t)=\\
(2\pi h)^{-1}\int \exp \Bigl(ih^{-1}\bigl(S(x_3,t,\zeta)-y_3 \zeta )\Bigr)B(x_3,t,\zeta)\,d\zeta
\label{16-6-85}
\end{multline}
where $S$ solves
\begin{equation}
S_t = S_{x_3}^2 +V(x),\qquad S|_{t=0}=x_3\zeta
\label{16-6-86}
\end{equation}
and $B\sim \sum_k B_k h^k$ in the standard way.

Then one can prove easily that
\begin{gather}
S=\bar{S} + O\Bigl(x_3 ^2|t|+ |x_3|(|\zeta|+|t|)|t|+\zeta^2 |t|^3 +|t|^5\Bigr)
\label{16-6-87}\\
\shortintertext{where}
\bar{S}=x_3\zeta+t(\zeta^2-\beta x_3)+\zeta \beta t^2 +\frac{2}{3}\beta^2t^3
\label{16-6-88}
\end{gather}
solves the same problem for pilot-model $V=-\beta x_3$. Rescaling
$x_3\to x_3/\beta$, $h\mapsto h/\beta^2$ shows one can calculate $S$ and $\bar{S}$ for $\beta=0$, then plug $x_3=x_3/\beta$, $\zeta=\zeta/\beta$ and multiply the result by $\beta^2$; (\ref{16-6-87}) obviously survives but (\ref{16-6-88}) improves to
\begin{gather}
S=\bar{S} +
O\Bigl(x_3 ^2|t|+ |x_3|(|\zeta|+\beta |t|)|t|+\zeta^2 |t|^3 +|t|^5\Bigr).
\tag*{$\textup{(\ref*{16-6-87})}^*$}\label{16-6-87-*}
\end{gather}
Obviously phase function
\begin{equation}
\bar{\phi}(x_3,y_3,t,\zeta)\Def \bar{S}(x_3,t,\zeta)-y_3\zeta
\label{16-6-89}
\end{equation}
is equivalent to phase function
\begin{equation}
\bar{\phi}\Def (\zeta +\beta t)x_3 - (\zeta -\beta t)y_3 -2\zeta^2 t -\frac{2}{3}\beta^2t^3
\label{16-6-90}
\end{equation}
we used earlier.

Furthermore, obviously

\begin{claim}\label{16-6-91}
Phase function
\begin{equation}
\phi(x_3,y_3,t,\zeta)\Def S(x_3,t,\zeta)-y_3\zeta
\label{16-6-92}
\end{equation}
is $1\D$-action and therefore the total exponent is also $3\D$-action associated with the generalized pilot-model.
\end{claim}

Note that for a phase function (\ref{16-6-89}) $\bar{\phi}_\zeta =0$ iff
$x_3 -y_3 + 2t\zeta + \beta t^2=0$ and then $\bar{\phi}_{\zeta\zeta}= -4t$ and therefore for $|t|\gg h$ we can apply a stationary phase method for both $S$ and $\bar{S}$ and while the principal term would be of magnitude $(h|t|)^{-\frac{1}{2}}$ the error in $l$-term approximation will be
of magnitude $(h|t|)^{-\frac{1}{2}} (h/|t|)^l$ or after all the substitution to $\mathsf{U}$ and rescalings it will be
\begin{equation}
C\mu^{\frac{3}{2}}h^{-\frac{3}{2}}|k|^{-\frac{3}{2}} \times (\mu h/|k|)^l
\label{16-6-93}
\end{equation}
in the Tauberian expression. While almost useless for $\mu h$ close to $1$ and $k=1$ the last factor is very important for larger $k$.

Then the contribution of zone $\{k:|k|\ge m\}$ to such error does not exceed (as $l=1$)
\begin{equation}
C\mu^{\frac{5}{2}}h^{-\frac{1}{2}}m ^{-\frac{3}{2}}
\label{16-6-94}
\end{equation}
and it is $O(\mu^{-\frac{3}{2}}h^{-1})$ as
$m\ge \mu ^{\frac{2}{3}}h^{\frac{1}{3}}$. Obviously $m\ll \mu h^{\frac{1}{5}}$ and $m\ll \mu^2h$ which means that in zone $k:|k|\le m$ we could perfectly deal with successive approximations.

On the other hand, replacing $B$ by $1$ and $\phi_{\zeta\zeta}$ by $-4t$ leads to an error
\begin{equation}
C\mu^{\frac{3}{2}}h^{-\frac{3}{2}}|k|^{-\frac{3}{2}} \times \mu^{-2}k^2\times
\min \bigl( (\mu^2h/|k|)^{\frac{1}{2}}, 1\bigr)
\label{16-6-95}
\end{equation}
which after summation becomes $\mu^{\frac{1}{2}}h^{-1}\times \mu$  which is not as good as  the Tauberian estimate. This spares us from more complicated formula. We leave to the reader

\begin{Problem}\label{Problem-16-6-13}
(i) Write the approximation due to described combination of successive approximations and WKB method.

\medskip\noindent
(ii) Consider $\alpha \le 1$.
\end{Problem}

\subsection{Magnetic Weyl approximation}
\label{sect-16-6-6-3}

In virtue of the above results the result of subsubsection~\ref{-16-5-6-1} hold for general operators.
\section{Strong and super-strong magnetic field}
\label{sect-16-6-7}

The standard method of successive approximations shows that replacing operator by the generalized pilot-model brings an error with the contribution of $\{t:|t|\asymp T\}$  (before rescaling) not exceeding
\begin{equation}
C\mu h^{-2}\times
\bigl(\mu^{-1}h + C\mu^{-\frac{1}{2}}h^{\frac{1}{2}}(T^{\frac{1}{2}} + h^\delta) \bigr)
\label{16-6-96}
\end{equation}
and summation with respect to $T$ running from $h$ to $1$ returns
\begin{equation*}
C\mu h^{-2}\times
\bigl(\mu^{-1}h +C\mu^{-\frac{1}{2}}h^{\frac{1}{2}}(1+h^\delta|\log h| )\bigr)
\asymp
C\mu^{\frac{1}{2}}h^{-\frac{3}{2}} + C\mu^{\frac{1}{2}}h^{\delta-\frac{3}{2}}|\log \mu|
\end{equation*}
which is less than the Tauberian estimate. Therefore

\begin{theorem}\label{thm-16-6-14}
Let  conditions \textup{(\ref{16-2-1})}--\textup{(\ref{16-2-3})}, \textup{(\ref{16-6-7})}--\textup{(\ref{16-6-8})},
\textup{(\ref{16-6-48})}, \textup{(\ref{16-0-5})}--\textup{(\ref{16-0-6})} be fulfilled\footref{foot-16-6}. Then as $\mu h \gtrsim 1$
\begin{equation}
|e(x,x,0)-e^0(x,x,0)|\le C \mu h^{-\frac{3}{2}}.
\label{16-6-97}
\end{equation}
\end{theorem}

Similarly, transition to the pilot model brings an error not exceeding
\begin{equation}
C\mu h^{-2}\times \bigl(\beta^2 T^4+ h^{\frac{ 4}{3}}\bigr)
\label{16-6-98}
\end{equation}
and summation with respect from $h$ to $T_*$ returns
$C\mu h^{-\frac{3}{2}} + C\mu h^{-2} \beta^2 T_*^4$. On the other hand, Tauberian approach shows that the contribution of $|t|\ge T_*$  should not exceed $C\mu h^{-\frac{3}{2}}T_* ^{-1}$ and minimizing it sum we arrive to
$T_*=\beta^{-\frac{2}{5}}h^{\frac{1}{10}}$ and it is
\begin{equation}
C\mu h^{-\frac{3}{2}} \Bigl( 1+ \beta^{\frac{2}{5}}h^{-\frac{1}{10}}\Bigr).
\label{16-6-99}
\end{equation}
Therefore

\begin{theorem}\label{thm-16-6-15}
Let  conditions \textup{(\ref{16-2-1})}--\textup{(\ref{16-2-3})}, \textup{(\ref{16-6-7})}--\textup{(\ref{16-6-8})},
\textup{(\ref{16-6-48})}, \textup{(\ref{16-0-5})}--\textup{(\ref{16-0-6})} be fulfilled\footref{foot-16-6}. Then as $\mu h \gtrsim 1$
\begin{equation}
|e(x,x,0)-\bar{e}(x,x,0)|\le C\mu h^{-\frac{8}{5}}.
\label{16-6-100}
\end{equation}
\end{theorem}

Further, if we pick up $T_*=\min( \beta^{-\frac{1}{2}}h^{\frac{1}{8}},1)$ we conclude that the contribution of zone $\{t:T_*\le |t|\le \epsilon\}$ to an error does not exceed $C\mu h^{-\frac{3}{2}}$.  So we need to consider zone $\{t:|t|\ge h^{\frac{1}{8}}\}$. However in this zone we can apply stationary phase expression for $U_{(31)}$ which gives an error
\begin{equation*}
C\int_T^1 \mu h^{-2}\times \bigl(\frac{h}{t}\bigr)^{\frac{3}{2}}\,dt \asymp
C\mu h^{-2} \bigl(\frac{h}{T}\bigr)^{\frac{3}{2}}
\end{equation*}
which is $O(\mu h^{-\frac{3}{2}})$ as $T\ge h^{\frac{2}{3}}$.

\begin{Problem}\label{Problem-16-6-16}
Construct this stationary phase approximation.
\end{Problem}

Again for the magnetic Weyl approximation error we refer to subsubsection~\ref{-16-5-2}.

\section{Micro-averaging}
\label{sect-16-6-8}

We are interested only in Tauberian, Weyl and magnetic Weyl estimates as we consider pilot-model and generalized pilot-model approximations a bit too complicated. We consider here only isotropic micro-averaging (i.e. with $\gamma_3=\gamma$).

\subsection{Tauberian estimates}
\label{sect-16-6-8-1}

Basically everything remains as in section \ref{-16-5-4} but we need to answer a question: ``what is $\ell$?'' Or rather ``what $T^*$ we need to take?'' Clearly $T^*=\epsilon \min (\alpha,\beta)$ fits but can we do better than this?

\medskip\noindent
(i) Let $\beta \le \alpha$. Then as we take $\gamma_3=\gamma$ we have $\nu(\boldgamma)\asymp \alpha \gamma$ so averaging with respect to $x_3$ does not matter and therefore we do not need to assume that
$\gamma_3\le \epsilon \beta$ (which in this case would mean
$\gamma\le\epsilon \beta$). In this case our assumption is
$\gamma \le\epsilon \alpha$.

Then as $\nabla _{\perp \mathbf{F}} V/F$ has a less oscillation then $\epsilon\alpha$ with $\alpha=|\nabla _{\perp \mathbf{F}} V/F|\bigr|_{t=0}$ for time $T^*=\epsilon _1 \alpha$ we can take $T^*=\epsilon _1 \alpha$.

\medskip\noindent
(ii) Let $\beta \ge \alpha$. Then as we take $\gamma_3=\gamma$ we have $\nu(\boldgamma)\asymp \beta \gamma$ so averaging with respect to $x'$ does not matter as well as the shift and therefore we do not need to assume that $\gamma\le \epsilon \alpha$ (which in this case would mean
$\gamma_3\le \epsilon\beta$).

Then as $\nabla _{\parallel \mathbf{F}} V/F$ has a less oscillation then $\epsilon\beta$ with $\beta =|\nabla _{\parallel \mathbf{F}} V/F|\bigr|_{t=0}$ for time $T^*=\epsilon _1 \beta$ we can take $T^*=\epsilon _1 \beta$. However in this case we should not include factor $(\mu h/\alpha |k|)^{\frac{1}{2}}$
which is a loss as $\mu h/\alpha < h/\beta\gamma $ i.e. as
$\mu \gamma \le \alpha\beta $.

\medskip

So, in both cases we can take $\gamma_3=\gamma$ and $T^*\asymp |\nabla V/F|$ (but we obviously need $\gamma \le |\nabla V/F|$).

\subsection{Weyl and magnetic Weyl estimates}
\label{sect-16-6-8-2}

As we do not care about $T^*$ anymore we just take
$\gamma_3=\gamma \le |\nabla V/F|$.

\begin{Problem}\label{Problem-16-6-17}
As $\beta$ is rather an obstacle in the pilot-model approximation, it would be interesting investigate this approximation with micro-averaging.
\end{Problem}

\chapter{Dirac energy: $3\D$-estimates}
\label{sect-16-7}

In this and the next sections we consider asymptotics of $\I$ defined by (\ref{16-0-2}).

\section{Tauberian formula}
\label{sect-16-7-1}

Let us consider first contribution of zone $\{(x,y):\ |x-y|\ge C\gamma\}$.

\begin{proposition}\label{prop-16-7-1} (cf. proposition~\ref{prop-16-3-1}).
Let $\mu \le h^{-1}$. Then under conditions \textup{(\ref{16-0-5})}--\textup{(\ref{16-0-7})} the contribution of zone $\{(x,y):\ |x-y|\ge C\gamma\}$, to the remainder is $O(h^{-2}\gamma ^{-\kappa})$ while the main part is given by the same expression \textup{(\ref{16-0-1})} with $e(x,y,0)$ replaced by its standard implicit Tauberian approximation with $T\asymp \epsilon  $ \textup{(\ref{16-0-11})}.
\end{proposition}

\begin{proof}
Proof repeats one of proposition~\ref{prop-16-3-1} albeit according to Chapter~\ref{book_new-sect-13} estimate (\ref{16-3-4}) is replaced by
\begin{multline}
\| E(\tau, \tau' ) \bar{\psi} \|_1\le Ch^{-3}\bigl(|\tau -\tau'|+ CT^{-1}h\bigr) \\
\forall \tau ,\tau' \in [-\epsilon,\epsilon], \ T=\epsilon \mu
\label{16-7-1}
\end{multline}
with $T=\epsilon$.
\end{proof}

\begin{proposition}\label{prop-16-7-2} (cf. proposition~\ref{prop-16-3-3}).
Let conditions \textup{(\ref{16-0-5})}--\textup{(\ref{16-0-7})} be fulfilled. Then

\medskip\noindent
(i) As $0<\kappa < 3$ and \underline{either} $\kappa\ne 1,2$ \underline{or} $\kappa=1,2$ and $\omega(x,y)$ is replaced by
$\omega(x,y)-\varkappa (\frac{1}{2}(x+y)) |x-y|^{-\kappa}$ with an appropriate smooth coefficient $\varkappa(x)$, with the error $O(h^{-2-\kappa})$ one can replace $e(x,y,\tau)$ by its standard Tauberian expression \textup{(\ref{16-0-11})} in the formula \textup{(\ref{16-0-11})} for $\I$.

\medskip\noindent
(ii) As $\kappa= 1,2$ and $\omega=\varkappa (\frac{1}{2}(x+y)) |x-y|^{-\kappa}$, with the error $O(h^{-2-\kappa}|\log h|)$ one can replace $e(x,y,\tau)$ by its standard Tauberian expression \textup{(\ref{16-0-11})} in the formula \textup{(\ref{16-0-11})} for $\I$.
\end{proposition}

\begin{proposition}\label{prop-16-7-3}(cf. proposition~\ref{prop-16-3-6}).
Let conditions \textup{(\ref{16-0-5})} and  \textup{(\ref{16-0-6})} be fulfilled.

\medskip\noindent
(i) Further, let \underline{either} condition  \textup{(\ref{16-0-8})}\,\footnote{\label{foot-16-18} Which actually could be replaced by much weaker non-degeneracy condition of Chapter~\ref{book_new-sect-13} for $d=3$.}  be fulfilled \underline{or}  $\mu \le h^{\delta-1}|\log h|^{-1}$ with an arbitrarily small exponent.  Then \textup{(\ref{16-0-11})} and statements (i), (ii) of proposition \ref{prop-16-7-2} hold.

\medskip\noindent
(ii) Furthermore,  let  $h^{-1-\delta}\mu \le  h^{-1}$.  Then \textup{(\ref{16-0-11})} and statements (i), (ii) of proposition \ref{prop-16-7-2} hold  with an extra factor $(1+\mu h ^{1-\delta})$ in the right-hand expressions.
\end{proposition}

\section{Superstrong magnetic field}
\label{sect-16-7-2}

Consider Schr\"odinger-Pauli operator as $\mu h\ge \epsilon_0$ which is a bit more tricky:

\begin{proposition}\label{prop-16-7-4}(cf. proposition~\ref{prop-16-3-8}).
Let $\mu h\ge \epsilon_0$. Then

\medskip\noindent
(i) Contribution of zone $\{(x,y):\ |x-y|\ge \gamma\}$ to $\I$ does not exceed
$C\mu h^{-2}\gamma^{-\kappa}$;

\medskip\noindent
(ii) Further,
\begin{equation}
|\I|\le C\mu h^{-2}
\left\{\begin{aligned}
&h^{-\kappa}\qquad&&\text{as\ \ }0<\kappa<1,\\
&h^{-1}\bigl(1+(\log \mu h)_+\bigr)\qquad&&\text{as\ \ }\kappa=1,\\
&\mu ^{\frac{1}{2} (\kappa-1)}h ^{-\frac{1}{2}(\kappa+1)}\qquad
&&\text{as\ \ }1<\kappa<3.
\end{aligned}\right.
\label{16-7-2}
\end{equation}

\medskip\noindent
(iii)  Furthermore,  $\I=O(\mu h^{\infty})$ under condition \textup{(\ref{16-3-14})};
in particular it is the case as $\fz<1$ and $\mu h\ge C_0$.
\end{proposition}

\begin{proof}
Statement (i) trivially follows from the fact that $\sL^2$-norm of $e(.,.,\tau)$ does not exceed $C\mu^{\frac{1}{2}} h^{-1}$.

Meanwhile estimate of the contribution of the zone
\begin{equation*}
\{(x,y): x\in B(0,1) , y\in B(0,1), |x-y|\le \gamma\}
\end{equation*}
is more subtle: if $0<\kappa<1$ \ $\omega = \partial_{x_3} \omega'$ with
$\omega'= O(\gamma^{1-\kappa})$ as $|x-y|\le \gamma$. Then
\begin{multline*}
I_\gamma=\int |e(x,y,\tau)|^2 \omega (x,y)\psi '_\gamma(x-y)\,dxdy=\\
-2h^{-1}\Re i  \int hD_{x_3}e(x,y,\tau)\cdot e(y,x,\tau)
\omega' (x,y) \psi '_\gamma(x-y)\,dxdy-\\
\int |e(x,y,\tau)|^2 \omega (x,y,\tau)\cdot
\partial_{x_3}\psi '_\gamma(x-y)\,dxdy
\end{multline*}
with the first and second terms not exceeding
$C\gamma^{1-\kappa}h^{-1} \times \mu h^{-2}$ and
$C\gamma^{-\kappa} \times \mu h^{-2}$ respectively where we used the fact that
$\sL^2$-norm of $P_j e(.,.,\tau)$ does not exceed $C\mu^{\frac{1}{2}} h^{-1}$. Setting $\gamma=h$ finishes the proof of (ii) in this case.

As $1\le \kappa < 3$ we consider a partition of  $\{(x,y):\ |x-y|\le h\}$ to subzones $\{(x,y):\ |x'-y'|\gtrsim \sigma\}$ and
$\{(x,y):\  |x'-y'|\lesssim \sigma\}$  and in the former zone we repeat the same arguments as before albeit with
$\omega'=O(\log |x'-y'|/h)$, $\omega'=O( |x'-y'|^{1-\kappa})$ for $\kappa=1$ and $1<\kappa<3$ respectively; then contribution of this zone does not exceed
$Ch^{-1}|\log \sigma/h| \times \mu h^{-2}$ and
$Ch^{-1}\sigma^{1-\kappa} \times \mu h^{-2}$ respectively; as $\sigma=\mu^{-\frac{1}{2}}h^{\frac{1}{2}}$ we arrive to  estimate (\ref{16-7-2}) for contribution of this zone as well.

Finally, contribution of the zone
$\{(x,y):\  |x'-y'|\le \sigma, |x_3-y_3\le \gamma\}$ does not exceed
$ C\mu^2 h^{-4} \times \int |z|^{-\kappa}\,dz$ with the integral taken over cylinder  $\{|z'|\le \sigma, |z_3\le \gamma\}$ and not exceeding
$C(1+|\log \sigma/\gamma |) \sigma^2$ and $C\sigma^{2-\kappa}$ respectively as $|e(x,y,\tau)=O(\mu h^{-2})$. Again for selected $\gamma,\sigma$ we get estimate (\ref{16-7-2}) for contribution of this zone as well.

\medskip\noindent
Finally, statement (iii) is due to the fact that   $e(x,y,0)=O(\mu h^{\infty})$ as condition (\ref{16-3-14}) is fulfilled in one of points $x,y$; see subsection~\ref{book_new-sect-13-5-4}.
\end{proof}

\begin{proposition}\label{prop-16-7-5} (cf. proposition~\ref{prop-16-3-9}).
Let conditions \textup{(\ref{16-0-1})}--\textup{(\ref{16-0-5})} be fulfilled. Let $\mu h\ge \epsilon$ and one of the nondegeneracy conditions
\textup{(\ref{16-3-15}}) or \textup{(\ref{16-3-17}}) be fulfilled.

\medskip\noindent
(i) As  $\sigma \ge C_0\mu ^{-\frac{1}{2}}h^{\frac{1}{2}}$ contribution of the zone
\begin{equation*}
\{(x,y):\ |x'-y'|\ge C_0\mu^{-\frac{1}{2}}h^{\frac{1}{2}}, |x-y|\ge \gamma\}
\end{equation*}
to $\I$ does not exceed $C\mu h^{-1}\gamma^{-\kappa}$;

\medskip\noindent
(ii) As  $\gamma \ge C_0h$ contribution of the zone
\begin{equation*}
\{(x,y):\ |x'-y'|\le C_0\mu^{-\frac{1}{2}}h^{\frac{1}{2}}, |x_3-y_3|\ge \gamma\}
\end{equation*}
to $\I$ does not exceed $C\mu h^{-1}\gamma^{-1-\kappa}$;

\medskip\noindent
(iii) Estimate
\begin{multline}
|\I-\I^\T|\le\\
C\mu h^{-1}
\left\{\begin{aligned}
& h^{-\kappa}\qquad &&\text{as\ \ } 0<\kappa<1,\\
& h^{-1}\bigl(1+  (\log \mu  h)_+\bigr) \qquad &&\text{as\ \ }\kappa=1,\\
& \mu^{\frac{1}{2} (\kappa-1)} h^{-\frac{1}{2}(\kappa+1)} \qquad
&&\text{as\ \ }1<\kappa<3, \ \kappa\ne 2,\\
& \mu^{\frac{1}{2} } h^{-\frac{3}{2}}\bigl(1+  (\log \mu  h)_+\bigr) \qquad
&&\text{as\ \ } \kappa=2
\end{aligned}\right.
\label{16-7-3}
\end{multline}
holds.
\end{proposition}

\begin{proof}
To prove statement (i) recall that drift speed does not exceed $C\mu^{-1}$ and therefore Hilbert-Schmidt norm of $\psi E(\tau) \psi'$ does not exceed
$C\mu^{\frac{1}{2}} h^{-\frac{1}{2}}$ as $\psi,\psi'$ are $\sL^\infty_0$-functions with
$\dist (\supp \psi,\supp \psi')\ge  C_0\mu^{-\frac{1}{2}}h^{\frac{1}{2}}$.

Really, it is true for a Hilbert-Schmidt norm of
$\bigl(E(\tau )-E(\tau')\bigr)\psi$ with
$|\tau-\tau'|\le  h$ and then by Tauberian theorem it is true for
a Hilbert-Schmidt norm of $\bigl(E(\tau )-E^\T(\tau)\bigr)\psi'$ with $E^\T$ operator with the Schwartz kernel $e^\T$ with time $T\asymp 1$:
\begin{equation}
\|\bigl(E(\tau )-E^\T(\tau)\bigr)\psi'\|_\HS\le
C\mu^{\frac{1}{2}} h^{-\frac{1}{2}}.
\label{16-7-4}
\end{equation}

However
$\psi E^\T \psi'$ is negligible as $T\le \epsilon $ due to propagation results. This implies (i) obviously. Statement (ii) is proven in the same way but now $T\asymp \gamma$.

\medskip\noindent
To prove statement (iii) we need to estimate contribution to $\I-\I^\T$ of the zones  $\{(x,y):\ |x'-y'|\le h, h\le |x_3-y_3|\le \epsilon\}$ and
$\{(x,y):\ |x-y|\le h\}$. However estimate of the contribution of the former is trivial due to  estimate (\ref{16-7-4}). So, consider zone
$\{(x,y):\ |x-y|\le h\}$.

\medskip\noindent
(a) As $0<\kappa<1$ we can replace as in the proof of proposition~\ref{prop-16-7-4}
$\omega(x,y)$ by $\omega'(x,y)=O(h^{1-\kappa})$ in this zone and one copy of
$\bigl( e(x,y,\tau)- e^\T (x,y,\tau)\bigr)$ to
$D_3 \bigl( e(x,y,\tau)- e^\T (x,y,\tau)\bigr)$; however subsection~\ref{book_new-13-5-4} implies the following generalization of (\ref{16-7-4}):
\begin{equation}
\|P_3^k P^{\prime\, j} \bigl(E(\tau )-E^\T(\tau)\bigr)\psi'\|_\HS\le
C\mu^{\frac{1}{2}} h^{-\frac{1}{2}} (\mu h)^j
\label{16-7-5}
\end{equation}
with $P'=(P_1,P_2)$ and therefore contribution of this zone to $(\I-\I^\T)$ does not exceed $C\mu h^{-1-\kappa}$. In this case estimate (\ref{16-7-3}) is proven.

\medskip\noindent
(b) Let $1\le \kappa <3$. Then the above arguments imply
\begin{claim}\label{16-7-6}
As $1< \kappa <3$ contribution of the zone $\{(x,y):\ |x'-y'|\ge \sigma\}$ with
$\sigma \ge C_0\mu^{-\frac{1}{2}}h^{\frac{1}{2}}$ to $\I$ does not exceed
$C\mu h^{-2} \sigma^{1-\kappa}$; as $\kappa=1$ it does not exceed
$C\mu h^{-2}(1+|\log \mu h|)$.
\end{claim}

Note that with $\sigma=C_0\mu^{-\frac{1}{2}}h^{\frac{1}{2}}$ we get exactly the right-hand expression in (\ref{16-7-3}). Therefore we need to estimate contribution of the zone  $\{(x,y):\ |x'-y'\le \sigma, |x_3-y_3|\le h\}$ with  $\sigma=C_0\mu^{-\frac{1}{2}}h^{\frac{1}{2}}$.

Consider case  $1\le \kappa<2$ first. Note that
$\omega = (3-\kappa)^{-1}\sum_j  \partial_{z_j} z_j\omega$~\footnote{\label{foot-16-19} Recall that we replace $(x,y)$ by $(\mathsf{x},z)=(\frac{1}{2}(x+y), x-y)$.} and therefore we can replace $\omega$  either by $(x_j-y_j) \omega= O(|x'-y'| \cdot |x-y|^{-\kappa})$ with $j=1,2$  or by $(x_3-y_3) \omega  =O(|x-y|^{1-\kappa})$, simultaneously applying $h^{-1}P_j$ to one of the factors $\bigl( e(x,y,\tau)- e^\T (x,y,\tau)\bigr)$. Then  we can replace
$(x_j-y_j) \omega$ by $(x_j-y_j)\omega'$ simultaneously applying $h^{-1}P_j$ to one of the factors $\bigl( e(x,y,\tau)- e^\T (x,y,\tau)\bigr)$ (the same as before or another one).

So as $j=1,2$ we estimate the term by
$C\mu h^{-1} \times h^{-2}   (\mu h)^{\frac{1}{2}} \times \sigma \times \sigma^{1-\kappa}$ with factors
$(\mu h)^{\frac{1}{2}}$  and $\sigma$ coming from $P_j$ and $(x_j-y_j)$; as $\kappa=1$ the last factor is replaced by $|\log \sigma/\gamma|$.
As $j=3$ these factors are replaced by $1$ and  $h$ respectively with the same product. Note that when we drag the last $\partial _{x_3}$ through $x_3$ they may cancel one another but the power becomes $(1-\kappa)$ which we treated already (and while the gain in this power is exactly compensated by an extra factor $h^{-1}$.)

Plugging $\sigma$ we get a proper estimate (\ref{16-7-3}).

\medskip\noindent
(c) The same arguments work for $2\le  \kappa <3$ as well albeit now
$\omega = (3-\kappa)^{-2}\sum_{j,k}  \partial_{z_j}\partial_{z_k} z_jz_k\omega$.
Note that when we drag the last $\partial _{x_3}$ through $x_3$ they may cancel one another but the power becomes $(1-\kappa)$ which we treated already (and while the gain in this power is exactly compensated by an extra factor $h^{-1}$.) However as $\kappa=2$ we get power $-1$ and then logarithm comes in.
\end{proof}

\chapter{Dirac energy: $3\D$-calculations}
\label{sect-16-8}

\section{Pilot Model}
\label{sect-16-8-1}

\subsection{Transformations}
\label{sect-16-8-1-1}

Consider first the pilot-model operator (\ref{16-5-1}).

Let us rescale as before\footnote{\label{foot-16-20} We  need to add factor
$\mu^{-6+\kappa}=\mu^{-3}\times \mu^{-3}\times \mu^{-\kappa}$ where two factors $\mu^{-3}$ are coming from $dx$ and $dy$ and  $\mu ^\kappa$ from $\omega$.}. Then $U(x,y,t)$ in comparison with  $U_{(2)}(x',y',t)$ has factor $U_{(1)}(x_3,y_3,t)$ defined by (\ref{16-5-4}). Let us consider effect of all these changes ignoring other variables; again without any loss of the generality we assume that $\omega(x,y)$ is given by (\ref{16-4-3});

We can replace variables $x,y$ with new variables $\x \Def \frac{1}{2}(x+y)$ and $z\Def (x-y)$ and rescale. Note that the phase of $U_{(1)}$ is linear with respect to $\x_3$ and we can get rid off $\mu^{-1}d\x_3$ integration in the same manner we got rid off $\mu^{-1}\,dx_1$: we replace $\Omega (\ldots,\x_3,.)$ by
its partial Fourier transform $\x_3\to -\mu^{-1}h^{-1}(t'-t'')$ with an extra factor $2\pi$; thus in comparison with $2\D$ case we have
\begin{itemize}[leftmargin=*,label=-]
\item
$\hat{\Omega}\bigl(2\mu^{-1}h^{-1}\alpha (t'-t''),
-\mu^{-1}h^{-1}\beta (t'-t''),  z\bigr)$  \\
where $\hat{\Omega}(.,.,z)$ is a partial Fourier transform with respect to $(\x_1,\x_3)$ (instead of
$\hat{\Omega}\bigl(2\mu^{-1}h^{-1}\alpha (t'-t''),  z\bigr)$);
\item an extra factor  which after easy reductions becomes
\begin{multline*}
\frac{1}{4} h^{-1} (t't'')^{-\frac{1}{2}}
\exp \Bigl( i\hbar^{-1}(t'-t'')  \bigl( -\frac{1}{8}  (t't'')^{-1} z_3^2 -\frac{2}{3}\beta^2 \mu^{-2}(t^{\prime\,2}+t't''+t^{\prime\prime\,2})\Bigr),
\end{multline*}
\item integration with respect to $dz_3$ which enters both this factor and $\hat{\Omega}(.,.,z).$
\end{itemize}
Taking in account all these modifications we need to replace (\ref{16-4-4}) by
\begin{multline}
\I  =\\
\frac{1}{8}(2\pi)^{-3}\hbar^{-3}\mu^{\kappa+3} \iint\iiint
\hat{\Omega} (4\hbar^{-1}\alpha s, -\hbar^{-1}\beta s, z)\times \\[3pt] \csc(t+s)\csc(t-s)(t+s)^{-\frac{3}{2}}(t-s)^{-\frac{3}{2}} \times
\\[3pt]
\exp \Bigl(i\hbar^{-1}\Bigl[-\frac{1}{4}\bigl(\cot(t+s) -\cot(t-s)\bigr) \bigl(z_1 ^2 + (z_2+2t\mu^{-1}\alpha)^2  + 4s^2\mu^{-2}\alpha^2\bigr)\\[3pt]
- \bigl(\cot(t+s) +\cot(t-s)\bigr) s\mu^{-1}\alpha  -2s(\tau-\mu^{-2}\alpha^2) + \\[3pt]
2s  \bigl( -\frac{1}{8}  (t+s)^{-1}(t-s)^{-1} z_3^2 -\frac{2}{3}\beta^2 \mu^{-2}(3t^2+s^2)\Bigr]\Bigr)
\,dt ds\,\psi(\mu^{-1}z) dz_1dz_2dz_3.
\label{16-8-1}
\end{multline}

\subsection{Case $ \lambda^2 \gg \mu h$}
\label{sect-16-8-1-2}

Assume first that
\begin{equation}
\lambda^2 \Def \alpha^2+\beta^2\ge   h^{1-\delta}.
\label{16-8-2}
\end{equation}

\begin{remark}\label{rem-16-8-1}(cf. remark~\ref{rem-16-4-2}).
(i) We take $\lambda$-admissible with respect to $\x$ function $\Omega_\lambda=\Omega(\x/\lambda,.)$.

Then the virtue of the factor $\hat{\Omega}(4\hbar^{-1}\alpha \gamma s,-\hbar^{-1}\beta \gamma s,\cdot)$ under assumption (\ref{16-8-2})  we need to consider only $|s|\le h^\delta$  and therefore we can consider separately
$|t'|\le \epsilon_0,|t''|\le \epsilon_0$ and
$|t'|\ge \epsilon_0,|t''|\ge \epsilon_0$;

\medskip\noindent
(ii) Note that due to section~\ref{book_new-sect-6-3} contribution of zone
$\{|t'|\le \epsilon_0,|t''|\le \epsilon_0\}$ defined by integral expressions (\ref{16-8-1}) etc with an extra factor  $\bar{\chi}_{\epsilon_0}(t')$ or
$\bar{\chi}_{\epsilon_0}(t'')$ or $\bar{\chi}_{\epsilon_0}(t)$ differs from the same expression for non-magnetic Schr\"odinger operator  by
$O(\mu h^{-2-\kappa} \times \hbar^\kappa)=O(\mu^{\kappa+1}h^{-2})$ as
$\kappa \ne 1$ and by $O(\mu^2 h^{-2}|\log \mu|)$ as $\kappa = 1$.

\medskip\noindent
(iii) Furthermore, if we remove from this expression for a non-magnetic Schr\"odinger operator cut-off $\{|t'|\ge \epsilon_0\}$ then the error would not exceed the same expression as well.
\end{remark}

Let us consider contribution of zone
$\{|t'|\ge \epsilon_0,|t''|\ge \epsilon_0\}$ defined by an integral expressions (\ref{16-8-1}) etc with an extra factor  $\bigl(1-\bar{\chi}_{\epsilon_0}(t')\bigr)$. Due to remark~\ref{rem-16-8-1}(i) we need to consider only $t',t''$ belonging to the same tick.

Let us consider first zone
\begin{equation}
\{|s|\ge \hbar /\lambda^2, |\sin (t)|\ge C|s|\}.
\label{16-8-3}
\end{equation}
Then  integration by parts with respect to $z'$ delivers one of the factors (\ref{16-4-7}), (\ref{16-4-8}). Thus integrating by parts many times in the zone where both of these factors are less than $1$ we acquire factors
(\ref{16-4-9}), (\ref{16-4-10}) respectively.

On the other hand, integration by parts with respect to $z_3$ delivers factor $\hbar |s|^{-1}t^2 |z_3|^{2}$;  therefore integrating by parts many times in the zone $\{|z_3|\ge \hbar^{\frac{1}{2}} |s|^{-\frac{1}{2}}|t|\}$ delivers factor
\begin{equation}
\bigl( 1+  |s|\hbar^{-1}t^{-2}|z_3|^2\bigr)^{-l}.
\label{16-8-4}
\end{equation}

\paragraph{Case $1<\kappa<3$.}
\label{sect16-8-1-2-1}

Multiplying by $|z|^{-\kappa}$ and integrating we get after multiplication by $|\sin (t)|^{-2}$ expression (\ref{16-4-11}) albeit with $\kappa$ replaced by $(\kappa-1)$ i.e.
\begin{equation}
\asymp |\sin (t)|^{1-\kappa} |s|^{-\frac{3}{2}+\frac{1}{2}\kappa}\hbar^{\frac{3}{2}-\frac{1}{2}\kappa}.
\label{16-8-5}
\end{equation}
Integrating by $t$ over $k$-th tick intersected with $\{t:\ |\sin(t)|\ge |s|\}$ we get (\ref{16-4-12}) modified the same way i.e.
\begin{equation}
C\hbar^{\frac{3}{2}-\frac{1}{2}\kappa}\left\{\begin{aligned}
&|s|^{-\frac{3}{2}+\frac{1}{2}\kappa}\qquad &&1<\kappa<2,\\
&|s|^{-\frac{1}{2}}(1+|\log |s||)\qquad &&\kappa=2,\\
&|s|^{\frac{1}{2}-\frac{1}{2}\kappa}\qquad &&2<\kappa<3.
\end{aligned}\right.
\label{16-8-6}
\end{equation}

This expression (\ref{16-8-6}) must be \underline{either} integrated with respect to $s$ over $\{|s|\le \hbar/\lambda^2 \}$ \underline{or} multiplied by
$(\hbar/\lambda^2)^l |s|^{-l}$ due to factor $\hat{\Omega}$ and integrated over $\{|s|\ge \hbar/\lambda^2\}$, resulting in both cases in the same answer
which is the value of $\textup{(\ref{16-8-6})}\times |s|$ calculated as
$s= \hbar/\lambda^2$ which is similarly modified expression (\ref{16-4-13}) i.e.
\begin{equation}
C\hbar
\left\{\begin{aligned}
&\lambda^{1-\kappa} \qquad &&1<\kappa<2,\\
&(1+|\log (\hbar/\lambda^2)|)\lambda^{-1}\qquad &&\kappa=2,\\
&\hbar^{2-\kappa}\lambda^{-3+\kappa}\qquad &&2<\kappa<3.
\end{aligned}\right.
\label{16-8-7}
\end{equation}

In addition to zone (\ref{16-8-3}) we  need to consider  zone (\ref{16-4-14}) defined as $\{|\sin(t')|\asymp |s|,\ |\sin(t'')|\le |s| \}$;
its tween  $\{|\sin(t')|\asymp |s|,\ |\sin(t'')|\le |s| \}$ is considered in the same way.

In zone (\ref{16-4-14})  $|\cot (t')-\cot(t'')|\asymp |\sin (t'')|^{-1}$ and in this case factors (\ref{16-4-7}), (\ref{16-4-8}) are replaced by \ref{16-4-7-'}, \ref{16-4-8-'} and (\ref{16-4-9}), (\ref{16-4-10}) by \ref{16-4-9}, \ref{16-4-10-'} while factor (\ref{16-8-4}) is preserved.

Then, multiplying by $|z|^{-\kappa}$ and integrating we get  after multiplication by $|\sin(t'')|^{-1}|\sin(t')|^{-1}$ similarly modified expression \ref{16-4-11-'}
\begin{equation}
\hbar ^{\frac{3}{2}-\frac{1}{2}\kappa}
|\sin (t'')|^{\frac{1}{2}-\frac{1}{2}\kappa}|s|^{-1};
\tag*{$\textup{(\ref*{16-8-5})}'$}\label{16-8-5-'}
\end{equation}
then integrating by $|t''|$ over one tick but with $|\sin(t'')|\le |s|$ we get modified expression \ref{16-4-12-'} i.e.
\begin{equation}
\hbar ^{\frac{3}{2}-\frac{1}{2}\kappa}
|s|^{\frac{1}{2}-\frac{1}{2}\kappa}.
\tag*{$\textup{(\ref*{16-8-6})}'$}\label{16-8-6-'}
\end{equation}

Finally, \underline{either} integrating  over $\{|s|\le h/\varepsilon\lambda\}$ \underline{or} multiplying by $|s|^{-l}(h/\varepsilon\lambda)^l$ and integrating over $\{|s|\ge h/\varepsilon\lambda\}$ we get in both cases the same answer $\textup{\ref{16-8-6-'}}\times |s|$, calculated as $s= \hbar/\lambda^2$ not exceeding (\ref{16-8-7}).

Therefore the total contribution of zones (\ref{16-8-3}) and (\ref{16-4-14}) is given by expression (\ref{16-8-7}). Then multiplying by
$|k|^{-3}\mu^\kappa h^{-3}\lambda^3$ we get after summation with respect to $k:|k|\ge 1$ the value as $k=1$ i.e.
$\textup{(\ref{16-8-7})}\times \mu^\kappa h^{-3}\lambda^3$.

Therefore we arrive to analogue of proposition~\ref{prop-16-4-3}
\begin{multline}
|\I_lambda ^\T - \I_\lambda ^{\T\,\prime}|\le R^\W(\lambda)\Def \\
C\mu^{\kappa+1} h^{-2}\lambda^3
\left\{\begin{aligned}
&\lambda^{1-\kappa} \qquad &&1<\kappa<2,\\
&(1+|\log (\hbar/\lambda^2)|)\lambda^{-1}\qquad &&\kappa=2,\\
&\mu ^{2-\kappa}h^{2-\kappa}\lambda^{-3+\kappa}\qquad &&2<\kappa<3
\end{aligned}\right.
\label{16-8-8}
\end{multline}
where $\I_\alpha ^{\T\,\prime}$ is a Tauberian expression for $\I$ albeit with $T=\epsilon_0 \mu^{-1}$.

\paragraph{Case $0<\kappa\le 1$.}
\label{sect16-8-1-2-2}

In this case integration over $\{z: |z|\le C_0\}$~\footnote{\label{foot-16-21} Corresponding to $|x-y|\le C_0\mu^{-1}$.}
and multiplication by  $|\sin(t)|^{-2} $  results in
\begin{gather}
C|\sin (t)|^{-2}\rho^2
\left\{\begin{aligned}
&\zeta ^{1-\kappa}\qquad &&0<\kappa<1,\\
&(1+(\log \zeta/\rho))_+) \qquad &&\kappa=1
\end{aligned}\right.\label{16-8-9}\\
\shortintertext{with}
\zeta=\zeta(s,k)=\min\bigl((\hbar/|s|)^{\frac{1}{2}}|k|,1\bigr),
\label{16-8-10}\\
\rho = \min\bigl( (\hbar/|s|)^{\frac{1}{2}}, 1\bigr)
\label{16-8-11}
\end{gather}
(cf. expression (\ref{16-8-5})).

One can see easily that
\begin{claim}\label{16-8-12}
Integral by $t$ over one tick intersected with  $\{t:|\sin(t)|\ge |s|\}$ does not exceed
$C\hbar|s|^{-1}\zeta^{1-\kappa}$ as $|s|\ge \hbar$ and
$C(\hbar|s|^{-1})^{\frac{1}{2}}\zeta^{1-\kappa}$ as $|s|\le \hbar$.
As $\kappa=1$ one needs to include an extra factor $|\log \zeta/\rho|$, $\rho$ calculated as $|\sin(t)|=1$ for $|s|\ge \hbar$ and $\sin(t)|= (|s|/\hbar)^{\frac{1}{2}}$ for $|s|\le \hbar$.
\end{claim}

Meanwhile in the zone (\ref{16-4-14}) integration over $\{z: |z|\le C_0\}$ and multiplication by $s^{-1}|\sin(t'')|^{-1}$ results in expression (\ref{16-8-11}) and integration by $\{t'':|\sin(t'')|\le |s|\}$ results in much lesser term
$C\hbar \zeta^{1-\kappa}$, with some logarithmic factor as $\kappa=1$.

Now integral by $\{s: |s|\lesssim 1\}$ does not exceed $C\hbar |k|^{1-\kappa}$ as $0<\kappa<1$ and
$C\hbar (1+|\log \hbar|+|\log |k|)$ as $\kappa=1$. Finally, multiplication by $\mu^{\kappa+3} \hbar^{-3}|k|^{-2}$ and summation with respect to $k:|k|\ge 1$ results in $C\mu^{\kappa+1}h^{-2}$ as $0<\kappa <1$ and
$C\mu^2h^{-2}|\log \mu h|$ as $\kappa=1$.

Therefore we arrive to analogue of proposition~\ref{prop-16-4-3}:
\begin{multline}
|\I_lambda ^\T - \I ^{\T\,\prime}_\lambda |\le R^\W(\lambda)\Def\\
C\mu^{\kappa+1} h^{-2}\lambda^3
\left\{\begin{aligned}
&1 \qquad &&0<\kappa<1,\\
&(1+|\log (\mu h)|)\qquad &&\kappa=1.
\end{aligned}\right.
\label{16-8-13}
\end{multline}

Combining with (\ref{16-8-8}) we arrive to

\begin{proposition}\label{prop-16-8-2} (cf. corollary~\ref{prop-16-4-3}).
(i) For the pilot-model operator with  $1\le \mu \le h^{-1}$ under assumption \textup{(\ref{16-8-2})} estimates \textup{(\ref{16-8-8})},  \textup{(\ref{16-8-13})} hold for $1<\kappa<3$, $0<\kappa \le 1$ respectively.

\medskip\noindent
(ii) In particular, as $\lambda=1$
\begin{multline}
|\I ^\T - \I ^{\T\,\prime}|\le R^\W  \Def \\
C\mu^{\kappa+1} h^{-2}
\left\{\begin{aligned}
&1 \qquad &&0<\kappa<2,\kappa\ne 1,\\
&(1+|\log (\mu h)|)\qquad &&\kappa=1,2,\\
&\mu ^{2-\kappa}h^{2-\kappa}\qquad &&2<\kappa<3.
\end{aligned}\right.
\label{16-8-14}
\end{multline}
\end{proposition}

Following the same proof as in proposition~\ref{prop-16-4-4} one can prove easily that
\begin{equation}
|\I ^{\T\,\prime} -\cI^\W  |\le C\mu^2 h^{-1-\kappa}
\label{16-8-15}
\end{equation}
which does not exceed $R^\W(\lambda)$ implying

\begin{corollary}\label{cor-16-8-3}(cf. corollary~\ref{cor-16-4-5}).
In frames of proposition~\ref{prop-16-8-2}
\begin{equation}
|\I ^ \T _\lambda-\cI^\W_\lambda |\le CR^\W(\lambda).
\label{16-8-16}
\end{equation}
\end{corollary}

\begin{remark}\label{rem-16-8-4}
After summation with respect to partition we restore estimate (\ref{16-8-14}) under some non-degeneracy assumption. Comparing with the Tauberian remainder estimate $O(h^{-2-\kappa})$ we conclude that $R^\W$ is lesser as $\mu \le h^{-\kappa/(\kappa+1)}$ ($0<\kappa <2)$ and $\mu \le h^{-\frac{2}{3}}$
($2\le \kappa <3$) (with some correction for $\kappa=1,2$).

In particular, $\mu \le h^{-\frac{1}{2}}$ would suffice for $\kappa>1$ and therefore improvement as in subsubsection~\ref{sect-16-4-1-3}. ``\nameref{sect-16-4-1-3}'' is not needed.
\end{remark}

\subsection{Case $ \lambda^2 \not\gg \mu h$}
\label{sect-16-8-1-3}

Assume now that $\mu^{-1}\le \lambda$ and $\lambda^2  \le \mu h^{1+\delta}$. Then in contrast to the previous we will need to compute contributions of pair of ticks $(k',k'')$ with $k'\ne k''$. Then if $t'$, $t''$ belong to $k'$-th and $k''$-th ticks respectively we denote $r=k'-k''$ and $s=t'-t''-2\pi r$.

However we still need to distinguish cases $1<\kappa <3$ and $0<\kappa\le 1$.

\paragraph{Case $1< \kappa < 3$.}
\label{sect-16-8-1-3-1}
In this case we according to the previous subsection need just to take results for $d=2$, plug $\kappa\Def (\kappa-1)$, $\alpha=\lambda$  and multiply by $h^{-1}\lambda$ resulting in

\begin{proposition}\label{prop-16-8-5}(cf. proposition~\ref{prop-16-4-10}).
For the pilot-model operator with  $1\le \mu \le h^{-1}$, $1<\kappa<3$
\begin{gather}
|\I_{(\lambda)}^\T - \cI_{(\lambda)}^\W |\le \textup{(\ref{16-8-8})} + \textup{(\ref{16-8-18})}\qquad  \text{as\ \ } \lambda^2\ge \mu h, \label{16-8-17}\\
\intertext{with}
C\mu^{\frac{1}{2}\kappa-\frac{1}{2}} h^{-\frac{5}{2}-\frac{1}{2}\kappa} \lambda^3 (\mu h/\lambda^2)^l\label{16-8-18}\\
\intertext{and}
|\I_{(\lambda)}^\T - \cI_{(\lambda)}^\W |\le \textup{(\ref{16-8-20})} + \textup{(\ref{16-8-21})} + \textup{(\ref{16-8-22})}
\qquad \text{as\ \ } \lambda^2\le \mu h\label{16-8-19}\\
\intertext{with}
C\mu^{\frac{1}{2}\kappa +\frac{3}{2}} h^{-\frac{3}{2}-\frac{1}{2}\kappa}\lambda^3,
\label{16-8-20}\\
C\mu^{\frac{1}{2}+\frac{1}{2}\kappa} h^{-\frac{3}{2}-\frac{1}{2}\kappa}\lambda^3
\bigl(1+(\log (\mu h/ \lambda^2))_+\bigr)^2,
\label{16-8-21}\\
C\mu^{\frac{1}{2}\kappa-\frac{1}{2}} h^{-\frac{5}{2}-\frac{1}{2}\kappa} \lambda^3 \bigl(1+(\log (\mu h/ \lambda^2))_+\bigr)
\label{16-8-22}
\end{gather}
\end{proposition}

\begin{proof}
Easy details are left to the reader.
\end{proof}

\paragraph{Case $0< \kappa \le 1$.}
\label{sect-16-8-1-3-2}
In this case we need to modify arguments of the previous subsubsection.

\medskip\noindent
(i) Consider case $k'-k''=r=0$ first. As we did not use $\lambda^2 \gg \mu h$ estimate  (\ref{16-8-13}) for contributions of such terms remains valid.

\medskip\noindent
(ii) Consider now  case $k'-k''=r\ne 0$.  Note that then according to (\ref{16-8-4}) one should use $\zeta =  |k'+k''| \cdot (\hbar|k-k'|)^{\frac{1}{2}}$

As $\lambda^2 \ge \mu h$ the last expression before summation with respect to $k\ (=k')$  acquires factor $(|k'-k''|\lambda^2 /\mu h)^{-l}$ and  we get after summation with respect to $k'\ne 0$, $k''\ne 0$ the same right-hand expression of (\ref{16-8-13}).

Meanwhile, as  $\lambda^2 \le \mu h$ we get before summation (as $0<\kappa<1$)
\begin{multline}
C\mu^{\kappa+1} h^{-2}\lambda^3 |k'|^{-1}|k''|^{-1}
\bigl( |k'+k''|/|k'-k''|^{\frac{1}{2}})\bigr)^{1-\kappa}\times
\\
\bigl(1+|k'-k''|\lambda^2/\mu h\bigr)^{-l}
\label{16-8-23}
\end{multline}
(with a logarithmic factor as $\kappa=1$).

It is enough to sum with as $1\le k''$, $k'=k''+r$, $r\ge 1$; then the sum is
\begin{equation*}
\asymp C \mu^{\kappa+1}h^{-2}\lambda^3
\sum _{r\ge 1} r^{-\frac{1}{2}-\frac{1}{2}\kappa}
\bigl(1+r\lambda^2/\mu h\bigr)^{-l} \asymp
C\mu^{\kappa+1}h^{-2}\lambda^3
(\mu h/\lambda^2)^{\frac{1}{2}-\frac{1}{2}\kappa};
\end{equation*}
so we get
\begin{equation}
C
\left\{\begin{aligned}
&\mu^{\frac{3}{2}+\frac{1}{2}\kappa} h^{-\frac{3}{2}-\frac{1}{2}\kappa}\lambda^{2+\kappa} \qquad
&& 0<\kappa <1,\\
&\mu^2  h^{-2}\lambda^3\bigl(1+\log (\mu h)\bigr)^2\qquad
&& \kappa =1
\end{aligned}\right.
\label{16-8-24}
\end{equation}
where case $\kappa=1$ we left to the reader.

\medskip\noindent
(iii) Consider now  $k'= 0, k''\ne 0$; its tween case $k'\ne 0, k''= 0$ is addressed in the same way. To do this we need to modify our expression fo $\I$: namely, we do not divide by $t'$ and then take Fourier transform
$F_{t'\to \hbar^{-1}\tau}$; instead we take Fourier transform
$F_{t'\to \hbar^{-1}\tau'}$, then integrate by $\tau'$ to $\tau$ and divide by $\hbar$; modifying this way  (\ref{16-8-1})  we arrive instead of  (\ref{16-8-1})  to the expression similar to (\ref{16-4-30}).

Then repeating the similar arguments of section~\ref{sect-16-4} with the above modification, we arrive to
\begin{equation*}
C \mu^{\kappa}h^{-3}\lambda^3
|k''|^{-\frac{1}{2}-\frac{1}{2}\kappa}
\bigl(1+|k''|\lambda^2/\mu h\bigr)^{-l}
\end{equation*}
instead of (\ref{16-8-23}) and summation with respect to $k'':|k''|\ge 1$ returns $C \mu^{\kappa}h^{-3}\lambda^3 (\mu h/\lambda^2)^{l} $
as $\lambda^2\ge \mu h$ and $C\mu^{\kappa}h^{-3}\lambda^3
(\mu h/\lambda^2)^{\frac{1}{2}-\frac{1}{2}\kappa}$ as $\lambda^2 \le \mu h$.

So we arrive to
\begin{equation}
C (\mu h/\lambda^2)^{l}\mu^{\kappa}h^{-3}\lambda^3
\left\{\begin{aligned}
& 1   \qquad
&& 0<\kappa <1,\\
&\bigl(1+\log (\mu h)\bigr)^2\qquad
&& \kappa =1
\end{aligned}\right.
\label{16-8-25}
\end{equation}
as $\lambda^2 \ge \mu h$ and
\begin{equation}
C \mu^{\frac{1}{2}+\frac{1}{2}\kappa}h^{-\frac{5}{2}-\frac{1}{2}\kappa}
\lambda^{2+\kappa}
\left\{\begin{aligned}
&  1
\qquad
&& 0<\kappa <1,\\
&\bigl(1+\log (\mu h)\bigr)^2\qquad
&& \kappa =1
\end{aligned}\right.
\label{16-8-26}
\end{equation}
as $\lambda^2 \le \mu h$.

Thus we proved

\begin{proposition}\label{prop-16-8-6}
For the pilot-model operator with  $1\le \mu \le h^{-1}$, $0<\kappa\le 1$
\begin{gather}
|\I_{(\lambda)}^\T - \cI_{(\lambda)}^\W |\le \textup{(\ref{16-8-13})} + \textup{(\ref{16-8-25})}\qquad  \text{as\ \ } \lambda^2\ge \mu h, \label{16-8-27}\\
\intertext{and}
|\I_{(\lambda)}^\T - \cI_{(\lambda)}^\W |\le \textup{(\ref{16-8-13})} + \textup{(\ref{16-8-24})} + \textup{(\ref{16-8-26})}.
\qquad \text{as\ \ } \lambda^2\le \mu h
\label{16-8-28}
\end{gather}
\end{proposition}

\section{General operators}
\label{sect-16-8-2}

Consider now general operators satisfying conditions (\ref{16-0-5})--(\ref{16-0-6}) \underline{and} either (\ref{16-0-7}) or (\ref{16-0-8}) as $\tau=0$.

First, using proposition~\ref{prop-16-7-1}  and propagation results we conclude that

\begin{proposition}\label{prop-16-8-7}(cf. proposition~\ref{prop-16-4-11}).
Let conditions \textup{(\ref{16-0-5})}--\textup{(\ref{16-0-6})} be fulfilled.
Then

\medskip\noindent
(i) Contribution of the zone $\{(x,y):\ |x-y|\ge \epsilon_0\mu^{-1}\}$ (before rescaling) to $\I^\T$ does not exceed $C\mu^{\kappa+1} h^{-2}$ and

\medskip\noindent
(ii) Under condition \textup{(\ref{16-0-7})} contribution of time
$\{|t|\ge \epsilon_0\}$ (before rescaling) to $\I^\T$ does not exceed $C\mu^{\kappa+1} h^{-2}$.
\end{proposition}

Note that $O(\mu^{\kappa+1}h^{-1})$ either coincides or is smaller than the remainder estimate (\ref{16-8-14}) and as long as we are interested in Weyl approximation $\cI^\W$ we should be completely happy with it.

Under condition (\ref{16-0-7}) we are done arriving to statement (i) of theorem~\ref{thm-16-8-9} below.

Under condition (\ref{16-0-8}) let us introduce the scaling function
\begin{equation}
lambda (x)= \epsilon_0 |\nabla VF^{-1}|+ \frac{1}{2}\bar{\lambda},\qquad
\bar{\lambda}=C_0 \mu^{-1}
\label{16-8-29}
\end{equation}
and consider covering of $B(0,1)$ by $\lambda$-admissible elements. In virtue of proposition~\ref{prop-16-8-7}(i) we need to consider only pairs $(x,y)$ belonging to the same element. Due to the rescaling arguments, the contribution of one such element to $|I^\T-\cI^\W|$ does not exceed
$C\mu h^{-1-\kappa}\lambda^3$ and therefore

\begin{claim}\label{16-8-30}
Under condition (\ref{16-0-8}) contribution of the zone
\begin{equation*}
\{(x,y): |x-y|\le C_0\mu^{-1},\  |\nabla VF^{-1}| \le \lambda\}
\end{equation*}
to  $|I^\T-\cI^\W|$ does not exceed $C\mu  h^{-2-\kappa}\lambda^3$;
\end{claim}
as $\lambda=\bar{\lambda}$ we get $O(\mu^{-2}h^{-2-\kappa})$.

Therefore we need to consider only zone
\begin{equation*}
\{(x,y): |x-y|\le C_0\mu^{-1},\  |\nabla VF^{-1}| \ge \bar{\lambda}\}.
\end{equation*}
We are going to prove that

\begin{proposition}\label{prop-16-8-8}(cf. proposition~\ref{prop-16-4-12}).
Let conditions \textup{(\ref{16-0-5})}--\textup{(\ref{16-0-6})} be fulfilled.
Further, let $\mu \le h^{-1}$. Then  for each $z$
contribution of zone $\{(x,y)\in B(z,\lambda (z)),\ |x-y|\le C_0\mu^{-1}|\}$ with $\lambda(z)\ge C_0\mu^{-1}$ to $|\I^\T -\cI^\W|$ does  not exceed

\medskip\noindent
(i) $R^\W(\lambda)$  which is defined as the right-hand expression of \textup{(\ref{16-8-8})} or \textup{(\ref{16-8-13})} for $1<\kappa<3$, $0<\kappa\le 1$ provided $\lambda \ge (\mu h)^{\frac{1}{2}-\delta}$;

\medskip\noindent
(ii) the right-hand expression of \textup{(\ref{16-8-17})} or \textup{(\ref{16-8-27})} for $1<\kappa<3$, $0<\kappa\le 1$ provided
$\lambda \ge (\mu h)^{\frac{1}{2}}$;

\medskip\noindent
(iii) the right-hand expression of \textup{(\ref{16-8-19})} or \textup{(\ref{16-8-28})} for $1<\kappa<3$, $0<\kappa\le 1$ provided
$\lambda \le (\mu h)^{\frac{1}{2}}$.
\end{proposition}

Consider  components of $R^\W(\lambda)$. Note that  expressions (\ref{16-8-8}), (\ref{16-8-13}) and (\ref{16-8-20})--(\ref{16-8-22}), (\ref{16-8-24}), (\ref{16-8-26}) contain $\lambda$ in the positive powers while
expressions (\ref{16-8-18}), (\ref{16-8-25}) contain $\lambda$ in the negative powers. Therefore after integration with respect $\lambda^{-1}\,d\lambda$ we get expressions (\ref{16-8-8}) and (\ref{16-8-13}) as $\lambda=1$ (i.e. we get (\ref{16-8-14}) and we get all other expressions as
$\lambda=(\mu h)^{\frac{1}{2}}$; one can check easily that these latter expressions are less than $\textup{(\ref{16-8-14})}+Ch^{-2-\kappa}$.

Therefore we arrive to

\begin{theorem}\label{thm-16-8-9} (cf. theorem~\ref{thm-16-4-13}).
Let conditions \textup{(\ref{16-0-5})}--\textup{(\ref{16-0-6})}.
Let $\mu \le h^{-1}$. Then

\medskip\noindent
(i) Under condition \textup{(\ref{16-0-7})}  estimate
\begin{equation}
|\I  - \I ^\W|\le CR^\W + Ch^{-2-\kappa};
\label{16-8-31}
\end{equation}
\medskip\noindent
(ii) Under condition \textup{(\ref{16-0-8})}
\begin{equation}
|\I  - \I ^\W|\le CR^\W + C\mu ^{\frac{5}{2}}h^{-\frac{1}{2}-\kappa- \delta} + Ch^{-2-\kappa}.
\label{16-8-32}
\end{equation}
\end{theorem}

\begin{proof}[Proof of proposition~\ref{prop-16-8-8}]
First of all applying proposition~\ref{prop-16-2-5} and the same arguments as in the analysis of section~\ref{sect-16-2} we arrive to

\begin{claim}\label{16-8-33}
Contribution of the zone
$\{|\sin (\theta(t'))|\ge C_0\hbar , \ |\sin (\theta(t''))|\ge C_0\hbar \}$ to the error $|\I^\T-\cI^\W|$ does not exceed $R^\W(\lambda)$.
\end{claim}

Therefore we need to explore contributions of three remaining zones
\begin{gather}
\{|\sin (\theta(t'))|\le 2C_0\hbar, \ |\sin (\theta(t''))|\le 2C_0\hbar \}
\tag{\ref*{16-4-43}}
\shortintertext{and}
\{|\sin (\theta(t'))|\ge 2C_0\hbar, \ |\sin (\theta(t''))|\le C_0\hbar \}
\tag{\ref*{16-4-44}}
\end{gather}
(the same analysis would cover its tween  zone
$\{|\sin (\theta(t'))|\ge 2C_0\hbar,\ |\sin (\theta(t''))|\le C_0\hbar \}$).

We would like to follow the arguments of the proof of proposition~\ref{prop-16-4-12} but there is a problem: the magnetic drift is controlled by $\alpha$ and it may be much less than $\lambda$. However let us apply few general arguments.

\medskip\noindent
\emph{Zone \textup{(\ref{16-4-43}) as $\lambda^2\ge \mu h$}.\/} First, if we consider zone $\{(x,y):\ |x-y|\ge \ell\}$ (in the non-rescaled coordinates) we estimate the contribution of the corresponding pair $(k',k'')$ by
$C\hbar^2 \ell^{-\kappa}|k'|^{-1}|k''|^{-1}\times h^{-d}\lambda^d$ and summation with respect to $|k'|\le \mu \lambda$, $|k''|\le \mu \lambda$ returns
\begin{equation}
C\mu^2  h^{2-d}\ell^{-\kappa} (1+|\log \mu \lambda|)^2.
\label{16-8-34}
\end{equation}
Meanwhile, if we consider  zone $\{(x,y):\ |x-y|\ge \ell\}$ (in the non-rescaled coordinates as well ) then estimating $U(x,y,t)$ by $h^{-d}$ (no rescaling, $d=2,3$) we we estimate the contribution of the corresponding pair $(k',k'')$ by
$C\hbar^2 \times h^{-2d}\times \ell^{d-\kappa}\lambda^d $ and again summation with respect to $k',k''$ returns
\begin{equation}
C\mu^2 h^{2-2d}\ell^{d-\kappa}\lambda^d (1+|\log \mu \lambda|)^2.
\label{16-8-35}
\end{equation}
Comparing (\ref{16-8-34}) and (\ref{16-8-35}) we see that their sum is minimal
\begin{equation}
C\mu^2 h^{2-d-\kappa}\lambda^d (1+|\log \mu \lambda|)^2
\label{16-8-36}
\end{equation}
as $\ell = h $.

Comparing with (\ref{16-8-8}) and (\ref{16-8-13}) we conclude that for $d=3$ this is less than $R^\W$ as $0<\kappa<1$, $\lambda \ge \mu h$. To deal with
$1\le \kappa <3$ let us note first that there must be a factor
\begin{equation}
C\bigl( 1+ |k'-k''|\lambda^2 /\mu  h\bigr)^{-l}
\label{16-8-37}
\end{equation}
before summation. Really, as $|\xi_3+\eta_3|\ge \epsilon_0 \lambda$ this factor is due to the shift in $(x_3+y_3)$ and  as
$|\xi_3+\eta_3|\ge \epsilon_0 \lambda$  but $|\nabla 'V/F|\asymp \lambda$ or  $|\nabla _{x_3} V/F|\asymp \lambda$ this factor is due to the shift in $(\xi_1+\eta_1)$ or $(\xi_3+\eta_3)$ respectively.

Therefore
\begin{claim}\label{16-8-38}
As $\lambda^2\ge \mu h^{1-\delta}$ contribution of $k'\ne k''$ is negligible.
\end{claim}
This allows us also eliminate logarithmic factors in (\ref{16-8-36}) and (\ref{16-8-37}).

Let us apply cut-off $Q$ by $hD_{x_3}$ and $Q'$ by $hD_{xy_3}$. Assume first that \underline{either} $|\xi_3|\le C\rho$ on the support of (the symbol of) $Q$ \underline{or} $|\eta_3|\le C\rho$ on the support of (the symbol of) $Q'$, in our estimate we acquire factor $\rho$ and we get an expression below $R^\W(\lambda)$.

Therefore we can assume that

\begin{claim}\label{16-8-39}
$|\xi_3|\ge \rho\Def \max\bigl((\mu h)^{\frac{1}{2}},\mu^{-1}\bigr)$ on the support of (the symbol of) $Q$
\end{claim}
(or $|\eta_3|\ge \max\bigl((\mu h)^{\frac{1}{2}},\mu^{-1}\bigr)$ on the support of (the symbol of) $Q'$.

Let us consider corresponding classical dynamics. Assumption (\ref{16-8-39}) implies that in $k$ winding it travels the ``observable'' distance $k\xi_3$ and if there was no ``turn back'' the contribution of $k$-th winding to estimate would not exceed
\begin{equation*}
C\mu ^2 h^{-1}|k|^{-2-\kappa}\int_{|\xi_3|\ge \rho}
|\xi_3|^{-\kappa}\,d\xi_3 \lambda^3
\end{equation*}
where integral is $\asymp 1$ as $0<\kappa <1$, $\asymp \log \rho$ as $\kappa=1$  and $\asymp \rho^{1-\kappa}$ as $1<\kappa <3$.  Obviously multiplied by
$\mu^2 h^{-1}$ it does not exceed $R^\W(\lambda)$ and therefore summation with respect to $k$ returns $O(R^\W(\lambda))$.

Now we need to consider effect of the ``turn back''.
Let $\beta= |\nabla_{x_3} V/ F|$. Note first that as $|\xi_3|\ge \beta\lambda$
there is no turn back for time $|t|\le \epsilon\lambda$. Therefore we need to consider only $|\xi_3|\le \beta \lambda$.

There is come back on $k$-th winding as $\xi_3 $ is close to
$\zeta^*(k)\sim \const \beta \mu^{-1}k$ and the distance
$|x-y|\asymp \mu^{-1}|k| |\xi_3-\zeta(k)|$ and then contribution of $k$-th winding (as $|x-y|\ge Ch$) does not exceed
\begin{multline*}
C\mu ^{2+\kappa} h^{-1}|k|^{-2-\kappa}\int_{|\xi_3-\zeta(k)| \ge \mu h|k|^{-1}} |\xi_3-\zeta(k)|^{-\kappa}\,d\xi_3  \lambda^3 \le\\
C\lambda^3\left\{ \begin{aligned}
&\mu^{2+\kappa}h^{-1}|k|^{-2-\kappa}\quad &&0<\kappa<2,\\
&\mu^4 h^{1-\kappa}|k|^{-4} \qquad && 2<\kappa<3,
\end{aligned}\right.
\end{multline*}
with a logarithmic factor $(1+\log |k|+|\log \mu h|)$ as $\kappa=2$. The sum of the expression above is the same as this expression for $k=1$ (we denote it temporarily by $R\lambda^3 $) which obviously less than $R^\W(\lambda)$.

Finally, we need to consider the contribution of $k$-th winding as
$|x-y|\le h$. However it is done by the standard way as $U(x,y,t)$ after cut by $\xi_3$ and $\eta_3$ acquires factor $Ch^2$.

Therefore, under assumption $\lambda^2\ge \mu h^{1-\delta}$ we are done.

\medskip\noindent
\emph{Case $\lambda^2\ge \mu h$.\/}
As $\lambda^2\ge \mu h$ (but $\lambda^2\not\gg \mu h$) terms with $k'\ne k''$  are no more negligible but in the estimate we can ignore due to the factor (\ref{16-8-37}) them except when $k'=0$ or $k''=0$.

Applying the same arguments as in $2\D$-case i.e. replacing in term with $k'=0$ factor $|k'|^{-1}$ by  factor $\hbar^{-1}$ we arrive to
$R\lambda^3 \times 1/(\mu h)  \times (\mu h/\lambda^2)^l$ which is less than term (\ref{16-8-18}) or (\ref{16-8-25}) as $1<\kappa<3$ or $0<\kappa\le 1$ respectively.

\medskip\noindent
\emph{Zone \textup{(\ref{16-4-43}) as $\lambda^2\le \mu h$}.\/}
In this case we apply the same arguments but now effectively we sum over $k',k''$ with $|k'-k''|\le \mu h/\lambda^2$ which effectively boils out to the summation over $k'=1$, $k''=1,\dots, \mu h/\lambda^2$ or
$k'=0$, $k''=1,\dots, \mu h/\lambda^2$; the former results in
$R \lambda^3 \times \mu h/\lambda^2$ and the latter results
$R \lambda^3 \times \mu h/\lambda \times 1/(\mu h)$   and one can see easily that the result is less than  what we  claimed.

\medskip\noindent
\emph{Zone \textup{(\ref{16-4-44})} as $\lambda^2 \ge \mu h $.\/}
Now there is always shift  with respect to $(x-y)$ which is at least
$\asymp |\sin (t')|$ and therefore as $1<\kappa <3$ contribution of this zone does not exceed contribution of zone (\ref{16-4-43}). As $0<\kappa<1$  we get
$C\mu ^{\kappa+1}h^{-2}\lambda^3 |k|^{-2}$ and summation returns
$C\mu ^{\kappa+1}h^{-2}\lambda^3$ which is the answer as $\lambda ^2\gg \mu h$; let us denote it by $R\lambda^3$; for $\kappa=1$ it includes logarithmic factor $(1+|\log (\mu h)|)$.

The same answer estimates contribution of $k'\ne 0, k''\ne 0$ as
$\lambda^2\ge \mu h $; contributions of $k'=0, k''\ne 0$ and $k''\ne 0, k'=0$ do not exceed $R\lambda^3 \times 1/(\mu h) \times (\mu h/\lambda^2)^l$  and one can see easily that the result is less than  what we  claimed.

\medskip\noindent
\emph{Zone \textup{(\ref{16-4-44})} as $\lambda^2 \le \mu h $.\/}
Again as $1<\kappa<3$ contribution of (\ref{16-4-44}) does not exceed contribution of (\ref{16-4-43}) and as $0<\kappa \le 1$ we estimate contribution of $k'\ne 0, k''\ne 0$ by $R\lambda^3 \times \mu h/\lambda^2$, and contributions of $k'\ne 0, k''= 0$ and  $k'= 0, k''\ne 0$ by
$R\lambda^3 \times \mu h/\lambda^2\times 1/(\mu h)$  and one can see easily that the result is less than  what we  claimed.
\end{proof}

\section{Perturbations}
\label{sect-16-8-3}

Consider perturbations. However due to the presence of $x_3$, $hD_3$ there are rather different components here than in subsection~\ref{sect-16-4-3}.

\subsection{Perturbations. I}
\label{sect-16-8-3-1}

Let us start from the terms $O(|x'-y'|)$. During evolution these terms are $O(\mu^{-1})$ and therefore we need to multiply by $O(\mu^{-2}h^{-1}|k|)$.

\paragraph{$1<\kappa<3$.}
\label{sect-16-8-3-1-1}
Consider first case $k'=k''=k$. Assume that $\lambda^2\ge \mu h$.
Recall that $k$-th tick' contribution to $\I^\T$ was not exceeding expression (\ref{16-8-7}) multiplied by $\mu^\kappa h^{-3}\lambda^3|k|^{-3}$  i.e.
\begin{equation}
C \lambda^3|k|^{-3}
\left\{\begin{aligned}
&\mu^{\kappa+1} h^{-2}\lambda^{1-\kappa} \qquad &&1<\kappa<2,\\
&\mu^3 h^{-2}(1+|\log (\hbar/\lambda^2)|)\lambda^{-1}\qquad &&\kappa=2,\\
&\mu^3 h^{-\kappa}\lambda^{-3+\kappa}\qquad &&2<\kappa<3
\end{aligned}\right.
\label{16-8-40}
\end{equation}
(as $|\sin(t')|\le \hbar$ or $|\sin(t'')|\le \hbar$ one should use proposition~\ref{prop-16-8-8}) and multiplying by $C\mu^{-2}h^{-1}|k|^{-1}$ we get after summation with respect to $k:|k|\ge 1$
\begin{equation}
C \lambda^3
\left\{\begin{aligned}
&\mu^{\kappa-1} h^{-3}\lambda^{1-\kappa} \qquad &&1<\kappa<2,\\
&\mu h^{-3}(1+|\log (\hbar/\lambda^2)|)\lambda^{-1}\qquad &&\kappa=2,\\
&\mu  h^{-1-\kappa}\lambda^{-3+\kappa}\qquad &&2<\kappa<3
\end{aligned}\right.
\label{16-8-41}
\end{equation}
and summation with respect to $\lambda$ returns the same expression as $\lambda=1$ which does not exceed $Ch^{-2-\kappa}$.

Next we need to consider  $0\ne k'\ne k''\ne 0$. Recall that the contribution of such pair to $\I^\T$ does not exceed
\begin{multline}
C\lambda^3|k'|^{-\frac{3}{2}} |k''|^{-\frac{3}{2}}
(\mu h/\lambda^2|k'-k''|)^l\times\\
\left\{\begin{aligned}
&\mu^{\kappa+1} h^{-2} \lambda^{1-\kappa} \qquad &&1<\kappa<2,\\
&\mu^3 h^{-2} (1+|\log (\mu h/\lambda^2)|)\lambda^{-1}\qquad &&\kappa=2,\\
&\mu^3 h^{-\kappa}\lambda^{-3+\kappa}\qquad &&2<\kappa<3.
\end{aligned}\right.
\label{16-8-42}
\end{multline}
Multiplying by $C\mu^{-2}h^{-1} |k'|$ (as the perturbation goes to the first ``factor'') we get
\begin{multline}
C\lambda^3|k'|^{-\frac{1}{2}} |k''|^{-\frac{3}{2}}
(\mu h/\lambda^2|k'-k''|)^l\times\\
\left\{\begin{aligned}
&\mu^{\kappa-1} h^{-3} \lambda^{1-\kappa} \qquad &&1<\kappa<2,\\
&\mu  h^{-3} (1+|\log (\mu h/\lambda^2)|)\lambda^{-1}\qquad &&\kappa=2,\\
&\mu  h^{-1-\kappa}\lambda^{-3+\kappa}\qquad &&2<\kappa<3.
\end{aligned}\right.
\label{16-8-43}
\end{multline}
Then summation with respect to $k'$ of
$|k'|^{-\frac{1}{2}}(\mu h/\lambda^2|k'-k''|)^l$ returns
$|k''|^{-\frac{1}{2}}(\mu h/\lambda^2)^l$ as $\lambda^2\ge \mu h$  and summation with respect to $k''$ then returns the same expression as we got considering $k'=k''\ne 0$ with an extra factor $(\mu h/\lambda^2)^l$.

On the other hand, as $\lambda^2 \le \mu h$ one should replace (\ref{16-8-7}) by
$C(\mu h)^{\frac{3}{2}-\frac{1}{2}\kappa}$; multiplying by
$\mu^\kappa h^{-3}\lambda^3|k'|^{-\frac{3}{2}}|k''|^{-\frac{3}{2}}
(1+\lambda^2|k'-k''|/\mu h)^{-l} \times \mu^{-2}h^{-1}|k'|$ we get
\begin{equation}
C\mu ^{-\frac{1}{2}+\frac{1}{2}\kappa}
h^{-\frac{5}{2}-\frac{1}{2}\kappa}
\lambda^3|k'|^{-\frac{1}{2}}|k''|^{-\frac{3}{2}}
(1+ \lambda^2|k'-k''|/\mu h )^{-l}
\label{16-8-44}
\end{equation}
and one can see easily that summation with respect to $k',k''$ returns
\begin{equation*}
C\mu ^{-\frac{1}{2}+\frac{1}{2}\kappa}
h^{-\frac{5}{2}-\frac{1}{2}\kappa}
\lambda^3(\mu h/\lambda^2)^{\frac{1}{2}}\asymp
C\mu ^{\frac{1}{2}\kappa} h^{-2-\frac{1}{2}\kappa}\lambda^2
\end{equation*}
and summation with respect to $\lambda$ returns
$C\mu ^{\frac{3}{2}\kappa} h^{-1-\frac{1}{2}\kappa}$ which again is less than $Ch^{-2-\kappa}$.

If $k'\ne 0$, $k''= 0$ we recall that $|k''|^{-\frac{3}{2}}$ should be replaced by $\hbar^{-\frac{3}{2}}$ in (\ref{16-8-42})--(\ref{16-8-43}) and perturbation brings factor $\mu^{-2}h^{-1} |k'|$ so we get instead of (\ref{16-8-43})
\begin{multline*}
C\lambda^3|k'|^{-\frac{1}{2}} \mu^{-\frac{3}{2}}h^{-\frac{3}{2}}
(\mu h/\lambda^2|k'|)^l\times\\
\left\{\begin{aligned}
&\mu^{\kappa-1} h^{-3} \lambda^{1-\kappa} \qquad &&1<\kappa<2,\\
&\mu   h^{-3} (1+|\log (\mu h/\lambda^2)|)\lambda^{-1}\qquad &&\kappa=2,\\
&\mu  h^{-1-\kappa}\lambda^{-3+\kappa}\qquad &&2<\kappa<3
\end{aligned}\right.
\end{multline*}
and summation with respect to $|k'|$ and then $\lambda\ge (\mu h)^{\frac{1}{2}}$ returns
$C\mu^{\frac{1}{2}\kappa-\frac{1}{2}} h^{-\frac{5}{2}-\frac{1}{2}\kappa} $.

Similarly, as $\lambda^2\le \mu h$ we start from (\ref{16-8-44}) which is replaced by
\begin{equation*}
C\mu ^{-2+\frac{1}{2}\kappa} h^{-4-\frac{1}{2}\kappa}
\lambda^3|k'|^{-\frac{1}{2}}
(1+ \lambda^2|k'|/\mu h )^{-l}
\end{equation*}
and summation with respect to $k'$ and returns
$C\mu ^{-\frac{3}{2}+\frac{1}{2}\kappa}h^{-\frac{7}{2}-\frac{1}{2}\kappa}
\lambda^2$ and then summation with respect to $\lambda\le (\mu h)^{\frac{1}{2}}$ returns the same answer
$C\mu^{\frac{1}{2}\kappa-\frac{1}{2}} h^{-\frac{5}{2}-\frac{1}{2}\kappa} $.

As $|k'|< |k''|$ the estimate of the contribution of the pair $(k',k'')$ to an error is obviously less than  the same estimate for the pair $(k'',k')$; this takes care of $k'=0$, $k''\ne 0$. Finally for $k'=k''=0$ we can use the standard arguments to estimate the same way.

As the result we conclude that the error will not exceed $Ch^{-2-\kappa}$.

\paragraph{$0<\kappa\le 1$.}
The similar arguments work for $0<\kappa<1$. However one must reconsider zone $\{(x,y):\ |x-y|\ge C_0\mu^{-1}\}$ where contribution of the pair $(k',k'')$  to Weyl asymptotics does not exceed
\begin{equation}
C\mu^2 h^{-2}\zeta^{1-\kappa} |k|^{-3}\lambda^3\asymp
C\mu^{\kappa+1} h^{-2} |k|^{-2-\kappa}\lambda^3
\label{16-8-45}
\end{equation}
with $\zeta\asymp \mu^{-1}|k|$ as $k'=k''=k\ne 0$. Summation with respect to $k:|k|\ge 1$ returns then $C\mu^{\kappa+1}h^{-2}\lambda^3$ as it should.

As $k'=k''=k\ne 0$ it results in the approximation error not exceeding
\begin{equation}
C\mu^2 h^{-2}\zeta^{1-\kappa} |k|^{-3}\lambda^3\times \mu^{-2}h^{-1}|k|
\asymp C\mu^{\kappa-1} h^{-3} |k|^{-1-\kappa}\lambda^3
\label{16-8-46}
\end{equation}
and summation with respect to $k,\lambda$ returns $C\mu^{\kappa-1} h^{-3}$.

As $0\ne k'\ne k'\ne 0$, $|k'|\ge |k''|$  it results in the approximation error not exceeding
\begin{multline}
C\mu^2 h^{-2}\bigl(\mu^{-1} |k'|)\bigr)^{1-\kappa} |k'|^{-\frac{3}{2}}|k''|^{-\frac{3}{2}}\times \\
(1+\lambda^2|k'-k''|/\mu h)^{-l} \lambda^3\times \mu^{-2}h^{-1}|k'|
\asymp \\
C\mu^{\kappa-1} h^{-3} |k'|^{\frac{1}{2}-\kappa}|k''|^{-\frac{3}{2}}(1+\lambda^2|k'-k''|/\mu h)^{-l} \lambda^3
\label{16-8-47}
\end{multline}
and the sum with respect to $(k',k'')$ and $\lambda:\lambda^2\ge \mu h$ does not exceed $C\mu^{\kappa-1} h^{-2}$.

Meanwhile as $\lambda^2\le \mu h$ the sum with respect to $(k',k'')$ does not exceed $C\mu^{\kappa-1} h^{-3} (\mu h/\lambda^2)^{\frac{3}{2}-\kappa}\lambda^3$ and the sum with respect to $\lambda$ is less than $C\mu^{\kappa-1} h^{-3}$.

As $k'\ne 0$, $k''=0$ (\ref{16-8-47}) becomes
\begin{equation}
C\mu^{\kappa-1} h^{-3} |k'|^{\frac{1}{2}-\kappa}(\mu h)^{-\frac{3}{2}} (1+\lambda^2|k'|/\mu h)^{-l} \lambda^3
\label{16-8-48}
\end{equation}
and then summation with respect to $k'$ returns
$C\mu^{\kappa-1} h^{-3} (\mu h)^{-\frac{3}{2}} (\mu h/\lambda^2)^l \lambda^3$ as $\lambda^2\ge \mu h$ and then the sum  with respect to $\lambda$ is less than $C\mu^{\kappa-1} h^{-3}$.

Meanwhile as $\lambda^2\le \mu h$ the sum with respect to $k'$ does not exceed
$C\mu^{\kappa-1} h^{-3} (\mu h/\lambda^2) ^{\frac{3}{2}-\kappa}
(\mu h)^{-\frac{3}{2}}  \lambda^3$ and then the sum with respect to $\lambda$ is less than $C\mu^{\kappa-1} h^{-3}$.

As $|k'|< |k''|$ the estimate of the contribution of the pair $(k',k'')$ to an error is obviously less than  the same estimate for the pair $(k'',k')$; this takes care of $k'=0$, $k''\ne 0$. Finally, for $k'=k''=0$ we can use the standard arguments to estimate contribution by $Ch^{-2-\kappa}$.

For $\kappa=1$ we apply the same arguments with the obvious modifications, concluding that the error does not exceed  $Ch^{-3}(1+|\log \mu h|)$.

Note that in contrast to the case $1<\kappa<3$ now the error  is larger than $Ch^{-2-\kappa}$.

To improve this remainder estimate let us notice that the second term in the approximation is $0$ and the third term also acquires factor  $\mu^{-2}h^{-1}|k|$ (as $k'=k''=k\ne 0$);  so we get
\begin{equation*}
C\mu^{\kappa-1} h^{-3} |k|^{-1-\kappa}\times \mu^{-2}h^{-1}|k|\lambda^3\asymp C\mu^{\kappa-3} h^{-4} |k|^{-\kappa}\lambda^3
\end{equation*}
which we should sum as $|k|\le \mu^2h$;   we also should sum
\begin{equation*}
C\mu^2 h^{-2}\zeta^{1-\kappa} |k|^{-3}\lambda^3 \asymp C\mu^{\kappa+1}h^{-2}|k|^{-2-\kappa}\lambda^3
\end{equation*}
as $|k|\ge \mu^2h$ and both sums return
$C\mu^{\kappa+1}h^{-2}(\mu^2h)^{-1-\kappa}\lambda^3 \asymp C\mu^{-1-\kappa}h^{-3-\kappa}\lambda^3$. Summation with respect to $\lambda$ returns $C\mu^{-1-\kappa}h^{-3-\kappa}$. One can estimate contribution of all all other pairs $(k',k'')$ in the same way.
As $\kappa=1$ we get $C\mu^{-2}h^{-4}(1+|\log \mu h|)$.

So we arrive to

\begin{theorem}\label{thm-16-8-10}(cf. theorem~\ref{thm-16-4-21}).
Let conditions \textup{(\ref{16-0-5})}--\textup{(\ref{16-0-6})} and either \textup{(\ref{16-0-7})} or \textup{(\ref{16-0-8})} be fulfilled.
Let $h^{-\frac{1}{2}}\mu \le h^{-1}$. Then

\medskip\noindent
(i) As $1<\kappa <3$
\begin{equation}
|\I-\I^0| \le Ch^{-2-\kappa}
\label{16-8-49}
\end{equation}
where
\begin{equation}
\I^0\Def\\
\int \Omega _{1,1} \bigl(\frac{1}{2}(x+y),x-y\bigr) \cdot
|e^0  _{\frac{1}{2}(x+y)}(x,y,0)|^2 \,dxdy
\label{16-8-50}
\end{equation}
and $e^0_z(x,y,\tau)$ is the Schwartz kernel of the spectral projector of the generalized pilot-model operator
\begin{equation}
A^0_z \Def h^2D_1^2 + (hD_2-\mu x_1)^2+h^2D_3^2+
V^0(x),\quad
V^0(x)\Def V(z_1,z_2,x_3)
\label{16-8-51}
\end{equation}
which in contrast to \textup{(\ref{16-6-36})} $V^0$ does not contain linear terms with respect to $(x_1,x_2)$;

\medskip\noindent
(ii) As $0<\kappa <1$
\begin{equation}
|\I-\I^0| \le Ch^{-2-\kappa}+C\mu^{-1-\kappa}h^{-3-\kappa}
\label{16-8-52}
\end{equation}
and as $\kappa=1$
\begin{equation}
|\I-\I^0| \le Ch^{-3} +C\mu^{-2}h^{-4}|\log \mu h|.
\label{16-8-53}
\end{equation}
\end{theorem}

\begin{remark}\label{rem-16-8-11}
Note that the remainder estimate does not exceed  $Ch^{-2-\kappa}$  for
$\mu \ge h^{-1/(1+\kappa)}$  as $0<\kappa<1$  and for
$\mu \ge h^{-\frac{1}{2}}(1+|\log \mu h|) ^{\frac{1}{2}}$ as $\kappa=1$.
\end{remark}

\subsection{Perturbations. II}
\label{sect-16-8-3-2}

If we go from the generalized pilot-model approximation to the magnetic Weyl approximation then we need to counter $(x_3-y_3)\lambda$ terms and this is $O(\mu^{-1}|k'|\lambda)$ due to the finite speed of propagation. Then dealing as before we acquire an extra factor $|k'|\lambda$.

\paragraph{$1<\kappa<3$.}
\label{sect-16-8-3-2-1}
Then as $k'=k''=k\ne 0$ we have  $ |k|^{-1}$ rather than $ |k|^{-2}$ and the sum with respect to $k$ contains an extra logarithmic factor $(1+\log (\mu^2h))\lambda$  in comparison to what we got in subsubsection~\ref{sect-16-8-3}.1~``\nameref{sect-16-8-3-1}''.

\medskip
Consider case $0\ne k'\ne k''\ne 0$, $|k'|\ge|k''|$. Then as $\lambda^2\ge \mu h$ we arrive to the same expression (but with an extra factor $(\mu h/\lambda^2)^l)$.
However as $\lambda^2\le \mu h$ we need to consider $|k'|^{-\frac{3}{2}}|k''|^{-\frac{3}{2}}\times |k'|^2\lambda$  with
$k',k'':|k'-k''|\le \mu h/\lambda^2$ which in comparison to what we got before after summation returns an extra factor $\mu h/\lambda^2\times \lambda$ and the sum with respect to $\lambda$ does not exceed what we got above.

\medskip
Consider $0\ne k'$, $k''= 0$. Again no change as $\lambda^2\ge \mu h$ and  an extra factor $(\mu h/\lambda^2)\times \lambda$ as $\lambda^2\le \mu h$ and the sum with respect to $\lambda$ does not exceed what we got above.

\medskip
Therefore, as a result we get $C\mu ^{\kappa-1}h^{-3}(1+|\log \mu^2h|)$ as $1<\kappa<2$
and $C\mu h^{-1-\kappa}(1+|\log \mu^2h|)$ as $2<\kappa<3$ with an extra factor $(1+|\log \mu h|)$  as $\kappa=2$. In both cases it is less than $Ch^{-2-\kappa}$ unless $\mu \ge h^{-1}|\log h|^{-r}$.

To improve our result in this latter case we apply three term approximation and note that the second term is $0$ while the third term contains summation  of
$R^\W\lambda^3 |k|^{-3} \times (\mu^{-2}h|k|^2\lambda)^2 $ for
$k:|k|\le \mu h^{\frac{1}{2}}$ and summation of
$R^\W\lambda^3 |k|^{-3}$ for $k:|k|\ge \mu h^{\frac{1}{2}}$ as $k'=k''=k\ne 0$
where  $R^\W=\mu^3h^{-\kappa}$ as $2<\kappa <3$,
$R^\W=\mu^{\kappa+1}h^{-2}$ as $1<\kappa <2$ with an extra factor
$(1+|\log \mu h|)$ as $\kappa=2$.

Other pairs $(k',k'')$ should be modified in the same way and we leave easy but tedious details to the reader. As a result we get $R^\W \mu^{-2}h^{-1}$ without any logarithmic factor thus arriving to the same estimate as before but without logarithmic factor; as $1<\kappa<3$ it does not exceed $Ch^{-2-\kappa}$. So we arrive to statement (i) of theorem~\ref{thm-16-8-12} below.

\paragraph{$0<\kappa\le 1$.}
\label{sect-16-8-3-2-2}

Now expression (\ref{16-8-46}) acquires factor $|k|\lambda$ and becomes
$C\mu^{\kappa-1}h^{-3}|k|^{-\kappa}\lambda^4$ and summation with respect to $k:|k|\le \mu h^{\frac{1}{2}}\lambda^{-\frac{1}{2}}$ returns
\begin{equation*}
C\mu^{\kappa-1}h^{-3}(\mu h^{\frac{1}{2}}\lambda^{-\frac{1}{2}})^{1-\kappa} \lambda^4 \asymp
C h^{-\frac{5}{2}-\frac{1}{2}\kappa}\lambda^{-\frac{7}{2}+\frac{1}{2}\kappa}.
\end{equation*}
We also should sum
$C\mu ^{\kappa+1}h^{-2} |k|^{-2-\kappa}\lambda^3$ for
$|k|\ge \mu h^{\frac{1}{2}}\lambda^{-\frac{1}{2}}$ with the same result.

Finally summation with respect to $\lambda$ returns
$C h^{-\frac{5}{2}-\frac{1}{2}\kappa}$.

Other pairs $(k',k'')$ should be modified in the same way and we leave easy but tedious details to the reader. Three term approximation does not improve results. As $\kappa=1$ we acquire factor $(1+|\log  h|)$. So we arrive to statement (ii) of theorem~\ref{thm-16-8-12} below.

\begin{theorem}\label{thm-16-8-12}(cf. theorem~\ref{thm-16-4-19}).
Let conditions \textup{(\ref{16-0-5})}--\textup{(\ref{16-0-6})} and either \textup{(\ref{16-0-7})} or \textup{(\ref{16-0-8})} be fulfilled.
Let $h^{-\frac{1}{2}}\le \mu \le h^{-1}$. Then

\medskip\noindent
(i) As $1<\kappa <3$
\begin{equation}
|\I-\cI^\MW| \le Ch^{-2-\kappa};
\label{16-8-54}
\end{equation}
(ii) As $0<\kappa <1$
\begin{equation}
|\I-\cI^\MW| \le Ch^{-\frac{5}{2}-\frac{1}{2}\kappa}
\label{16-8-55}
\end{equation}
and as $\kappa=1$
\begin{equation}
|\I-\cI^\MW| \le Ch^{-3} (1+|\log \mu h|).
\label{16-8-56}
\end{equation}
\end{theorem}

\subsection{Perturbations. III}
\label{sect-16-8-3-3}

Now in the case $0<\kappa\le 1$ we consider the pilot-model approximation. As (\ref{16-8-55}) is larger than (\ref{16-8-53}) we consider improvement of the previous subsubsection  only. Now as perturbation is $(x_3-y_3)^2$ which translates into $\mu^{-2}|k'|^2$ rather than $\mu^{-1}$ expression (\ref{16-8-46}) acquires factor $\mu^{-1}|k'|^2$ and becomes
$C\mu^{\kappa-2} h^{-3} |k|^{1-\kappa}\lambda^3$ and we need to sum this as $|k|\le \mu h^{\frac{1}{3}}$ resulting in
\begin{equation*}
C\mu^{\kappa-2} h^{-3} (\mu h^{\frac{1}{3}})^{2-\kappa}\lambda^3\asymp
C h^{-\frac{7}{3}-\frac{1}{3}\kappa} \lambda^3.
\end{equation*}
We should sum $C\mu ^{\kappa+1}h^{-2} |k|^{-2-\kappa}\lambda^3$  for
$|k|\ge \mu h^{\frac{1}{3}}$ with the same result.

Finally, summation with respect to $\lambda$ returns
$C h^{-\frac{7}{3}-\frac{1}{3}\kappa}$.

Other pairs $(k',k'')$ should be modified in the same way and we leave easy but tedious details to the reader. Three term approximation does not improve results. As $\kappa=1$ we acquire factor $(1+|\log  h|)$. So we arrive to

\begin{theorem}\label{thm-16-8-13}
Let conditions \textup{(\ref{16-0-5})}--\textup{(\ref{16-0-6})} and either \textup{(\ref{16-0-7})} or \textup{(\ref{16-0-8})} be fulfilled.
Let $h^{-\frac{1}{2}}\le \mu \le h^{-1}$. Then as $0<\kappa <1$
\begin{equation}
|\I- \bar{\I}| \le Ch^{-2-\kappa}+ C h^{-\frac{7}{3}-\frac{1}{3}\kappa}+
C\mu^{-1-\kappa}h^{-3-\kappa}
\label{16-8-57}
\end{equation}
and as $\kappa=1$
\begin{equation}
|\I-\bar{\I}| \le Ch^{-3}+
C\mu^{-2}h^{-4}(1+|\log \mu h|)
\label{16-8-58}
\end{equation}
where $\bar{\I}$ is defined for a simplified pilot-model operator
\begin{multline}
\bar{A}_z \Def h^2D_1^2 + (hD_2-\mu x_1)^2+h^2D_3^2+\bar{V}(x),\\
\bar{V}(x)\Def V(z) +\partial_{z_3}V(z) (x_3-z_3).
\label{16-8-59}
\end{multline}
\end{theorem}

\section{Superstrong magnetic field}
\label{sect-16-8-4}
Consider pilot-model Schr\"odinger-Pauli operator (\ref{16-1-72}). Recall that according to propositions~\ref{prop-16-7-4}(ii), \ref{prop-16-7-5}(ii)  respectively
\begin{equation}
|\I|\le C\mu h^{-2}
\left\{\begin{aligned}
&h^{-\kappa}\qquad&&\text{as\ \ }0<\kappa<1,\\
&h^{-1}\bigl(1+(\log \mu h)_+\bigr)\qquad&&\text{as\ \ }\kappa=1,\\
&\mu ^{\frac{1}{2} (\kappa-1)}h ^{-\frac{1}{2}(\kappa+1)}\qquad
&&\text{as\ \ }1<\kappa<3
\end{aligned}\right.
\tag{\ref*{16-7-2}}
\end{equation}
and
\begin{multline}
|\I-\I^\T|\le\\
C\mu h^{-1}
\left\{\begin{aligned}
& h^{-\kappa}\qquad &&\text{as\ \ } 0<\kappa<1,\\
& h^{-1}\bigl(1+  (\log \mu  h)_+\bigr) \qquad &&\text{as\ \ }\kappa=1,\\
& \mu^{\frac{1}{2} (\kappa-1)} h^{-\frac{1}{2}(\kappa+1)} \qquad
&&\text{as\ \ }1<\kappa<3, \ \kappa\ne 2,\\
& \mu^{\frac{1}{2} } h^{-\frac{3}{2}}\bigl(1+  (\log \mu  h)_+\bigr) \qquad
&&\text{as\ \ } \kappa=2.
\end{aligned}\right.
\tag{\ref*{16-7-3}}
\end{multline}
Let us assume that $F=1$ to avoid some unpleasant correction terms as reduced to $F=1$.

\medskip\noindent
(i) Consider the generalized pilot-model approximation (\ref{16-8-51}). Then $\sigma \asymp h$ and under condition
\begin{equation}
|V+2m \mu h |\ge \epsilon_0\qquad \forall m\in \bZ^+
\label{16-8-60}
\end{equation}
the shift with respect to $x_3$ is $\asymp |t|$ which implies that we need to consider only $t'\asymp t''$ as the total contribution of other intervals will be less,  and that contribution of intervals $|t'|\le C_0h$, $|t''|\le C_0h$ does not exceed the right-hand of (\ref{16-8-2}) multiplied by $Ch$, and that contribution of intervals
$t'\asymp T$, $t''\asymp T$ ($T\ge C_0h$) does not exceed $C \mu h^{-1} \times T^{-\kappa}$ so summation with respect to $T^{-1}dT$ returns
$C\mu h^{-1-\kappa}$.

Then the error does not exceed the right-hand expression of (\ref{16-8-3}).

\medskip\noindent
(ii) Consider the magnetic Weyl approximation. In this case $\sigma \asymp T$ and again under assumption (\ref{16-8-60}) we can consider only $t'\asymp t''$, and we estimate contribution of intervals $|t'|\le C_0h$, $|t''|\le C_0h$ as above, but contribution of intervals
$t'\asymp T$, $t''\asymp T$ ($T\ge C_0h$) does not exceed $C \mu h^{-2} \times T^{1-\kappa}$ and summation with respect to $T^{-1}dT$ returns
$C\mu h^{-1-\kappa}$ as $1<\kappa <3$. As $0<\kappa< 1$ we sum this way only for $T\le h^{\frac{1}{2}}$ resulting in
$C \mu  h^{-\frac{3}{2}-\frac{1}{2}\kappa}$ and we sum $C\mu h^{-1}T^{-\kappa}$ for  $T\ge h^{\frac{1}{2}}$  with the same result. As $\kappa=1$ we get
$C\mu h^{-2}(1+|\log h|)$.

\medskip\noindent
(iii) Consider the pilot-model approximation as $0<\kappa<1$.  In this case
$\sigma \asymp T^2$ (as we pass from the generalized pilot-model approximation) and again under assumption (\ref{16-8-60}) we can consider only $t'\asymp t''$, and we estimate contribution of intervals $|t'|\le C_0h$, $|t''|\le C_0h$ as above, but contribution of intervals $t'\asymp T$, $t''\asymp T$ ($T\ge C_0h$) does not exceed $C \mu h^{-2} \times T^{2-\kappa}$ summation with respect to $T^{-1}dT$ returns $C\mu  h^{-\frac{4}{3}-\frac{1}{3}\kappa}$ as we sum with respect to $T^{-1}dT$ for $T\le h^{\frac{1}{3}}$; for $T\ge h^{\frac{1}{3}}$ we sum $C\mu h^{-1}T^{-1-\kappa}$ with the same result.

Then we arrive to

\begin{theorem}\label{thm-16-8-14}
(i) For the Schr\"odinger-Pauli operator  under conditions \textup{(\ref{16-0-5})}--\textup{(\ref{16-0-6})} and \textup{(\ref{16-8-60})}
\begin{equation}
|\I - \cI^\MW|\le C
\left\{\begin{aligned}
&\mu  h^{-\frac{3}{2}-\frac{1}{2}\kappa} \qquad &&0<\kappa<1,\\
& \mu h^{-2}(1+|\log h|) \qquad && \kappa=1,\\
&\mu ^{\frac{1}{2}(\kappa+1)}h^{-\frac{1}{2}(\kappa+3)}\qquad &&1<\kappa<3, \ \kappa\ne 2,\\
& \mu ^{\frac{3}{2}}h^{-\frac{5}{2}}(1+(\log \mu h)_+)\qquad &&\kappa= 2.\\
\end{aligned}\right.
\label{16-8-61}
\end{equation}
(ii) Further, as $0<\kappa\le 1$
\begin{gather}
|\I-\I^0|\le C  \left\{\begin{aligned}
&\mu h^{-1-\kappa}\qquad &&0<\kappa<1,\\
&\mu h^{-2}(1+(\log \mu h)_+) &&\kappa=1
\end{aligned}\right.
\label{16-8-62}\\
\shortintertext{and}
|\I-\bar{\I}|\le  C  \left\{\begin{aligned}
&\mu h^{-1-\kappa}+ \mu  h^{-\frac{4}{3}-\frac{1}{3}\kappa}\qquad &&0<\kappa<1,\\
&\mu h^{-2}(1+|\log  h|) &&\kappa=1
\end{aligned}\right.
\label{16-8-63}
\end{gather}
\end{theorem}

Let us get rid off condition (\ref{16-8-60}). First, we still can assume without any loss of the generality that (\ref{16-8-60}) holds for all $m\in \bZ^+$ save $m=\bar{m}$. Then making $\ell$-admissible partitions in $x$ and $y$ with
\begin{equation}
\ell =\epsilon |V+2\bar{m}\mu h|+h^{\frac{2}{3}}
\label{16-8-64}
\end{equation}
one can replace (\ref{16-8-60}) by (\ref{16-0-7}) and then making $\ell$-admissible partitions in $x$ and $y$ with
\begin{equation}
\ell =\epsilon \bigl(|V+2\bar{m}\mu h|+|\nabla V|^2\bigr)^{\frac{1}{2}}+ h^{\frac{1}{2}}
\label{16-8-65}
\end{equation}
one can replace (\ref{16-0-7}) by (\ref{16-0-8}) arriving to

\begin{theorem}\label{thm-16-8-15}
Theorem \ref{thm-16-8-14} remains valid with assumption  \textup{(\ref{16-8-60})} replaced by  \textup{(\ref{16-0-8})}.
\end{theorem}

Easy but tedious details are left to the reader.

\section{Problems and remarks}
\label{sect-16-8-5}

\begin{remark}\label{rem-16-8-16}
The main difference between cases $2<\kappa<3$, $1<\kappa<2$ and $0<\kappa <1$   that in the first case the main contribution to the remainder is delivered by $(x,y)$ close to one another $(|x-y| \ll \mu^{-1}$), in the second case  by
$(x, y)$ with $|x − y | \asymp \mu^{-1}$ and in the third case by $(x, y)$ with $|x − y | \asymp 1$.
\end{remark}

We can calculate $\cI^\MW$ and $\bar{\I}$ plugging corresponding expressions $e^\MW_x(x,x,0)$ and $\bar{e}_x(x,x,0)$ into $\I$.

\begin{Problem}\label{problem-16-8-17}
Find nice expressions for $\cI^\MW$ and $\bar{\I}$.
\end{Problem}

\begin{Problem}\label{problem-16-8-18}
As $\mu h\le 1$ get rid off condition (\ref{16-0-5}); according to Chapter~\ref{book_new-sect-13} we do not need it for estimate $|\I-\I^\T|$ but we want to get rid off it in estimates for  $|\I-\cI^\W|$, $|\I-\cI^\MW|$ and $|\I-\bar{\I}|$. To do this

\medskip\noindent
(i) Assume first that condition (\ref{16-0-7}) is fulfilled and consider the scaling function
$\ell(x)= \max\bigl(\epsilon |V(x)|,\mu h, C\mu^{-1}\bigr)$.

\medskip\noindent
(ii) Assuming then that condition (\ref{16-0-8}) is fulfilled and consider the scaling function $\ell(x)= \max\bigl(\epsilon (|V(x)|+|\nabla V(x)|^2)^{\frac{1}{2}},(\mu h)^{\frac{1}{2}}, C\mu^{-1}\bigr)$.
\end{Problem}

\begin{Problem}\label{problem-16-8-19}
Weaken non-degeneracy condition (\ref{16-0-8}) which seems to be excessive here (in contrast to $2\D$ case).
\end{Problem}

\begin{Problem}\label{problem-16-8-20}
Get rid off condition (\ref{16-0-6}) assuming instead that
$|\mathbf{F}|+|\nabla \otimes \mathbf{F}|\ge \epsilon$.
\end{Problem}

\chapter{Estimates of expression (\ref{16-0-4})}
\label{sect-16-9}

\section{Basics}
\label{sect-16-9-1}

In the multiparticle quantum theory we need to calculate expression (\ref{16-0-2}) approximating energy of the electron-electron interaction or, more often, to estimate expression
\begin{multline}
\K\Def\\
\iint \bigl(e(x,x,\tau)-h^{-d}\cN_x (\tau)\bigr)
\bigl(e(y,y,\tau)-h^{-d}\cN_x  (\tau)\bigr)\omega(x,y)\,dxdy
\tag{\ref{16-0-4}}
\end{multline}
where $\omega (x,y)=\omega_\kappa(x,y)$ satisfies assumption (\ref{16-0-3}) and $\cN$ is \emph{some\/} approximation. We concentrate on (\ref{16-0-4}); asymptotics of (\ref{16-0-2}) could be proven by the same arguments and we leave them to the reader.

Applying partition of unity we note that both of them depend on asymptotics of expression
\begin{equation}
\Gamma (e\psi)=\int e(x,x,\tau)\psi_\gamma(x)\,dx
\label{16-9-1}
\end{equation}
with $\gamma$-admissible function $\psi_\gamma$.

Let us denote by $\gamma^d R(\alpha,\gamma)$ an upper estimate of an error occurring when we replace in (\ref{16-9-1}) $e(x,x,0)$ by $h^{-d}\cN$ and let
\begin{equation}
M = Ch^{-d}+ C\mu h^{1-d}
\label{16-9-2}
\end{equation}
be an upper estimate for $e(x,x,0)$.

\begin{proposition}\label{prop-16-9-1}
Under condition \textup{(\ref{16-0-7})}
\begin{gather}
|\K| \le C\int \bigl(R(1,\gamma)\bigr)^2 \gamma^{d-1-\kappa}\,d\gamma
\label{16-9-3}
\shortintertext{and}
|\J-\cJ|\le CM \int R(1,\gamma) \gamma^{d-1-\kappa}\,d\gamma
\label{16-9-4}
\end{gather}
and under condition \textup{(\ref{16-0-8})}
\begin{multline}
|\K| \le C\iint_{\{\gamma\le \alpha\}}
\bigl(R(\alpha,\gamma)\bigr)^2 \gamma^{d-1-\kappa}\alpha^{d-1}
\,d\gamma d\alpha +\\
C\iint_{\{\alpha\le \beta\}}  R(\alpha,\alpha) R(\beta,\beta)
\alpha^{d-1}\beta^{d-1-\kappa}\, d\alpha  d\beta
\label{16-9-5}
\end{multline}
and
\begin{multline}
|\J-\cJ| \le CM\iint_{\{\gamma\le \alpha\}}
R(\alpha,\gamma)  \gamma^{d-1-\kappa}\alpha^{d-1}
\,d\gamma d\alpha +\\
CM\iint_{\{\alpha\le \beta\}}  R(\alpha,\alpha)\alpha^{d-1}
\beta^{d-1-\kappa}\, d\alpha  d\beta
\label{16-9-6}
\end{multline}
where for $\alpha \le C_0\mu^{-1}$ we need to modify $\alpha^{-1}R(\alpha,\gamma)$ (see below)\footnote{\label{foot-16-22} As $d=3$ we denote by $\alpha$ what we used to denote by $\gamma$ in sections~\ref{sect-16-5}--\ref{sect-16-8} (i.e. $(\alpha^2+\beta^2)^{\frac{1}{2}}$ in these sections notations.}.
\end{proposition}

\begin{proof}
Due to representation (\ref{16-3-1}) and definition of $R(\alpha,\gamma)$ and $M$ contribution of the zone
$\{(x,y):\ |x-z|\asymp \gamma, |y-z|\asymp \gamma\}$
(note that $R(\alpha,\gamma)\le M$) to $\K$ does not exceed
$C(R(\alpha,\gamma))^2 \gamma^{2d-\kappa}$ for fixed $z$ and $\gamma$ where $\alpha =\alpha (z)$ provided $\gamma \le 3\epsilon \alpha$.

Under assumption (\ref{16-0-7}) summation over $z$ results in
$C(R(1,\gamma))^2 \gamma^{d-\kappa}$ and summation over $\gamma$ results in estimate (\ref{16-9-3}); estimate (\ref{16-9-4}) is proven in the same way.

On the other hand, under assumption (\ref{16-0-8}) summation over
$z: \alpha(z)\asymp \alpha$ results in
$C(R(\alpha,\gamma))^2 \alpha^d \gamma^{d-\kappa}$ and summation over $\gamma$, and then over $\alpha$ results  the first term in the right-hand expression of (\ref{16-9-5}); it estimates contribution of zone
$\{(x,y):\ |x-y|\le 3\epsilon \alpha\alpha(x)\}$ (in which $\alpha(y)\asymp\alpha(x)$).

Further, under assumption (\ref{16-0-8}) if $|x-y|\ge 2\epsilon \alpha $, $\alpha \Def \alpha(x)$, $\beta \Def \alpha(y)$  then  $|x-y|\asymp \max(\alpha,\beta)$. Therefore as $\beta \ge \alpha$ the contribution of zone \begin{equation*}
\{(x,y): |x-y|\asymp \beta, \alpha (y)\asymp \alpha, \alpha(y)\asymp \beta\}
\end{equation*}
does not exceed $CR(\alpha,\alpha) \alpha^d \times R(\beta,\beta)\beta^{d-\kappa}$ and summation over $\alpha$, $\beta$ returns the second term in the in the right-hand expression of (\ref{16-9-5}); it estimates contribution of zone
$\{(x,y):\ |x-y|\ge \epsilon \bigl(\alpha (x)+\alpha(y)\bigr) \}$.
\end{proof}

\section{Case $d=2$}
\label{sect-16-9-2}

\subsection{Tauberian asymptotics}
\label{sect-16-9-2-1}

(i) Consider first case $\mu \le h^{-1}$. Now according to (\ref{16-1-58}) we need just to plug
\begin{equation}
R(\alpha,\gamma)=\mu^{-1}h^{-1}+\alpha^{-1} R^\T(\alpha,\gamma)
\label{16-9-7}
\end{equation}
as $h^{-2}\cN_x =e^\T(x,x,0)$. Obviously the effect of $C\mu^{-1}h^{-1}$ term to both (\ref{16-9-3}) and (\ref{16-9-5}) will be term $C\mu^{-2}h^{-2}$.

Plugging into (\ref{16-9-3}) or (\ref{16-9-5})
$R(\alpha, \gamma)=\alpha^{-1} R^\T(\alpha,\gamma)$ with $R^\T(\alpha,\gamma)$  defined by (\ref{16-1-57}) we obviously eliminate from integration the last two zones (defined in (\ref{16-1-57})) as a power of $\gamma$ is negative~\footnote{\label{foot-16-23} Here and below talking about powers of $\alpha$ or $\beta$ or $\gamma$ we mean in the integral with respect to $\gamma^{-1}d\gamma$ or $\alpha^{-1}d\alpha$ or $\beta^{-1}d\beta$.}  and the first one as a power of $\gamma$ is positive. Thus $\gamma$ snaps to its minimal value in the last two zones and to its  maximal value in the first one, and only the second zone matters.

Further, in virtue of (\ref{16-1-57})
$R^\T(\alpha,\gamma)= C\mu^{\frac{1}{2}}\gamma^{-\frac{1}{2}}$  in the remaining second zone $\{h\le  \gamma \le \min(\mu h/\alpha,\mu^{-1})\}$.

Then the power of $\gamma$ is positive, $0$ or negative as $0<\kappa < 1$, $\kappa=1$ or $1<\kappa<2$ respectively. Thus  $\gamma$ snaps to
$\min(\mu h/\alpha,\mu^{-1})$ as $0<\kappa<1$ and  $\gamma$ snaps to $h$ as $1<\kappa <2$.

Then (\ref{16-9-3}) becomes
\begin{equation}
|\K^\T|\le C\mu^{-2}h^{-2}+ C\left\{\begin{aligned}
&\mu h^{1-\kappa} \qquad && 1<\kappa <2,\\
&\mu \min(|\log \mu h|, |\log \mu|) \qquad && \kappa=1,\\
&\mu \min  \bigl( (\mu h)^{1-\kappa}, \mu^{-1+\kappa}\bigr) \qquad
&& 0<\kappa <1.
\end{aligned}\right.
\label{16-9-8}
\end{equation}

Consider the first term  in the right-hand expression of estimate (\ref{16-9-5}). If $1<\kappa<2$ as we mentioned $\gamma$ snaps to $h$ thus resulting in the integrand $C\mu h^{1-\kappa}$  and in
$C\mu h^{1-\kappa}|\log h|$ after integration with respect to
$\alpha^{-1} d\alpha$.

Let $0<\kappa\le 1$. Then $\gamma$ snaps to $\gamma(\alpha)\Def \min(\mu h/\alpha,\mu^{-1})$, thus resulting in the integrand
$C\mu\min\bigl( (\mu h/\alpha)^{1-\kappa},\mu^{\kappa-1}\bigr)$ as $0<\kappa <1$ and in $C\mu \log \min \bigl(\mu /\alpha, (\mu h)^{-1})\bigr)$ as $\kappa=1$.

Then after integration with respect to $\alpha^{-1}d\alpha$ we get
$O(\mu^\kappa |\log h|)$ as $0<\kappa <1$  or $O(\mu^\kappa |\log h| \cdot |\log (\mu h)|)$ as $\kappa=1$  and this  is less than the estimate of the contribution of zone (\ref{16-9-9}) below unless $\mu \ge h^{-1}|\log h|^{-1}$.

However in (\ref{16-9-5}) we need also calculate the contribution of the zone
\begin{equation}
\{(x,y): \ \alpha(x) \le C_0\mu^{-1},\ \alpha(y)\le C_0\mu^{-1}\}
\label{16-9-9}
\end{equation}
which is not covered by our arguments and also the second term.

The former obviously does not exceed $C\mu^2 h^{-2}\times \mu^{-4+\kappa}= C\mu^{-2+\kappa}h^{-2}$ where the second factor is $ \int |x-y|^{-\kappa}dxdy$ over  zone (\ref{16-9-9}) and this also covers zone  $\{\beta\le C_0\mu^{-1}\}$ in the second term of (\ref{16-9-5}). Obviously $C\mu^{-2+\kappa}h^{-2}$ is larger than anything we got before.

Consider the second term in the right-hand expression of (\ref{16-9-5}). Here we have $R^\T (\alpha,\alpha)=\alpha^{-1}$ as
$(\mu h)^{\frac{1}{2}} \ge \alpha \ge \mu^{-1}$ and
$\mu^{-1}\alpha^{-1} h^{-1}
\min\bigl((\mu^2 h/\alpha)^{\frac{1}{2}},1\bigr)\times
(\mu h/\alpha)^l$ as $(\mu h)^{\frac{1}{2}} \le \alpha\le 1$ and therefore
$\beta $ is in the negative power and snaps to $\alpha$; so we arrive to
\begin{equation*}
\int \bigl(R^\T (\alpha,\alpha)\bigr)^2 \alpha^{1-\kappa}\,d\alpha.
\end{equation*}
In this integral $\alpha$ is in the negative power and snaps to $\mu^{-1}$ resulting in
$\bigl(R^\T (\mu^{-1},\mu^{-1})\bigr)^2 \mu^{\kappa-2}=O( \mu^\kappa)$ which is less than we have already.

Finally, we need to consider contribution of zone
\begin{equation}
\{(x,y):\ \alpha(x) \le C_0\mu^{-1}, \alpha(y)\ge C_0\mu^{-1}\}.
\label{16-9-10}
\end{equation}
So, we get
$C\mu^{-2}\times \mu h^{-1}\int R^\T (\beta,\beta)\beta^{1-\kappa}\,d\beta$;
in this integral  $\beta$ is in the negative power for sure as $1<\kappa<2$ and it snaps to $\beta=\mu^{-1}$ resulting in
$C \mu^{-1}h^{-1} R^\T (\mu^{-1},\mu^{-1})\mu^ {\kappa-2}$ which is less than we got already.

As $0<\kappa< 1$, $\beta$ is also in the negative power except zone
$\{\mu^{-1}\le \beta \le (\mu h)^{\frac{1}{2}}\}$ which is possible as
$\mu \ge h^{-\frac{1}{3}}$; then as $\beta$ snaps to $(\mu h)^{\frac{1}{2}}$ resulting in $C(\mu h) ^{-\frac{1}{2}-\frac{1}{2}\kappa}$ and one can see easily that it is less than what we got already. The same is true for $\kappa=1$ as well.

So we arrive to estimate
\begin{equation}
|\K^\T|\le C\mu^{-2+\kappa}h^{-2}+
C |\log h| \mu \bigl(\mu^{\kappa-1}+h^{1-\kappa}\bigr).
\label{16-9-11}
\end{equation}

\noindent
(ii) Let $\mu \ge h^{-1}$. Now we need just to plug $R(\alpha,\gamma)=\alpha^{-1} R^\T(\alpha,\gamma)\ge 1$ with $R^\T(\alpha,\gamma)$ delivered by (\ref{16-1-76}). Then in (\ref{16-9-3}) we need to snap $\gamma$ to $\mu^{-\frac{1}{2}}h^{\frac{1}{2}}$ and set
$R(\gamma)=\mu ^{\frac{1}{2}}h^{-\frac{1}{2}}$. So (\ref{16-9-3}) becomes
\begin{equation}
|\K^\T|\le C \mu ^{\frac{1}{2}\kappa} h^{-\frac{1}{2}\kappa}.
\label{16-9-12}
\end{equation}

Meanwhile, under assumption \ref{16-0-8-'}  in the first term of the right-hand expression of (\ref{16-9-5}) $\alpha$ is in power $0$ and logarithmic factor appears; we get $C \mu ^{\frac{1}{2}\kappa} h^{-\frac{1}{2}\kappa}|\log \mu|$.

Further, in (\ref{16-9-5}) we need to estimate also contribution of zones \begin{gather}
\{(x,y):\ \alpha(x) \le C_0\mu^{-\frac{1}{2}}h^{\frac{1}{2}}, \alpha(y) \le C_0\mu^{-\frac{1}{2}}h^{\frac{1}{2}}\},\label{16-9-13}\\
\{(x,y):\ \alpha(x) \le C_0\mu^{-\frac{1}{2}}h^{\frac{1}{2}}, \alpha(y) \ge C_0\mu^{-\frac{1}{2}}h^{\frac{1}{2}}\}\label{16-9-14}
\end{gather}
and of the second term in the right-hand expression and one can see easily that these contributions do not exceed
$C \mu ^{\frac{1}{2}\kappa} h^{-\frac{1}{2}\kappa}|\log \mu|$; and so we get
\begin{equation}
|\K^\T|\le C \mu ^{\frac{1}{2}\kappa} h^{-\frac{1}{2}\kappa}|\log \mu|.
\label{16-9-15}
\end{equation}

Thus we proved

\begin{proposition}\label{prop-16-9-2}
(i) For $\mu \le h^{-1}$ under condition \textup{(\ref{16-0-7})} estimate \textup{(\ref{16-9-8})} holds and under condition \textup{(\ref{16-0-8})} estimate \textup{(\ref{16-9-11})} holds;

\medskip\noindent
(ii) Let  $\mu \ge h^{-1}$. Then under condition \ref{16-0-7-'}    estimate \textup{(\ref{16-9-12})} holds and under condition \ref{16-0-8-'} estimate \textup{(\ref{16-9-15})} holds.
\end{proposition}

\begin{Problem}\label{problaem-16-9-13}
Prove that  estimates \textup{(\ref{16-9-8})} and \textup{(\ref{16-9-12})} hold under conditions \ref{16-0-8-+} and \ref{16-0-8-+'} respectively.
\end{Problem}

\subsection{Weyl approximation}
\label{sect-16-9-2-2}

Now we should add to what we got for $\K^\T$ corresponding expressions with
$R(\alpha,\gamma)=R^\W(\alpha,\gamma)$ calculated according to (\ref{16-1-80})--(\ref{16-1-81}) and therefore $\gamma$ must be snapped to
$\min (\alpha, \mu h \alpha^{-1})$. Then (\ref{16-9-3}) becomes
\begin{multline}
|\K^{\W}| \le C\mu^{-2}h^{-2}+ C (R^\W)^2 (\mu h)^{2-\kappa}=\\
C\mu^{-2}h^{-2}+
C\left\{\begin{aligned}
&\mu^4h^{-1} (\mu h)^{2-\kappa}
&&\text{as\ \ } \mu \le h^{-\frac{1}{2}},\\
&\mu^2 h^{-2} (\mu h)^{2-\kappa}
&&\text{as\ \ } h^{-\frac{1}{2}}\le \mu\le h^{-1}
\end{aligned}\right.
\label{16-9-16}
\end{multline}
which obviously is larger than the right-hand expressions of (\ref{16-9-8})  and so we can neglect the Tauberian error.

\medskip
Consider the first term in the right-hand expression of (\ref{16-9-5}). Note that for $\alpha \le (\mu h)^{\frac{1}{2}}$ $\gamma$ snaps to $\alpha$ and then $\alpha$ is in the positive power. Also note that as $\alpha \le \mu^2h$ (\ref{16-1-80}) defines $R^\W=\mu h^{-1}$ and $\alpha$ is in the positive power as well. Therefore as $h^{-\frac{1}{2}}\le \mu \le h^{-1}$ we snap $\alpha$ to $1$ resulting in $C\mu^2 h^{-2}(\mu h)^{2-\kappa}$ prescribed by (\ref{16-9-16}) and $\mu \le h^{-\frac{1}{2}}$ we need to consider interval
$\max( \mu^{-1},\mu^2 h)\le \alpha \le 1$; then automatically
$\alpha \ge (\mu h)^{\frac{1}{2}}$ and we arrive to
\begin{equation}
C\int \mu^4 h^{-1} (\mu h/\alpha)^{2-\kappa} d\alpha.
\label{16-9-17}
\end{equation}
Here as $1<\kappa<2$, $\alpha$ is in the positive power  and snaps to $1$ resulting in $C\mu^4 h^{-1}(\mu h)^{2-\kappa}$ prescribed by (\ref{16-9-16}).

On the other hand, as $0<\kappa<1$, $\alpha$ is in the negative degree and snaps to $\max(\mu^{-1}, \mu^2h)$ resulting in
$C\mu^4 h^{-1}(\mu h)^{2-\kappa} \mu ^{1-\kappa}$ for $\mu\le h^{-\frac{1}{3}}$ and $C\mu^4 h^{-1} (\mu h)^{2-\kappa} (\mu^2h)^{\kappa-1}$ for
$h^{-\frac{1}{3}} \le \mu\le h^{-\frac{1}{2}}$
and while the former expression is less than $C\mu^{-2+\kappa}h^{-2}$ the latter is not (at their respective intervals).

As $\kappa=1$ we get $C\mu^5 (1+|\log \mu^2h|)$.

\medskip
Consider the second term in the right-hand expression of (\ref{16-9-5}). Then both $\alpha\le \beta$ effectively are bounded from above by
$(\mu h)^{\frac{1}{2}}$ which leaves interval empty for
$\mu \le h^{-\frac{1}{3}}$. On the other hand then $\beta \le \mu^2 h$ for
$h^{-\frac{1}{3}}\le \mu \le h^{-1}$ and we arrive to
\begin{equation*}
C\mu^2 h^{-2} \iint_{\{\alpha\le \beta \le (\mu h)^{\frac{1}{2}}\} }
\alpha \beta^{1-\kappa}\,d\alpha d\beta \asymp
C\mu^2 h^{-2} (\mu h)^{2-\frac{1}{2}\kappa}
\end{equation*}
which is less than the right-hand expression of (\ref{16-9-16}) unless
$h^{-\frac{1}{3}}\le \mu \le h^{-\frac{1}{2}}$ and $0< \kappa<1$ in which case it is less than $C\mu^4 h^{-1} (\mu h)^{2-\kappa} (\mu^2h)^{\kappa-1}$.

Therefore we arrive to
\begin{multline}
|\K^{\W}| \le
C\mu^{-2+\kappa}h^{-2}+
C\left\{\begin{aligned}
&\mu^4h^{-1} (\mu h)^{2-\kappa}&&\text{as\ \ } \mu \le h^{-\frac{1}{2}},\\
&\mu^2 h^{-2} (\mu h)^{2-\kappa}&&\text{as\ \ }  h^{-\frac{1}{2}}\le \mu\le h^{-1}
\end{aligned}\right. \; +\\[3pt]
C\left\{\begin{aligned}
&\mu^4 h^{-1} (\mu h)^{2-\kappa} (\mu^2h)^{\kappa-1} \qquad
&&0<\kappa<1,\ & \mu \le h^{-\frac{1}{2}},\\
&\mu^5(1+|\log \mu^2h|) &&\kappa=1,\ &\mu \le h^{-\frac{1}{2}},\\
&0 \qquad &&\text{otherwise}.
\end{aligned}\right.
\label{16-9-18}
\end{multline}

Thus we proved

\begin{proposition}\label{prop-16-9-4}
For $\mu \le h^{-1}$ under condition \textup{(\ref{16-0-7})} estimate \textup{(\ref{16-9-16})}  holds and under condition \textup{(\ref{16-0-8})} estimate \textup{(\ref{16-9-18})} holds.
\end{proposition}

\subsection{Magnetic Weyl approximation}
\label{sect-16-9-2-3}
Finally, let us deal with magnetic Weyl approximation. Now we should add to what we got for $\K^\T$ corresponding expressions with
$R(\alpha,\gamma)=R^\MW(\alpha,\gamma)$ calculated according to (\ref{16-1-83}) for $\mu \le h^{-1}$ and (\ref{16-1-85}) for $\mu \ge h^{-1}$.

\medskip\noindent
(i) Consider case $\mu \le h^{-1}$ first. Then according to (\ref{16-1-83}) $\gamma$ snaps to $\min (\mu^{-1}, \mu h \alpha^{-1})$ and
$R(\alpha,\gamma)=\mu h^{-1}$  and (\ref{16-9-3}) becomes
\begin{equation}
|\K^\MW|\le
C\mu^{-2}h^{-2}+
C\mu^2 h^{-2}\left\{\begin{aligned}
&(\mu h)^{2-\kappa}  \qquad &&\mu \le h^{-\frac{1}{2}},\\
&\mu^{\kappa-2} \qquad && h^{-\frac{1}{2}}\le \mu \le h^{-1}.
\end{aligned}\right.
\label{16-9-19}
\end{equation}

Meanwhile in the first term of (\ref{16-9-5}) $\alpha$ snaps to its largest possible value i.e $1$. Further, in the second term $\alpha$ snaps to its largest value which is $\beta$ and then $\beta$ snaps to its largest value $(\mu h)^{\frac{1}{2}}$ as $h^{-\frac{1}{3}}\le \mu \le h^{-1}$ and it becomes
$C\mu^2 h^{-2}(\mu h)^{2-\frac{1}{2}\kappa}$.

Furthermore, contributions of zones (\ref{16-9-9}) and (\ref{16-9-10}) are estimated as before; we arrive in the end to
\begin{multline}
|\K^\MW|\le
C\mu^{-2+\kappa}h^{-2}+ C\mu^2 h^{-2}
\left\{\begin{aligned}
&(\mu h)^{2-\kappa}  \quad &&\mu \le h^{-\frac{1}{2}},\\[3pt]
&\mu^{\kappa-2} \quad && h^{-\frac{1}{2}}\le \mu \le h^{-1}
\end{aligned}\right. \ +\\
C\left\{\begin{aligned}
& \mu^2 h^{-2}(\mu h)^{2-\frac{1}{2}\kappa} \qquad
&&h^{-\frac{1}{3}}\le \mu \le h^{-1},\\
&0 &&\text{otherwise}.
\end{aligned}\right.
\label{16-9-20}
\end{multline}

\noindent
(ii) As $\mu h\ge 1$ $\gamma$ snaps to $\mu^{-\frac{1}{2}}h^{\frac{1}{2}}$ and (\ref{16-9-3}) becomes
\begin{equation}
|\K^\MW|\le
C\mu^{1+\frac{1}{2}\kappa } h^{-1+\frac{1}{2}\kappa}
\label{16-9-21}
\end{equation}

Consider the first term in expression (\ref{16-3-5}); integrand contains $\alpha$ in the positive power and therefore $\alpha$ snaps to $1$. Meanwhile in the second term both $\alpha\ge C_0\mu^{-\frac{1}{2}}h^{\frac{1}{2}}$ and $\beta\ge C_0\mu^{-\frac{1}{2}}h^{\frac{1}{2}}$ are in the negative powers and snap to $\mu^{-\frac{1}{2}}h^{\frac{1}{2}}$. Estimating as before contributions of zones (\ref{16-9-13}) and (\ref{16-9-14}) we arrive to estimate (\ref{16-9-21}) again.

All these above estimates are worse than for $\K^\T$; thus we arrive to

\begin{proposition}\label{prop-16-9-5}
(i) For $\mu \le h^{-1}$ under condition \textup{(\ref{16-0-7})} estimate \textup{(\ref{16-9-19})} holds and under condition \textup{(\ref{16-0-8})} estimate \textup{(\ref{16-9-20})} holds;

\medskip\noindent
(ii) For $\mu \ge h^{-1}$ under either condition \ref{16-0-7-'}, \ref{16-0-8-'}   estimate \textup{(\ref{16-9-21})} holds.
\end{proposition}

\section{Case $d=3$}
\label{sect-16-9-3}

We consider case $d=3$. Note that

\begin{remark}\label{rem-16-9-6}
(i) We use notations in (\ref{16-9-5}), (\ref{16-9-6}): $\lambda$ instead of $\alpha$ and $\lambda'$ instead of $\beta$ (in the second term only);

\medskip\noindent
(ii) We cannot expect estimate better than $O(h^{-4})$ for $\mu h \le 1$. Meanwhile under condition (\ref{16-0-8}) contribution of zone
\begin{equation}
\{(x,y): \ \lambda(x) \le \bar{\lambda},\ \lambda(y)\le \bar{\lambda}\}
\label{16-9-22}
\end{equation}
with $\bar{\lambda}=C_0\mu^{-1}$ to any estimate does not exceed
$C\mu^2 h^{-4} \times \mu^{-6+\kappa}=O(h^{-4})$. We can even take  $\bar{\lambda}=\mu^{-2/(6-\kappa)}$ here.
\end{remark}

\subsection{Tauberian estimates}
\label{sect-16-9-3-1}

We will use results of subsection~\ref{-16-5-4} combined with arguments of subsubsection~\ref{sect-16-6-8-1} always setting $\gamma_3=\gamma$.

\paragraph{Case $\mu h\le 1$}
\label{sect-16-9-3-1-1}

As $\gamma_3=\gamma$ we have $\nu(\boldgamma)\asymp \lambda \gamma$,
$\ell \asymp \lambda$ and (\ref{16-5-25}) becomes
\begin{equation}
R^\T(\gamma)=
\left\{\begin{aligned}
&\mu  h^{-\frac{3}{2}}  \lambda ^{\frac{1}{2}}
\qquad && \text{as\ \ }
\lambda^2\gamma  \le h,\\
&\mu  h^{-1}\lambda^{-\frac{1}{2}}\gamma^{-\frac{1}{2}}
\qquad && \text{as\ \ }
\lambda^2\gamma \ge h, \lambda\gamma\le \mu h,\\
&\mu^ {\frac{1}{2}}h^{-\frac{3}{2}}(\mu h/\lambda\gamma)^l\qquad
&&\text{as\ \ } \lambda\gamma\ge \mu h.
\end{aligned}\right.
\label{16-9-23}
\end{equation}
To estimate   $|\K^\T |$ we need to calculate
\begin{equation}
K(\lambda) \Def \int\bigl(R^\T(\gamma)\bigr)^2\gamma^{2-\kappa}\,d\gamma
\label{16-9-24}
\end{equation}
and as
$\lambda \ge \mu^{-1}$ it is equal to
$C\mu^2 h^{-2}\lambda^{-1}\gamma^{2-\kappa}$ calculated for
$\gamma=\min(\mu h/\lambda,\lambda)$ as $0<\kappa<2$ and for
$\gamma= h/\lambda^2$ as $2<\kappa<3$  resulting in
\begin{multline}
K(\lambda) \le  Ch^{-4}\lambda^{3-\kappa}+\\[3pt]
C\left\{\begin{aligned}
&C\mu^2 h^{-2}\lambda^{-1}(h/\lambda^2)^{2-\kappa} &&2<\kappa<3,\\
&C\mu^2 h^{-2}\lambda^{-1}(1+|\log \lambda\mu|) \qquad &&\kappa=2,\\
&\mu^2 h^{-2}\lambda^{-1}\min(\mu h/\lambda,\lambda)^{2-\kappa}
&&0<\kappa<2.
\end{aligned}\right.
\label{16-9-25}
\end{multline}
In particular, we conclude that

\begin{claim}\label{16-9-26}
As $\mu h\le 1$, $|\nabla V/F|\asymp 1$
\begin{equation}
|\K^\T |\le  Ch^{-4} +
C\left\{\begin{aligned}
&\mu^2 h^{-\kappa} \qquad &&2<\kappa<3,\\
&\mu^2h^{-2}|\log \mu| \qquad &&\kappa=2,\\
&0 \qquad &&0<\kappa<2.
\end{aligned}\right.
\label{16-9-27}
\end{equation}
\end{claim}

On the other hand,  under assumption (\ref{16-0-8}) we need to calculate the first term in the right-hand expression of (\ref{16-9-5}) which is equal to $\int K(\lambda)\lambda^2\,d\lambda$ and one can see easily that it is (\ref{16-9-25}) calculated as $\lambda=1$.

Finally, one needs to calculate the second term in the right-hand expression of (\ref{16-9-5}) and we leave to the reader to prove that it allows the same estimate. Thus we arrive to statement (i) of proposition~\ref{prop-16-9-7} below.

On the other hand, assume that assumption (\ref{16-6-57}) is fulfilled i.e.
$|\nabla _{\perp \mathbf{F}}V/F|\asymp 1$. Then (\ref{16-5-25}) becomes
\begin{multline}
R^\T(\boldgamma)= \\
\left\{\begin{aligned}
& \mu^{\frac{3}{2}} h^{-1}
\ \bigl(1+ |\log (\mu h)|\bigr)
\qquad && \text{as\ \ }
\gamma \le h,\\
& \mu^{\frac{3}{2}} h^{-1}
\bigl(1+ (\log (1/\mu \gamma))_+\bigr)
\qquad && \text{as\ \ }
h \le \gamma \le \min(\mu h ,  \mu^{-1}),\\
&\mu  h^{-1}\gamma^{-\frac{1}{2}}
\qquad && \text{as\ \ }
\mu^{-1} \le \gamma \le \mu h,\\
&\mu^ {\frac{1}{2}}h^{-\frac{3}{2}}(\mu h/\gamma))^l\qquad
&&\text{as\ \ } \gamma \ge \mu h.
\end{aligned}\right.
\label{16-9-28}
\end{multline}
and in this case instead of $\gamma=h$ the lower pick is $\gamma=\mu^{-1}$ in $C\mu^2 h^{-2} \gamma^{2-\kappa}$ which becomes
$C h^{-2}  \mu^\kappa $ and we arrive to statement (ii) of proposition~\ref{prop-16-9-7}:

\begin{proposition}\label{prop-16-9-7}
Let $\mu h\le 1$. Then

\medskip\noindent
(i)  Under either non-degeneracy condition either \textup{(\ref{16-0-7})} or \textup{(\ref{16-0-8})} estimate \textup{(\ref{16-9-27})} holds.

In particular, $|\K^\T|=O(h^{-4})$ as \underline{either}  \ $0< \kappa<2$
\underline{or} \ $\kappa=2$ and $\mu \le (h |\log h|)^{-1}$ \underline{or}
$2<\kappa<3$ and $\mu \le h^{-2+\frac{1}{2}\kappa}$ (and this is true for sure as $\mu\le h^{-\frac{1}{2}}$).

\medskip\noindent
(ii)  Under non-degeneracy condition \textup{(\ref{16-6-57})}, $2< \kappa <3$
\begin{equation}
|\K^\T |\le  Ch^{-4} +
Ch^{-2} \mu^{\kappa}.
\label{16-9-29}
\end{equation}

In particular,  $|\K^\T|=O(h^{-4})$ as $\mu \le h^{-2/\kappa}$ (and this is true for sure as $\mu\le h^{-\frac{2}{3}}$).
\end{proposition}

\paragraph{Case $\mu h\ge 1$}
\label{sect-16-9-3-1-2}

\begin{proposition}\label{prop-16-9-8}
Let $\mu h\ge 1$. Then under either non-degeneracy condition \ref{16-0-7-'} or \ref{16-0-8-'} estimate
\begin{equation}
|\K^\T |\le  C\mu^2 h^{-2} +
C\left\{\begin{aligned}
&\mu^2 h^{-\kappa} \qquad &&2<\kappa<3,\\
&\mu^2h^{-2}|\log \mu| \qquad &&\kappa=2,\\
&0 \qquad &&0<\kappa<2.
\end{aligned}\right.
\label{16-9-30}
\end{equation}
holds.
\end{proposition}

\begin{proof}
An easy proof based on proposition \ref{prop-16-5-9} is left to the reader.
\end{proof}

\subsection{Weyl asymptotics}
\label{sect-16-9-3-2}

\begin{proposition}\label{prop-16-9-9}
Let $\mu h\le 1$. Then

\medskip\noindent
(i) Under non-degeneracy condition \textup{(\ref{16-0-7})}
\begin{equation}
|\K^\W|\le Ch^{-4} + C \mu^3 h^{-3} (\mu h)^{3-\kappa};
\label{16-9-31}
\end{equation}
(ii) Under non-degeneracy condition \textup{(\ref{16-0-8})}
\begin{equation}
|\K^\W|\le Ch^{-4} + C \mu^3 h^{-3} (\mu h)^{\frac{1}{2}(3-\kappa)};
\label{16-9-32}
\end{equation}
\medskip\noindent
(iii)  As $\mu \le h^{-\frac{1}{2}}$ under non-degeneracy condition \textup{(\ref{16-6-57})}
\begin{equation}
|\K^\W|\le Ch^{-4} + C \mu^5 h^{-2} (\mu h)^{3-\kappa}.
\label{16-9-33}
\end{equation}
\end{proposition}

\begin{proof}
An easy proof based on proposition~\ref{prop-16-5-6} is left to the reader.
\end{proof}

\subsection{Magnetic Weyl asymptotics}
\label{sect-16-9-3-3}

We will apply proposition~\ref{prop-16-5-14}. In the isotropic settings  $R^\MW(\boldgamma)$  given by (\ref{16-5-47}) becomes
\begin{multline}
R^\MW (\gamma)\Def\\
C\mu \left\{\begin{aligned}
&  \lambda^{-\frac{1}{2}}\gamma^{-\frac{5}{2}}
\bigl(\mu h/\lambda\gamma\bigr)^l\qquad
&&\text{as\ \ } \gamma\ge \mu h/\lambda,\\
&\lambda^{-\frac{1}{2}}\gamma^{-\frac{5}{2}}
\qquad&&\text{as\ \ } \lambda^{-\frac{1}{3}}h^{\frac{2}{3}}\le
\gamma \le \mu h/\lambda, \\
&\lambda^{\frac{1}{3}}h^{-\frac{5}{3}}
\qquad&&\text{as\ \ } \gamma\le \lambda^{-\frac{1}{3}}h^{\frac{2}{3}},
\end{aligned}\right.
\label{16-9-34}
\end{multline}
which immediately implies statements (i), (ii)  of proposition~\ref{prop-16-9-10} below.

Meanwhile as $\alpha\asymp 1$, $\mu h\le 1$ (\ref{16-5-47}) becomes
\begin{multline}
R_1^\MW (\gamma)\Def\\
C\mu  \left\{\begin{aligned}
& \gamma ^{-\frac{5}{2}} (\mu h/\gamma)^l\qquad
&&\text{as\ \ } \gamma\ge \mu h,\ 1 \le \mu^2h,\\
&  \gamma^{-\frac{5}{2}}
\qquad&&\text{as\ \ } \gamma \le \mu h,\ \mu \gamma \ge 1,\\
&\gamma ^{-\frac{5}{2}}(\mu^2h)^{\frac{1}{2}}
\bigl(\mu h/\gamma\bigr)^l\qquad
&&\text{as\ \ } \gamma\ge \mu h,\ 1 \ge \mu^2h,\\
& \gamma ^{-\frac{5}{2}}(\mu \gamma )^{\frac{1}{2}}
\qquad&&\text{as\ \ } \gamma\le \mu h,\ \mu \gamma \le 1
\end{aligned}\right.
\label{16-9-35}
\end{multline}
which immediately imply statement  (iii)   of proposition~\ref{prop-16-9-10} below.

\begin{proposition}\label{prop-16-9-10}
(i) Let $\mu h\le 1$. Then  under either non-degeneracy condition   \textup{(\ref{16-0-7})} or \textup{(\ref{16-0-8})}
\begin{equation}
|\K^\MW |\le C h^{-4}+  C\mu^2h^{-\frac{2}{3}(2+\kappa)};
\label{16-9-36}
\end{equation}
(ii) Let $\mu h\ge 1$. Then under either non-degeneracy condition\ref{16-0-7-'} or \ref{16-0-8-'}
\begin{equation}
|\K^\MW |\le C \mu^2 h^{-2}+  C\mu^2h^{-\frac{2}{3}(2+\kappa)};
\label{16-9-37}
\end{equation}
(iii) Let $\mu \le h^{-\frac{2}{3}}$. Then  under non-degeneracy condition   \textup{(\ref{16-6-57})}
\begin{equation}
|\K^\MW |\le C h^{-4}+  C\mu^3 h^{-\frac{2}{3}(1+\kappa)}.
\label{16-9-38}
\end{equation}
\end{proposition}

\section{Problems}
\label{sect-16-9-4}

\begin{problem}\label{problem-16-9-11}
For $d=2,3$ estimate $|\J-\J^\T|$, $|\J-\cJ^\W|$ and $|\J-\cJ^\MW|$.
\end{problem}

\begin{Problem}\label{problem-16-9-12}
For $d=2,3$ consider pilot-model approximation.
\end{Problem}

\begin{problem}\label{problem-16-9-13}
For $d=2,3$ compare estimates for $|\K^\W|$ and $|\J-\cJ^\W|$ to those for
$|\K^\MW|$ and $|\J-\cJ^\MW|$.
\end{problem}

\begin{Problem}\label{problem-16-9-14}
For $d=2,3$ as $\mu \le h^{-1}$ get rid off condition (\ref{16-0-5}) replacing it by (\ref{16-0-7}), or even \ref{16-0-8-+}, or even  (\ref{16-0-8}).
\end{Problem}

\begin{Problem}\label{problem-16-9-15}
Get rid off condition (\ref{16-0-6}) assuming instead that
$|F|+|\nabla F|\ge \epsilon$ for $d=2$ (which will affect estimates for sure) and  $|\mathbf{F}|+|\nabla \otimes \mathbf{F}|\ge \epsilon$ for $d=3$ (which most likely would not affect estimates)
\end{Problem}

\begin{Problem}\label{problem-16-9-16}
For $d=3$ investigate $\K^\T$, $\K^\W$, $\K^\MW$ under non-degeneracy conditions
(\ref{16-0-7})  and  $|\nabla_{\perp \mathbf{F}} V/F |\le \epsilon \implies
|\det \Hess_{\perp \mathbf{F}} V/F |\asymp 1$.
\end{Problem}

\bibliographystyle{alpha}

\providecommand{\bysame}{\leavevmode\hbox to3em{\hrulefill}\thinspace}

\vglue .06truein

\begin{tabular}{rrl}
&{\hskip 200 pt} &Department of Mathematics,\cr
&&University of Toronto,\cr
&&40, St.George Str.,\cr
&&Toronto, Ontario M5S 2E4\cr
&&Canada\cr
&&ivrii@math.toronto.edu\cr
&&Fax: (416)978-4107\cr
\end{tabular}

\end{document}